\documentclass[a4paper,
pagesize,
oneside, 				
numbers=noenddot,
DIV=12, 
headinclude=false, 
bibtotoc, 
fontsize=10pt, 
headsepline=true, 
open=right,
abstracton,
cleardoublepage=empty, 
footinclude=false,
english
]{scrartcl}

\usepackage{setspace}
\usepackage{standalone}

\usepackage{authblk}

\usepackage{mathtools}

\usepackage{stmaryrd}

\usepackage{pdflscape}

\usepackage{subfig}

\usepackage{pst-pdf}

\usepackage{url}

\usepackage{upgreek}

\usepackage{sidecap}

\usepackage{xcolor}
\definecolor{darkred}{rgb}{0.9,0.1,0.1}
\definecolor{darkblue}{rgb}{0,0,0.7}
\definecolor{darkgreen}{rgb}{0,0.5,0}
\definecolor{bluegray}{rgb}{0.4, 0.6, 0.8}
\definecolor{cadmiumorange}{rgb}{0.93, 0.53, 0.18}
\definecolor{darkcerulean}{rgb}{0.03, 0.27, 0.49}
\usepackage[pagebackref, colorlinks=false, linkcolor=darkblue, linktocpage=true, linkbordercolor=red, citecolor=darkblue]{hyperref} %

\usepackage[english]{babel}
\usepackage[utf8]{inputenc}
\usepackage[T1]{fontenc} 
\usepackage{lmodern} 
\usepackage{amsmath, amsfonts, amssymb}
\usepackage{mathtools}
\usepackage{mathrsfs}
\usepackage{verbatim}
\usepackage{stmaryrd}
\usepackage[framed,amsmath,thmmarks,hyperref]{ntheorem}
\usepackage{commath}

\usepackage[labelfont=bf]{caption}

\usepackage{longtable}

\allowdisplaybreaks

\usepackage[english=american]{csquotes}

\usepackage{graphicx}

\usepackage{comment}

\usepackage{framed}

\usepackage[]{listofsymbols}

\setkomafont{sectioning}{\bfseries} 

\setkomafont{pagenumber}{\bfseries}
\setkomafont{pageheadfoot}{\normalfont}
\setkomafont{pagenumber}{\normalfont}

\usepackage{enumitem} 

\usepackage{ifthen}

\usepackage{tikz,pgfplots}
\pgfplotsset{compat=newest}
\usetikzlibrary{cd}
\usetikzlibrary{calc}
\usetikzlibrary{arrows, scopes}
\usetikzlibrary{decorations.pathmorphing}
\usetikzlibrary{positioning}
\usetikzlibrary{shapes.geometric}
\usetikzlibrary{backgrounds}
\usetikzlibrary{arrows,chains,matrix,positioning,scopes}

\usepackage[multiple,hang,flushmargin]{footmisc}

\usepackage{graphicx}

\usepackage[colorinlistoftodos, disable]{todonotes}  
\setuptodonotes{inline}

\usepackage{marginnote}

\usepackage[framemethod=TikZ]{mdframed}

\mdfdefinestyle{temporary}{linewidth=3pt, linecolor=red}
\mdfdefinestyle{onlyframe}{middlelinecolor=black!100,roundcorner=1ex, skipbelow=15pt}
\mdfdefinestyle{theorem}{middlelinecolor=black!25, backgroundcolor=black!15,roundcorner=1ex, skipbelow=15pt}
\mdfdefinestyle{onlyback}{middlelinecolor=white!100, backgroundcolor=black!10,roundcorner=1ex, skipbelow=15pt}
\mdfdefinestyle{main}{middlelinecolor=darkblue!25, backgroundcolor=darkblue!15,roundcorner=1ex, skipbelow=15pt}

\usepackage{commath}

\usepackage{enumitem}

\usepackage{nicefrac}

\usepackage{rotating}

\usepackage{relsize}

\usepackage{ctable}
\usepackage{multicol}
\usepackage{multirow}

\usepackage{accents}

\usepackage{ucs}

\usepackage{fixltx2e} 
\usepackage{lmodern} 

\usepackage{setspace,graphicx,tikz,tabularx} 
\usepackage[draft=false,babel,tracking=true,kerning=true,spacing=true]{microtype}

\usepackage{private}

\DeclareSymbolFont{extraup}{U}{zavm}{m}{n}
\DeclareMathSymbol{\varclub}{\mathalpha}{extraup}{84} 
\DeclareMathSymbol{\varspade}{\mathalpha}{extraup}{85}

\usepackage{tikz,pgfplots}
\pgfplotsset{width=\textwidth,height=5cm}
\usepackage{mhequ}

\usepackage{mathrsfs}

\usepackage{multirow}

\usepackage[framemethod=TikZ]{mdframed}

\mdfdefinestyle{temporary}{linewidth=3pt, linecolor=red}
\mdfdefinestyle{theorem}{middlelinecolor=black!25, backgroundcolor=black!15,roundcorner=1ex, skipbelow=15pt}
\mdfdefinestyle{main}{middlelinecolor=darkblue!25, backgroundcolor=darkblue!15,roundcorner=1ex, skipbelow=15pt}

\def\CB{\mathcal{B}}
\def\CC{\mathcal{C}}

\def\CE{\mathcal{E}}

\def\CH{\mathcal{H}}
\def\CI{\mathcal{I}}

\def\CL{\mathcal{L}}
\def\CM{\mathcal{M}}

\def\CS{\mathcal{S}}
\def\CT{\mathcal{T}}

\def\CW{\mathcal{W}}

\def\EE{\mathscr{E}}

\def\II{\mathscr{I}}
\def\JJ{\mathscr{J}}

\def\MM{\boldsymbol{\CM}}

\def\VV{\mathscr{V}}

\def\fy{\mathfrak{y}}
\def\fz{\mathfrak{z}}

\newcommand{\mbE}{\mathbb{E}}

\newcommand{\mbI}{\mathbb{I}}

\newcommand{\mbN}{\mathbb{N}}

\newcommand{\mbR}{\mathbb{R}}

\newcommand{\mbT}{\mathbb{T}}

\newcommand{\mbZ}{\mathbb{Z}}

\newcommand{\msA}{\mathscr{A}}

\newcommand{\msI}{\mathscr{I}}

\newcommand{\msL}{\mathscr{L}}

\newcommand{\msS}{\mathscr{S}}

\newcommand{\msV}{\mathscr{V}}

\newcommand{\mcB}{\mathcal{B}}
\newcommand{\mcC}{\mathcal{C}}

\newcommand{\mcE}{\mathcal{E}}
\newcommand{\mcF}{\mathcal{F}}

\newcommand{\mcL}{\mathcal{L}}

\newcommand{\mcS}{\mathcal{S}}

\newcommand{\mfB}{\mathfrak{B}}

\newcommand{\mfy}{\mathfrak{y}}
\newcommand{\mfz}{\mathfrak{z}}

\newcommand{\bz}{\boldsymbol{Z}}

\newcommand{\cembed}{\hookrightarrow\mathrel{\mspace{-15mu}}\rightarrow}

\DeclareMathOperator*{\esssup}{ess\,sup}
\newcommand{\supp}{\text{supp}}

\newcommand{\tzero}{|_{t=0}}

\newcommand{\RN}[1]{%
\textup{\uppercase\expandafter{\romannumeral#1}}%
}

\newcommand{\scal}[1]{\langle #1 \rangle}

\newcommand{\dd}{\mathop{}\!\mathrm{d}}

\colorlet{testcolor}{green!60!black}
\colorlet{testcolor2}{blue!100!black}
\colorlet{lightblue}{darkblue!80}
\colorlet{grayblue}{bluegray!100}
\colorlet{orange}{cadmiumorange!100}

\usetikzlibrary{calc}
\usetikzlibrary{shapes.misc}
\usetikzlibrary{shapes.symbols}
\usetikzlibrary{shapes.geometric}
\usetikzlibrary{snakes}
\usetikzlibrary{decorations}
\usetikzlibrary{decorations.markings}

\makeatletter
\pgfdeclareshape{crosscircle}
{
\inheritsavedanchors[from=circle] 
\inheritanchorborder[from=circle]
\inheritanchor[from=circle]{north}
\inheritanchor[from=circle]{north west}
\inheritanchor[from=circle]{north east}
\inheritanchor[from=circle]{center}
\inheritanchor[from=circle]{west}
\inheritanchor[from=circle]{east}
\inheritanchor[from=circle]{mid}
\inheritanchor[from=circle]{mid west}
\inheritanchor[from=circle]{mid east}
\inheritanchor[from=circle]{base}
\inheritanchor[from=circle]{base west}
\inheritanchor[from=circle]{base east}
\inheritanchor[from=circle]{south}
\inheritanchor[from=circle]{south west}
\inheritanchor[from=circle]{south east}
\inheritbackgroundpath[from=circle]
\foregroundpath{
\centerpoint%
\pgf@xc=\pgf@x%
\pgf@yc=\pgf@y%
\pgfutil@tempdima=\radius%
\pgfmathsetlength{\pgf@xb}{\pgfkeysvalueof{/pgf/outer xsep}}%
\pgfmathsetlength{\pgf@yb}{\pgfkeysvalueof{/pgf/outer ysep}}%
\ifdim\pgf@xb<\pgf@yb%
\advance\pgfutil@tempdima by-\pgf@yb%
\else%
\advance\pgfutil@tempdima by-\pgf@xb%
\fi%
\pgfpathmoveto{\pgfpointadd{\pgfqpoint{\pgf@xc}{\pgf@yc}}{\pgfqpoint{-0.707107\pgfutil@tempdima}{-0.707107\pgfutil@tempdima}}}
\pgfpathlineto{\pgfpointadd{\pgfqpoint{\pgf@xc}{\pgf@yc}}{\pgfqpoint{0.707107\pgfutil@tempdima}{0.707107\pgfutil@tempdima}}}
\pgfpathmoveto{\pgfpointadd{\pgfqpoint{\pgf@xc}{\pgf@yc}}{\pgfqpoint{-0.707107\pgfutil@tempdima}{0.707107\pgfutil@tempdima}}}
\pgfpathlineto{\pgfpointadd{\pgfqpoint{\pgf@xc}{\pgf@yc}}{\pgfqpoint{0.707107\pgfutil@tempdima}{-0.707107\pgfutil@tempdima}}}
}
}
\makeatother

\colorlet{symbols}{blue!90!black}
\colorlet{trees}{black!90!black}
\colorlet{cameron}{darkgreen!90!black}
\colorlet{testcolor}{green!60!black}
\colorlet{darkblue}{blue!60!black}
\colorlet{darkgreen}{green!60!black}

\def\barnorm#1{\lfloor \hspace{-0.29em} \rceil #1 \lfloor \hspace{-0.29em} \rceil}
\def\sbarnorm#1{\lfloor \hspace{-0.24em} \rceil #1 \lfloor \hspace{-0.24em} \rceil}
\def\threebars{|\!|\!|}

\def\symbol#1{\textcolor{symbols}{#1}}
\def\1{\mathbf{1}}
\def\X{\symbol{X}}

\usetikzlibrary{calc}
\usetikzlibrary{shapes.misc}
\usetikzlibrary{shapes.symbols}
\usetikzlibrary{shapes.geometric}
\usetikzlibrary{snakes}
\usetikzlibrary{decorations}
\usetikzlibrary{decorations.markings}

\def\drawx{\draw[-,solid] (-3pt,-3pt) -- (3pt,3pt);\draw[-,solid] (-3pt,3pt) -- (3pt,-3pt);}
\tikzset{
root/.style={draw=symbols,circle,inner sep=0pt,minimum size=0.5mm,fill=white},	
smalldot/.style={circle,fill=symbols,draw=symbols,inner sep=0pt,minimum size=0.5mm},
dot/.style={circle,fill=black,inner sep=0pt,minimum size=1mm},
bigdot/.style={circle,fill=black,inner sep=0pt,minimum size=4mm},
noiseedge/.style={black,semithick,decorate, decoration={snake,segment length=4pt,amplitude=1pt}},
smallestdot1/.style={circle,fill=symbols,draw=symbols,inner sep=0pt,minimum size=0.2mm},
smallestdot2/.style={circle,fill=cameron,draw=cameron,inner sep=0pt,minimum size=0.2mm},
smallnoiseedge1/.style={draw=symbols,decorate, decoration={snake,segment length=1.75pt,amplitude=0.55pt}},
smallnoiseedge2/.style={draw=cameron,decorate, decoration={snake,segment length=1.75pt,amplitude=0.55pt}},
noiseedge1/.style={draw=symbols,decorate, decoration={snake,segment length=0.9pt,amplitude=0.55pt}},
noiseedge2/.style={draw=cameron,decorate, decoration={snake,segment length=0.9pt,amplitude=0.55pt}},
smallnoisenode1/.style={draw=symbols,circle,inner sep=0pt,minimum size=0.5mm,fill=white},	
smallnoisenode2/.style={draw=cameron,circle,inner sep=0pt,minimum size=0.5mm,fill=white},
poly/.style={draw=symbols,circle,inner sep=0pt,minimum size=0.5mm,fill=black},
dot/.style={circle,fill=black,inner sep=0pt,minimum size=1mm},
int/.style={circle,fill=black,draw=black,inner sep=0pt,minimum size=0.7mm},
circ/.style={circle,draw=black,inner sep=0pt, minimum size=1mm},
var/.style={circle,fill=black!10,draw=black,inner sep=0pt, minimum size=2mm},
dotred/.style={circle,fill=black!50,inner sep=0pt, minimum size=2mm},
generic/.style={semithick,shorten >=1pt,shorten <=1pt},
oddfunc/.style={generic, dotted},
dist/.style={ultra thick,draw=testcolor,shorten >=1pt,shorten <=1pt},
testfcn/.style={ultra thick,testcolor,shorten >=1pt,shorten <=1pt,<-},
testfunction/.style={ultra thick,testcolor,shorten >=1pt,shorten <=1pt},
testfcnx/.style={ultra thick,testcolor,shorten >=1pt,shorten <=1pt,<-,
postaction={decorate,decoration={markings,mark=at position 0.6 with {\drawx}}}},
kprime/.style={semithick,shorten >=1pt,shorten <=1pt,densely dashed,->},
kprimex/.style={semithick,shorten >=1pt,shorten <=1pt,densely dashed,->,
postaction={decorate,decoration={markings,mark=at position 0.4 with {\drawx}}}},
kernel/.style={semithick,shorten >=1pt,shorten <=1pt,->,draw=black},
Pkernel/.style={ultra thick,shorten >=1pt,shorten <=1pt,->,draw=blue},
PkernelBig/.style={very thick,shorten >=1pt,shorten <=1pt,decorate, draw=blue, decoration={zigzag,amplitude=1.5pt,segment length = 3pt,pre length=2pt,post length=2pt}},
multx/.style={shorten >=1pt,shorten <=1pt,
postaction={decorate,decoration={markings,mark=at position 0.5 with {\drawx}}}},
kernelx/.style={semithick,shorten >=1pt,shorten <=1pt,->,
postaction={decorate,decoration={markings,mark=at position 0.4 with {\drawx}}}},
kernel1/.style={->,semithick,shorten >=1pt,shorten <=1pt,postaction={decorate,decoration={markings,mark=at position 0.45 with {\draw[-] (0,-0.1) -- (0,0.1);}}}},
kernel2/.style={->,semithick,shorten >=1pt,shorten <=1pt,postaction={decorate,decoration={markings,mark=at position 0.45 with {\draw[-] (0.05,-0.1) -- (0.05,0.1);\draw[-] (-0.05,-0.1) -- (-0.05,0.1);}}}},
kernelBig/.style={semithick,shorten >=1pt,shorten <=1pt,decorate, decoration={zigzag,amplitude=1.5pt,segment length = 3pt,pre length=2pt,post length=2pt}},
kernelBigg/.style={thick,shorten >=1pt,shorten <=1pt,decorate, decoration={zigzag,amplitude=3.5pt,segment length = 7pt,pre length=2pt,post length=2pt}},
kernelBigg1/.style={thick,shorten >=1pt,shorten <=1pt,decorate, decoration={saw,amplitude=3.5pt,segment length = 7pt,pre length=2pt,post length=2pt}},
kernelBigg2/.style={thick,shorten >=1pt,shorten <=1pt,decorate, decoration={bumps,amplitude=3.5pt,segment length = 7pt,pre length=2pt,post length=2pt}},
rho/.style={dotted,semithick,shorten >=1pt,shorten <=1pt},
k/.style={color=trees, snake=zigzag, segment length=1pt, segment amplitude=0.3pt},
r/.style={color=trees, double, double distance=0.1pt},
renorm/.style={shape=circle,fill=white,inner sep=1pt},
labl/.style={shape=rectangle,fill=white,inner sep=1pt},
xi/.style={circle,fill=symbols!30,draw=symbols,inner sep=0pt,minimum size=0.8mm},
xibig/.style={circle,fill=symbols!30,draw=symbols,inner sep=0pt,minimum size=1.1mm},
cmn/.style={circle,fill=cameron!50,draw=symbols,inner sep=0pt,minimum size=1.2mm},
xix/.style={crosscircle,fill=symbols!10,draw=symbols,inner sep=0pt,minimum size=1.2mm},
xib/.style={circle,fill=symbols!30,draw=symbols,inner sep=0pt,minimum size=1.4mm},
xiblack/.style={circle,fill=trees,draw=trees,inner sep=0pt,minimum size=0.7mm},
xigray/.style={circle,fill=symbols!30,draw=symbols,inner sep=0pt,minimum size=0.8mm},
xibx/.style={crosscircle,fill=symbols!10,draw=symbols,inner sep=0pt,minimum size=1.6mm},
not/.style={circle,fill=symbols,draw=symbols,inner sep=0pt,minimum size=0.5mm},
>=stealth,
}
\makeatletter
\def\DeclareSymbol#1#2#3{\expandafter\gdef\csname MH@symb@#1\endcsname{\tikz[baseline=#2,scale=0.13,draw=symbols]{#3}}\expandafter\gdef\csname MH@symb@#1s\endcsname{\scalebox{0.6}{\tikz[baseline=#2,scale=0.15,draw=symbols]{#3}}}}
\def\<#1>{\csname MH@symb@#1\endcsname}
\makeatother

\DeclareSymbol{xi}{-3}{\node[xibig] {};}
\DeclareSymbol{wn}{-3}{\node[xibig] {};}
\DeclareSymbol{xigray}{-3}{\node[xigray] {};}
\DeclareSymbol{X}{-2.4}{\node[dot] {};}
\DeclareSymbol{1}{0}{\draw[white] (-.4,0) -- (.4,0); \draw (0,0)  -- (0,1.2) node[xi] {};}
\DeclareSymbol{2}{0}{\draw (-0.5,1.2) node[xi] {} -- (0,0) -- (0.5,1.2) node[xi] {};}
\DeclareSymbol{3}{0}{\draw (0,0) -- (0,1.2) node[xi] {}; \draw (-.7,1) node[xi] {} -- (0,0) -- (.7,1) node[xi] {};}
\DeclareSymbol{31}{-3}{\draw (0,0) -- (0,-1) -- (1,0) node[xi] {}; \draw (0,0) -- (0,1.2) node[xi] {}; \draw (-.7,1) node[xi] {} -- (0,0) -- (.7,1) node[xi] {};}
\DeclareSymbol{30}{-3}{\draw (0,0) -- (0,-1); \draw (0,0) -- (0,1.2) node[xi] {}; \draw (-.7,1) node[xi] {} -- (0,0) -- (.7,1) node[xi] {};}
\DeclareSymbol{10}{-3}{\draw (0,0) -- (0,-1); \draw (-.8,1) node[xi] {} -- (0,0);}
\DeclareSymbol{32}{-3}{\draw (0,0) -- (0,-1) -- (1,0) node[xi] {}; \draw (0,0) -- (0,-1) -- (-1,0) node[xi] {}; \draw (0,0) -- (0,1.2) node[xi] {}; \draw (-.7,1) node[xi] {} -- (0,0) -- (.7,1) node[xi] {};}
\DeclareSymbol{22}{-3}{\draw (0,0.3) -- (0,-1) -- (1,0) node[xi] {}; \draw (0,0.3) -- (0,-1) -- (-1,0) node[xi] {};\draw (-.7,1) node[xi] {} -- (0,0.3) -- (.7,1) node[xi] {};}
\DeclareSymbol{12}{-3}{\draw (0,0) -- (0,-1) -- (1,0) node[xi] {}; \draw (0,0) -- (0,-1) -- (-1,0) node[xi] {}; \draw (-.8,1) node[xi] {} -- (0,0);}
\DeclareSymbol{20}{-3}{\draw (0,0) -- (0,-1);\draw (-.7,1) node[xi] {} -- (0,0) -- (.7,1) node[xi] {};}

\DeclareSymbol{xib}{-3}{\node[xiblack] {};}
\DeclareSymbol{X}{-2.4}{\node[dot] {};}
\DeclareSymbol{1b}{0}{\draw[white] (-.4,0) -- (.4,0); \draw[trees] (0,0)  -- (0,1.2) node[xiblack] {};}
\DeclareSymbol{2b}{0}{\draw[trees] (-0.5,1.2) node[xiblack] {} -- (0,0) -- (0.5,1.2) node[xiblack] {};}
\DeclareSymbol{3b}{0}{\draw[trees] (0,0) -- (0,1.2) node[xiblack] {}; \draw[trees] (-.7,1) node[xiblack] {} -- (0,0) -- (.7,1) node[xiblack] {};}
\DeclareSymbol{31b}{-3}{\draw[trees] (0,0) -- (0,-1) -- (1,0) node[xiblack] {}; \draw[trees] (0,0) -- (0,1.2) node[xiblack] {}; \draw[trees] (-.7,1) node[xiblack] {} -- (0,0) -- (.7,1) node[xiblack] {};}
\DeclareSymbol{30b}{-3}{\draw[trees] (0,0) -- (0,-1); \draw[trees] (0,0) -- (0,1.2) node[xiblack] {}; \draw[trees] (-.7,1) node[xiblack] {} -- (0,0) -- (.7,1) node[xiblack] {};}
\DeclareSymbol{10b}{-3}{\draw[trees] (0,0) -- (0,-1); \draw[trees] (-.8,1) node[xiblack] {} -- (0,0);}
\DeclareSymbol{32b}{-3}{\draw[trees] (0,0) -- (0,-1) -- (1,0) node[xiblack] {}; \draw[trees] (0,0) -- (0,-1) -- (-1,0) node[xiblack] {}; \draw[trees] (0,0) -- (0,1.2) node[xiblack] {}; \draw[trees] (-.7,1) node[xiblack] {} -- (0,0) -- (.7,1) node[xiblack] {};}
\DeclareSymbol{22b}{-3}{\draw[trees] (0,0.3) -- (0,-1) -- (1,0) node[xiblack] {}; \draw[trees] (0,0.3) -- (0,-1) -- (-1,0) node[xiblack] {};\draw[trees] (-.7,1) node[xiblack] {} -- (0,0.3) -- (.7,1) node[xiblack] {};}
\DeclareSymbol{12b}{-3}{\draw[trees] (0,0) -- (0,-1) -- (1,0) node[xi] {}; \draw[trees] (0,0) -- (0,-1) -- (-1,0) node[xiblack] {}; \draw[trees] (-.8,1) node[xiblack] {} -- (0,0);}
\DeclareSymbol{20b}{-3}{\draw[trees] (0,0) -- (0,-1);\draw[trees] (-.7,1) node[xiblack] {} -- (0,0) -- (.7,1) node[xiblack] {};}

\DeclareSymbol{G}{1}{\draw[white] (-.4,0) -- (.4,0); \draw[trees, semithick] (0,0)  -- (0,2.5) {};}
\DeclareSymbol{K}{1}{\draw[white] (-.4,0) -- (.4,0); \draw[trees, semithick, snake=zigzag, segment length=2pt, segment amplitude=0.5pt] (0,0)  -- (0,2.5)  {};}
\DeclareSymbol{R}{1}{\draw[white] (-.4,0) -- (.4,0); \draw[color=trees, double, double distance=0.4pt] (0,0)  -- (0,2.5) {};}
\DeclareSymbol{I}{1}{\draw[white] (-.4,0) -- (.4,0); \draw[symbols, semithick] (0,0)  -- (0,2.5) {};}

\DeclareSymbol{1r}{0}{\draw[white] (-.4,0) -- (.4,0); \draw[r] (0,0)  -- (0,1.2) node[xiblack] {};}
\DeclareSymbol{1k}{0}{\draw[white] (-.4,0) -- (.4,0); \draw[k] (0,0)  -- (0,1.2) node[xiblack] {};}

\DeclareSymbol{2kr}{0}{\draw[k] (-0.5,1.2) node[xiblack] {} -- (0,0); \draw[r] (0,0) -- (0.5,1.2) node[xiblack] {};}
\DeclareSymbol{2rk}{0}{\draw[k] (-0.5,1.2) node[xiblack] {} -- (0,0); \draw[r] (0,0) -- (0.5,1.2) node[xiblack] {};}
\DeclareSymbol{2kk}{0}{\draw[k] (-0.5,1.2) node[xiblack] {} -- (0,0); \draw[k] (0,0) -- (0.5,1.2) node[xiblack] {};}
\DeclareSymbol{2rr}{0}{\draw[r] (-0.5,1.2) node[xiblack] {} -- (0,0); \draw[r] (0,0) -- (0.5,1.2) node[xiblack] {};}

\DeclareSymbol{3rrr}{0}{\draw[r] (0,0) -- (0,1.2) node[xiblack] {}; \draw[r] (-.7,1) node[xiblack] {} -- (0,0) -- (.7,1) node[xiblack] {};}
\DeclareSymbol{3kkk}{0}{\draw[k] (0,0) -- (0,1.2) node[xiblack] {}; \draw[k] (-.7,1) node[xiblack] {} -- (0,0) -- (.7,1) node[xiblack] {};}
\DeclareSymbol{3kkr}{0}{\draw[k] (0,0) -- (0,1.2) node[xiblack] {}; \draw[k] (-.7,1) node[xiblack] {} -- (0,0); \draw[r] (0,0) -- (.7,1) node[xiblack] {};}
\DeclareSymbol{3krr}{0}{\draw[r] (0,0) -- (0,1.2) node[xiblack] {}; \draw[k] (-.7,1) node[xiblack] {} -- (0,0); \draw[r] (0,0) -- (.7,1) node[xiblack] {};}

\DeclareSymbol{3rrr}{0}{\draw[r] (0,0) -- (0,1.2) node[xiblack] {}; \draw[r] (-.7,1) node[xiblack] {} -- (0,0) -- (.7,1) node[xiblack] {};}
\DeclareSymbol{3kkk}{0}{\draw[k] (0,0) -- (0,1.2) node[xiblack] {}; \draw[k] (-.7,1) node[xiblack] {} -- (0,0) -- (.7,1) node[xiblack] {};}
\DeclareSymbol{3kkr}{0}{\draw[k] (0,0) -- (0,1.2) node[xiblack] {}; \draw[k] (-.7,1) node[xiblack] {} -- (0,0); \draw[r] (0,0) -- (.7,1) node[xiblack] {};}
\DeclareSymbol{3krr}{0}{\draw[r] (0,0) -- (0,1.2) node[xiblack] {}; \draw[k] (-.7,1) node[xiblack] {} -- (0,0); \draw[r] (0,0) -- (.7,1) node[xiblack] {};}

\DeclareSymbol{30k}{-3}{\draw[k] (0,0) -- (0,-1); \draw[trees] (0,0) -- (0,1.2) node[xiblack] {}; \draw[trees] (-.7,1) node[xiblack] {} -- (0,0) -- (.7,1) node[xiblack] {};}
\DeclareSymbol{30r}{-3}{\draw[r] (0,0) -- (0,-1); \draw[trees] (0,0) -- (0,1.2) node[xiblack] {}; \draw[trees] (-.7,1) node[xiblack] {} -- (0,0) -- (.7,1) node[xiblack] {};}
\DeclareSymbol{3rrr0k}{-3}{\draw[k] (0,0) -- (0,-1); \draw[r] (0,0) -- (0,1.2) node[xiblack] {}; \draw[r] (-.7,1) node[xiblack] {} -- (0,0) -- (.7,1) node[xiblack] {};}
\DeclareSymbol{3kkk0k}{-3}{\draw[k] (0,0) -- (0,-1);\draw[k] (0,0) -- (0,1.2) node[xiblack] {}; \draw[k] (-.7,1) node[xiblack] {} -- (0,0) -- (.7,1) node[xiblack] {};}
\DeclareSymbol{3kkr0k}{-3}{\draw[k] (0,0) -- (0,-1);\draw[k] (0,0) -- (0,1.2) node[xiblack] {}; \draw[k] (-.7,1) node[xiblack] {} -- (0,0); \draw[r] (0,0) -- (.7,1) node[xiblack] {};}
\DeclareSymbol{3krr0k}{-3}{\draw[k] (0,0) -- (0,-1);\draw[r] (0,0) -- (0,1.2) node[xiblack] {}; \draw[k] (-.7,1) node[xiblack] {} -- (0,0); \draw[r] (0,0) -- (.7,1) node[xiblack] {};}

\DeclareSymbol{3rrr0r}{-3}{\draw[r] (0,0) -- (0,-1); \draw[r] (0,0) -- (0,1.2) node[xiblack] {}; \draw[r] (-.7,1) node[xiblack] {} -- (0,0) -- (.7,1) node[xiblack] {};}
\DeclareSymbol{3kkk0r}{-3}{\draw[r] (0,0) -- (0,-1);\draw[k] (0,0) -- (0,1.2) node[xiblack] {}; \draw[k] (-.7,1) node[xiblack] {} -- (0,0) -- (.7,1) node[xiblack] {};}
\DeclareSymbol{3kkr0r}{-3}{\draw[r] (0,0) -- (0,-1);\draw[k] (0,0) -- (0,1.2) node[xiblack] {}; \draw[k] (-.7,1) node[xiblack] {} -- (0,0); \draw[r] (0,0) -- (.7,1) node[xiblack] {};}
\DeclareSymbol{3krr0r}{-3}{\draw[r] (0,0) -- (0,-1);\draw[r] (0,0) -- (0,1.2) node[xiblack] {}; \draw[k] (-.7,1) node[xiblack] {} -- (0,0); \draw[r] (0,0) -- (.7,1) node[xiblack] {};}

\DeclareSymbol{30b}{-3}{\draw[trees] (0,0) -- (0,-1); \draw[trees] (0,0) -- (0,1.2) node[xiblack] {}; \draw[trees] (-.7,1) node[xiblack] {} -- (0,0) -- (.7,1) node[xiblack] {};}

\DeclareSymbol{xib}{-3}{\node[xiblack] {};}
\DeclareSymbol{X}{-2.4}{\node[dot] {};}
\DeclareSymbol{1b}{0}{\draw[white] (-.4,0) -- (.4,0); \draw[trees] (0,0)  -- (0,1.2) node[xiblack] {};}
\DeclareSymbol{2b}{0}{\draw[trees] (-0.5,1.2) node[xiblack] {} -- (0,0) -- (0.5,1.2) node[xiblack] {};}
\DeclareSymbol{3b}{0}{\draw[trees] (0,0) -- (0,1.2) node[xiblack] {}; \draw[trees] (-.7,1) node[xiblack] {} -- (0,0) -- (.7,1) node[xiblack] {};}
\DeclareSymbol{31b}{-3}{\draw[trees] (0,0) -- (0,-1) -- (1,0) node[xiblack] {}; \draw[trees] (0,0) -- (0,1.2) node[xiblack] {}; \draw[trees] (-.7,1) node[xiblack] {} -- (0,0) -- (.7,1) node[xiblack] {};}
\DeclareSymbol{30b}{-3}{\draw[trees] (0,0) -- (0,-1); \draw[trees] (0,0) -- (0,1.2) node[xiblack] {}; \draw[trees] (-.7,1) node[xiblack] {} -- (0,0) -- (.7,1) node[xiblack] {};}
\DeclareSymbol{10b}{-3}{\draw[trees] (0,0) -- (0,-1); \draw[trees] (-.8,1) node[xiblack] {} -- (0,0);}
\DeclareSymbol{32b}{-3}{\draw[trees] (0,0) -- (0,-1) -- (1,0) node[xiblack] {}; \draw[trees] (0,0) -- (0,-1) -- (-1,0) node[xiblack] {}; \draw[trees] (0,0) -- (0,1.2) node[xiblack] {}; \draw[trees] (-.7,1) node[xiblack] {} -- (0,0) -- (.7,1) node[xiblack] {};}
\DeclareSymbol{22b}{-3}{\draw[trees] (0,0.3) -- (0,-1) -- (1,0) node[xiblack] {}; \draw[trees] (0,0.3) -- (0,-1) -- (-1,0) node[xiblack] {};\draw[trees] (-.7,1) node[xiblack] {} -- (0,0.3) -- (.7,1) node[xiblack] {};}
\DeclareSymbol{12b}{-3}{\draw[trees] (0,0) -- (0,-1) -- (1,0) node[xi] {}; \draw[trees] (0,0) -- (0,-1) -- (-1,0) node[xiblack] {}; \draw[trees] (-.8,1) node[xiblack] {} -- (0,0);}
\DeclareSymbol{20b}{-3}{\draw[trees] (0,0) -- (0,-1);\draw[trees] (-.7,1) node[xiblack] {} -- (0,0) -- (.7,1) node[xiblack] {};}

\usepackage{caption}

\usepackage{lipsum}

\usepackage{cite}

\setkomafont{dictumauthor}{\normalfont}
\setkomafont{dictumtext}{\normalfont}

\usepackage{colortbl}

\allowdisplaybreaks[1]

\usepackage{scalerel}

\makeatletter
\newcommand\bigwp{\mathop{\mathpalette\bigDi@mond\relax}}
\newcommand\bigDi@mond[2]{%
\vcenter{\hbox{\m@th
\scalebox{\ifx#1\displaystyle 2\else1.2\fi}{$#1\diamond$}%
}}%
}
\makeatother

\author{Tom Klose and Avi Mayorcas}

\date{\today}

\title{Large Deviations of the $\boldsymbol{\Phi^4_3}$ Measure \\ via Stochastic Quantisation }

\numberwithin{equation}{section}

\newcommand{\bbzm}{\bar{\Pi}}
\newcommand{\Phim}{\Phi}

\newcommand{\bbzmeps}{\bar{\Pi}^\eps}
\newcommand{\bzm}{\Pi}
\newcommand{\MMm}{\boldsymbol{\CM}^{\minus}}

\begin{document}

\maketitle	

\begin{abstract}
The~$\Phi^4_3$ measure is one of the easiest non-trivial examples of a Euclidean quantum field theo\-ry~(EQFT) whose rigorous construction in the 1970's has been one of the ce\-le\-bra\-ted achievements of constructive quantum field theory.
In recent years, progress in the field of singular stochastic~PDEs, initiated by the theory of regularity structures, has allowed for a new construction of the~$\Phi^4_3$ EQFT as the invariant measure of a previously ill-posed Langevin dynamics---a strategy originally proposed by Parisi and Wu~(`81) under the name \emph{stochastic quantisation}.
We apply the same methodology to obtain a large deviation principle (LDP) for the family of periodic $\Phi^4_3$ measures at varying temperature. In addition, we show that the rate functional of the LDP and the $\Phi^4_3$ action functional coincide up to a constant. 
We wish to highlight that while our main result had previously been obtained by Barashkov~(2022), the main focus of this work is on the approach through the stochastic quantisation equation.
\end{abstract}

\noindent
\emph{Key words and phrases.} \\
\noindent
Euclidean Quantum Field Theory, $\Phi^4_3$ Measure, Large Deviations, Stochastic Quantisation 

\vspace{0.5em}
\noindent
\emph{2020 Mathematics Subject Classification.} \\
\noindent
Primary: 81S20, Secondary: 60F10, 60H17, 60L30


\setcounter{tocdepth}{3} 		

\tableofcontents

\section{Introduction}\label{sec:intro}
The $\Phi^4_3$ measure on the $3$-dimensional torus~$\mbT^3$ at temperature~$\eps > 0$ is given by the \emph{formal} expression
\begin{equation}\label{eq:phi43_formal_measure}
	\mu_\eps  \,\propto \,  \exp\left(-\frac{1}{\eps^2}\int_{\mbT^3} \left(\frac{1}{2}|\nabla \phi(x)|^2 + \frac{1}{4}\phi(x)^4 + \frac{m^2}{2}\phi(x)^2  \right) \dd x \right)  \prod_{x\in \mbT^3} \dd \phi(x).
\end{equation}
where~\enquote{$\propto$} denotes  a suitable normalisation so as to give a probability measure. We refer the reader to Section~\ref{sec:tail_bounds} for a more rigorous description of this object.
Taking this definition for granted, the measure defined in \eqref{eq:phi43_formal_measure} gives a Euclidean version of a quartic, Bosonic, relativistic quantum field theory. In the programme of constructive quantum field theory, Euclidean quantum field theory (EQFT) measures are related to relativistic theories through verification of the Osterwalder--Schrader~(OS) axioms~\cite{osterwalder_schrader_I, osterwalder_schrader_II}; see Glimm and Jaffe~\cite[Sec.~6.1]{glimm_jaffe} for a pedagogical account and the introduction by Gubinelli and Hofmanova~\cite{gubinelli_hofmanova} for a more detailed overview. 
Informally, the integration variable $\phi$ in \eqref{eq:phi43_formal_measure} is to be understood as an element of~$\mcS'(\mbT^3)$, the space of distributions on the three dimensional torus. As such, the expression \eqref{eq:phi43_formal_measure} is purely formal for two reasons: firstly, the product measure over $\prod_{x\in \mbT^3}\dd \phi(x)$ is not well-defined and secondly, even if it were, the quartic non-linear terms are generically ill-defined for $\phi\in \mcS'(\mbT^3)$.

\vspace{-0.5em}
\paragraph{Rigorous Constructions of the Measure.}
Despite these challenges, rigorous meaning can be given to a renormalised modification of \eqref{eq:phi43_formal_measure}; this was first achieved in both finite and infinite volumes in the series of works~\cite{glimm_68_boson,eckmann_osterwalder_71_uniqueness,glimm_jaffe_73_positivity,feldman_74_finite_vol,feldman_osterwalder_76_wightman}. A major outcome of this programme was the verification of the (OS) axioms for the Euclidean $\Phi^4$ measure at small couplings, i.e. small relative strength of the non-linear term, mediated through a parameter absent in our presentation. However, modern approaches have recently succeeded in constructing the same measure at all coupling strengths, as well as showing promise for more systematic approaches to a wider range of physical theories. One of these approaches, which is also the approach of this paper, is the programme of \emph{stochastic quantisation}, proposed originally by Nelson~\cite{nelson_66_derivation} in the context of quantum mechanics and then by Parisi and Wu~\cite{parisi-wu} in the context of EQFT; see Damgaard and H\"uffel~\cite{damgaard_hueffel_1987_stoch_quant} for a review article on the subject. 
In the case of the Euclidean $\Phi^4_3$ measure, the approach of Parisi and Wu leverages the Gibbsian nature of (the renormalised version of) \eqref{eq:phi43_formal_measure} to view it as the invariant measure of the (renormalised) Langevin dynamics
\begin{equation}\label{eq:phi43_formal_SPDE}
	\partial_t \bar{u}_\eps - \Delta \bar{u}_\eps = - \bar{u}_\eps^3 - m^2 \bar{u}_\eps + \infty \, \bar{u}_\eps + \eps \xi, \quad \bar{u}_\eps(0,\,\cdot\,) = \fz,
\end{equation}
where $\xi$ is a space-time white noise, $\fz$ is a suitably regular initial condition, and the term $\infty \bar{u}_\eps$ signifies a diverging counter-term which must be included in the equation to balance the singularity of the non-linear term. We give rigorous meaning to this equation in Definition~\ref{def:phi43_rigorous_SPDE} below. 

A local solution theory for \eqref{eq:phi43_formal_SPDE} on $\mbT^3$ was first obtained by Hairer~\cite{hairer_rs} via the novel theory of regularity structures, then by Catellier and Chouk~\cite{catellier_chouk_paracontrolled_phi43} via the theory of paracontrolled calculus as developed in~\cite{gip}, and subsequently via renormalisation group methods by Kupiainen and Duch, see for example~\cite{kupiainen_renorm_group_spde,duch_22_flow,duch_23_renormalisation}.
Global well-posedness and existence of an invariant measure for the dynamical equation,
on the finite volume, was then shown by Mourrat and Weber~\cite{mourrat_weber_infinity} while uniqueness of the measure, for all coupling strengths, in the finite volume case,\footnote{In~$2$D, global well-posedness on the plane was shown by the same authors in~\cite{mourrat_weber_17_GWP}.} was obtained in combination by Hairer and Mattingly~\cite{hairer-mattingly} and Hairer and Schönbauer~\cite{hairer_schoenbauer_support}. In \cite{hairer_matetski_18_discretisation}, Hairer and Matetski further showed that this measure, $\mu_\eps$, agrees with the one obtained in the programme of CQFT, for example by \cite{feldman_74_finite_vol}, for coupling strengths where the latter is defined.

Alternative constructions of the same measure were also recently given by Barashkov and Gubinelli via a variational approach~\cite{barashkov_gubinelli_20_variational} and Girsanov's theorem~\cite{barashkov_gubinelli_21_girsanov} as well as by Gubinelli and Hofmanova~\cite{gubinelli_hofmanova} via an approximating family of PDEs on the lattice. 
In the latter, it was also shown that the measure satisfies all OS axioms in the form of~\cite{eckmann_osterwalder_71_uniqueness}, apart from rotation invariance and the clustering property which are only relevant on the infinite volume. 

As described above, a consequence of the stochastic quantisation approach is that, in certain circumstances, it allows for the transference of properties obtained for dynamical solutions, where tools of stochastic analysis and PDE theory may be applied, to the invariant Euclidean QFT measure.
One pertinent example of this is the work by Hairer and Steele~\cite{hairer_steele_22} where sub-Gaussian tail estimates are obtained for 
(a modified version of) the dynamical $\Phi^4_3$ equation.
These tail bounds for the dynamics are their main technical contribution which, by invariance of the EQFT measure, quite simply carry over to the renormalised version of \eqref{eq:phi43_formal_measure}.
The present work is another example of a similar me\-tho\-do\-lo\-gy; we transfer the large deviation principle for the periodic $\Phi^4_3$ dynamics, obtained by Hairer and Weber \cite{hairer_weber_ldp}, onto the EQFT measure as informally described by~\eqref{eq:phi43_formal_measure}. 
However, in our case the transference, compared to \cite{hairer_steele_22}, is more technical and builds on analytic arguments developed by Sowers~\cite{sowers_ldp_measure} and Cerrai and R\"ockner~\cite{cerrai_roeckner_ldp_measure} in the non-singular SPDE setting.
We believe that the method developed in our work may readily be applied to other physically relevant models for which the theory of regularity structures applies, see the paragraph on future directions at the end of the next section.

\subsection{Main Result and Discussion.} \label{sec:main_result_discussion}
In order to state our main result, we introduce the  $\Phi^4_3$ \emph{action functional}~$\msS: \mcS'(\mbT^3)\to [0,+\infty]$ by setting 
\begin{equation}\label{eq:phi43_action}
	\msS(\phi) 
	\coloneqq  
	\begin{cases}
		\int_{\mbT^3} \left(\frac{1}{2} |\nabla \phi(x)|^2 + \frac{1}{4} |\phi(x)|^4 +\frac{m^2}{2} |\phi(x)|^2 \right)\dd x, & \phi \in H^1(\mbT^3),\\
		+\infty, & \phi \in \mcC^{\alpha}(\mbT^3)\setminus H^1(\mbT^3).
	\end{cases}
\end{equation}
The following theorem is the main result of this paper.\footnote{Note that the factor~\enquote{$2$} in the rate function~$2\msS$ can simply be removed by rescaling~$\eps \rightsquigarrow \sqrt{2}\eps$ or by considering an SPDE like~\eqref{eq:phi43_formal_SPDE} with formal gradient~$-\frac{1}{2} D\msS(u_\eps)$ rather than~$-D\msS(u_\eps)$. 
}
\begin{theorem}[LDP for the~$\boldsymbol{\Phi^4_3}$ measure]\label{th:main_result}
	Let $\alpha \in (-\nicefrac{2}{3},-\nicefrac{1}{2})$ be fixed. Then, the measures~$(\mu_\eps)_{\eps > 0}$, formally given by \eqref{eq:phi43_formal_measure}, satisfy a large deviation principle on $\mcC^\alpha(\mbT^3)$ with good rate function~$2 \msS$ and speed $\eps^2$. That is,
	\begin{enumerate}[label=(\roman*)]
		\item \label{it:LDP_compactness} the map~$\msS$ introduced in~\eqref{eq:phi43_action} is lower-semicontinuous, has compact sub-level sets, and is not identical to~$+ \infty$,
		\item \label{it:LDP_lower} for any open set $\mathring{A} \subset \mcC^\alpha(\mbT^3)$, we have
		\begin{equation}\label{eq:main_result_lower}
			\liminf_{\eps \to 0} \eps^2 \ln (\mu_\eps(\mathring{A}))\geq 	-\inf_{\mfz\in \mathring{A}} 2\msS(\mfz),
		\end{equation}
		\item \label{it:LDP_upper} for any closed set $\bar{A}\subset \mcC^\alpha(\mbT^3)$, we have
		\begin{equation}\label{eq:main_result_upper}
			\limsup_{\eps \to 0} \eps^2 \ln (\mu_\eps(\bar{A}))\leq -\inf_{\mfz\in \bar{A}} 2\msS(\mfz). 
		\end{equation}
	\end{enumerate}
\end{theorem}
\noindent
\paragraph{Intuition for LDPs and Semi-Classical Physics.}
One may naturally ask what additional information a large deviation principle provides for the family $(\mu_\eps)_{\eps > 0}$. Mathematically speaking, it expresses the idea that as random fluctuations are removed, the measure $\mu_{\eps}$ should concentrate on classical minimisers of the action functional $\msS$. From a physical point of view, up to normalising other dimensionless constants, we may identify $\hbar = \eps^2$ where $\hbar$ is Planck's constant. 
Thus, we may think of $\eps>0$ in \eqref{eq:phi43_formal_measure} as either a measure of temperature, or the discrepancy between the classical theory and its quantum analogue. We expect it to quantify the amplitude of expected fluctuations of the quantum system around its natural, classical analogue.

The large deviation principle captures the size of these fluctuations on a logarithmic scale. Informally, it expresses the probability of witnessing \emph{rare events}~$A$ as $\eps\to 0$, that is
\begin{equation*}
	\mu_\eps(A) \simeq_{\tiny \log} \exp\del[2]{- \frac{1}{\eps^2} \inf_{\fz \in A} 2\msS(\fz)} \quad \text{for} \quad \eps \ll 1,
\end{equation*}
where the notation~$\simeq_{\tiny \log}$ is a shorthand for  \emph{logarithmic} equivalence in the sense of the bounds~\eqref{eq:main_result_lower} and~\eqref{eq:main_result_upper}.
It is one of the main results in the theory of large deviations that an LDP is equivalent to a~\emph{Laplace principle}, namely the approximation
\begin{equation} \label{eq:laplace_principle}
	\int e^{-\frac{1}{\eps^2} (F + \msS)(\phi)} \dif \phi
	\propto
	\mbE_{\mu_\eps}\sbr[1]{ e^{-\frac{1}{\eps^2} F(\phi)}} 
	\simeq_{\tiny \log} \exp\del[2]{-\frac{1}{\eps^2} \inf_{\phi} (F + 2\msS)(\phi)} 	
\end{equation}
for suitable functionals~$F: \supp(\mu_\eps) \to \mbR$. In fact, this is the starting point for the so-called \emph{weak convergence approach} to large deviations developed by Dupuis and Ellis~\cite{dupuis_ellis_97_weak} and applied to stochastic PDEs by Budhiraja, Dupuis, and Maroulas~\cite{budhiraja_dupuis_maroulas_2008_ldp_weak_conv}.

In the context of Euclidean QFT, the Laplace principle as expressed by~\eqref{eq:laplace_principle} is an example of~\emph{semi-classical} limits. The study of these limits generalises the idea described above, i.e. they are concerned with the discrepancy, due to thermal fluctuations, between quantum field theoretic descriptions and their classical counterpart. The LDP is only one example of such a limit. A tighter relationship between the quantum and classical theories is given by the notion of \emph{sharp Laplace asymptotics} or \emph{semi-classical expansions}, sometimes also referred to as \emph{low temperature expansions}.
For the~$\Phi^4_2$ measure on~$\T^2$, such an expansion has recently been obtained by Gess, Seong, and Tsatsoulis~\cite{gess_seong_tsatsoulis_24}. 
For an example of how to obtain \emph{sharp Laplace asymptotics} for solutions to a family of singular SPDEs we refer to \cite{friz_klose_22_precise, klose_2025}. 

We also mention that LDPs have been studied in the context of Euclidean gauge theories by Lévy and Norris \cite{levy_norris_06}, obtaining the first rigorous connection between the previously constructed $2$D Yang--Mills measure and the Yang--Mills action functional. In the direction of semi-classical limits, large deviations and sharp Laplace asymptotics have been studied in the context of Louville quantum field theories by Lacoin, Rhodes, and Vargas \cite{lacoin_rhodes_vargas_17_semiclassical_liouville,lacoin_rhodes_vargas_22_semiclassical_conformal}. In the context of relativistic QFT, the asymptotics described by LDP and sharp Laplace asymptotics on the Euclidean side are related to the \emph{approximations of stationary phase} on the relativistic side. 
For finite-dimensional SDE dynamics, this connection has been explored by Ben Arous \cite{benArous_88_stationary}, motivated by earlier works of Azencott and Doss~\cite{doss_1985, azencott_doss_1985} on the semi-classical limits of Schr\"odinger's equation.

\vspace{-0.5em}
\paragraph{Large Deviation Principles for SPDEs and their Invariant Measures.}
As described earlier in the introduction, the dynamical $\Phi^4_3$ equation is not naively well-posed, due to the irregularity of space-time white noise. However, if one considers a sufficiently regular approximation to the noise, for example, stetting~$\xi_\kappa := \xi * \rho_\kappa$ with $\rho_\kappa$  a standard mollifier at scale~$\kappa>0$, then the equation
\begin{equation} \label{eq:phi43_SPDE_mollified_unrenormalised}
	\begin{cases}
		\partial_t u_{\eps,\kappa}(t,x) - \Delta u_{\eps,\kappa}(t,x) 
		= - u_{\eps,\kappa}(t,x)^3 - m^2 u_{\eps,\kappa}(t,x) + \eps \xi_\kappa(t,x),&  \\
		u_{\eps,\kappa}(0,\,\cdot\,) = \fz, \quad (t,x) \in [0,T] \x \T^d,
	\end{cases}
\end{equation}
has a classically well-defined, unique solution for all~$d \geq 1$. Cerrai and Röckner analysed such equations in a series of works~\cite{cerrai_2003, cerrai_roeckner_ldp_dynamics} and,
in fact, obtained an LDP for a large class of reaction-diffusion systems, with non-Lipschitz reaction term and multiplicative, regular noise, even in the case of unbounded, degenerate diffusion coefficient. In the subsequent work~\cite{ cerrai_roeckner_ldp_measure} the same authors demonstrated that under suitable structural assumptions on the coefficients, the LDP can be transferred to the unique invariant measure of the system; this generalised earlier work by Sowers~\cite{sowers_ldp_dynamics,sowers_ldp_measure} which treated the case of a single equation with more restrictive conditions on both the reaction and the diffusion terms.
For a succinct, intuitive description of the relation between LDPs for stochastic PDEs and those for their invariant measures, we direct the reader to the introduction of~\cite{freidlin88}.
More recently, the works \cite{cerrai_paskal_22_NS_LDP,bai_feng_zhao_24_burgers} have applied the same methodology to stochastic fluid equations and their invariant measures.

When~$d = 1$, the equation~\eqref{eq:phi43_SPDE_mollified_unrenormalised} is well-posed even in the limit~$\kappa \to 0$ and an LDP (also in the case of a double-well potential, i.e. making the replacement~$-m^2\mapsto +m^2$ in~\eqref{eq:phi43_SPDE_mollified_unrenormalised}) was obtained by Faris and Jona-Lasinio~\cite{faris-jona-lasinio}.
In contrast, equation~\eqref{eq:phi43_SPDE_mollified_unrenormalised} with fixed~$\kappa = 0$ is naively ill-posed for any~$d \geq 2$. 
In fact, when~$d = 2$ and~$\eps > 0$ is fixed, Hairer, Ryser, and Weber~\cite{hairer-ryser-weber} showed that~$u_{\eps,\kappa}$ converges to $0$ in a space of distributions as~$\kappa \to 0$.
In contrast, it had already been proved by Cerrai and Freidlin~\cite{cerrai_freidlin_11_rde_quasi} on the level of large deviations that the rate function~$\msI_{0,T}^{(\kappa)}$ which governs the LDP of the random variables~$(u_{\eps,\kappa})_{\kappa > 0}$ $\Gamma$-converges \emph{in any dimension}~$d \geq 1$ to a limiting rate function\footnote{This is with the appropriate understanding that $\msI_{0,T}(v)$ is infinite whenever the integral on the right hand side is not finite, see~\eqref{eq:uniform_ldp_rf} in Appendix~\ref{sec:uniform_ldp} for details.}
\begin{equation} \label{eq:rate_function_phi43}
	\msI_{0,T}(v) = \frac{1}{2} \int_0^T \int_{\T^d} (\partial_t v - \Delta v + v^3 + m^2 v)^2 \dif x \dif t
\end{equation}  
as~$\kappa \to 0$.
For~$d = 1$, this functional does indeed coincide with the rate function obtained by Faris and Jona-Lasinio \cite{faris-jona-lasinio}, but when~$d \geq 2$, the problem of finding a suitable stochastic dynamics~$\bar{u}_{\eps,\kappa}$, whose limit~$\bar{u}_\eps$ (as $\kappa \to 0$) is both non-trivial and whose large deviations are governed by~$\msI_{0,T}$, had eluded progress at the time. 
In dimension~$d \in \{2,3\}$, this problem was overcome by Hairer and Weber~\cite{hairer_weber_ldp} using the novel theory of regularity structures.
Indeed, they showed that~$\msI_{0,T}$ as given by \eqref{eq:rate_function_phi43} is the rate function for an LDP governing solutions~$(\bar{u}_{\eps,\kappa})_{\eps > 0}$ to a suitably \emph{renormalised} version of \eqref{eq:phi43_SPDE_mollified_unrenormalised},
\begin{equation}  \label{eq:phi43_SPDE_mollified_renormalised}
	\begin{cases}
		\partial_t \bar{u}_{\eps,\kappa}(t,x) - \Delta \bar{u}_{\eps,\kappa}(t,x) 
		= - \bar{u}_{\eps,\kappa}(t,x)^3 - \del[1]{m^2  -3 \eps^2 C_{\kappa}^{(1)} +9 \eps^4 C_{\kappa}^{(2)}} \bar{u}_{\eps,\kappa}(t,x) + \eps \xi_\kappa(t,x),&  \\
		\bar{u}_{\eps,\kappa}(0,\,\cdot\,) = \fz,
		\quad (t,x) \in [0,T] \x \T^d,
	\end{cases}
\end{equation}
where the smoothing parameter $\kappa$ is chosen as a non-negative function of the amplitude $\eps$ so that $\kappa(\eps) \to 0$ as $\eps \to 0$. 
Importantly, their result includes the case~$\kappa = 0$, for which the process~$\bar{u}_{\eps}$ only formally solves the stochastic PDE given by~\eqref{eq:phi43_formal_SPDE}.\footnote{This rigorous limiting procedure is more fully described in Appendix~\ref{sec:technical_proofs}.}
In the case~$d=2$, this extends earlier partial results by Jona-Lasinio and Mitter~\cite{jona-lasinio_mitter_ldp_phi42}.
For context, we point out that the article~\cite{hairer_weber_ldp} also considers the situation when, additionally to~$\kappa(\eps) \to 0$, one has~$\eps \kappa(\eps)^{-1} \to \lambda^2 \in [0,\infty)$ as~$\eps \to 0$; in that case, the \emph{un}renormalised dynamics~$(u_{\eps,\kappa(\eps)})_{\eps > 0}$ given in~\eqref{eq:phi43_SPDE_mollified_unrenormalised} satisfy an LDP with good rate function similar to~\eqref{eq:rate_function_phi43}, but with some explicit effective mass~$m_\lambda$.
Cerrai and Debussche~\cite{cerrai_debussche_19_ldp_phi2nd} have recently generalised this result to the case of arbitrary polynomial non-linearities in any dimension~$d \geq 1$.

One way to view the result of \cite{hairer_weber_ldp} is to show that the $\Gamma$-convergence result of~\cite{cerrai_freidlin_11_rde_quasi}  correctly predicted that~$\msI_{0,T}$ should be the rate function of a limiting dynamics~$\bar{u}_\eps$ that is independent of~$\kappa$. The current article demonstrates that a similar result holds at the level of the invariant measure. Starting from the dynamic rate functions $\msI^{(\kappa)}_{[0,T]}$ (resp. $\msI_{[0,T]}$) the natural candidate rate function for an LDP governing the invariant measure of $u_{\eps,\kappa}$ solving \eqref{eq:phi43_SPDE_mollified_unrenormalised} (resp. $\bar{u}_\eps \coloneqq \lim_{\kappa \to 0} \bar{u}_{\eps,\kappa}$ solving \eqref{eq:phi43_SPDE_mollified_renormalised}) are the \emph{quasi-potentials}\footnote{Note that we do not impose any regularity constraints on the functions~$v$ in~\eqref{eq:kappa_quasi_potential} and~\eqref{eq:quasi_potential}: Simply requiring~$\msV^{(\kappa)}(v)$ resp. $\msV(v)$ to be finite already ensures the requisite regularity.}
\begin{align}\label{eq:kappa_quasi_potential}
	\msV^{(\kappa)} (\fz) \coloneqq 
	\inf \cbr[1]{\msI^{(\kappa)}_{0,T}(v)\,:\, T > 0, \, v:[0,T]\times \mbT^3 \to \mbR, \, v(0) = 0, \, v(T) = \mfz} 
\end{align}
and 
\begin{align}\label{eq:quasi_potential_intro}
	\msV(\mfz) \coloneqq 
	\inf \cbr[1]{\msI_{0,T}(v)\,:\, T > 0, \, v:[0,T]\times \mbT^3 \to \mbR, \, v(0) = 0, \, v(T) = \mfz}.
\end{align}

Given~\cite[Thm.~6.1]{cerrai_freidlin_11_rde_quasi}, in which the authors show that $\msV^{(\kappa)}$ converges pointwise to~$\msV$ in~$H^1(\mbT^3)$, it is natural to conjecture that~$\msV$ governs the large deviations of the family of measures $\mu_\eps$ which are invariant for the renormalised dynamics~$\bar{u}_\eps$ (i.e.  \eqref{eq:phi43_SPDE_mollified_renormalised} with $\kappa=0$).
A consequence of our work, similar in this regard to \cite{hairer_weber_ldp}, is to show that is indeed the case.
Furthermore, we identify $\msV$ as being proportional to the Euclidean $\Phi^4_3$ action~$\msS$ defined in~\eqref{eq:phi43_action} above; 
this is natural to expect since it has already been shown that the invariant measure~$\mu_\eps$ for the dynamics~$\bar{u}_\eps$ is exactly the Euclidean $\Phi^4_3$ measure.
We give more details of our proof strategy and this identification, in particular, in the next paragraph.

Let us further mention that the same LDP for the periodic $\Phi^4_3$ measure\footnote{In two dimensions, the analogous result on the infinite volume is proved in~\cite{barashkov_gubinelli_22_variational_volume}.} has been obtained in~\cite[Thm.~2.40]{barashkov_phd_thesis} in which it is phrased equivalently as a Laplace principle, see~\eqref{eq:laplace_principle} above;
the proof then uses weak convergence arguments in the spirit of Dupuis and Ellis~\cite{dupuis_ellis_97_weak} which are based on the Boué--Dupuis formula as in~\cite{barashkov_gubinelli_20_variational}.
We maintain two motivations for re-deriving the result via stochastic quantisation: Firstly, our method builds on the systematic approach of regularity structures~\cite{hairer_rs, bhz, chandra_hairer, rs_renorm} that is applicable to a wide class of singular SPDEs and requires only a small number of additional, model-specific ingredients. 
In principle, this would make it possible to obtain analogous results for a large class of models using a similar approach. Secondly, a major goal of constructive QFT is to give rigorous meaning to gauge theories such as Yang--Mills(--Higgs) theories which form an integral part in the standard model of particle physics. 
At present, a variational formulation similar to~\cite{barashkov_gubinelli_20_variational} is not available for gauge theories of this type. In contrast, however, the approach via stochastic quantisation has shown success in applications to gauge theories, see Shen~\cite{shen_abelian}, Chandra et al.~\cite{chandra_chevyrev_hairer_shen_22_2dYM,chandra_chevyrev_hairer_shen_22_3dYMH} (and Chevyrev's article~\cite{chevyrev_review_ym23d} reviewing the two works) as well as the recent work of Chevyrev and Shen~\cite{chevyrev_shen_23}.
This, and other possible future directions are further discussed in the final paragraph of the introduction.

\vspace{-0.5em}
\paragraph{Strategy of Proof.} 
We give a short description of our proof strategy, highlighting in particular where we appeal to model specific properties and where we make use of more general tools. 
For the most part, our strategy is inspired by the above-mentioned works of Cerrai and R\"ockner~\cite{cerrai_2003,cerrai_roeckner_ldp_dynamics,cerrai_roeckner_ldp_measure} in which the authors deal with structurally more general but non-singular equations.
One of the main challenges in their setting, and in ours, is to find a good representation of the quasi-potential~$\msV$ that is easier to control than the expression in~\eqref{eq:quasi_potential_intro}.
In a first step, we observe that the definition of~$\msI_{0,T}$ immediately implies that one can represent the quasi-potential~$\msV$ as
\begin{equation} \label{eq:quasi_potential_skeleton}
	\msV(\fz) \coloneqq \inf \cbr[3]{\frac{1}{2}\norm[0]{h}^2_{L^2_TL^2_x}\,:\, T > 0, \, w^h(T) = \mfz}
\end{equation}
where $\fz$ is in the H\"older--Besov space $\mcC^\alpha(\mbT^3)$ for $\alpha < -1/2$, the support of the invariant measure~$\mu_\eps$, and $w^{h}$ solves the \emph{skeleton equation},
\begin{equation}\label{eq:nl_skeleton_intro}
	\partial_t w^{h} - \Delta w^{h} = -(w^{h})^3-m^2 w^{h} + h,\qquad 
	w^{h}\tzero = 0,
\end{equation}
on $[0,T]$ where the driver, $h\in L^2_{[0,T]}L^2(\mbT^3)$, is an element of the Cameron--Martin space associated to the space-time white noise.
Crucially for our purposes, the rate function~$\II_T$ in~\eqref{eq:rate_function_phi43} is independent of the requisite infinite renormalisation; 
in turn, this makes \eqref{eq:nl_skeleton_intro} the correct skeleton equation, even for the renormalised dynamics~\eqref{eq:phi43_formal_SPDE}. 
In~\cite{cerrai_roeckner_ldp_measure}, the authors further simplify the expression for~$\msV$ by appealing to so-called ancient solutions to~\eqref{eq:nl_skeleton_intro}; broadly speaking, they fix the common time horizon~$(-\infty,0]$ for the infimisation problem in~\eqref{eq:quasi_potential_skeleton} by demanding that~$w^h$ hits~$\fz$ at time~$0$ and decays to zero as~$t \to -\infty$. 

In our case, we eschew this argument and instead make use of the additive noise and gradient flow structure of~\eqref{eq:phi43_formal_SPDE} to more directly show that the \emph{quasi-potential} and \emph{action functional} coincide, i.e.
\begin{equation}\label{eq:action_quasipotential_intro}
	\msV(\mfz) = 2 \msS(\mfz)\quad \text{for all}\quad \mfz \in \CC^\alpha(\mbT^3).
\end{equation}
We expect that a similar strategy will apply for other physically relevant models, see the paragraph on future directions below and Section~\ref{sec:rate_functional} for details.
We point out that~\eqref{eq:action_quasipotential_intro} is the direct analogue of the finite-dimensional result by Freidlin and Wentzell~\cite[Chap.~4.3, Thm.~3.1]{freidlin-wentzell} but, to the best of our knowledge, our method of proof seems new, even in that case. 
Since the action functional $\msS$ is only finite when restricted to the Sobolev space~$H^1(\mbT^3)$, a direct consequence of \eqref{eq:action_quasipotential_intro} is that the quasi-potential~$\msV$ has compact sub-level sets in $\mcC^{\alpha}(\mbT^3)$.
That is, $\msV =2\msS$ is a \emph{good} rate functional, a fact that is required in the proof of the LDP upper bound, see Section~\ref{sec:LDP_upper}.

We then combine the techniques of Cerrai and Röckner~\cite{cerrai_roeckner_ldp_measure} with the LDP for the $\Phi^4_3$ equation as obtained by Hairer and Weber~\cite{hairer_weber_ldp}, see Section~\ref{sec:uniform_ldp} for a recap.
In fact, we need a mild generalisation of their result which is locally uniform in the initial condition, which we prove in Proposition~\ref{prop:uniform_ldp} in Appendix~\ref{sec:uniform_ldp}.
We view it as an advantage of this approach that the techniques in~\cite{hairer_weber_ldp} are applicable to a large class of singular stochastic dynamics with minimal change, essentially all equations which are amenable to the theory of regularity structures. We have aimed to continue this philosophy in our article, appealing to as few model specific facts as possible. 
Essentially, our analysis relies only on two conditions, both of which are reasonable to expect from other physical models: 
\begin{enumerate}[label=(\arabic*)]
	\item \label{model_specific:1} Sufficient \emph{tail estimates on the invariant measure}. 
	In our case, we appeal to the work of Moinat and Weber~\cite{moinat_weber_20_phi43Loc} which gives stretched exponential moments~(see Propositions~\ref{prop:cdfi_model} and~\ref{prop:measure_tail_bounds} below for a recap).
	We stress that their work \emph{only} relies on analysing the dynamics and does not appeal to any a priori information on the invariant measure. While these stretched exponential moments were upgraded to sub-Gaussian tail bounds in  \cite{hairer_steele_22}, also appealing to analysis of the dynamics, we point out that these stronger tail estimates are not required for our method. In fact, these sub-Gaussian tail bounds are also stronger than those required to verify the OS axioms. 
	In our proof, the stretched exponential tail estimates are required both in the proof of the LDP lower bound, see Section~\ref{sec:LDP_lower}, and that of the LDP upper bound, see Section~\ref{sec:ldp_upper_bound_proof}.
	\item \label{model_specific:2} \emph{Asymptotic stability of the deterministic dynamics}. In our case, we are able to show decay of the classical, deterministic dynamics (i.e. \eqref{eq:phi43_SPDE_mollified_unrenormalised} with $\eps =0$) to a unique, asymptotically stable solution as well as long time decay in $H^1$ of solutions to the associated skeleton equation, \eqref{eq:nl_skeleton_intro}. The former is guaranteed by the complimentary signs of the non-linear and linear terms in \eqref{eq:phi43_formal_SPDE} (see Lemma~\ref{lem:action_zeros}) while we only rely on the positive mass parameter to obtain long time decay of the skeleton equation, see Remark~\ref{rem:nl_skeleton_CDFI}. With regard to generalisations of the former we expect that if minimisers of the classical dynamics form a single connected component in the state space then our approach should be extendable without major change. In the case of non-unique, disconnected minimisers  some additional arguments may be required. The second property, of skeleton solution decay in $H^1$, should be similarly adaptable to other relevant cases.
\end{enumerate}
We close this paragraph by highlighting the main ways in which the \emph{singular} nature of \eqref{eq:phi43_formal_SPDE} requires us to improve upon the methods of \cite{cerrai_roeckner_ldp_dynamics,cerrai_roeckner_ldp_measure}. The main point of departure from those works stems from the low regularity of the invariant measure support, which is the natural state space for our LDP, $\mcC^{\alpha}(\mbT^3)$. A key, related issue is that the LDP rate functional and $\Phi^4_3$ action functional are both infinite on this state space, and only finite on the solution space of the skeleton equation, i.e. paths taking values in $H^1(\mbT^3)$. This causes technical challenges, for example, in the proof of the LDP upper bound, Theorem~\ref{thm:ldp_upper_bound} and requires more careful analysis of the skeleton equation in scale of weighted, path space norms, allowing for prescribed blow-up as $t\to 0$. See the proof of Lemma~\ref{lem:energy_bounds} and attendant discussions. On the other hand, it is to our benefit that despite the singular nature of the stochastic equation, the skeleton equation remains classically well-posed. Appealing to suitable Sobolev embeddings one sees that the cubic non-linearity is well defined for weak solutions to \eqref{eq:nl_skeleton_intro} in $C_T H^1(\mbT^3)$. A final, major, difference between ours and the previous work \cite{cerrai_roeckner_ldp_measure} is that, taking advantage of the gradient flow structure at hand, we are able to identify the limiting rate function for the invariant measures as the $\Phi^4_3$ action, Section~\ref{sec:rate_functional}. As discussed above, this identification greatly simplifies later proofs, in particular, since it immediately implies compactness of sublevel sets for the quasipotential. We point out here that our proof of this identification requires global, exact controllability of the skeleton equation, see Appendix~\ref{sec:skeleton_controllability} and the proof of Theorem~\ref{th:V_equal_2S}.
\paragraph{Future Directions.}

As described above, we expect the strategy developed in the current paper to be  applicable, at least in principle, to a number of other physically relevant theories beyond $\Phi^4_3$ and also open the way to address more challenging questions.

\begin{itemize}
	\item A natural extension of our work would be to consider the~$\Phi^4$ model in the \emph{full subcritical regime}, that is, the family of dynamic $\Phi^4_{4-\iota}$ models for any~$\iota > 0$. Two potential approaches to studying analogues of \eqref{eq:phi43_formal_measure} for $\iota\in (0,1)$ (i.e. $d\in (3,4)$), through the stochastic quantisation equation, are to either consider \eqref{eq:phi43_formal_SPDE} with $\xi$ replaced by $\Delta^{s} \xi$ for suitably chosen $s >0$ or to replace the Laplacian on the left hand side with a fractional Laplacian of suitable order. By tuning either the regularity of the noise or the linear differential operator one can capture the equivalent ultra-violet singularities present in the model as $d\to 4$. Note that in the case of modified Gaussian noise the appropriate local theory is covered by \cite{rs_renorm}. This approach was also taken in \cite{chandra_moinat_weber_23}, obtaining suitable a priori bounds on solutions to an appropriate modification of \eqref{eq:phi43_formal_SPDE}. More recently \cite{duch_gubinelli_rinaldi_24_parabolic,esquivel_weber_24_apriori} studied the case of modified differential operator, which leads to a more physically relevant construction of the measure in fractional dimensions. Taking an approach via the flow equation, \cite{duch_gubinelli_rinaldi_24_parabolic} showed existence but not uniqueness of the appropriate measure on the whole volume and obtained reflection positivity, translation invariance and quartic tail bounds following the core ideas of \cite{hairer_steele_22}. Meanwhile, \cite{esquivel_weber_24_apriori} appeals to a regularity structures formulation to handle the ultra-violet singularities and similar methods as \cite{chandra_moinat_weber_23} to obtain global, a priori estimates. Note that the case of fractional Laplacian studied by \cite{esquivel_weber_24_apriori}, the local theory is not exactly covered by existing regularity structures frameworks due to the fractional Laplacian. In summary, while many of the necessary ingredients are already available for the subcritical $\Phi^4$ equation, much work would remain to apply our strategy in full.
	´´
	\item A related, potential extension would be to the Euclidean $\Phi^4_3$ measure constructed on a compact Riemannian manifold, see the recent series of works by Bailleul et al.~\cite{bailleul_dang_ferdinand_to_23_phi43_harmonic,bailleul_dang_ferdinand_to_23_manifold_measures,bailleul_23_phi43_uniqueness}. In the same direction we also refer to Hairer and Singh~\cite{hairer_singh_23_RegStruct_manifold} for a treatment of more general singular SPDEs on compact, Riemannian manifolds. Here, we expect the main challenges to be largely technical rather than fundamental.
	\item
	More challenging further extensions one might consider are to the sine-Gordon and exponential models, \cite{albeverio_deVecchi_gubinelli_21_elliptic,gubinelli_hofmanova_rana_23_exp_decay,hairer_shen_16_dyn_sineGordon,chandra_hairer_shen_18_sineGordon_subCrit} as well as the recently studied Langevin dynamics for \emph{tensor field theories}, see Chandra and Ferdinand~\cite{chandra_ferdinand_23_tensor}. With regards the sine-Gordon model, we mention that large deviation principles as well as other important properties have been derived for these models on the infinite volume $\mbR^2$ via an FBSDE, stochastic control approach, first by Barashkov \cite{barashkov_22_stochcontrol_sine_gordon} in the parameter regime $\beta^2\in (0,4\pi)$ and recently by Gubinelli and Meyer \cite{gubinelli_meyer_24_sineGordon} in the regime $\beta^2\in (0,6\pi)$.  Concerning approaches through the stochastic quantisation equation, \cite{bringmann_cao_24_sineGordon,chandra_feltes_weber_24_apriori} recently obtained suitable a priori estimates and global well-posedness of the dynamic sine-Gordon equation on $\mbT^2$ first in the regime $\beta^2 \in (0,16\pi/3)$ by \cite{chandra_feltes_weber_24_apriori} then in the regime $\beta^2 \in (0,6\pi)$ by \cite{bringmann_cao_24_sineGordon}. Note that local existence is provided in the full sub-critical regime $\beta^2 \in (0,8\pi)$ by the theory of regularity structures, see \cite{chandra_hairer_shen_18_sineGordon_subCrit}.

	\item As originally suggested by Parisi and Wu~\cite{parisi-wu}, the method of stochastic quantisation has recently lead to progress in the rigorous construction and analysis of~\emph{gauge theories}, see the above-mentioned series of works~\cite{shen_abelian, chandra_chevyrev_hairer_shen_22_2dYM,chandra_chevyrev_hairer_shen_22_3dYMH, chevyrev_review_ym23d, chevyrev_shen_23} and also Chatterjee and Cao's independent approach~\cite{cao_chatterjee_1, cao_chatterjee_2} that lead to similar results. 
	A fundamental obstacle, at the time of writing, to applying the methodology of the current paper to semi-classical limits for gauge theories is the lack of suitable a piori global estimates for the associated stochastic quantisation equations. In \cite{chevyrev_shen_23} the authors consider the $2$D Yang--Mills measure and construct a suitable Markov process, which exists globally in time and is invariant under the same measure. However, this procedure does not construct the measure by flowing the stochastic quantisation equation to $t=+\infty$. Note also that a large deviation principle for the pure $2$D Yang--Mills measure on a compact manifold has already been obtained by \cite{levy_norris_06}, albeit on a less refined state space than that constructed in \cite{chevyrev_19_YM} on $\mbT^2$. On the other hand \cite{bringmann_cao_24_abelianHiggs} has recently obtained suitable a priori estimates to show global well-posedness of the stochastic quantisation equation associated to the Euclidean Abelian Yang--Mills--Higgs equation.
	\item The programme of stochastic quantisation has also shown success  in the construction of Euclidean, Fermionic theories, see recent work by Chandra, Hairer and Peev \cite{chandra_hairer_peev_23_yukawa} giving a stochastic quantisation approach to the construction of a $1$-dimensional, mixed Fermionic-Bosonic quantum field theory. While we do not expect any simple modification of the present arguments to be directly applicable to Fermionic theories, we simply wish to highlight the general success of stochastic quantisation in constructing physical quantum field theories, as well as other, rapidly developing methodologies,  \cite{devecchi_fresta_gubinelli_22_fermionic,albeverio_borasi_deVecchi_gubinelli_22_grassmannian}.
\end{itemize}		
\paragraph{Organisation of this Article.}
In the remainder of this section we describe some frequently used notation and useful preliminaries, Section~\ref{sec:notation}. The following three sections of the paper are dedicated to the proof of Theorem~\ref{th:main_result}. In Section~\ref{sec:rate_functional} we show that the quasi-potential and $\Phi^4_3$ action coincide. Section~\ref{sec:LDP_lower} gives a self-contained proof of the LDP lower bound while Section~\ref{sec:LDP_upper} does the same for the LDP upper bound. This concludes the proof of our main result. These sections are supplemented by two appendices. In Appendix~\ref{sec:skeleton_equation} we provide necessary analysis of the skeleton equation; the main points being global well-posedness, long time decay and global, exact controllability. Appendix~\ref{sec:technical_proofs} provides a brief recap of the necessary aspects of regularity structures for the dynamical $\Phi^4_3$ equation as well as proofs of necessary ancillary results; a locally uniform version of Hairer--Weber's dynamic LDP, \cite{hairer_weber_ldp}, and suitable asymptotic tail bounds on $\mu_{\eps}$ in terms of $\eps\in (0,1)$.

\subsection{Notations, Conventions and Preliminaries} \label{sec:notation}
In this subsection, we collect notational conventions and preliminaries.
\paragraph{Sets and Sequences} We write $\mbZ$ for the set of integers, $\mbN \coloneqq  \{1,\ldots\}$ for the strictly positive integers, $\mbN_0 \coloneqq  \{0\}\cup \mbN$ for the non-negative integers, $\mbR$ for the real numbers and $\mbR_+ \coloneqq  [0,\infty)$ for the non-negative, real numbers. We set $\mbT \coloneqq  \mbR/\mbZ$ and define the three dimensional torus $\mbT^3 \coloneqq  (\mbR/\mbZ)^3$. Given a Banach space $E$ and an index set $I\subseteq \mbN$ we use the (slight) abuse of notation $(a_n)_{n\in I}\subset E$ to denote a sequence of $E$ valued elements. When the context is clear we simply write $(a_n)_{n\in I}$. Given a real number $a\in \mbR$, we occasionally use the shorthands
\begin{equation*}
	a- \coloneqq \{ b\in \mbR\,:\, b<a\}\quad \text{and}\quad a+ \coloneqq \{b\in \mbR\,:\, b>a\}.
\end{equation*}
\paragraph{Inequalities, Embeddings and Operators} When stating an inequality we either write $A \lesssim _{a,\,b,\ldots} B$ to indicate that the inequality holds up to a constant depending on the parameters $a,\,b,\,\ldots$ or write that there exists a $C\coloneqq  C(a,b,\ldots)>0$ such that $A\leq CB$. If we write $A\lesssim B$ without explicit dependents we mean that the inequality holds up to a constant depending on parameters that we do not keep track of. These constants may always change in absolute size from line to line. If we write $A \simeq B$ we mean that there exists a constant $c\geq 1$ such that
\begin{equation*}
	\frac{1}{c} B \leq A \leq c  B.
\end{equation*}
Given two topological spaces, $X,\,Y$, we write $X \embed Y$ to denote that $X$ embeds continuously into $Y$ and $X \cembed Y$ to denote that $X$ embeds compactly into $Y$.

Given a metric space $(X,d_X)$, $x \in X$ and $\lambda \in \mbR$ we write,
\begin{equation*}
	B^X_\lambda(x) \coloneqq \{ y \in X\,:\, d_X(x,y) <\lambda\}\quad \text{and}\quad \bar{B}^X_\lambda(x)\coloneqq \{y\in X\,:\, d_X(x,y)\leq \lambda\}.
\end{equation*}
When the context is clear we will simplify our notation to $B_\lambda(x),\,\bar{B}_\lambda(x)$ and when $x=0$ we simplify even further, writing, $B_\lambda,\, \bar{B}_\lambda$.

Given two Banach spaces $X,\,Y$, we write $\mcL(X,Y)$ for the set of bounded, linear operators from $X$ to $Y$. As is standard we write $X^* \coloneqq \mcL(X,\mbR)$.

\paragraph{Lebesgue Spaces on $\mbT^3$} For $p\in[1,\infty)$ (respectively $p=\infty$) we write $L^p(\mbT^3)$ for the space of $p$-integrable (respectively essentially bounded) functions $f:\mbT^3 \to \mbR$, equipped with the norms,
\begin{equation*}
	\|f\|_{L^p_x} \coloneqq  \begin{cases}
		\left(\int_{\mbT^3}|f(x)|^p \dd x\right)^{\nicefrac{1}{p}}, & p\in [1,\infty),\\
		\esssup_{x\in \mbT^3} |f(x)|,&p=\infty.
	\end{cases}
\end{equation*}
Correspondingly, for $p\in [1,\infty)$ (respectively $p=\infty$) we write $\ell^p(\mbZ^3)$ for the space of $p$-summable (respectively bounded), real valued sequences, indexed by $\mbZ^3$.
\paragraph{Differentiable Functions and Distributions} Given $k\in \mbN_0$ we write $C^k(\mbT^3)$ for the functions $f:\mbT^3\to \mbR$ which are $k$-times continuously differentiable, we define $C^\infty(\mbT^3) \coloneqq  \cap_{k \in \mbN_0} C^k(\mbT^3)$ and write $\mcS'(\mbT^3)$ for the set of distributions on $\mbT^3$, identified as the dual of $C^\infty(\mbT^3)$. For $0\leq s<t<\infty$, the distributions $\mcS'([s,t]\times \mbT^3)$ are defined analogously.
\paragraph{Fourier Transform on $\mbT^3$} For $f\in L^1(\mbT^3)$ and $g\in \ell^1(\mbZ^3)$ we define the Fourier transform 
\begin{equation*}
	\mcF(f)(k) \coloneqq   \int_{\mbT^3} f(x)e^{-2\pi i k \cdot x}\dd x \quad \text{and inverse}\quad \mcF^{-1}g(x) \coloneqq  \sum_{k\in \mbZ^3} g(k)e^{2\pi i k\cdot x}. 
\end{equation*}
For concision, where appropriate, we write $\hat{f}(k)\coloneqq  \mcF(f)(k)$ and recall the well-known isometry,
\begin{equation*}
	\|f\|_{L^2_x} = \|\hat{f}\|_{\ell^2_k} \coloneqq  \left(\sum_{k\in \mbZ^3} |\hat{f}(k)|^2\right)^{\nicefrac{1}{2}},
\end{equation*}
The Fourier transform naturally extends to $\mcS'(\mbT^3)$ by density and duality.
\paragraph{Sobolev Spaces on $\mbT^3$} Given $\alpha\in \mbR$, we let $H^\alpha(\mbT^3)$ denote the functions $f:\mbT^3\to \mbR$ (respectively distributions $f\in \msS'(\mbT^3)$ when $\alpha<0$) which are finite under the norm,
\begin{equation*}
	\|f\|_{H^\alpha_x} \coloneqq  \left\|(1+|\,\cdot\,|^2)^{\nicefrac{\alpha}{2}} \hat{f}(\,\cdot\,) \right\|_{\ell^2_k} = \left(\sum_{k\in \mbZ^3}(1+|k|^2)^{\alpha} \hat{f}(k) \right)^{\frac{1}{2}}.
\end{equation*}
For all $\alpha\in \mbR$, the space $H^{\alpha}(\mbT^3)$ is Hilbert, with the natural inner product and one  has the duality
\begin{equation*}
	\left(H^{\alpha}(\mbT^3)\right)^* = H^{-\alpha}(\mbT^3).
\end{equation*}
When $\alpha=0$, one has $H^\alpha(\mbT^3)=L^2(\mbT^3)$ and when $\alpha \in \mbN$, one has the equivalence,
\begin{equation*}
	\|f\|^2_{H^\alpha_x} \simeq \sum_{k=0}^\alpha \|\partial^k_x f\|^2_{L^2_x}.
\end{equation*}
Finally, for $\alpha' \geq \alpha$ one directly has
\begin{equation}\label{eq:sob_reg_embed}
	\|f\|^2_{H^\alpha_x} 
	\leq \|f\|^2_{H^{\alpha'}_x}.
\end{equation}
As is standard we use angle brackets $\langle\,\cdot\,,\,\cdot\,\rangle$ to denote both the inner product associated to a given Hilbert space as well as the canonical pairing between a Hilbert space and is dual. Where necessary, to avoid confusion we will decorate the brackets with suitable subscripts; for example
\begin{equation*}
	\langle f,g\rangle_{H^1_x;H^{-1}_x} \quad \text{or}\quad \langle f,g\rangle_{L^2_x;L^2_x}.
\end{equation*}
\paragraph{Besov Spaces on $\mbT^3$} For $\alpha \in \mbR$ and $p,\,q \in[1,\infty]$ we define the Besov space, $B^\alpha_{p,q}(\mbT^3)$ as the completion of the smooth functions, $C^\infty(\mbT^3)$, under the topology induced by the norm
\begin{equation*}
	\|f\|_{\mcB^\alpha_{p,q}} \coloneqq  \begin{cases}
		\left(\sum _{j\geq-1} 2^{jq\alpha}\|\Delta_j f\|_{L^p}^q\right)^{\frac{1}{q}}, & q \in [1,\infty),\\
		\sup_{j\geq -1} 2^{jq\alpha}\|\Delta_j f\|_{L^p},& q =\infty.
	\end{cases}
\end{equation*}
where $\Delta_j$ denotes the $j^{\text{th}}$ Littlewood--Paley projector, we refer to \cite[Ch.~2]{bahouri_chemin_danchin_11_fourier} for details. Note that in contrast to the Sobolev spaces, $H^\alpha(\mbT^3)$, we define the Besov spaces as \emph{completions under the norm} rather than \emph{finite under the norm}. This ensures that the Besov spaces are Polish. When $p=q=\infty$ we use the shorthand $\mcC^{\alpha}(\mbT^3) \coloneqq  \mcB^{\alpha}_{\infty,\infty}(\mbT^3)$.

For all $p\in [1,\infty]$ we have the compact embedding,
\begin{align}
	&\mcB^{\alpha'}_{p,q'}(\mbT^3) \cembed \mcB^{\alpha}_{p,q}(\mbT^3) &&\,\,\Rightarrow && \|f\|_{\mcB^{\alpha}_{p,q}}\leq \|f\|_{\mcB^{\alpha'}_{p,q'}},\quad \alpha'>\alpha\,\,\&\,\, \& \,\, q' \geq q,\label{eq:besov_reg_embed}
\end{align}
Recall that as a result, given an $f\in \mcS'(\mbT^3)$, it holds that
\begin{equation}\label{eq:besov_finite_norm_embed}
	\|f\|_{\mcB^{\alpha}_{p,q}}<\infty \quad \Rightarrow \quad f\in \mcB^{\alpha'}_{p,q}(\mbT^3) \quad \text{for all}\quad  \alpha' <\alpha. 
\end{equation}
Conversely, $f\in \mcB^\alpha_{p,q}(\mbT^3)$ implies $\|f\|_{\mcB^{\alpha}_{p,q}} <\infty$. 

We also have the following continuous embeddings, for all $\alpha,\,\alpha'\in \mbR$ $p,\,p'\in [1,\infty]$ and $q,\,q' \in [1,\infty]$,
\begin{align}
	&\mcB^{\alpha}_{p,q}(\mbT^3) \embed \mcB^{\alpha}_{p,q'}(\mbT^3) &&\iff && \|f\|_{\mcB^{\alpha}_{p,q'}}\lesssim \|f\|_{\mcB^{\alpha}_{p,q}}, &&q'>q, \label{eq:besov_micro_embed}\\
	&\mcB^{\alpha}_{p,q}(\mbT^3)  \embed \mcB^{\alpha-3\left(\nicefrac{1}{p}-\nicefrac{1}{p'}\right)}_{p',q}(\mbT^3)&& \iff && \|f\|_{\mcB^{\alpha-3\left(\nicefrac{1}{p}-\nicefrac{1}{p'}\right)}_{p',q}}\lesssim \|f\|_{\mcB^{\alpha}_{p,q}}, && p'>p, \label{eq:besov_integrabillity_embed}\\
	&\mcB^0_{p,1}(\mbT^3)\embed L^p(\mbT^3) \embed \mcB^{0-}_{p,\infty}(\mbT^3)&& \iff&&  \|f\|_{\mcB^{0-}_{p,\infty}}\lesssim \|f\|_{L^p_x} \lesssim \|f\|_{\mcB^{0}_{p,\infty}}, && p\in [1,\infty]. \label{eq:besov_Lp_embed}
\end{align}
It is readily checked through Plancherel's theorem that
\begin{equation*}
	\|f\|_{\mcB^{\alpha}_{2,2}} = \|f\|_{H^\alpha_x}\quad \text{for all}\quad \alpha \in \mbR.
\end{equation*} 
Hence, considering \eqref{eq:besov_finite_norm_embed}, \eqref{eq:besov_integrabillity_embed} and \eqref{eq:besov_reg_embed}, we have the embeddings
\begin{equation*}
	\mcB^\alpha_{2,2}(\mbT^3) \embed H^\alpha(\mbT^3)\embed \mcC^{\alpha^\prime}(\mbT^3)\quad \text{for all}\quad \alpha^{\prime} <\alpha-\frac{3}{2}
\end{equation*}
and
\begin{equation*}
	H^{\alpha}(\mbT^3) \embed B^{\alpha^{\prime}}_{2,2}(\mbT^3)\embed \mcC^{\alpha^\prime -\frac{3}{2}}(\mbT^3)\quad \text{for all}\quad \alpha^\prime <\alpha.
\end{equation*}
Our most common application of these embeddings will be to combine them with \eqref{eq:besov_Lp_embed} and \eqref{eq:besov_reg_embed} to see that for $\alpha \in (-\nicefrac{2}{3},-\nicefrac{1}{2})$, there exists a $q\in (\nicefrac{3}{2},2)$ such that
\begin{equation}\label{eq:H1_to_L3q_to_Calpha_embed}
	H^1(\mbT^3)\cembed L^{3q}(\mbT^3) \embed \mcC^{\alpha}(\mbT^3).
\end{equation}
Furthermore, one has the classical, continuous Sobolev embedding $H^1(\mbT^3)\embed L^6(\mbT^3) $ which has the important consequence,
\begin{equation*}
	f\in H^1(\mbT^3)  \quad \Rightarrow \quad f^3 \in L^2(\mbT^3).
\end{equation*} 
We refer to \cite{bahouri_chemin_danchin_11_fourier,gip,triebel_92_theory} for details and proofs of the above.
\paragraph{Ornstein--Uhlenbeck Semi-Group} For $m>0$, we define the action of the Ornstein--Uhlenbeck semi-group on $L^1(\mbT^3)$ by setting,
\begin{equation*}
	e^{t(\Delta-m^2)}f\coloneqq  \mcF^{-1}\left(e^{-2\pi it(|\,\cdot\,|^2+m^2) } \hat{f}(\,\cdot\,)\right),\quad \text{for all } t\geq 0.
\end{equation*}
The following two inequalities hold for distributions $f\in \mcS'(\mbT^3)$ and $m \geq 0$,
\begin{align}
	\|e^{t(\Delta-m^2)}f\|_{\mcB^{\beta}_{p,q}} \lesssim &\,e^{-tm^2}(1\vee t)^{-\frac{\beta-\alpha}{2}-\frac{3}{2}\left(\frac{1}{p'}-\frac{1}{p}\right)}\|f\|_{\mcB^{\alpha}_{p',q'}}, &&\beta\geq \alpha,\, p'\geq p,\, q\geq q',\label{eq:ou_reg}\\
	\|(1-e^{t(\Delta-m^2)})f\|_{\mcB^{\alpha}_{p,q}} \lesssim &\,\left((1\vee t)^{\frac{\beta-\alpha}{2}} + m^2(1\vee t)\right) \|f\|_{\mcB^{\beta}_{p,q}}, && \beta\geq\alpha.\label{eq:ou_identity_approx}
\end{align}
\paragraph{Banach Space Valued Paths} Given a Banach space, $E$ and a finite interval $[s,t]\subset \mbR$ for $0\leq s<t <\infty$, we define the space of continuous, $E$ valued paths, $C_{[s,t]}E$, equipped with the supremum norm,
\begin{equation*}
	\|f\|_{C_{[s,t]}E} \coloneqq  \sup_{r\in [s,t]} \|f(r)\|_E.
\end{equation*}
When, $s=0$, we simplify notation by writing $C_tE$. Given a weight $\eta>0$, and horizon $T>0$, we define the family of weighted, path spaces by setting,
\begin{equation*}
	\|f\|_{C_{\eta;T}E}\coloneqq  \sup_{t\in (0,T]} (t\wedge 1)^{\eta}\|f(s)\|_E, \quad	C_{\eta;t}E\coloneqq  \{f\in C_T E\,:\, \|f\|_{C_{\eta;T}E}<\infty \,\}.
\end{equation*}
Note, that given $\eta_1 <\eta_2$,
\begin{equation}\label{eq:weight_space_embed}
	C_{\eta_1;T}E \embed C_{\eta_2;T}E\quad \iff \quad	\|f\|_{C_{\eta_2;T}E}\,\lesssim_T\,\|f\|_{C_{\eta_1;T};E}.	\end{equation}
\paragraph{Bochner--Sobolev Spaces} We work with a pair of Bochner--Sobolev spaces. For $0\leq s <t <\infty$ and $\alpha \in \mbR$ we set,
\begin{equation*}
	\|f\|^2_{L^2_{[s,t]} H^\alpha(\mbT^3)} \coloneqq \int_{s}^t \|f(r)\|^2_{H^\alpha_x} \dd r,\quad 	L^2_{[s,t]}H^\alpha(\mbT^3) \coloneqq  \left\{f:[s,t]\to \mbT^3\,:\, 	\|f\|_{H^1_{[s,t]}H^\alpha_x}  <\infty\right\}
\end{equation*}
and
\begin{equs}
	\|f\|^2_{H^1_{[s,t]} H^{-1}(\mbT^3)} & \coloneqq \|f\|^2_{L^2_{[s,t]}H^{-1}_x} + \|\partial_t f\|_{L^2_{[s,t]} H^{-1}_x }, \\ 	H^1_{[s,t]}H^1(\mbT^3) & \coloneqq  \left\{f:[s,t]\to \mbT^3\,:\, 	\|f\|_{L^2_{[s,t]}H^1_x}  <\infty\right\},
\end{equs}
where $\partial_t f$ is understood in the weak sense.

As usual, when $s=0$ we simplify notation by writing $L^2_T H^\alpha(\mbT^3)$ and  $H^1_T H^1(\mbT^3)$.

Note that we have the trivial embedding $C_TH^1(\mbT^3) \embed L^2_TH^1(\mbT^3)$ as well as  the following interpolation result
\begin{equation*}
	f\in L^2_T H^1(\mbT^3)\cap H^1_T H^{-1}(\mbT^3) \quad \Rightarrow \quad f \in C_TL^2(\mbT^3)
\end{equation*}
and generalisation of the chain rule,
\begin{equation}\label{eq:hilbert_chain}
	\frac{\dd}{\dd t} \|f(t)\|^2_{L^2(\mbT^3)} = 2\langle  f(t),\partial_t f(t)\rangle_{L^2_x},\quad \text{for a.e.} \quad t\in [0,T].
\end{equation}
See \cite[Sec.~5.9.2, Thm.~3]{evans_10_partial}.
\paragraph{Rate Functions} Three rate functions appear (with some mild variations) in the main body of the paper,
\begin{itemize}
	\item $\msI_{[s,t]}:C_{[s,t]}\mcC^{\alpha}(\mbT^3)\to [0,\infty]$, for $0\leq s<t<\infty$,  is the rate function on path space associated to the $\Phi^{4}_3$ equation; see Appendix~\ref{sec:uniform_ldp} for a detailed definition.
	\item $\msV:\mcC^{\alpha}(\mbT^3)\to [0,\infty]$ is the associated \emph{quasipotential}; see  \eqref{eq:quasi_potential} for a detailed definition.
	\item $\msS:\mcC^{\alpha}(\mbT^3)\to [0,\infty]$ is the (natural extension of) the $\Phi^4_3$ action to a rate function; see \eqref{eq:phi43_action} for a detailed definition.
\end{itemize}
Given a Polish space $E$, a rate function $\msI : E\to [0,\infty]$ and $\theta \in \mbR$, we define the sub-level set,
\begin{equation*}
	\msI[\theta] \coloneqq  \{ f\in E\,:\, \msI(f) \leq \theta\}.
\end{equation*}
For example, we will write $\msI_{[s,t]}[\theta],\, \msV[\theta]$ and $\msS[\theta]$ for the respective sub-level sets. In the case of the dynamic rate function, when $s=0$ we simplify notation by writing, $\msI_t(f)$ and $\msI_t[\theta]$.

\paragraph{Generalised Gr\"onwall Inequality} We make use of the following generalisation of Gr\"onwall's inequality, which can be found with proof as \cite[Lem.~B.1]{tsatsoulis_weber_20_exponential}. Let $f:[0,T]\to \mbR$ be a measurable function, $m^2>0$, $a,\,b\in \mbR$ and $\sigma_1+\sigma_2<1$ be such that
\begin{equation*}
	f(t)\leq e^{-m^2 t}a + b \int_0^t e^{-m^2(t-s)}(t-s)^{-\sigma_1}s^{-\sigma_2} f(s)\, \dd s.
\end{equation*}
Then, there exist constants $c,\,C>0$ such that
\begin{equation}\label{eq:gen_gronwall}
	f(t)\leq C\exp\left(-m^2 t + c b^{\frac{1}{1-\sigma_1-\sigma_2}} t\right) a.
\end{equation}
\section{Quasipotential and $\boldsymbol{\Phi^4_3}$ Action Coincide} 
\label{sec:rate_functional}

Let us introduce re-introduce the quasi-potential, $\msV:\mcC^{\alpha}(\mbT^3)\to [0,\infty]$ associated to the $\Phi^4_3$ dynamics, 
\begin{equation} \label{eq:quasi_potential}
	\msV(\mfz) \coloneqq  \inf \left\{\msI_T(v)\,:\, T>0,\, v \in C_T\mcC^{\alpha}(\mbT^3), \, v(0)=0, \, v(T)=\fz\right\},
\end{equation}
Recalling the definition of the $\Phi^4_3$ action, $\msS$, given by \eqref{eq:phi43_action}, the main result of this section is to show that the functionals $\msV$ and $\msS$ agree up to a constant and have compact sublevel sets in $\mcC^{\alpha}(\mbT^3)$.

\begin{theorem} \label{th:V_equal_2S}
	For any $\mfz\in \mcC^{\alpha}(\mbT^3)$ one has
	\begin{equation}\label{eq:V_equal_2S}
		\msV(\mfz) = 2\msS(\mfz).
	\end{equation}
\end{theorem}
While this result is a central to our strategy, its proof is somewhat technical and therefore, we recommend that on a first reading, provided the result is appreciated, the proof can be safely skipped.

Before proving Theorem~\ref{th:V_equal_2S}, we state and prove a technical lemma regarding properties of the $\Phi^4_3$ action.
\begin{lemma} \label{lem:action_zeros}
	For $\msS:\mcS'(\mbT^3)\to [0,+\infty]$ the $\Phi^4_3$ action as defined by \eqref{eq:phi43_action} the following all hold,
	\begin{enumerate}[label = \roman*)]
		\item \label{it:S_pos_def} $\msS(\phi) = 0$ if and only if $H^1(\mbT^3)\ni\phi =0$ and $\msS(\phi) > 0$ for all~$\phi \neq 0 \in H^1(\mbT^3)$. 
		\item \label{it:S_Frech_diff} The restriction of $\msS$ to $H^1(\mbT^3)$ is Fr\'echet differentiable and at each $\phi \in H^1_x$, for any $\varphi \in H^1(\mbT^3)$ one has 
		\begin{equation}\label{eq:S_Frech_deriv}
			D\msS(\phi)(\varphi) = \langle D\msS(\phi),\varphi \rangle_{H^{-1}_x,H^1_x}
			= -\langle\Delta \phi,\varphi \rangle_{H^{-1}_x;H^{1}_x} + m^2 \langle \phi,\varphi\rangle_{L^2_x;L^2_x}+  \langle \phi^3,\varphi \rangle_{L^2_x;L^2_x}.
		\end{equation}
		As a result, we may use the shorthand,
		\begin{equation*}
			D\msS(\phi) = -\Delta \phi + m^2  \phi+   \phi^3 .
		\end{equation*}
		\item \label{it:S_global_min} In the operator theoretic sense, $D\msS(\phi) \equiv 0$ if and only if $\phi =0$ in $H^1(\mbT^3)$. I.e. $\phi= 0$ is the global minimizer of $\msS$.
	\end{enumerate}
\end{lemma}
\begin{proof}
	To prove \ref{it:S_pos_def} it suffices to recall the embeddings~$H^1(\mbT^3) \embed L^6(\mbT^3) \embed L^4(\mbT^3) \embed L^2(\mbT^3)$ so that for all~$\phi \in H^1_x$ one has
	\begin{equation*}
		\msS(\phi) = \frac{1}{2}\norm[0]{\nabla \phi}_{L^2_x}^2+ \frac{m^2}{2} \norm[0]{\phi}_{L^2_x}^2 + \frac{1}{4} \norm[0]{\phi}_{L^4_x}^4 <\infty.
	\end{equation*}
	The claim then follows directly from the positive definiteness property of norms. 
	
	To show \ref{it:S_Frech_diff}, let~$\phi, \psi \in H^1(\mbT^3)$, define~$A_\phi(\psi) \coloneqq  A^0_\phi(\psi) + A^1_\phi(\psi)$ by 
	\begin{equation*}
		A_\phi^0(\psi) 
		\coloneqq  \scal{-\Delta \phi + m^2 \phi,\psi}_{H^{-1}_x \x H^1_x}, \qquad 
		A^1_\phi(\psi) \coloneqq  \scal{\phi^3, \psi}_{L^2_x}
	\end{equation*}
	and then we claim that
	\begin{equation}
		\abs[0]{\msS(\phi + \psi) - \msS(\phi) - A_\phi(\psi)} = o\del[1]{\norm[0]{\psi}_{H^1_x}}\quad \text{as} \quad \norm[0]{\psi}_{H^1_x} \to 0.
		\label{lem:frech_diff_s_pf_claim}
	\end{equation}
	To this end, note that
	\begin{equation*}
		A_\phi^0(\psi) 
		= \scal{\nabla \phi, \nabla \psi}_{L^2_x} + m^2 \scal{\phi, \psi}_{L^2_x}
		= 2 B(\phi,\psi)
	\end{equation*}
	where the \emph{bilinear} form~$B: H^1_x \x H^1_x \to \R$ corresponds to the \emph{quadratic} part~$S_0$ in~$\msS$, i.e.
	\begin{equation*}
		B(\phi,\psi) = \frac{1}{2} \scal{\nabla \phi, \nabla \psi}_{L^2_x} + \frac{m^2}{2} \scal{\phi,\psi}_{L^2_x}, \qquad S_0(\psi) = B(\psi,\psi).
	\end{equation*}
	Therefore, we have
	\begin{equs}[][lem:frech_diff_s_pf_quad]
		\abs[0]{S_0(\phi + \psi) - S_0(\phi) - A^0_\phi(\psi)}
		& =
		\abs[0]{B(\phi + \psi, \phi + \psi) - B(\phi,\phi) - 2 B(\phi,\psi)} \\
		& =
		B(\psi,\psi) 
		= 
		S_0(\psi)
		\simeq
		\norm[0]{\psi}_{H^1_x}^2
	\end{equs}
	For the \emph{quartic} part,~$S_1$ of~$\msS$, we find that 
	\begin{equs}
		S_1(\phi + \psi) - S_1(\phi) - A^1_\phi(\psi)
		& =
		\frac{1}{4} \int_{\T^3} 6 \phi(x)^2 \psi(x)^2 + 4 \phi(x) \psi(x)^3 + \psi(x)^4 \dif x \\
		& \simeq 
		\scal{\phi^2, \psi^2}_{L^2_x} + \scal{\phi, \psi^3}_{L^2_x} + \norm[0]{\psi}^4_{L^4_x}
	\end{equs}
	and thus 
	\begin{equation}
		\abs[0]{S_1(\phi + \psi) - S_1(\phi) - A^1_\phi(\psi)}
		\leq
		\norm[0]{\phi}_{L^4_x}^2 \norm[0]{\psi}_{L^4_x}^2 + \norm[0]{\phi}_{L^2_x} \norm[0]{\psi}_{L^6_x}^3 + \norm[0]{\psi}_{L^4_x}^4 
		= 
		o\del[1]{\norm[0]{\psi}_{H^1_x}}, 
		\label{lem:frech_diff_s_pf_quart}
	\end{equation}
	as~$\norm[0]{\psi}_{H^1_x} \to 0$. Note that we have used the Cauchy--Schwarz inequality and the same embeddings as above.
	The claim in~\eqref{lem:frech_diff_s_pf_claim} now follows from~\eqref{lem:frech_diff_s_pf_quad} and~\eqref{lem:frech_diff_s_pf_quart}.
	
	Finally, to obtain \ref{it:S_global_min}, we begin with an auxiliary observation.
	Let~$b: \R \to \R$ be given by~$b(x) = x^3 + m^2 x$ and observe that
	\begin{equation*}
		\frac{\dif}{\dif x} b(x) = 3 x^2 + m^2 \geq 0, \quad x \in \R. 
	\end{equation*}
	Thus,~$b$ is (strictly) increasing and we see that
	\begin{equation}
		\del[0]{b(x) - b(y)}\del[0]{x-y} \geq 0, \quad x,y \in \R. 
		\label{lem:elliptic_eq_unique_sol_pf_monotonicity}
	\end{equation} 
	From \ref{it:S_Frech_diff} it follows that the zeros of $D\msS$ in $H^1(\mbT^3)$ are characterised by weak solutions of the elliptic PDE problem,
	\begin{equation}
		-\Delta \phi = -\phi^3 - m^2 \phi.
		\label{eq:elliptic_eq}
	\end{equation}
	It is obvious that~$\phi \equiv 0$ is a solution to~\eqref{eq:elliptic_eq}, so we only have to prove uniqueness.
	For two, distinct, solutions~$\phi, \psi \in H^1(\mbT^3)$ of \eqref{eq:elliptic_eq} we have
	\begin{equs}
		\norm[0]{\nabla(\phi-\psi)}^2_{L^2_x}
		& = 
		\scal{\nabla(\phi-\psi),\nabla(\phi-\psi)}_{L^2_x} \\
		& = -\scal{\Delta(\phi-\psi),\phi-\psi}_{H^{-1}_x \x H^1_x} \\
		& = - \scal{\phi^3 - \psi^3 + m^2(\phi-\psi), \phi-\psi}_{L^2_x} \\
		& = - \int_{\T^3} \del[1]{b(\phi(x)) - b(\psi(x))} \del[1]{\phi(x) - \psi(x)} \dif x \\
		& \leq 0,
	\end{equs}
	where we have used integration by parts on~$\T^3$ in the second line, eq.~\eqref{eq:elliptic_eq} in the third line and eq.~\eqref{lem:elliptic_eq_unique_sol_pf_monotonicity} in the last line. It thus follows that~$\norm[0]{\nabla(\phi-\psi)}^2_{L^2_x} = 0$.
	By the Poincaré inequality, we finally have
	\begin{equation*}
		0 \leq \norm[0]{\phi-\psi}^2_{L^2_x} 
		\lesssim \norm[0]{\nabla(\phi-\psi)}^2_{L^2_x} 
		= 0
	\end{equation*}
	and therefore~$\phi = \psi$ in~$H^1(\mbT^3)$.
\end{proof}

We are now ready to prove Theorem~\ref{th:V_equal_2S}, which claims that the quasipotential~$\msV$ associated to the $\Phi^4_3$ dynamics and the $\Phi^4_3$ action~$\msS$ coincide up to a constant.

\begin{proof}[Proof of Theorem~\ref{th:V_equal_2S}]
	By definition $\msS(\phi)=+\infty$ if $\phi \in \CC^\alpha(\T^3) \setminus H^1(\T^3)$, while it follows from Theorem~\ref{th:nl_skeleton_gwp} that the same holds for $\msV$. Hence, if~$\phi \in \CC^\alpha(\T^3) \setminus H^1(\T^3)$, 
	\begin{equation*}
		\msV(\phi) = +\infty = \msS(\phi)
	\end{equation*}
	and there is nothing to show. Therefore, we need only consider ~$\phi \in H^1(\T^3)$. 	
	Firstly, we observe that given~$T > 0$ and any~$v \in L^2_TH^1(\mbT^3)\cap H^1_TH^{-1}(\mbT^3)$, it follows the chain rule in Hilbert spaces (c.f. \cite[Sec.~5.9.2, Thm.~3]{evans_10_partial} and \eqref{eq:hilbert_chain}), along with the explicit expression for the Fr\'echet derivative of the action, \eqref{eq:S_Frech_deriv}, that 
	\begin{equs}[][prop_rf_action_pf_eq1]
		\msS(v(t))-\msS(v(0)) & = \int_0^t \langle D\msS(v(s)),\partial_t v(s)\rangle_{H^{-1}_x;H^{-1}_x}\dd s \\
	\end{equs}
	
	From Proposition~\ref{prop:nl_skeleton_control} there exists an~$h  \in L^2_T L^2(\mbT^3)$ such that, $w^{h,0}$ solving
	\begin{equation*}
		\partial_t w^{h,0} =  \Delta w^{h,0} - \del[0]{w^{h,0}}^3  - m^2 w^{h,0} + h = -D\msS(w_t^{h,0}) + h, \quad w^{h,0}(0) = 0
	\end{equation*}
	satisfies~$w^{h,0}(T) = \phi$.

	Then, by Theorem~\ref{th:nl_skeleton_gwp} it holds that $w^{h,0} \in C_TH^1(\mbT^3)\cap H^1_T H^{-1}(\mbT^3) \embed L^2_TH^1(\mbT^3)\cap H^1_T H^{-1}(\mbT^3) $ and by~\eqref{prop_rf_action_pf_eq1} along with suitably integrating by parts one finds
	\begin{equs}[][prop_rf_action_pf_eq2]
		\msS(\phi)
		=
		\msS(\phi) - \msS(0)
		&	=
		\msS(w^{h,0}(T)) - \msS(w^{h,0}(0)) \\
		& =
		- \int_0^T \norm[0]{D \msS(w^{h,0}_t)}_{H^{-1}_x}^2 \dif t + \int_0^T \scal{D \msS(w^{h,0}_t),h_t}_{H_x^{-1};H^1_x} \dif t.
	\end{equs}
	\paragraph*{$\boldsymbol{\triangleright}$ \textbf{Proof that} $\boldsymbol{2S \leq V}$.}
	By Cauchy--Schwarz, Young's product inequality,\footnote{We use the inequality in the form~$ab \leq a^2 + \frac{b^2}{4}$ for~$a, b \geq 0$ which is sometimes also referred to as Peter--Paul inequality with parameter~$\nicefrac{1}{2}$.}~\eqref{prop_rf_action_pf_eq2} and the ordering of Sobolev spaces, \eqref{eq:sob_reg_embed}, we find
	\begin{equs}[][eq:SPhi_bound_1]
		\msS(\phi) 
		& \leq 
		- \int_0^T \norm[0]{D \msS(w^{h,0}_t)}_{H^{-1}_x}^2 \dif t + \int_0^T \norm[0]{D \msS(w^{h,0}_t)}_{H_x^{-1}} \norm[0]{h_t}_{H_x^{-1}} \dif t 
		\leq
		\frac{1}{4} \int_0^T \norm[0]{h_t}_{H_x^{-1}}^2 \dif t \\
		& \leq
		\frac{1}{4} \int_0^T \norm[0]{h_t}_{L^2_x}^2 \dif t.
	\end{equs}
	Since~$h$ and~$T$ were arbitrary, recalling the definition of $\msV$ as given by \eqref{eq:quasi_potential}, the claim follows directly.
	\paragraph*{$\boldsymbol{\triangleright}$ \textbf{Proof that} $\boldsymbol{2S \geq V}$.}
	
	The proof of the reverse direction is supplemented by Figure~\ref{fig_prop_rf_action}, which schematically describes the main ideas.
	At first, we recall that we may write any solution to the equation,
	\begin{equation*}
		\partial_t w^{0,\phi}(t) - \Delta w^{0,\phi}(t) = - w^{0,\phi}(t)^3 -m^2 w^{0,\phi}(t),\qquad w^{0,\phi}(0)= \phi \in H^1(\T^3), 
	\end{equation*}
	as a solution to the equation,
	\begin{equation*}
		\partial_t w^{0,\phi}(t) = - D\msS(w^{0,\phi}(t)),\qquad w^{0,\phi}(0) = \phi \in H^1(\T^3),
	\end{equation*}
	which is properly understood as an identity in $H^{-1}(\mbT^3)$. 
	\begin{figure}
		\begin{minipage}{.48\textwidth}
			\centering
			\subfloat[]{%
				\centering
				\includegraphics[scale=0.05]{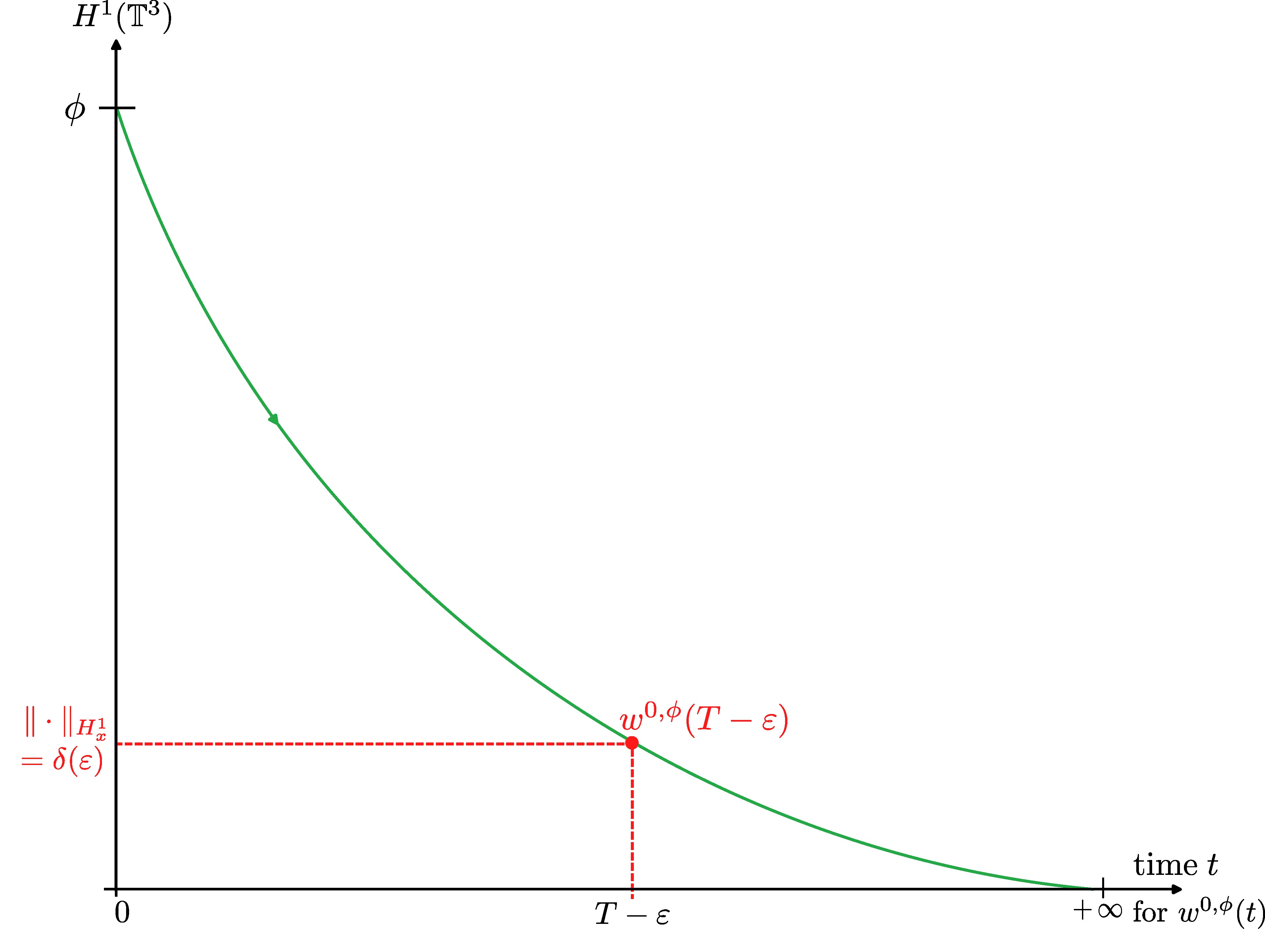}
				\label{fig_1_left}
			}
		\end{minipage}\hfill
		\hspace{2em}
		\begin{minipage}{.48\textwidth}
			\subfloat[]{%
				\centering
				\includegraphics[scale=0.05]{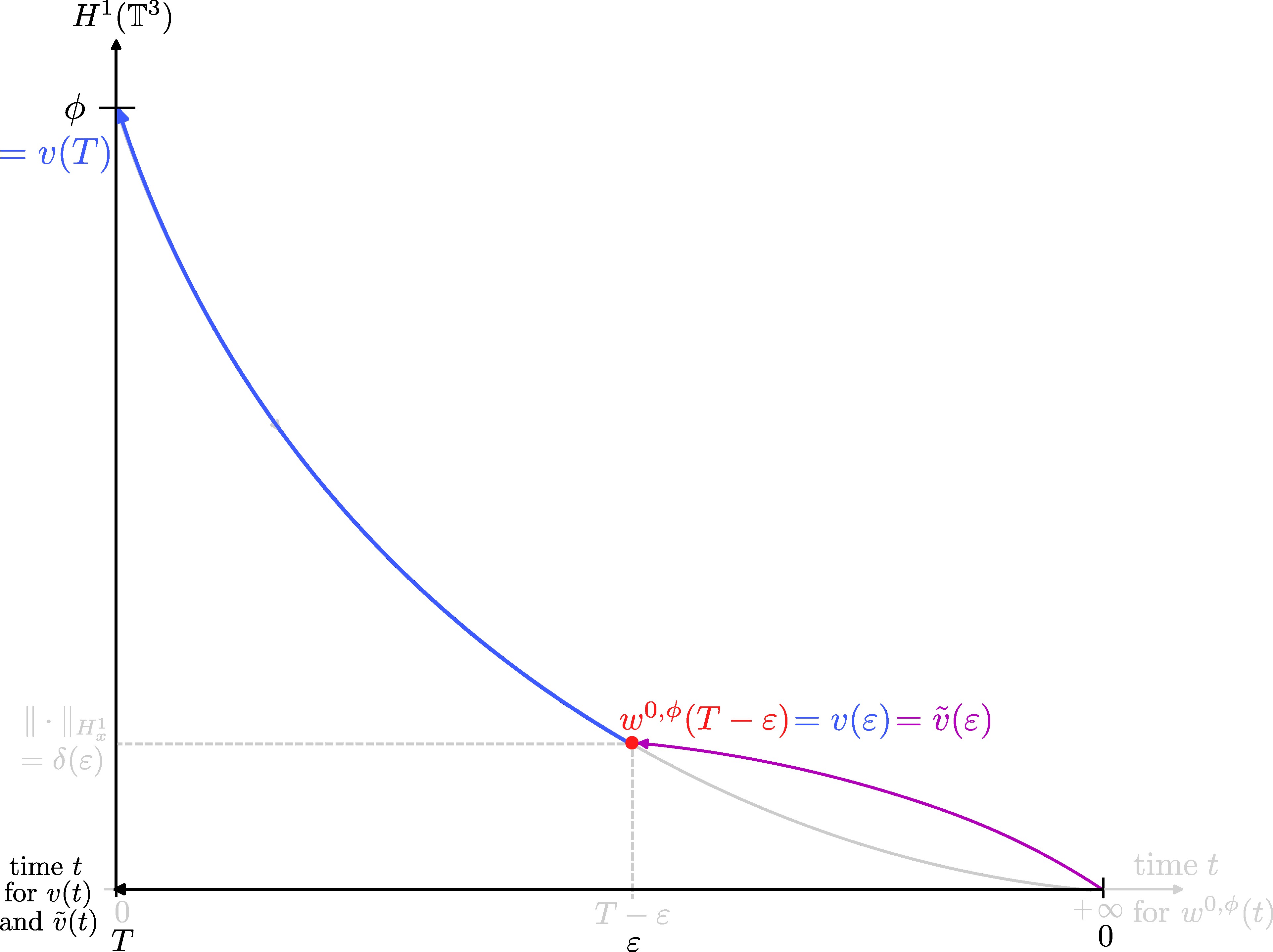}
				\label{fig_1_right}
			}
		\end{minipage}
		\captionsetup{format=plain, width=0.8\textwidth}
		\caption{\textbf{Schematic overview of the proof idea}. \\[0.5em]
			\protect\subref{fig_1_left} In green, the flow~$t \mapsto w^{0,\phi}(t)$ \emph{down the gradient} of~$\msS$ is depicted. The time~$T(\eps)$ is chosen such that~$w^{0,\phi}(T-\eps) \in \partial B_{\delta(\eps)}(0) \subseteq H^1(\T^3)$. \\
			\protect\subref{fig_1_right} In purple, the flow~$t \mapsto \tilde{v}(t) = w^{\tilde{h}^{\eps,T},0}(t)$ for~$t \in [0,\eps]$ is depicted, where---by local exact controllability---the element~$\tilde{h}^{\eps,T} \in L^2_{[0,\eps]}L^2(\T^3)$ is chosen such that~$\tilde{v}(\eps) = w^{0,\phi}(T-\eps)$. In blue, the flow~$t \mapsto v(t) = w^{0,\phi}(T-t)$ for~$t \in [\eps,T]$ \emph{up the gradient} of~$\msS$ is pictured. The path~$v_\star$ is the concatenation of~$\tilde{v}$ (on~$[0,\eps]$) and~$v$ (on~$[\eps,T]$).
		}
		\label{fig_prop_rf_action}
	\end{figure}
	In the remainder of the proof, we make use of \emph{global, exact controllability} of the non-linear skeleton equation, c.f. Proposition~\ref{prop:nl_skeleton_control}.
	In particular, for~$\eps > 0$ and $\psi \in H^1(\mbT^3)$, there exists an $h\in L^2_\eps L^2(\mbT^3)$ such that
	\begin{equs}[][prop_rf_action_pf_localexactctrl]
		w^{h,0}(\eps) = \psi. 
	\end{equs}
	Setting~$\psi \coloneqq  w^{0,\phi}(T - \eps) \in H^1(\mbT^3)$, we denote the corresponding~$h \in L^2_\eps L^2(\mbT^3)$ from~\eqref{prop_rf_action_pf_localexactctrl} by~$\tilde{h}_{\eps, T} \coloneqq  h$. That is,
	\begin{equation*}
		w^{\tilde{h}^{\eps,T},0}(\eps) = w^{0,\phi}(T - \eps).
	\end{equation*}
	We now consider the \emph{backwards flow} up the gradient of $\msS$,
	\begin{equation*}
		\partial_t v(t) = D\msS(v(t)), \quad v(\eps) = w^{0, \phi}(T-\eps).
	\end{equation*}
	Note that by construction, the reverse gradient flow $v$ satisfies the identity
	\begin{equation*}
		v(t) = w^{0,\phi}(T-t)\quad \text{for all} \quad t \in [\eps,T].
	\end{equation*}
	As a result, we also have $v(T) = w^{0,\phi}(0) = \phi$ and so we may write
	\begin{equation*}
		\partial_t v(t) = -D\msS(v(t)) + 2D\msS(v(t)),\quad v(\eps)=w^{0,\phi}(T-\eps),\quad v(T)=\phi.
	\end{equation*}
	In this way, we see the term $t\mapsto 2D\msS(v(t))$ is the necessary control to drive the \emph{forward gradient flow}, started from the point $w^{0,\phi}(T-\eps)$ at time~$\eps$, to equal $\phi$ at time $T$. Using the same identities as in~\eqref{prop_rf_action_pf_eq2} and the equation solved by $v(t)$, we find that
	\begin{equs} \label{prop_rf_action_pf_DSctrl}
		\frac{1}{2}\|D\msS(v(\,\cdot\,))\|^2_{L^2_{[\eps,T]}L^2_x} = 2	\int_{\eps}^T \|D\msS(v(s))\|_{L^2_x}^2  \dd s 	= & 2 \sbr[1]{\msS(v(T))-\msS(v(\eps))}	\\
		= &2 \sbr[1]{\msS(\phi) - \msS(v(\eps))} \\
		< &\infty.
	\end{equs}
	In other words,~$2D\msS(v) \in L^2_{[\eps,T]} L^2(\mbT^3)$ and so it is a member of the Cameron--Martin space of the space-time white noise and is therefore a suitable control. 
	
	Next, recalling that~$\tilde{v} \coloneqq  w^{\tilde{h}_{\eps,T},0}$  by definition solves
	\begin{equation*}
		\partial_t \tilde{v}(t) = - D\msS(\tilde{v}(t)) + \tilde{h}^{\eps;T}(t),
		\qquad 
		t \in [0,\eps],
		\quad
		\tilde{v}(0)=0,
	\end{equation*}
	by global, exact controllability and our choice of $\tilde{h}_{\eps,T}$ we have
	\begin{equation*}
		\tilde{v}(\eps) = w^{0,\phi}(T-\eps) = v(\eps).
	\end{equation*}
	Although this fact is not required directly by our proof here; note that by Theorem~\ref{th:nl_skeleton_gwp}, point \ref{it:nl_skeleton_global_decay}, we have
	\begin{equation*}
		\lim_{t\rightarrow \infty} \|w^{0,\phi}(t)\|_{H^1_x} =0 \quad \Rightarrow \quad \lim_{T \to \infty} \|w^{0,\phi}(T-\eps)\|_{H^1_x} =0
	\end{equation*} 
	which demonstrates that this target point can be chosen arbitrarily small, by taking $T$ arbitrarily large.

	Setting~$v_\star(t) \coloneqq   \tilde{v}(t) \mathbf{1}_{[0,\eps]}(t) + v(t) \mathbf{1}_{[\eps,T]}(t)$ for~$t \in [0,T]$, we see that~$v_\star = w^{h_\star,0}$ for the control
	\begin{equation*}
		h_\star \coloneqq  \tilde{h}^{\eps;T}\1_{[0,\eps]} + D\msS(v(\,\cdot\,))\1_{[\eps,T]} \in L^2_{[0,T]} L^2(\mbT^3),		
	\end{equation*}
	i.e.
	\begin{equation*}
		\partial_t v_\star(t) = -D\msS(v_\star(t)) + h_\star(t), \quad
		v_\star(0) = \tilde{v}(0) = 0, \quad
		v_\star(T) = v(T) = \phi.
	\end{equation*}	
	Finally, the definition of~$h_\star$ and the calculation in~\eqref{prop_rf_action_pf_DSctrl} imply that
	\begin{equation*}
		\int_0^T \norm[0]{h_\star(s)}^2_{L^2_x}\dif s 
		= 4\sbr[1]{\msS(\phi)-\msS(v_\star(\eps))} + \int_0^\eps \|\tilde{h}^{\eps,T}(s)\|^2_{L_x^2}\dd s.	
	\end{equation*}
	Rearranging this equality, recalling that~$\msS$ takes values in~$[0,\infty]$ and so taking a lower bound gives
	\begin{equation*} \label{prop_rf_action_pf_S_lowerbd}
		\msS(\phi) \,\geq \, 
		\frac{1}{4}	\int_0^T \norm[0]{h_\star(s)}^2_{L_x^2}\dd s  - \frac{1}{4}\int_0^\eps \|\tilde{h}^{\eps;T}(s)\|^2_{L_x^2} \dif s.
	\end{equation*}
	Since~$T = T(\eps) > 0$ and 
	\begin{equs}
		\frac{1}{2}	\int_0^T \norm[0]{h_\star(s)}^2_{L^2}\dd s
		& =
		\II_{0,T}(v_\star) \\
		& \geq
		\inf\left\{\II_{0,T'}(v'): \ T' > 0, \ v' \in C_{T'} \CC^\alpha(\T^3), \ v'(0) = 0, \ v'(T') = \phi\right\} \\[0.5em]
		& =
		\msV(\phi),
	\end{equs}
	we can further bound~$\msS$ from below by
	\begin{equation*}
		\msS(\phi) \,\geq \, 
		\frac{1}{2}	\msV(\phi)  
		- \frac{1}{4}\int_0^\eps \|\tilde{h}_{\eps;T}(s)\|^2_{L_x^2} \dif s.
	\end{equation*} 
	Since~$\eps > 0$ was arbitrary, the claim follows and the proof is finished.
\end{proof}
We have the following corollary of Theorem~\ref{th:V_equal_2S}.
\begin{corollary} \label{coro:cpc_sublevel_sets}
	The set
	\begin{equation*}
		\left\{ \mfz \in \mcC^{\alpha}(\mbT^3)\,:\, \msV(\mfz) \leq \theta \right\} = \left\{ \mfz \in \mcC^{\alpha}(\mbT^3)\,:\, \msS(\mfz) \leq \frac{\theta}{2} \right\} 
	\end{equation*}
	is a compact subset of $\mcC^{\alpha}(\mbT^3)$.
\end{corollary}

\begin{proof}
	By definition of~$\msS$, this follows from the compact embedding~$H^{1}(\T^3) \cembed \CC^\alpha(\T^3)$, c.f. \eqref{eq:H1_to_L3q_to_Calpha_embed}.
\end{proof}

\section{LDP Lower Bound for the $\boldsymbol{\Phi^4_3}$ Measure}\label{sec:LDP_lower}
In this section we establish the large deviation lower bound, i.e.  Theorem~\ref{th:main_result}\;\ref{it:LDP_lower}.
By~Remark~\ref{rmk:locally_uniform_ldp}~\ref{rmk:locally_uniform_ldp:i}, this statement is equivalent to the following theorem:\footnote{Note that in contrast to the analogous equivalence for the LDP upper bound, this does \emph{not} require that~$\msS$ has compact sublevel sets, see Section~\ref{sec:LDP_upper} for further discussion.}
\begin{theorem}[LDP lower bound] \label{thm:ldp_lower_bound}
	For any~$\gamma, \delta > 0$ and~$\fz \in \mcC^{\alpha}(\mbT^3)$, there exists an~$\eps_0 \coloneqq \eps_0 (\gamma,\delta,\fz)> 0$ such that for all $\eps \leq \eps_0$,
	\begin{equation*}
		\mu_\eps\del[1]{\fy \in \mcC^{\alpha}: \ \norm{\fy-\fz}_{\mcC^{\alpha}} < \delta}
		\geq 
		\exp\del[3]{-\frac{2\msS(\fz) + \gamma}{\eps^2}}.
	\end{equation*}
\end{theorem}
In order to prove Theorem~\ref{thm:ldp_lower_bound}, we will require the following lemma which amounts to a uniform, quantitative estimate on the control required to drive the solution of the skeleton equation~\eqref{eq:phi43_formal_SPDE},  started from a ball of radius $\rho$ to a neighbourhood of any terminal state, in terms of the quasi-potential at that terminal state. Note that the lemma assumes we take an $\mfz \in \mcC^{\alpha}(\mbT^3)$ such that $\msV(\mfz)<\infty$. While the statement is true under this, possibly vacuous, assumption, the global, exact controllability shown for the skeleton equation by Proposition~\ref{prop:nl_skeleton_control} demonstrates that the set infimised over in the definition of $\msV$ is non-empty, hence such a $\mfz\in \mcC^{\alpha}(\mbT^3)$ exists.
\begin{lemma} \label{lem:quant_lwr_bnd_control}
	Let  $\mfz\in \mcC^{\alpha}(\mbT^3)$ be such that $\msV(\mfz)<\infty$,
	and $\gamma>0$.
	Then, there exists a $\bar{T}>0$ such that for any~$\rho,\,\delta > 0$, there exist $T_0 > 0$ and $h_0 \in L^2_{\bar{T} + T_0}L^2(\mbT^3)$ with the property that
	\begin{equation}\label{eq:quant_lwr_bnd_control}
		\frac{1}{2}\norm{h_0}^2_{L^2_{T_0+\bar{T} }L^2_x} \leq \msV(\mfz) + \frac{\gamma}{2}\quad \text{and}\quad 
		\sup_{\norm[0]{\fy}_{\mcC^{\alpha}} \leq \rho} \norm[1]{w^{h_0,\fy}(T_0+\bar{T} ) - \fz}_{\mcC^{\alpha}} \leq \frac{\delta}{2}.
	\end{equation}
\end{lemma}
To give an intuition for the result, we can think of the required control on $\|h_0\|_{L^2_{T_0+\bar{T}}L^2_x}$ as a budget to spend and the distance $\delta/2$ as the accuracy we wish to achieve.
\begin{proof}[Proof of Lemma~\ref{lem:quant_lwr_bnd_control}]
	To prove the statement, first recall that using Theorem~\ref{th:nl_skeleton_gwp} and the assumption that $\msV(\mfz)<\infty$ it must hold that in fact $\mfz \in H^1(\mbT^3) \hookrightarrow L^{3q}(\mbT^3)$. Furthermore, it follows from the assumption that $\msV(\mfz)<\infty$ that there must exist a $\bar{T}>0$ and $\bar{h}\in L^2_{\bar{T}}L^2(\mbT^3)$ such that
	\begin{equation*}
		w^{\bar{h};0}(\bar{T})=\mfz.
	\end{equation*}
	Therefore, up to modifying the path if necessary, it is always possible to choose $\bar{h}$ as above and additionally such that
	\begin{equation}
		\frac{1}{2}\norm[0]{\bar{h}}_{L^2_{\bar{T}}L^2_x}^2 
		\leq \msV(\mfz) + \frac{\gamma}{2}.
		\label{lem:ldp_lower_bound:hbar}
	\end{equation}
	We now take a~$T > 0$, which will chosen appropriately later, and define~$h_0:[0,T+\bar{T}]\to L^2(\mbT^3)$ by setting 
	\begin{equation*}
		h_0(t) \coloneqq  
		\begin{cases}
			0 & \text{if} \quad t \in [0,T], \\
			\bar{h}(t-T) & \text{if} \quad t \in (T,T+\bar{T}].
		\end{cases}
	\end{equation*}
	We note that in particular,
	\begin{equation}\label{eq:h0_hBar_shift}
		\norm[0]{h_0}_{L^2_{T+\bar{T} }L^2_x} = \norm[0]{\bar{h}}_{L^2_{\bar{T}}L^2_x}.
	\end{equation}
	Thus the first bound claimed in \eqref{eq:quant_lwr_bnd_control} follows from \eqref{lem:ldp_lower_bound:hbar} and \eqref{eq:h0_hBar_shift}, with $T>0$ to be fixed later as $T_0$.
	
	To prove the second bound of \eqref{eq:quant_lwr_bnd_control} we recall that we have chosen $\bar{T}>0$ and $\bar{h}\in L^2_{\bar{T}}L^2(\mbT^3)$ such that $w^{\bar{h},0}(\bar{T})=\mfz$. Furthermore, by definition of $h_0$, for any $\mfy\in \mcC^{\alpha}(\mbT^3)$ and $t\in [0,T+\bar{T}]$ one has that
	\begin{equation*}
		w^{h_0,\fy}(t) = 
		\begin{cases}
			w^{0,\fy}(t) & \text{if} \quad t \in [0,T], \\
			w_T^{h_0,w^{0,\fy}(T)}(t) & \text{if} \quad t \in (T,T+\bar{T}]. 
		\end{cases}
	\end{equation*}
	Then, given $t\in [0,\bar{T}]$, by the semi-group property of the skeleton equation (see for example the proof of Proposition~\ref{prop:nl_skeleton_LWP}) and a simple change of variables, for all $t\in [0,\bar{T}]$ we have,
	\begin{equs}
		\thinspace
		& w^{h_0,\fy}(T+t) \\
		= & \ e^{t(\Delta-m^2)} w^{0,\fy}(T) - \int_{T}^{T+t}e^{(T+t-s)(\Delta-m^2)} w^{h_0,\fy}(s)^3 \dd s + \int_{T}^{T+t} e^{(T+t-s)(\Delta-m^2)} h_0(s) \dd s\\
		= & \ e^{t(\Delta-m^2)}  w^{0,\fy}(T)
		- \int_0^{t} e^{(t-s)(\Delta-m^2)}w^{h_0,\fy}(T+s)^{3} \dif s + \int_0^t e^{(t-s)(\Delta-m^2)}  h_0(T+s) \dif s.
	\end{equs}
	Using the fact that for all~$s\in [0,\bar{T}]$, $h_0(s+T) = \bar{h}(s)$; comparing $w^{h_0,\fy}(T+\,\cdot\,)$ to $w^{\bar{h},0}(\,\cdot\,)$ we find
	\begin{equation*}
		w^{h_0,\fy}(T+t) - w^{\bar{h},0}(t) 
		=
		e^{t(\Delta-m^2)} w^{0,\fy}(T)
		- \int_0^{t} e^{(t-s)(\Delta-m^2)}  \left(w^{h_0,\fy}(T+s)^3 - w^{\bar{h},0}(s)^3\right) \dif s.
	\end{equation*}
	By \eqref{eq:nl_skeleton_global_increment_bound_0_initial_2} of Theorem~\ref{th:nl_skeleton_gwp}, item~\ref{it:nl_skeleton_quant_bounds} in combination with \eqref{eq:nl_skeleton_Lp_growth} of Proposition~\ref{prop:nl_skeleton_Lp_apriori}, for $q \in (\nicefrac{3}{2},2)$ such that the embedding \eqref{eq:H1_to_L3q_to_Calpha_embed} holds, one sees that,
	\begin{equs}[][eq:lower_bnd_lem_apriori]
		\thinspace & 
		\sup_{t \in [0,\bar{T}]} \norm[0]{w^{\bar{h},0}(t)}_{L^{3q}_x} \lesssim_{q,m}  \norm[0]{\bar{h}}_{L^2_{\bar{T}}L^2_x} \\  
		\text{and} \quad & 
		\sup_{t \in[0,\bar{T}]} \norm[0]{w^{h_0,\fy}(T+s)}_{L^{3q}_x} \lesssim_{T,q,m,\alpha}  \norm[0]{\fy}_{\mcC^{\alpha}}+\norm[0]{h_0}_{L^2_{\bar{T}+T}L^2_x}.
	\end{equs}
	So making use of the Besov embedding, \eqref{eq:besov_Lp_embed} 
	along with the regularising effect of the Ornstein--Uhlenbeck semi-group~\eqref{eq:ou_reg} and the fact that 
	$$\norm[0]{h_0}_{L^2_{\bar{T} + T}L^2_x} = \norm[0]{\bar{h}}_{L^2_{\bar{T}}L^2_x}$$
	one sees that for all $t\in [0,\bar{T}]$ 
	\begin{equs}
		\thinspace &
		\norm[0]{w^{h_0,\fy}(T+t) - w^{\bar{h},0}(t)}_{L^{3q}_x} \\ 
		\leq \ & \norm[0]{	e^{t(\Delta-m^2)} w^{0,\fy}(T)}_{L^{3q}_x}
		+ \int_0^{t} \norm[2]{ e^{(t-s)(\Delta-m^2)}\left(w^{h_0,\fy}(T+s)^3 - w^{\bar{h},0}(s)^3\right)}_{L^{3q}_x} \dif s\\ 
		\lesssim \ & \norm[0]{	e^{t(\Delta-m^2)} w^{0,\fy}(T)}_{L^{3q}_x}
		+ \int_0^{t} \norm[2]{ e^{(t-s)(\Delta-m^2)}\left(w^{h_0,\fy}(T+s)^3 - w^{\bar{h},0}(s)^3\right)}_{B^0_{3q,\infty} \dif s}\\
		\lesssim \ &
		\norm[0]{w^{0,\fy}(T)}_{L^{3q}_x}
		+ \int_0^{t} (t-s)^{-\frac{1}{q}} \norm[1]{w^{h_0,\fy}(T+s)^3 - w^{\bar{h},0}(s)^3}_{B_{3q,3q}^{-2/q}} \dif s \\
		\lesssim \ &
		\norm[0]{w^{0,\fy}(T)}_{L^{3q}_x}
		+ \int_0^{t} (t-s)^{-\frac{1}{q}} \norm[1]{w^{h_0,\fy}(T+s)^3 - w^{\bar{h},0}(s)^3}_{L^q_x} \dif s \\
		\lesssim \ & \norm[0]{w^{0,\fy}(T)}_{L^{3q}_x} \\
		& \quad + \left(\|w^{h_0,\fy}\|^2_{C_{T+\bar{T}}L^{3q}_x} \vee \|w^{\bar{h},0}\|^2_{C_{\bar{T}}L^{3q}_x} \vee 1 \right) \int_0^t (t-s)^{-\frac{1}{q}} \norm[1]{w^{h_0,\fy}(T+s) - w^{\bar{h},0}(s)}_{L^{3q}_x} \dd s\\
		\lesssim \ & 
		\norm[0]{w^{0,\fy}(T)}_{L^{3q}_x}+ \int_0^{t} (t-s)^{-\frac{1}{q}} \norm[1]{w^{h_0,\fy}(T+s) - w^{\bar{h},0}(s)}_{L^{3q}_x} \dif s,
	\end{equs}
	where the final proportionality constant depends only on $q,m,\|\fy\|_{\mcC^{\alpha}},\norm[0]{\bar{h}}_{L^2_{\bar{T}}L^2_x}$, due to the  a priori bounds, \eqref{eq:lower_bnd_lem_apriori}.  Therefore, by Gr\"onwall's inequality, there exists a possibly different constant $C>0$, but with the same dependencies as above such that
	\begin{equation*}
		\norm[0]{w^{h_0,\fy}(T+t) - w^{\bar{h},0}(t)}_{L^{3q}_x}
		\lesssim C
		\exp\del[1]{t^{1-\frac{1}{q}}} \norm[0]{w^{0,\fy}(T)}_{L^{3q}_x},
		\qquad
	\end{equation*}
	In particular, choosing~$t = \bar{T}$ leads to
	\begin{equation*}
		\norm[1]{w^{h_0,\fy}(T+\bar{T}) - \fz}_{L^{3q}_x}
		\leq 
		c(\bar{T}) \norm[0]{w^{0,\fy}(T)}_{L^{3q}_x}
	\end{equation*}
	and the claim then follows from item~\ref{it:nl_skeleton_global_decay} of Theorem~\ref{th:nl_skeleton_gwp} after choosing $T=T_0\coloneqq T_0(\rho,\,\delta)>0$ appropriately large and applying the embedding \eqref{eq:H1_to_L3q_to_Calpha_embed}. 
\end{proof}
We are now in a position to prove the main result of this section, Theorem~\ref{thm:ldp_lower_bound}.
\begin{proof}[Proof of Theorem~\ref{thm:ldp_lower_bound}]
	Recalling the notation $u^\fy_\eps$ for the solution to the renormalised, dynamic $\Phi^4_3$ equation with temperature $\eps>0$ and started from $\fy\in \mcC^{\alpha}(\mbT^3)$ as described rigorously by Definition~\ref{def:phi43_rigorous_SPDE},
	and given $w^{h_0,\fy}$ as in Lemma~\ref{lem:quant_lwr_bnd_control}, one has 
	\begin{align*}
		\norm[0]{\bar{u}^\fy_\eps(\bar{T} + T_0) - \fz}_{\mcC^{\alpha}}
		&\leq \norm[0]{\bar{u}^\fy_\eps(\bar{T} + T_0) - w^{h_0,\fy}(\bar{T} + T_0)}_{\mcC^{\alpha}}  + \norm[0]{w^{h_0,\fy}(\bar{T} + T_0)-\mfz}_{\mcC^{\alpha}}\\
		&\leq 
		\norm[0]{\bar{u}^\fy_\eps(\bar{T} + T_0) - w^{h_0,\fy}(\bar{T} + T_0)}_{\mcC^{\alpha}} + \frac{\delta}{2}.
	\end{align*}
	Since~$\mu_\eps$ is invariant for~$u_\eps$, we apply this bound to see that
	\begin{align}
		\mu_\eps\del[1]{\fy \in \mcC^{\alpha}(\mbT^3): \ \norm{\fy-\fz}_{\mcC^{\alpha}} < \delta} 
		& =
		\int_{\mcC^{\alpha}} \P\del[1]{\norm[0]{\bar{u}^\fy_\eps(\bar{T} + T_0) - \fz}_{\mcC^{\alpha}} < \delta} \mu_\eps(\dif \fy) \notag \\
		& \geq
		\int_{\mcC^{\alpha}} \P\del[1]{\norm[0]{\bar{u}^\fy_\eps(\bar{T} + T_0) - w^{h_0,\fy}(\bar{T} + T_0)}_{\mcC^{\alpha}} < \nicefrac{\delta}{2}} \mu_\eps(\dif \fy) \label{pf:ldp_lower_bound:eq1}\\
		& \geq
		\int_{\|\eta\|_{\mcC^{\alpha}}\leq \rho} \P\del[1]{\norm[0]{\bar{u}^\fy_\eps - w^{h_0,\fy}}_{\CC_{\bar{T} + T_0}\mcC^{\alpha}} < \nicefrac{\delta}{2}} \mu_\eps(\dif \fy) \notag
	\end{align} 
	The locally, uniform LDP, obtained for~$(u^{\fy}_\eps)_{\eps > 0}$ by Proposition~\ref{prop:uniform_ldp}, implies that for~$\rho> 0$ there exists an~$\eps_0 > 0$ such that for any~$\eps \leq \eps_0$ and~$\norm[0]{\fy}_{\mcC^{\alpha}} \leq \rho$, one has
	\begin{align*}
		\P\del[1]{\norm[0]{\bar{u}^\fy_\eps - w^{h_0,\fy}}_{\CC_{\bar{T} + T_0}\mcC^{\alpha}} < \nicefrac{\delta}{2}}
		& \geq
		\exp\del[4]{-\frac{\II_{\bar{T}+T_0}\del[1]{w^{h_0,\fy}} + \gamma}{2\eps^2}}\\
		&=
		\exp\del[4]{-\frac{\norm[0]{h_0}_{L^2_{\bar{T}+T_0}L^2_x}^2 + \gamma}{2\eps^2}} \\
		& \geq
		\exp\del[4]{-\frac{\msV(\mfz) + \gamma}{\eps^2}},
	\end{align*} 
	where the last inequality is due to the first bound of~\eqref{eq:quant_lwr_bnd_control}. Hence, it follows from \eqref{pf:ldp_lower_bound:eq1} that
	\begin{equation*}
		\mu_\eps\del[1]{\fy \in \mcC^{\alpha}(\mbT^3): \ \norm{\fy-\fz}_{\mcC^{\alpha}} < \delta}  \geq 	\mu_{\eps}\del[1]{B_\rho}
		\exp\del[4]{-\frac{\msV(\mfz) + \gamma}{\eps^2}}
	\end{equation*}
	In order to conclude therefore, we only need establish that there exists some~$\rho_\star > 0$ such that
	\begin{equation*}
		\lim_{\eps \to 0} \mu_{\eps}\del[1]{B_{\rho_\star}} = 1.
	\end{equation*}
	This, in turn, is a straightforward consequence of Proposition~\ref{prop:measure_tail_bounds}. With the notations therein, let~$\theta \coloneqq  1$ and $\rho \coloneqq \rho(\theta)$. Then, for each~$\eps \in (0,1]$, one has
	\begin{equation*}
		1 \geq
		\mu_\eps(B_{\rho_\star})
		= 
		1 - \mu_\eps(B_{\rho_\star}^{\texttt{c}})
		\geq
		1 - K\exp\del[3]{-\frac{1}{\eps^{2}}} 
	\end{equation*} 
	Since the right hand side goes to~$1$ as~$\eps \to 0$ the proof is complete.
\end{proof}

\section{LDP Upper Bound for the $\boldsymbol{\Phi^4_3}$ Measure}\label{sec:LDP_upper}
In this section we establish the large deviation upper bound, Theorem~\ref{th:main_result}~\ref{it:LDP_upper}.
By Remark~\ref{rmk:locally_uniform_ldp}~\ref{rmk:locally_uniform_ldp:i}, this statement is equivalent to the following theorem since we have already shown that~$\msS$  has compact sublevel sets, see Corollary~\ref{coro:cpc_sublevel_sets}.
\begin{theorem}[LDP upper bound] \label{thm:ldp_upper_bound}
	For any~$\theta,\,\delta,\,\gamma > 0$ there exists an~$\eps_0 \coloneqq(\delta,\, \theta ,\gamma)> 0$ such that for all $\eps \leq \eps_0$,
	\begin{equation*}
		\mu_\eps\del[1]{\fz \in \mcC^{\alpha}(\mbT^3): \ \operatorname{dist}_{\mcC^{\alpha}}(\fz,\msS[\theta]) \geq \delta}
		\leq
		\exp\del[3]{-\frac{\theta - \nicefrac{\gamma}{2}}{\eps^2}}.
	\end{equation*}
\end{theorem}
The proof of Theorem~\ref{thm:ldp_upper_bound} requires several preliminary lemmas and propositions; the final proof is concluded in Section~\ref{sec:ldp_upper_bound_proof}. The first required result, is the following lemma, which allows us to define the set of paths~$t\mapsto v(t)\in\mcC^{\alpha}(\mbT^3)$ which begin in a given ball but take values outside a second ball at all integer time points.
\begin{lemma}[Dynamic energy bounds eventually imply static energy bounds] \label{lem:energy_bounds}
	Let~$\theta,\,\delta > 0$ and for any~$\lambda>0$ recall the definition of the following set for any $t>0$, 
	\begin{equation}
		\II^{\bar{B}_\lambda}_t[\theta]
		\coloneqq  
		\left\{
		v \in C_t\mcC^{\alpha}(\mbT^3): \ v(0) \in \bar{B}_{\lambda}, 
		\ v \in \II_{t}[\theta] \right\}
		\label{lem:energy_bounds:aux}
	\end{equation}
	Then, there exists some~$\lambda(\delta,\,\theta) =: \lambda > 0$ and~$\bar{T}(\delta,\,\theta) =: \bar{T} > 0$ such that for all~$t \geq \bar{T}$, we have
	\begin{equation}
		\left\{v(t) \in \mcC^\alpha(\mbT^3): \ v \in \II^{\bar{B}_\lambda}_t[\theta]\right\} \subseteq  
		\left\{\fz \in \mcC^{\alpha}(\mbT^3): \ \operatorname{dist}_{\mcC^{\alpha}}(\fz,\msV[\theta]) < \frac{\delta}{2}\right\}.
		\label{lem:energy_bounds:eq}
	\end{equation}
	where $\msV$ is the quasi-potential defined by \eqref{eq:quasi_potential} and $\msV[\theta]$ denotes the associated sub-level set.
\end{lemma}
The claim made in Lemma~\ref{lem:energy_bounds} is illustrated in Figure~\ref{fig_lem_energy_bounds}. Intuitively, it says that a path which starts close to equilibrium~$0$ and has bounded dynamic energy will eventually have bounded static energy.
\begin{figure}
	\begin{minipage}{.48\textwidth}
		\centering
		\subfloat{%
			\centering
			\includegraphics[scale=0.85]{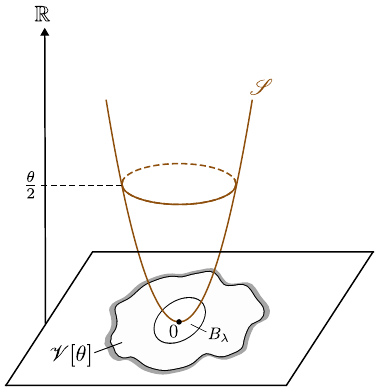} 
		}
	\end{minipage}\hfill
	\hspace{2em}
	\begin{minipage}{.48\textwidth}
		\subfloat{%
			\centering
			\includegraphics[scale=0.7]{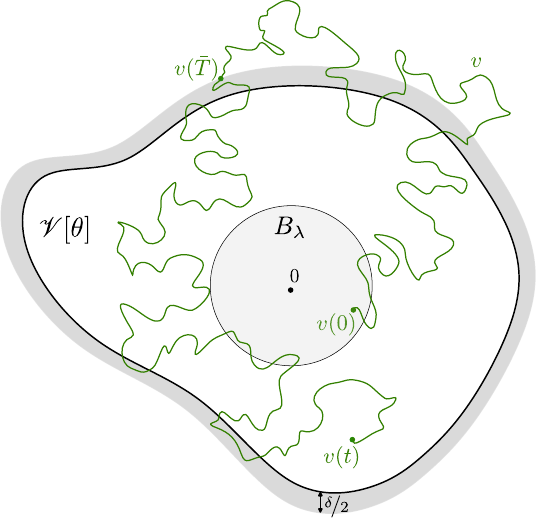} 
		}
	\end{minipage} 
	\captionsetup{format=plain, width=0.8\textwidth}
	\caption{\textbf{Schematic overview of the statement in Lemma~\ref{lem:energy_bounds}.} \\
		The green path~$v$ is an element in~$\II_t^{\bar{B}_\lambda}[\theta]$ and, according to the set inclusion in~\eqref{lem:energy_bounds:eq}, therefore never leaves the $\nicefrac{\delta}{2}$-neighbourhood of~$\VV[\theta]$ during the time interval~$[\bar{T},t]$. 
	}
	\label{fig_lem_energy_bounds}
\end{figure}

\begin{proof}
	Let us assume that the claimed inclusion does not hold, in other words, there exist $\delta,\,\theta>0$ such that for all $T,\,\lambda>0$ there exists a $t\geq T$ such that,
	\begin{equation*}
		\left\{v(t) \in \mcC^\alpha(\mbT^3): \ v \in \II^{\bar{B}_\lambda}_t[\theta]\right\} \not\subseteq  
		\left\{\fz \in \mcC^{\alpha}(\mbT^3): \ \operatorname{dist}_{\mcC^{\alpha}}(\fz,\msV[\theta]) < \frac{\delta}{2}\right\}.
	\end{equation*}
	As a consequence, given some $\theta,\,\delta>0$, there exist sequences $\lambda_n \searrow 0$, $T_n\nearrow +\infty$ and $v_n \in \II^{\bar{B}_{\lambda_n}}_{T_n}[\theta]$ for $n \in \N$, such that
	\begin{equation}
		\inf_{n\geq 0} \dist_{\mcC^\alpha} \left(v_n(T_n\right), \msV[\theta])\geq \frac{\delta}{2}.
		\label{lem:energy_bounds_pf_aux1}
	\end{equation}
	Since $v_n \in \II^{\bar{B}_{\lambda_n}}_{T_n}[\theta]$, using~\eqref{lem:energy_bounds:aux}, it holds in particular that~$\II_{T_n}(v_n) \leq \theta < \infty$.
	Thus, by definition of the rate function, there exists an $h_n \in L^2_{T_n}L^2_x$ such that~$v_n = w^{h_n,\fz_n}$ for $\fz_n \coloneqq  v_n(0)$, i.e.
	\begin{equation} \label{lem:energy_bounds_pf_aux2}
		\begin{aligned}
			\partial_t v_n - \Delta v_n &= - v_n^3 - m^2 v_n + h_n,\\
			\fz_n &= v_n(0)  \in \bar{B}_{\lambda_n}(0) \subseteq \CC^\alpha(\mbT^3)
		\end{aligned}%
		\qquad
		\text{and}\qquad \frac{1}{2}\|h_n\|^2_{L^2_{T_n}L^2_x} \leq \theta.
	\end{equation}
	In particular, as $\lambda_n \searrow 0$, it holds that 
	\begin{equation}
		\fz_n \to 0 \quad \text{in} \quad \CC^\alpha(\T^3) \quad \text{as} \quad n \to \infty.
		\label{lem:energy_bounds:pf_conv_IC}
	\end{equation}
	Since the rest of the proof is quite technical, at this stage we pause the formal proof to give an intuitive idea behind the rest of the strategy, which relies essentially on the gradient flow structure of the dynamic equation.
	\paragraph*{--- Intuition of the Strategy:} 
	Suppose we knew instead that~$\fz_n \in H^1(\mbT^3)$ and the convergence in~\eqref{lem:energy_bounds:pf_conv_IC} actually took place in~$H^1(\mbT^3)$.
	Then, arguing as in the proof of Theorem~\ref{th:V_equal_2S} (in particular using analogues of \eqref{prop_rf_action_pf_eq2} and  \eqref{eq:SPhi_bound_1}) we would be able to arrive at the estimate
	\begin{equation*}
		\msS(v_n(T_n)) - \msS(\fz_n) \leq \frac{1}{4} \norm[0]{h_n}_{L^2_{T_n} L^2_x}^2 \leq \frac{1}{2} \theta.
	\end{equation*}	
	By continuity of~$\msS$ on~$H^1(\mbT^3)$ and the fact that~$\msS(0) = 0$, this would then imply that~$\msS(v_n(T_n)) \leq \theta$ for~$n$ large enough.
	Because the sublevel set~$\msV[2\theta] = \msS[\theta]= \{\phi \in H^1(\mbT^3): \msS(\phi) \leq \theta\}$ is compact in~$\CC^{\alpha}(\mbT^3)$ (c.f. Corollary~\ref{coro:cpc_sublevel_sets}), there would exist a subsequence~$(n_k)_{k\in \mbN}$ and $\fy\in \mcC^{\alpha}(\mbT^3)$ such that
	\begin{equation}
		\lim_{k \to \infty} \|v_{n_k}(T_{n_k}) -\fy\|_{\mcC^{\alpha}}=0.
		\label{lem:energy_bounds_pf_aux3}
	\end{equation}
	By lower semicontinuity of~$\msS$, we could then obtain the estimate
	\begin{equation}
		\msS(\fy) \leq \liminf_{k \to \infty} \msS(v_{n_k}(T_{n_k}) \leq \lim_{k \to \infty} \msS(\fz_n) + \frac{1}{2} \theta = \msS(0) + \frac{1}{2} \theta =  \frac{1}{2} \theta
		\label{lem:energy_bounds_pf_aux4}
	\end{equation}
	from which, in combination with Theorem~\ref{th:V_equal_2S}, we would see that~$\fy$ is an element of the sublevel set~$\msV[\theta]$, a contradiction to~\eqref{lem:energy_bounds_pf_aux1}. ---
	
	The obstacle to applying this relatively simple strategy is that we do not necessarily have convergence of the sequence $\mfz_n\to 0$ with respect to the $H^1(\mbT^3)$ topology.  The remainder of the proof, however, demonstrates that one can make technical amendments to the previous argument to arrive at the same conclusion only using that $\mfz_n\to 0$ in the weaker $\mcC^{\alpha}(\mbT^3)$ topology.  
	
	\paragraph*{Step~$\boldsymbol{1}$: } The key idea is to define a family of solutions to approximate equations, similar to \eqref{lem:energy_bounds_pf_aux2} but with zero initial data and forcing that only takes effect after a short positive time. Let $(\lambda_n)_{n\in \mbN},\, (T_n)_{n\in \mbN}$ be the respectively decreasing and increasing sequences from previously, which we may assume are respectively such that $\lambda_n\in (0,1]$ and $T_n\in (1,\infty]$. We then define a strictly decreasing sequence $(s_n)_{n\in \mbN}$ by setting
	\begin{equation}\label{eq:sn_def}
		s_n \coloneqq  \lambda_n^{\frac{1}{2}} \wedge e^{-\frac{T^2_n}{\eps}}\lambda_n^{-\frac{1}{\eps}}\wedge 1, 
	\end{equation}
	where the parameters $\eta,\, \eps$ and $\alpha$ are chosen, as in Proposition~\ref{prop:nl_skeleton_LWP}, to satisfy
	\begin{equation}\label{eq:energy_lem_parameters}
		\alpha  \in \del[2]{-\frac{2}{3},-\frac{1}{2}},\quad\eta \in \del[2]{-\frac{5\alpha}{12},-\frac{\alpha}{3}} \quad \text{\&}\quad \eps \in \del[1]{0,-\alpha-3\eta}.
	\end{equation}
	We also fix $p=\frac{1}{1+\alpha}\in (6,9)$ so that applying \eqref{eq:besov_reg_embed} we have $L^p(\mbT^3)\cembed \mcC^{\alpha}(\mbT^3)$.
	
	We then define a sequence of forcing elements,
	\begin{equation*}
		\tilde{h}_n(t) 
		\coloneqq  
		\begin{cases}
			0 & \text{if} \ t \in [0,s_n), \\
			h_n(t-s_n) & \text{if} \ t \in [s_n, s_n + T_n].
		\end{cases}
	\end{equation*}
	and a family of paths $(\tilde{v}_n)_{n\in \mbN}$, solving the analogous equation to~\eqref{lem:energy_bounds_pf_aux2} with $h_n$ replaced by $\tilde{h}_n$, explicitly given by setting
	\begin{equation}
		\tilde{v}_n(t) 
		\coloneqq  w^{\tilde{h}_n,\fz_n}(t)
		=
		\begin{cases}
			w^{0,\fz_n}(t) & \text{if} \ t \in [0,s_n), \\
			w_{s_n}^{\tilde{h}_n,w^{0,\fz_n}(s_n)}(t) & \text{if} \ t \in [s_n, s_n + T_n].
		\end{cases}
		\label{lem:energy_bounds_pf_aux4a}
	\end{equation}
	The notation $w^{h,\mfz}_s$ here indicates that $w^{h,\mfz}_s(t)\big|_{t=s}=\mfz$.
	
	Since $\tilde{v}_n(s) = w^{0,\fz_n}(s)$ for $s\in [0,s_n)$ and $s_n\in (0,1]$, we may apply Proposition~\ref{prop:nl_skeleton_LWP}, specifically~\eqref{eq:nl_skeleton_H1_growth}, to see that there exists an~$\eta \coloneqq 1+\eta_0 \in \left(1-\frac{5\alpha}{12},1-\frac{\alpha}{3}\right)\subset \left(1+\frac{10}{36},1+\frac{2}{9}\right)$ and $\delta \coloneqq  \delta(\alpha,\eta) > 0$ such that
	\begin{equation*}
		\sup_{s \in (0,s_n]} (s \wedge 1)^{\eta} \|\tilde{v}_n(s)\|_{H^1_x} \lesssim 
		s_n^\delta \|\fz_n\|_{\mcC^{\alpha}}
	\end{equation*}
	From which it follows that,
	\begin{equation}\label{lem:energy_bounds_pf_aux4b}
		\|\tilde{v}_n(s_n)\|_{H^1_x} 
		\lesssim s_n^{\delta-\eta} \norm[0]{\fz_n}_{\CC^\alpha}
		\leq s_n^{\delta-\eta} \lambda_n \leq \lambda_n^{\frac{\delta}{2} + 1-\frac{\eta}{2}}.
	\end{equation}
	where we have used the assumption that~$\fz_n \in B_{\lambda_n}(0) \subseteq \CC^\alpha(\T^3)$ along with the definition of the sequence $(s_n)_{n\in \mbN}$ given by \eqref{eq:sn_def}.
	Since $\frac{\delta}{2}+1-\frac{\eta}{2}>0$ we find,
	\begin{equation*}
		\tilde{v}_n(s_n) \to 0 \quad \text{in} \quad H^1(\T^3) \quad \text{as} \quad n \to \infty
	\end{equation*}
	and so $\tilde{v}_n(s_n)$ provides an sequence to which we can apply our intuitive strategy presented previously. In particular, 
	making the replacements
	\begin{equation*}
		v_n(T_n) \rightsquigarrow \tilde{v}_n(s_n + T_n), \quad
		v_{n_k}(T_{n_k}) \rightsquigarrow v_{n_k}(s_{n_k} + T_{n_k}), \quad
		\fz_n \rightsquigarrow \tilde{v}_n(s_n).
	\end{equation*}
	We obtain versions of \eqref{lem:energy_bounds_pf_aux3} and~\eqref{lem:energy_bounds_pf_aux4} which now read, for some $\mfy \in \mcC^{\alpha}(\mbT^3)$,
	\begin{equation*}
		\lim_{k \to \infty} \|\tilde{v}_{n_k}(T_{n_k} + s_{n_k}) -\mfy\|_{\mcC^{\alpha}}=0\quad \text{and} \quad
		\msS(\fy) \leq \frac{1}{2} \theta.
		\label{lem:energy_bounds_pf_aux5}
	\end{equation*}
	Therefore, in order to arrive at a contradiction to~\eqref{lem:energy_bounds_pf_aux1} and complete the proof, it remains to show
	\begin{equation}\label{lem:energy_bounds_pf_aux6}
		\lim_{n \to \infty} \tilde{v}_n(T_n + s_n) = \lim_{n \to \infty} v_n(T_n) \quad \text{in} \quad \CC^\alpha(\T^3)
	\end{equation}
	if both of these limits exist. 
	Note that we have re-labelled the subsequences to depend on~$n$ only for the sake of presentation.
	We establish the equality in the limit of \eqref{lem:energy_bounds_pf_aux6} in \textbf{Step~$\boldsymbol{2}$} below.
	\paragraph*{Step~$\boldsymbol{2}$: }
	We write the mild formulation for~$\tilde{v}_n(\,\cdot\, + s_n)$ and~$v_n$, starting with the former.
	Let~$t \in (0,T_n]$ so that using~\eqref{lem:energy_bounds_pf_aux4a}, we obtain 
	\begin{equs}
		\tilde{v}_n(t + s_n) 
		=&\, e^{t(\Delta - m^2)} w^{0,\fz_n}(s_n) - \int_{s_n}^{t+s_n} e^{(t + s_n - r)(\Delta - m^2)}w_{s_n}^{\tilde{h}_n,w^{0,\fz_n}(s_n)}(r)^3 \dif r \\
		& + 
		\int_{s_n}^{t+s_n} e^{(t + s_n - r)(\Delta - m^2)}\tilde{h}_n(r) \dif r \\
		=&\, e^{t(\Delta - m^2)} w^{0,\fz_n}(s_n)
		- \int_{0}^{t} e^{(t  - r)(\Delta - m^2)}\tilde{v}_n(r + s_n)^3 \dif r  + 
		\int_{0}^{t} e^{(t - r)(\Delta - m^2)} h_n(r) \dif r \\
	\end{equs}
	and also
	\begin{equation*}
		v_n(t) 
		= 
		e^{t(\Delta - m^2)} \fz_n - \int_0^{t} e^{(t - r)(\Delta - m^2)} v_n(r)^3 \dif r 
		+ \int_0^{t} e^{(t - r)(\Delta - m^2)} h_n(r) \dif r 
	\end{equation*}
	so that we obtain the following equation for their difference,
	\begin{equs}[][lem:energy_bounds_pf_aux7]
		\tilde{v}_n(t + s_n) - v_n(t) 
		= &\,
		e^{t(\Delta - m^2)} \del[1]{w^{0,\fz_n}(s_n) - \fz_n}\\
		& -
		\int_{0}^{t} e^{(t  - r)(\Delta - m^2)}\del[1]{\tilde{v}_n(r + s_n)^3 - v_n(r)^3} \dif r.
	\end{equs}
	Note that the integral term with the forcing~$h_n$ has cancelled out exactly.
	We aim to apply (a generalised version of) Gronwall's inequality and, to this end, analyse the two terms on the right hand side of~\eqref{lem:energy_bounds_pf_aux7} separately.
	
	\paragraph*{Step~$\boldsymbol{2a}$: Comparing Initial Data}
	Using the mild formulation of the skeleton equation we directly have,
	\begin{align*}
		\|e^{t(\Delta-m^2)}(w^{0,\fz_n}(s_n) - \fz_n)\|_{L^p_x} \leq  & \, e^{-m^2 t} \|e^{t\Delta} (e^{s_n(\Delta-m^2)} \fz_n - \fz_n)\|_{L^p} \\
		&+ e^{-m^2 t} \left\|e^{t\Delta}\int_0^{s_n} e^{-(s_n-r) (\Delta-m^2)} w^{0,\fz_n}(r)^3\,\dd s\right\|_{L^p_x}\\
		& \eqqcolon (\RN{1}) + (\RN{2}).
	\end{align*}
	
	To bound the first term we use the regularising effect of the heat semi-group, \eqref{eq:ou_reg}, along with the continuity at zero of the Ornstein--Uhlenbeck semi-group, \eqref{eq:ou_identity_approx}, to see that for any $t>1$, $s_n\in (0,1)$ and $\eps \in (0,-\alpha-3\eta)$
	\begin{equation*}
		(\RN{1}) \lesssim  e^{-t m^2}
		\left\|\del[2]{e^{s_n(\Delta -m^2)}  -1}\mfz_n\right\|_{B^{\nicefrac{5\alpha}{6}-\eps}_{p,p}}  \lesssim (s_n^{\eps} +m^2 s_n)\|\mfz_n\|_{B^{\nicefrac{5\alpha}{12}}_{p,p}} \lesssim  s_n^\eps \|\mfz_n\|_{\mcC^\alpha}\leq s_n^\eps \lambda_n.
	\end{equation*}
	For the second term we proceed similarly, first using that $e^{t(\Delta-m^2)}$ is a uniformly bounded family of operators from $L^p(\mbT^3)$ to itself and then applying \eqref{eq:ou_reg} to give
	\begin{align*}
		(\RN{2}) \lesssim & \,\int_0^{s_n} e^{-(s_n-r)m^2}(s_n-r)^{-\frac{3}{p}} \|w^{0,\mfz_n}(r)^3\|_{L^{\nicefrac{p}{3}}_x} \dd r\\
		\lesssim &\, \|w^{0,\mfz_n}\|^3_{C_{\eta;s_n}L^p_x} \int_0^{s_n} e^{-(s_n-r)m^2}(s_n-r)^{-\frac{3}{p}} r^{-3\eta} \dd r\\
		\lesssim & \, \|w^{0,\mfz_n}\|^3_{C_{\eta;s_n}L^p_x}  s_n^{1-\frac{3}{p}-3\eta}.
	\end{align*}
	So applying estimate \eqref{eq:nl_skeleton_Lp_growth} of Proposition~\ref{prop:nl_skeleton_LWP} and using the assumptions that $\lambda _n \leq 1$, $s_n\in (0,1)$ and $0<\eps  <-\alpha-3\eta_0 = 1-\frac{3}{p}-3\eta_0$ we find
	\begin{align}
		\|e^{t(\Delta-m^2)}(w^{0,\fz_n}(s_n) - \fz_n)\|_{L^p_x}  &\lesssim s_n^\eps \lambda_n +s_n^{1-\frac{3}{p}-3\eta_0} \|w^{0,\mfz_n}\|^3_{C_{\eta_0;s_n}L^p}   \notag\\
		&\leq s_n^\eps \lambda_n + s_n^{1-\frac{3}{p}-3\eta_0} \|\mfz_n\|^3_{\mcC^{\alpha}}\notag\\
		&\lesssim s_n^{\eps}\lambda_n.\label{eq:energy_bounds_pf_initial_data_time_one}
	\end{align}
	\paragraph*{Step~$\boldsymbol{2b}$: Comparing the Non-Linear Term}
	Recall that~$L^{\nicefrac{p}{3}}(\mbT^3) \embed B_{p,p}^{-\nicefrac{6}{p}}(\mbT^3)$ by Besov embedding and let~$t \in [0,T_n]$.
	For the second term in~\eqref{lem:energy_bounds_pf_aux7}, we find 
	\begin{align}
		\thinspace \ 
		\int_{0}^{t} &\norm[1]{e^{(t  - r)(\Delta - m^2)}\sbr[1]{\tilde{v}_n(r + s_n)^3 - v_n(r)^3}}_{L^p_x} \dif r \notag \\
		&\leq \ 
		\int_0^{t} e^{- (t - r) m^2} (t - r)^{-\nicefrac{3}{p}} \norm[0]{\tilde{v}_n(r + s_n)^3 - v_n(r)^3}_{L^{\nicefrac{p}{3}}_x} \dif r \label{lem:energy_bounds_pf_cubic_term}\\
		&\lesssim \ 
		\int_0^{t} e^{- (t - r) m^2} (t - r)^{-\nicefrac{3}{p}} \del[1]{\norm[0]{\tilde{v}_n(r + s_n)}^2_{L^{p}_x} + \norm[0]{v_n(r)}^2_{L^{p}_x}} \norm[0]{\tilde{v}_n(r + s_n) - v_n(r)}_{L^{p}_x} \dif r \notag\\
		&\leq \ 
		C(t) \int_0^{t} e^{- (t - r) m^2} (t - r)^{-\nicefrac{3}{p}} (r \wedge 1)^{-2\eta_0} \norm[0]{\tilde{v}_n(r + s_n) - v_n(r)}_{L^{p}_x} \dif r  \notag
	\end{align}
	where we have set~
	\begin{equation*}
		C(t) \coloneqq  \sup_{r \in [0,t]} (r \wedge 1)^{2\eta_0}  \del[1]{\thinspace\norm[0]{\tilde{v}_n(r + s_n)}^2_{L^{p}_x} + \norm[0]{v_n(r)}^2_{L^p_x}}.
	\end{equation*}

	\paragraph*{Step~$\boldsymbol{2c}$: Conclusion and Gr\"onwall Estimate}
	Since~the family~$\{C(T_n)\}_{n \in \N}$ is uniformly bounded (see Theorem~\ref{th:nl_skeleton_gwp}, item~\ref{it:nl_skeleton_quant_bounds}) we combine~\eqref{lem:energy_bounds_pf_aux7}, \eqref{eq:energy_bounds_pf_initial_data_time_one}, and~\eqref{lem:energy_bounds_pf_cubic_term} to give the following, estimate which holds for all~$t \in [1,T_n]$:
	\begin{equs} \label{lem:energy_bounds_pf_gronwall_1}
		\thinspace &
		\norm[0]{\tilde{v}_n(t + s_n) - v_n(t)}_{L^p_x} \\
		\lesssim & \ 
		e^{-tm^2} s_n^\eps \lambda_n 
		+ \int_0^{t} e^{- (t - r) m^2} (t - r)^{-\nicefrac{3}{p}} (r \wedge 1)^{-2\eta_0} \norm[0]{\tilde{v}_n(r + s_n) - v_n(r)}_{L^p_x} \dif r
	\end{equs}
	Further, observe that the choice of parameters in~\eqref{eq:energy_lem_parameters} guarantees that
	\begin{equation*}
		\frac{3}{p} + 2 \eta_0 < 1.
	\end{equation*}
	Therefore, we may apply the generalised Gr\"onwall inequality~\eqref{eq:gen_gronwall} to see that there exist constants $c_1\coloneqq c_1(\alpha,\eta_0),\, c_2\coloneqq c_2(\alpha,\eta_0)>0$ such that for all $t\in [1,T_n]$ such that,
	\begin{equation} \label{lem:energy_bounds_pf_gronwall_2}
		\norm[0]{\tilde{v}_n(t + s_n) - v_n(t)}_{L^p} 
		\leq c_1 s_n^\eps \lambda_n \exp\del[3]{\del[1]{c_2-m^2  }t}
	\end{equation}
	Taking~$t = T_n$ and using that 
	\begin{equation*}
		s_n \leq   e^{-\frac{T^2_n}{\eps}}\lambda_n^{-\frac{1}{\eps}}
	\end{equation*}
	we find,
	\begin{equation*}
		\norm[0]{\tilde{v}_n(T_n + s_n) - v_n(T_n)}_{L^p_x} \leq c_1 e^{(c_2-m^2)T_n-T^2_n}
	\end{equation*}
	Since the RHS goes to~$0$ as~$n \to \infty$ (and $T_n\nearrow \infty$), the claim of~\eqref{lem:energy_bounds_pf_aux6} follows by Besov embedding $L^p(\mbT^3)\cembed \mcC^{\alpha}(\mbT^3)$, c.f. \eqref{eq:besov_reg_embed}.
\end{proof}
\begin{remark}[Relation between~$\lambda$ and~$\rho$.] \label{rmk:lambda_rho}
	If one could quantify the dependence of~$\lambda$ on~$\delta$ and~$\theta$ in Lemma~\ref{lem:energy_bounds}, such that it matched that of~$\rho = \rho(\theta)$ from Proposition~\ref{prop:measure_tail_bounds} (cf. also Remark~\ref{rmk:quantitative_tail_bound}), then it would be possible to drastically simplify the proof of the LDP upper bound.
	To wit, choosing~$\rho = \lambda$, one could directly apply the estimate~\eqref{coro:energy_bounds:eq1} below to the situation that~$\bar{u}^\fy_\eps$ begins from~$\fy \in B_\lambda(0)$.
	As a consequence, we would not have to introduce the set~$\CE_{\rho,\delta,\theta}(n)$ below and could conclude without considering the, proximately introduced, scenarios~$\RN{2}$ and~$\RN{3}$~separately. However, since our proof of Lemma~\ref{lem:energy_bounds} is non-constructive, this simpler proof does not seem immediately available.
\end{remark}

We are now able to define the \emph{excursion set} $\CE_{\rho,\delta,\theta}(n)$ which will play an important role in the arguments of this section.
\begin{definition}[The excursion set~$\CE_{\rho,\delta,\theta}$]\label{def:excursion}
	Let~$\delta,\, \theta > 0$ and~$\lambda \coloneqq  \lambda(\delta,\, \theta ) > 0$ be the associated parameter obtained in Lemma~\ref{lem:energy_bounds}. 
	Then, for any~$n \in \N$ and~$\rho > 0$, we define the set
	\begin{equation*}
		\CE_{\rho,\delta,\theta}(n)
		\coloneqq 
		\left\{
		v \in C_n\mcC^\alpha(\mbT^3): \ \|v(0)\|_{\mcC^{\alpha}}\leq \rho , \ \|v(j)\|_{\mcC^{\alpha}} \geq  \lambda, \ \text{for all} \ j \in \{ 1,\ldots,n \}
		\right\}.
	\end{equation*}
\end{definition}
The purpose of defining this set is to give us access to a useful decomposition of the event
\begin{equation}\label{eq:dynamic_large_energy}
	\{\operatorname{dist}_{\mcC^{\alpha}(\mbT^3)}(\bar{u}_\eps^{\mfy}(t),\msV[\theta]) \geq \delta\}.
\end{equation}
By statistical invariance of the dynamics with respect to the measure $\mu_{\eps}$, provided $\fy$ is sampled from $\mu_{\eps}$, exponential control on the event described by \eqref{eq:dynamic_large_energy} is equivalent to the  LDP upper bound claimed by Theorem~\ref{thm:ldp_upper_bound}, c.f. Remark~\ref{rmk:locally_uniform_ldp}, item~\ref{rmk:locally_uniform_ldp:i}. Hence, a proof of Theorem~\ref{thm:ldp_upper_bound} will follow from sufficient exponential control on \eqref{eq:dynamic_large_energy}, averaged over $\fy \sim \mu_{\eps}$. To this end we decompose  the event described by \eqref{eq:dynamic_large_energy} into three disjoint scenarios:
\begin{enumerate}[label=\bfseries \Roman*{:}]
	\item \label{scenario_1} when $\mfy \in B_\rho(0)^{\texttt{c}}\subset \mcC^{\alpha}(\mbT^3)$ and $\operatorname{dist}_{\mcC^{\alpha}(\mbT^3)}(\bar{u}_\eps^{\mfy}(t),\msV[\theta]) \geq \delta$,
	\item \label{scenario_2} when $\mfy \in B_{\rho}(0) \subset  \mcC^{\alpha}(\mbT^3)$, $\bar{u}_\eps^{\mfy} \in \CE_{\rho,\delta,\theta}(\bar{n})$ for some $\bar{n} \in \mbN, \bar{n} \leq t$, \\[0.5em] and $\operatorname{dist}_{\mcC^{\alpha}(\mbT^3)}(\bar{u}_\eps^{\mfy}(t),\msV[\theta]) \geq \delta$,
	\item \label{scenario_3} when $\mfy \in B_{\rho}(0) \subset  \mcC^{\alpha}(\mbT^3)$, $\bar{u}_\eps^{\mfy} \in \CE_{\rho,\delta,\theta}(\bar{n})^{\texttt{c}}$ for the same $\bar{n}\in \N$, \\[0.5em] and $\operatorname{dist}_{\mcC^{\alpha}(\mbT^3)}(\bar{u}_\eps^{\mfy}(t),\msV[\theta]) \geq \delta$.
\end{enumerate}
\begin{remark}[Choice of~$t$ and~$\bar{n}$.] \label{rmk:choice_t_nbar}
	We emphasise that, at this stage, both~$t$ and~$\bar{n}$ are \emph{parameters}.	The latter will be determined in Lemma~\ref{lem:bound_prob_H} below whereas the former picks up the constraint~$t \geq \bar{n}$ in~\eqref{pf:ldp_up:outside_H:eq3} below and~$t \geq \bar{T}$ in Lemma~\ref{lem:energy_bounds} above. Consequently, we later choose~$t := \bar{T} + \bar{n}$, see the proof of Proposition~\ref{prop:ldp_upper_outside_H}.
\end{remark} 
The three scenarios are graphically represented in Figure~\ref{fig_scenarios} and the strategies to controlling their probabilities, informally, described below. At the end of each paragraph, we direct the reader to the relevant rigorous results which make these heuristics precise.

\begin{figure}
	\begin{minipage}{.5\textwidth}
		\centering
		\subfloat[\textbf{Scenario I} \label{fig:scenario_1}]{%
			\centering
			\raisebox{1.2em}{
				\includegraphics[scale=0.7]{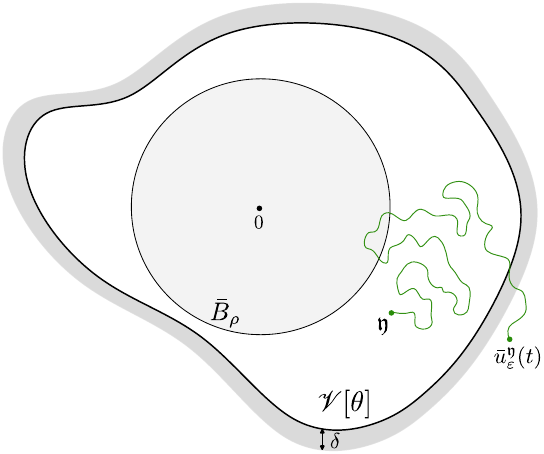} 
			}
		}
	\end{minipage}%
	\hfill %
	\begin{minipage}{.5\textwidth}
		\centering
		\subfloat[\textbf{Scenario II} \label{fig:scenario_2}]{%
			\centering
			\raisebox{0ex}{
				\includegraphics[scale=0.7]{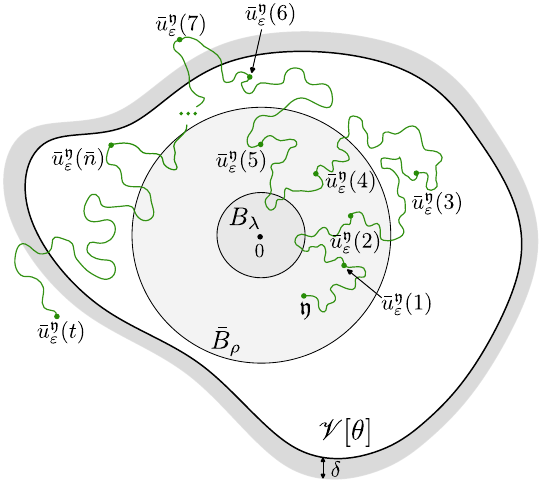} 
			}	
		}
	\end{minipage} \hfill%
	\hspace{3em}
	\centering
	\subfloat[\textbf{Scenario III} \label{fig:scenario_3}]{%
		\centering
		\includegraphics[scale=0.7]{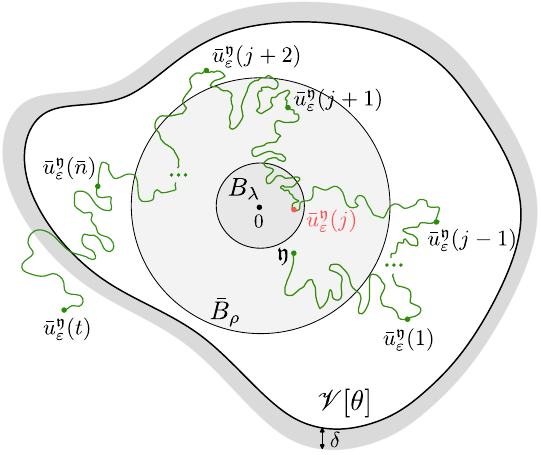} 
	}
	\captionsetup{format=plain, width=0.8\textwidth}
	\caption{\textbf{Schematic overview of the different scenarios.} \\[0.5em]
		The scenarios I, II, and III depicted in Figures~\ref{fig_scenarios}\protect\subref{fig:scenario_1}, \ref{fig_scenarios}\protect\subref{fig:scenario_2}, and~\ref{fig_scenarios}\protect\subref{fig:scenario_3} above are explained in detail in the paragraphs below.
	} 
	\label{fig_scenarios}
\end{figure}

\paragraph*{Figure~\ref{fig_scenarios}\subref{fig:scenario_1} -- Scenario~$\boldsymbol{\RN{1}}$: Initial condition \emph{outside} a ball~$\boldsymbol{B_\rho}$.} 
This is the easiest scenario to control, we simply discard the constraint  $\operatorname{dist}_{\mcC^{\alpha}(\mbT^3)}(\bar{u}_\eps^{\mfy}(t),\msV[\theta]) \geq \delta$ and use the fact that $\fy \sim \mu_\eps$ to obtain a bound on the event $\fy \in B_\rho(0)^{\texttt{c}}$ via the appropriate \emph{tail bounds} on the invariant measure, c.f. Proposition~\ref{prop:measure_tail_bounds} and~\eqref{eq:ldp_upper_step1}.

\paragraph*{Figure~\ref{fig_scenarios}\subref{fig:scenario_2} -- Scenario~$\boldsymbol{\RN{2}}$: Initial condition \emph{inside} a ball~$\boldsymbol{B_\rho}$, \emph{excursions} from~$\boldsymbol{B_\lambda}$ at all integer times up to~$\boldsymbol{\bar{n}}$.}  Similarly to Scenario~I, we discard the second condition, $\operatorname{dist}_{\mcC^{\alpha}(\mbT^3)}(\bar{u}_\eps^{\mfy}(t),\msV[\theta]) \geq \delta$, on the distance between the dynamic path $\bar{u}^{\mfy}_\eps$ and the sub-level set of the quasi-potential,  $\msV[\theta]$. Instead, to control Scenario~II we focus on the condition $\bar{u}_\eps^{\mfy} \in \CE_{\rho,\delta,\theta}(\bar{n})$, see Definition~\ref{def:excursion}. By definition, if the realised path lies in this excursion set then at each integer time $j\in \{1,\ldots,\bar{n}\}$ we must have that $\bar{u}^\fy_\eps(j) \in B^{\texttt{c}}_\lambda$. Since $0$ is the unique, asymptotically stable attractor of the deterministic dynamics, these $\lambda$-\emph{excursions} require repeated inputs of energy from the noise, which the locally uniform LDP for the dynamics (see Proposition~\ref{prop:uniform_ldp}) guarantees is an exponentially vanishing event as $\eps\to 0$. This argument is carried out in Lemma~\ref{lem:bound_prob_H} and Proposition~\ref{prop:bound_prob_H}, which are made use of in~\eqref{eq:ldp_upper_step2} below.

\paragraph*{Figure~\ref{fig_scenarios}\subref{fig:scenario_3} --  Scenario~$\boldsymbol{\RN{3}}$: Initial condition \emph{inside} a ball~$\boldsymbol{B_\rho}$, \emph{entrance} into~$\boldsymbol{B_\lambda}$ between time~$\boldsymbol{1}$ and~$\boldsymbol{\bar{n}}$.}
To control this scenario we finally make use of the assumption that $\operatorname{dist}_{\mcC^{\alpha}(\mbT^3)}(\bar{u}_\eps^{\mfy}(t),\msV[\theta]) \geq \delta$. The core idea is to appeal to Lemma~\ref{lem:energy_bounds} which, informally, tells us that, taking $t>0$ sufficiently large, realisations of the path which lie at distance $\delta>0$ or more from the sub-level set $\msV[\theta]$ must also have have dynamic energy larger than $\theta$. The locally uniform LDP for the dynamics, given by Proposition~\ref{prop:uniform_ldp},  ensures that this is an exponentially vanishing event, uniformly in the initial data $\fy \in B_\rho$, as $\eps\to 0$. Technically, we make use of the additional assumption that $\bar{u}^\fy_\eps\in \CE_{\rho,\delta,\theta}(\bar{n})^{\texttt{c}}$ to guarantee that for at least some integer time $j \in \{1,\ldots,\bar{n}\}$ one has $\bar{u}^\fy_\eps(j) \in B_\lambda$. This allows us to leverage the \emph{Markov property} (c.f. Figure~\ref{fig_Markov_property}) in order to properly apply Lemma~\ref{lem:energy_bounds} for the equation restarted from $B_\lambda$. This argument is made rigorous in Lemma~\ref{lem:prob_energy_bounds} and Proposition~\ref{prop:ldp_upper_outside_H} which are made use of in~\eqref{eq:ldp_upper_step3} below

\subsection[Control on the Dynamic Remaining in $\CE_{\rho,\delta,\theta}$]{Control on the Dynamic Remaining in $\boldsymbol{\CE_{\rho,\delta,\theta}}$}\label{sec:control_inside_excursion}
The following lemma identifies a time horizon~$\bar{n} \in \N$ (cf. Remark~\ref{rmk:choice_t_nbar}) for which the set $\CE_{\rho,\delta,\theta}(\bar{n})$ is in the complement of the $\theta$-sub-level set of the dynamic rate function~$\II_{\bar{n}}$.
In other words: Excursion paths~$v \in \CE_{\rho,\delta,\theta}(\bar{n})$ have energy~$\II_{\bar{n}}(v)$ larger than~$\theta$.
\begin{lemma} \label{lem:bound_prob_H}
	For each~$\rho,\,\delta,\, \theta > 0$,  there exists some~$\bar{n} = \bar{n}(\rho,\,\delta,\, \theta ) \in \N$ such that 
	\begin{equation*}
		\beta_{\bar{n}} \coloneqq  \inf_{v \in \CE_{\rho,\delta,\theta}(\bar{n})} \II_{\bar{n}}(v) > \theta.
	\end{equation*}
	Since, by definition, all paths~$v \in \CE_{\rho,\delta,\theta}(\bar{n})$ start in~$\bar{B}_\rho$, the previous estimate is, in fact, locally uniform in the initial data, i.e. it is equivalent to
	\begin{equation*}
		\inf_{\fz \in \bar{B}_\rho} \inf_{v \in \CE_{\rho,\delta,\theta}(\bar{n})} \II^\fz_{\bar{n}}(v) > \theta.
	\end{equation*}
\end{lemma}

\begin{proof}
	Let us suppose for a contradiction that~$\sup_{n \in \N} \beta_n \leq \theta$.
	Then, for each~$n \in \N$, we can fix a path~$v_n \in \CE_{\rho,\delta,\theta}(n)$ such that
	\begin{equation}
		\II_n(v_n) \leq \beta_n + 1 \leq \theta + 1.
		\label{lem:bound_prob_H:pf_contradiction}
	\end{equation}
	By Theorem~\ref{th:nl_skeleton_gwp}, item~\ref{it:nl_skeleton_quant_bounds}, applied with~$p >6$ (such that by \eqref{eq:besov_reg_embed} one has~$L^p(\T^3) \cembed \CC^\alpha(\T^3)$), and recalling the definition of $\II_t(v_t)$ given in~\eqref{eq:uniform_ldp_rf}, for any~$n \in \N$ we have
	\begin{equation}\label{pf_lem_bound_prob_H_eq1}
		\sup_{t \in [1,n]} \norm[0]{v_n(t)}_{\mcC^{\alpha}} \leq 	\sup_{t \in [1,n]} \norm[0]{v_n(t)}_{L^p_x} 
		\lesssim_{m,\alpha} \norm[0]{v_n(0)}_{\CC^\alpha}+\II_n(v_n)^{\frac{1}{2}}
		\leq \rho + (1+\theta)^{\frac{1}{2}}=: c_{\rho,\theta} > \rho.
	\end{equation}
	With~$\lambda \coloneqq  \lambda(\delta,\, \theta )$ again as in Lemma~\ref{lem:energy_bounds}, we then introduce the set
	\begin{equation*}
		\tilde{\CE}^{\mathtt{EP}}_{\rho,\,\delta,\, \theta }(\bar{k})
		\coloneqq  
		\{v \in C_{\bar{k}}\mcC^{\alpha}(\mbT^3): \ \norm[0]{v(0)}_{\CC^\alpha} \leq c_{\rho,\theta}, \ \norm[0]{v(\bar{k})}_{\mcC^{\alpha}} \geq \lambda\}. 
	\end{equation*}
	of \enquote{end-point excursions} starting in a ball of radius~$c_{\rho,\theta} > \rho$.
	We now  reduce ourselves to prove that there exists some~$\bar{k} \in \N$ large enough such that 
	\begin{equation}
		\iota_{\bar{k}} \coloneqq  \inf\{\II_{\bar{k}}(v): \ v \in \tilde{\CE}^{\mathtt{EP}}_{\rho,\,\delta,\, \theta }(\bar{k})\} > 0.
		\label{pf_lem_bound_prob_H_eq2}
	\end{equation}
	To see that~\eqref{pf_lem_bound_prob_H_eq2} leads to a contradiction, given any~$n \in \N$, denote by
	\begin{equation}
		v_{n\bar{k}}^{(m)}(t) \coloneqq  v_{n\bar{k}}\sVert[0]_{[m\bar{k},(m+1)\bar{k}]}(m\bar{k}+t), \quad t \in [0,\bar{k}], \quad m \in \{0,\ldots,n-1\},
		\label{lem:bound_prob_H:pf_v_concatenation}
	\end{equation}
	a reparametrisation of~$v_{n\bar{k}}\sVert[0]_{[m\bar{k},(m+1)\bar{k}]}$ that lies in~$C_{\bar{k}}\mcC^{\alpha}(\mbT^3)$; see Figure~\ref{fig:path_vnkbar} for a visualisation.
	Then, note that due to~\eqref{pf_lem_bound_prob_H_eq1} and because~$v_{n\bar{k}} \in \CE_{\rho,\delta,\theta}(n\bar{k})$, one has $v_{n\bar{k}}^{(m)} \in \tilde{\CE}^{\mathtt{EP}}_{\rho,\,\delta,\, \theta }(\bar{k})$ for each~$m \in \{ 0,\ldots,n-1 \}$ such that we find
	\begin{equation}
		\theta + 1 
		\geq 
		\II_{n\bar{k}}(v_{n\bar{k}})
		=
		\sum_{m=0}^{n-1} \II_{[m\bar{k},(m+1)\bar{k}]}\del[1]{v_{n\bar{k}}\sVert[0]_{[m\bar{k},(m+1)\bar{k}]}}
		=
		\sum_{m=0}^{n-1} \II_{\bar{k}}\del[1]{v_{n\bar{k}}^{(m)}}
		\geq 
		n \iota_{\bar{k}}
		\overset{\eqref{pf_lem_bound_prob_H_eq2}}{>} 0.
		\label{lem:bound_prob_H:pf_contradiction_pf}
	\end{equation}
	Since~$n \in \N$ can be chosen arbitrarily large we obtain a contradiction.

	\begin{figure}
		\begin{minipage}{.48\textwidth}
			\centering
			\subfloat{%
				\centering
				\includegraphics[scale=0.4]{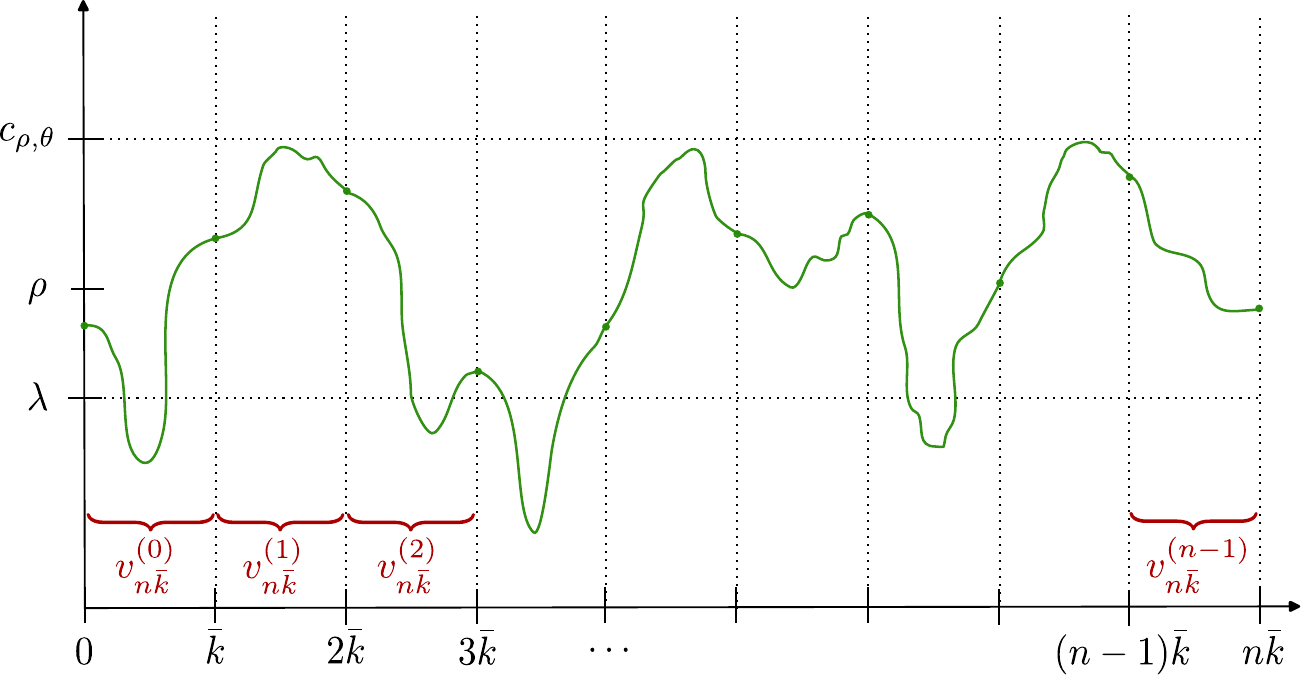}
			}
		\end{minipage}\hfill
		\hspace{7em}
		\begin{minipage}{.48\textwidth}
			\subfloat{%
				\centering
				\raisebox{0.5em}{
					\includegraphics[scale=0.6]{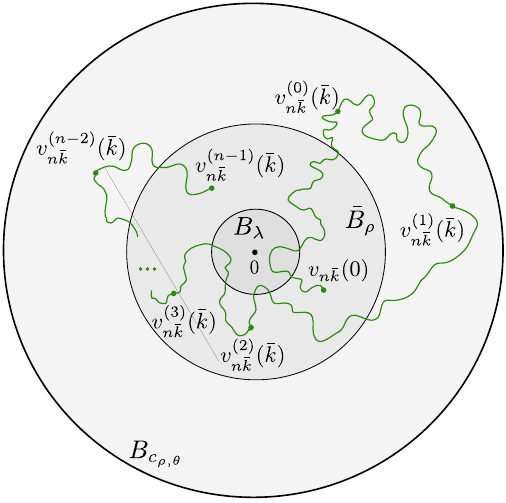}
				}
			}
		\end{minipage} 
		\captionsetup{format=plain, width=0.8\textwidth}
		\caption{\textbf{Visualisation of the path}~$\boldsymbol{v_{n\bar{k}}}$.} 
		\label{fig:path_vnkbar}
	\end{figure}
	
	We are therefore only required to show~\eqref{pf_lem_bound_prob_H_eq2}. By Theorem~\ref{th:nl_skeleton_gwp}, item ~\ref{it:nl_skeleton_global_decay}, there exists some time~$\bar{t} \geq 1$ such that 
	\begin{equation}
		\norm[0]{w_1^{0,\fz}(t)}_{\mcC^{\alpha}}\leq  \norm[0]{w_1^{0,\fz}(t)}_{L^p_x} \leq \frac{\lambda}{2}
		\quad \text{for all} \quad t \geq \bar{t},
		\label{pf_lem_bound_prob_H_eq3}
	\end{equation}
	where we have use the notation $w^{0,\fz}_1$ to indicate that the initial condition $\mfz$ holds at $t=1$, i.e. $w^{0,\fz}_1(1) =\mfz$. This trivially implies that, for any integer~$\bar{k} \geq \bar{t}$, we have
	\begin{equation}
		w_1^{0,\fz}\sVert[0]_{[1,\bar{k}]} \notin \left\{v \in C_{[1,\bar{k}]}\mcC^{\alpha}(\mbT^3): \ \norm[0]{v(\bar{k})}_{\CC^\alpha} \geq \lambda \right\}.
		\label{pf_lem_bound_prob_H_eq4}
	\end{equation}
	We claim that for any such~$\bar{k}$, we have~$\iota_{\bar{k}} > 0$. 
	Suppose otherwise, i.e. that $\iota_{\bar{k}} = 0$. In that case, there exists a minimising sequence
	\begin{equation*}
		\{\tilde{v}_n\}_{n \in \N} \subseteq \tilde{\CE}^{\mathtt{EP}}_{\rho,\,\delta,\, \theta }(\bar{k}), \quad
		\lim_{n \to \infty} \II_{\bar{k}}(\tilde{v}_n) = 0.
		\label{pf_lem_bound_prob_H_minseq}
	\end{equation*}
	For each~$n \in \N$, we have~$\norm[0]{\tilde{v}_n(0)}_{\mcC^\alpha} \leq c_{\rho,\theta}$ by definition of~$\tilde{\CE}^{\mathtt{EP}}_{\rho,\,\delta,\, \theta }(\bar{k})$ and, therefore, the same argument as in~\eqref{pf_lem_bound_prob_H_eq1} implies that
	\begin{equation*}
		\norm[0]{\tilde{v}_n(1)}_{L^p_x} 
		\leq c_{\rho,\theta} + \II_{\bar{k}}(\tilde{v}_n)^{\frac{1}{2}}
		\leq c_{\rho,\theta} + C
	\end{equation*}
	since~$\{\II_{\bar{k}}(\tilde{v}_n)\}_{n \in \N}$ is a convergent sequence and therefore bounded by some~$C < \infty$.
	As a consequence,~$\{{v}_n(1)\}_{n \in \N} \subseteq L^p(\mbT^3)$ is a bounded sequence and so by the compact embedding $L^p(\mbT^3)~\cembed~\mcC^{\alpha}(\mbT^3)$ (due to our choice of $p>6$) there exist a subsequence~$\{\tilde{v}_{n_j}(1)\}_{j \in \N}$ and~$\bar{\fz} \in \mcC^{\alpha}(\T^3)$ such that 
	\begin{equation*}
		\lim_{j \to \infty} \tilde{v}_{n_j}(1) = \bar{\fz} \quad \text{in} \quad \mcC^{\alpha}(\mbT^3)\
	\end{equation*}
	Since, by construction~$\II_{\bar{k}}(\tilde{v}_n) \to 0$ as~$n \to \infty$, the definition of~$\II_{\bar{k}}$ implies the existence of a sequence~$(h_n)_{n \in \N} \subseteq L^2_{[1,\bar{k}]}L^2(\mbT^3)$ such that 
	\begin{equation*}
		\tilde{v}_n |_{[1,\bar{k}]} = w^{h_n,\tilde{v}_n(1)}_1 \quad \text{and} \quad \lim_{n \to \infty} h_n = 0 \quad \text{in} \quad L^2_{[1,\bar{k}]}L^2(\mbT^3).
	\end{equation*}
	Therefore, by joint continuity of the solution map in the initial data and forcing, we have
	\begin{equation*}
		\lim_{j \to \infty} \tilde{v}_{n_j}\sVert[0]_{[1,\bar{k}]} 
		=
		w_1^{0,\bar{\fz}} 
		\quad
		\text{in} 
		\quad
		C_{[1,\bar{k}]}\mcC^{\alpha}(\mbT^3)
	\end{equation*}
	and then~$\norm[0]{w_1^{0,\bar{\fz}}(\bar{k})}_{\mcC^{\alpha}} \geq \lambda$ since~$\tilde{v}_n \in \tilde{\CE}^{\mathtt{EP}}_{\rho,\,\delta,\, \theta }(\bar{k})$.
	This contradicts~\eqref{pf_lem_bound_prob_H_eq4} and we indeed have~$\iota_{\bar{k}} > 0$.
	The claim in~\eqref{pf_lem_bound_prob_H_eq2} follows. 
	The proof is complete.
\end{proof}
The following proposition uses the findings of Lemma~\ref{lem:bound_prob_H} to establish an upper bound on the probability that the $\Phi^4_3$ dynamics remains inside the set~$\CE_{\rho,\delta,\theta}(\bar{n})$.
For technical reasons, proving local uniformity with respect to the initial condition in that bound requires a uniform LDP upper bound \emph{different from}---and in general \emph{not equivalent to}---that in Proposition~\ref{prop:uniform_ldp}, cf. Remark~\ref{rmk:locally_uniform_ldp}.
In our situation, however, we are able to show the required bound (see Corollary~\ref{cor:DZULDP} in the Appendix) essentially because the solution map associated to singular stochastic PDEs is \emph{locally Lipschitz continuous} in the initial condition.
\begin{proposition}[Scenario~$\RN{2}$] \label{prop:bound_prob_H}
	Let $\rho,\, \delta,\, \theta >0$ and~$\bar{n}\coloneqq \bar{n}(\rho,\,\delta,\, \theta ) \in \N$ be given by Lemma~\ref{lem:bound_prob_H}. Then, given~$\gamma > 0$, there exists an~$\eps_1 > 0$ such that
	\begin{equation*}
		\sup_{\fy \in \bar{B}_\rho} \P\del[1]{\bar{u}^\fy_\eps \in \CE_{\rho,\delta,\theta}(\bar{n})} 
		\leq
		\exp\del[3]{-\frac{\theta-\nicefrac{\gamma}{2}}{\eps^2}}
		\quad \text{for all }\,\,\eps \leq \eps_1.
	\end{equation*}
\end{proposition}

\begin{proof}
	Observe that the map	
	\begin{equation*}
		\CC^\alpha(\mbT^3) \to [0,1], \quad \fy \mapsto \P\del[2]{\bar{u}^\fy_\eps \in \CE_{\rho,\delta,\theta}(\bar{n})}
	\end{equation*}
	is continuous by the strong Feller property of the semigroup generated by~$\bar{u}_\eps$, see~\cite{hairer-mattingly}. 
	Since $H^1(\T^3)$ is dense in~$\CC^\alpha(\T^3)$, we have 
	\begin{equation*}
		\sup_{\fy \in \bar{B}_\rho} \P\del[2]{\bar{u}^\fy_\eps \in \CE_{\rho,\delta,\theta}(\bar{n})}
		=
		\sup_{\fy \in \bar{B}_\rho \cap H^1(\T^3)} \P\del[2]{\bar{u}^\fy_\eps \in \CE_{\rho,\delta,\theta}(\bar{n})} \,.
	\end{equation*}
	Note that $\bar{B}_\rho$ is closed and~$H^1(\mbT^3)$ is compact in~$\CC^\alpha(\mbT^3)$, so $\bar{B}_\rho \cap H^1(\T^3)$ is compact in~$\CC^\alpha(\mbT^3)$ as well.
	As a consequence, the DZULDP upper bound from Corollary~\ref{cor:DZULDP} applies to the closed set~$\CE_{\rho,\delta,\theta}(\bar{n})$ and reads
	\begin{equation*}
		\limsup_{\eps \to 0} \eps ^2 \log \sup_{\fy \in \bar{B}_\rho} \P\del[1]{\bar{u}^\fy_\eps \in \CE_{\rho,\delta,\theta}(\bar{n})} 
		\leq 
		-\inf_{\fy \in \bar{B}_\rho} \inf_{v \in \CE_{\rho,\delta,\theta}(\bar{n})} \II_T^\fz(v)
		< - \theta 
	\end{equation*}
	where the last estimate uses Lemma~\ref{lem:bound_prob_H}.
	The claim follows.
\end{proof}
\subsection[Control when the Dynamic Leaves $\CE_{\rho,\delta,\theta}$]{Control when the Dynamic Leaves $\boldsymbol{\CE_{\rho,\delta,\theta}}$}\label{sec:control_outside_excursion}
We first require the following version of Lemma~\ref{lem:energy_bounds} and record its specialisation to~$\hat{u}^\fy_\eps$.
\begin{lemma} \label{lem:prob_energy_bounds}
	Let~$\delta,\, \theta>0$ and $\lambda \coloneqq  \lambda(\delta,\,\theta)>0$ be as in Lemma~\ref{lem:energy_bounds}. Then, recalling the notation
	\begin{equation}
		\II^\mfy_{t}[\theta] = \left\{ v \in C_t\mcC^{\alpha}(\mbT^3)\,:\, v(0) = \mfy,\quad \msI_{t}(v)\leq \theta\right\},
		\label{lem:prob_energy_bounds_eq0}
	\end{equation}
	there exists some~$\bar{T} > 0$ such that for any~$\mfy\in \bar{B}_\lambda \subset \mcC^{\alpha}(\mbT^3)$ and~$t \geq \bar{T}$ one has
	\begin{equs}
		\thinspace
		& \left\{v \in C_t \CC^\alpha: \, v(0) = \fy, \, \operatorname{dist}_{\mcC^{\alpha}}\del[1]{v(t),\msV[\theta]} \geq \delta\right\} \\
		\subseteq \ &  
		\left\{v \in C_t \CC^\alpha: \, v(0) = \fy, \, \operatorname{dist}_{C_t \mcC^{\alpha}}\del[1]{v, \II^\mfy_{t}[\theta]} \geq \frac{\delta}{2} \right\}.
		\label{coro:energy_bounds:claim}
	\end{equs}
	In particular, we have 
	\begin{equation}
		\P\del[1]{\operatorname{dist}_{\mcC^{\alpha}}\del[1]{\bar{u}^\fy_\eps(t),\msV[\theta]} \geq \delta}
		\leq 
		\P\del[3]{\operatorname{dist}_{C_t \mcC^{\alpha}}\del[1]{\bar{u}^\fy_\eps, \II^\fy_{t}[\theta]} \geq \frac{\delta}{2}}.
		\label{coro:energy_bounds:eq1}
	\end{equation}  
\end{lemma}
\begin{proof}
	Let~$L$ denote the set on the LHS of~\eqref{coro:energy_bounds:claim} and~$R$ the set on its RHS.
	The statement in~\eqref{coro:energy_bounds:eq1} then simply reads
	\begin{equation*}
		\P\del[1]{\bar{u}^\fy_\eps \in L} \, \leq  \, \P\del[1]{\bar{u}^\fy_\eps \in R}
	\end{equation*}
	and thus follows from~\eqref{coro:energy_bounds:claim} by monotonicity of probability measures with respect to set inclusion.
	
	In order to prove~\eqref{coro:energy_bounds:claim}, let~$v \in L$. 
	Taking the complement on both sides of~\eqref{lem:energy_bounds:eq} we see that we must therefore have
	\begin{equation*}
		v \notin \II^{\bar{B}_\lambda}_t[\theta] \quad \text{for all} \quad t \geq \bar{T}.
	\end{equation*}
	In words, any~$v \in L$ cannot be an element in any of the sets~$\II^{\bar{B}_\lambda}_t[\theta]$,~$t \geq \bar{T}$. 
	However, it is not a priori excluded that there exists a~$v \in L$ and some~$t \geq \bar{T}$ such that~$v$ is \emph{arbitrarily close} to the set~$\II^{\bar{B}_\lambda}_t[\theta]$.
	
	In order to rule that out, assume for a contradiction that~$\operatorname{dist}_{C_t \mcC^{\alpha}}\del[1]{v, \II^{\bar{B}_\lambda}_t[\theta]} < \nicefrac{\delta}{2}$ for some $t \geq \bar{T}$. 
	Then we may find~$\tilde{v} \in \II^{\bar{B}_\lambda}_t[\theta]$ such that
	\begin{equation*}
		\norm[0]{v(t) - \tilde{v}(t)}_{\mcC^{\alpha}}
		\leq
		\norm[0]{v - \tilde{v}}_{C_t \mcC^{\alpha}}
		< \frac{\delta}{2}.
	\end{equation*}
	However, Lemma~\ref{lem:energy_bounds} applied to~$\tilde{v} \in \II^{\bar{B}_\lambda}_t[\theta]$ shows that~$\operatorname{dist}_{\mcC^{\alpha}}(\tilde{v}(t),\msV[\theta]) < \nicefrac{\delta}{2}$. 
	That is, there exists some~$\fz \in \msV[\theta]$ such that~$\norm[0]{\tilde{v}(t) - \fz}_{\mcC^{\alpha}} < \nicefrac{\delta}{2}$.
	As a consequence, the estimate
	\begin{equation*}
		\norm[0]{v(t) - \fz}_{\mcC^{\alpha}}
		\leq 
		\norm[0]{v(t) - \tilde{v}(t)}_{\mcC^{\alpha}} + \norm[0]{\tilde{v}(t) - \fz}_{\mcC^{\alpha}}
		< \frac{\delta}{2} + \frac{\delta}{2}
		=
		\delta
	\end{equation*} 
	contradicts our assumption that~$v \in L$ and so we actually have 
	\begin{equation*}
		\operatorname{dist}_{C_t \mcC^{\alpha}}\del[1]{v, \II^{\bar{B}_\lambda}_t[\theta]} \geq \nicefrac{\delta}{2}.
	\end{equation*}
	Finally, note that~$\fy \in \bar{B}_\lambda$ implies~$\II^{\fy}_{t}[\theta] \subseteq \II^{\bar{B}_\lambda}_t[\theta]$ which, together with the previous estimate, leads to the lower bound
	\begin{equation*}
		\operatorname{dist}_{C_t \mcC^{\alpha}}\del[1]{v, \II^\fy_t[\theta]}
		\geq
		\operatorname{dist}_{C_t \mcC^{\alpha}}\del[1]{v, \II^{\bar{B}_\lambda}_t[\theta]}
		\geq \frac{\delta}{2}.
	\end{equation*}
	Since~$v \in L$ by definition also entails that~$v(0) = \fy$, we arrive at the conclusion that~$v \in R$. 
	This establishes the claim in~\eqref{coro:energy_bounds:claim}.
\end{proof}

\begin{proposition}[Scenario~$\RN{3}$]\label{prop:ldp_upper_outside_H}
	Let $\delta,\,\gamma,\, \rho,\, \theta>0$ and $\bar{n}\in \mbN$. Then, there exists a $t>0$ and $\eps_0(t,\bar{n},\delta,\,\theta,\gamma)\coloneqq  \eps_0>0$ such that 
	\begin{equation*}
		\sup_{\mfy \in \bar{B}_\rho}\P\del[2]{\operatorname{dist}_{\mcC^{\alpha}}(\bar{u}^\fy_\eps(t),\msV[\theta])\geq \delta, \ \bar{u}^\fy_\eps \in \CE_{\rho,\delta,\theta}(\bar{n})^{\texttt{c}}}   \leq \bar{n}	\exp\del[3]{-\frac{\theta-\nicefrac{\gamma}{2}}{\eps^2}}, \quad \text{for all}\,\,\, \eps\in (0,\eps_0).
	\end{equation*}
\end{proposition}
\begin{proof}
	We first note that using the definition of $\CE_{\rho,\delta,\theta}(\bar{n})$, for all $\mfy \in \bar{B}_\rho$ it holds that, 
	\begin{equs}[pf:ldp_up:outside_H:eq1]
		\thinspace 
		\P&\del[2]{\operatorname{dist}_{\mcC^{\alpha}}(\bar{u}^\fy_\eps(t),\msV[\theta])\geq \delta, \ \bar{u}^\fy_\eps \in \CE_{\rho,\delta,\theta}(\bar{n})^{\texttt{c}}} \\
		&	\leq \ 
		\P\del[4]{\thinspace\bigcup_{j=1}^{\bar{n}} 
			\left\{ 
			\operatorname{dist}_{\mcC^{\alpha}}(\bar{u}^\fy_\eps(t)\msV[\theta])\geq \delta, \ \bar{u}^\fy_\eps(j) \in \bar{B}_\lambda
			\right\}
		}
	\end{equs}
	Using the fact that $t\mapsto \bar{u}^\fy_\eps(t)$ is a Markov process~(see for example~\cite[Coro.~1.3]{hairer_matetski_18_discretisation}) we may restart the equation at any time~$j \in \{ 1,\ldots,\bar{n}\}$ when~$\hat{u}^\fy_\eps(j) \in \bar{B}_\lambda$
	to find that for any $t >\bar{n}$, which will be chosen later, 
	\begin{equation}	
		\P\del[4]{\thinspace\bigcup_{j=1}^{\bar{n}} 
			\left\{ 
			\operatorname{dist}_{\mcC^{\alpha}}(\bar{u}^\fy_\eps(t),\msV[\theta])\geq \delta, \ \bar{u}^\fy_\eps(j) \in \bar{B}_\lambda
			\right\}
		}
		\leq 
		\sum_{j=1}^{\bar{n}} \sup_{\fz \in \bar{B}_\lambda} 
		\P\del[2]{\operatorname{dist}_{\mcC^{\alpha}}(\bar{u}^\fz_\eps(t-j),\msV[\theta])\geq \delta};
		\label{pf:ldp_up:outside_H:eq2}
	\end{equation}
	see Figure~\ref{fig_Markov_property} for a graphical representation of this restarting procedure.

	\begin{figure}
		\subfloat[]{%
			\centering
			\includegraphics[scale=0.7]{Images/scenario_3.pdf} 
			\label{fig_5_left}
		}
		\tikzstyle{smallblock3bold} = [rectangle, draw, text width=6em, text centered, rounded corners, minimum height=1em]
		\tikzstyle{line} = [draw, -latex']
		\begin{tikzpicture}[node distance=0.5cm, auto]
			\node (init) {};
			\node (L) at (-1,2) {};
			\node [right=0.5cm of L] (R) {};
			\draw [line] (L) edge node[fill=white, above=0.35cm, pos=0.8]{} (R) ;
		\end{tikzpicture}
		\hspace{-2em}
		\subfloat[]{%
			\centering
			\includegraphics[scale=0.7]{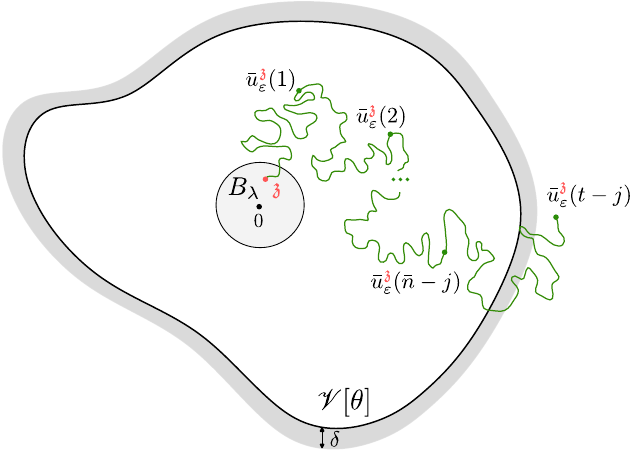} 
			\label{fig_5_right}
		}
		\captionsetup{format=plain, width=0.8\textwidth}
		\caption{\textbf{Application of the Markov property.} \\[0.5em]
			\protect\subref{fig_5_left} The figure shows a realisation of~$\bar{u}_\eps^\fy$ that hits the ball~$B_\lambda$ at time~$j$, at which point the Markov property is applied. \\[0.5em]
			\protect\subref{fig_5_right} The figure presents a realisation of the restarted dynamics~$\bar{u}_\eps^\fz$ after the probability has additionally been maximised over all possible points~$\fz \in B_\lambda$ that~$\bar{u}^\fy_\eps(j)$ could have hit.
		}
		\label{fig_Markov_property}
	\end{figure}
	Now recalling Lemma~\ref{lem:prob_energy_bounds} and the $\bar{T}\coloneqq \bar{T}(\delta,\,\theta)>0$ defined therein, we choose $t= \bar{T}+\bar{n}$ so that applying~\eqref{coro:energy_bounds:eq1} we obtain the bound
	\begin{align}	
		\thinspace
		& \sum_{j=1}^{\bar{n}} \sup_{\fz \in \bar{B}_\lambda} 
		\P\del[2]{\operatorname{dist}_{\mcC^{\alpha}}(\bar{u}^\fz_\eps(t-j),\msV[\theta])\geq \delta} \notag\\ 
		\leq & \ \bar{n}\max_{j= 1,\ldots,\bar{n}}  \sup_{\fz \in \bar{B}_\lambda} 	\P\del[2]{\operatorname{dist}_{\mcC^{\alpha}}(\bar{u}^\fz_\eps(t-j),\msV[\theta])\geq \delta} \label{pf:ldp_up:outside_H:eq3} \\
		\leq & \ \bar{n} 	\max_{j= 1,\ldots,\bar{n}} \sup_{\fz \in \bar{B}_\lambda}  \P\del[3]{\operatorname{dist}_{C_{t-j} \mcC^{\alpha}}\del[1]{\bar{u}^\fz_\eps, \II^\fy_{t-j}[\theta]} \geq \frac{\delta}{2}} \notag
	\end{align}
	We now recall the locally uniform LDP satisfied by the process $\hat{u}^\fy_\eps$ as given by Proposition~\ref{prop:uniform_ldp}. Using the notation of that proposition, given any $\gamma>0$ and $\delta,\,\theta>0$ as above, let us set
	\begin{equation*}
		\eps_0 
		\coloneqq  \eps_0(t,\bar{n},\delta, \gamma,\theta) 
		\coloneqq  \min\left\{\eps(t-j,\delta, \gamma,\theta): \ j \in \{1,\ldots,\bar{n} \} \right\},
	\end{equation*}
	so that applying Proposition~\ref{prop:uniform_ldp},~item \ref{prop:uniform_ldp_3}, we find, for any $\gamma>0$,
	\begin{equation}\label{pf:ldp_up:outside_H:final}
		\max_{j= 1,\ldots,\bar{n}} \sup_{\fz \in \bar{B}_\lambda} \P\del[3]{\operatorname{dist}_{C_{t-j} \mcC^{\alpha}}\del[1]{\bar{u}^\fz_\eps, \II^\fz_{t-j}[\theta]} \geq \frac{\delta}{2}} \leq \exp\del[3]{-\frac{\theta-\nicefrac{\gamma}{2}}{\eps^2}},\quad \text{for all}\quad \eps<\eps_0.
	\end{equation}
	Note that the respective right hand sides of \eqref{pf:ldp_up:outside_H:eq2}, \eqref{pf:ldp_up:outside_H:eq3} and \eqref{pf:ldp_up:outside_H:final} do not depend on~$\fy \in \bar{B}_\rho$ any more;
	combined with~\eqref{pf:ldp_up:outside_H:eq1}, they establish the claim and conclude the proof.
\end{proof}

\subsection{Proof of the LDP Upper Bound} \label{sec:ldp_upper_bound_proof}

We are now in a position to prove the LDP upper bound claimed in Theorem~\ref{thm:ldp_upper_bound}.

\begin{proof}[Proof of Theorem~\ref{thm:ldp_upper_bound}]
	First, by linearity of the integral, invariance of the measure $\mu_{\eps}$ for the dynamic, renormalised process, $u_{\eps}$ and additivity of measures, for any $\rho>0,\, t>0$, which will be further restricted later on, and $\eps>0$ , it holds that,
	\begin{align}
		\thinspace 
		& \mu_\eps\del[1]{\fz \in \mcC^{\alpha}(\mbT^3): \ \operatorname{dist}_{\mcC^{\alpha}}(\fz,\msV[\theta])\geq \delta} \notag\\
		= \ & 
		\int_{\mcC^{\alpha}} \P\del[2]{\operatorname{dist}_{\mcC^{\alpha}}(\bar{u}^\fy_\eps(t),\msV[\theta])\geq \delta} \mu_\eps(\dif \fy) \notag\\
		= \ & 
		\int_{B_\rho^{\texttt{c}}} \P\del[2]{\operatorname{dist}_{\mcC^{\alpha}}(\bar{u}^\fy_\eps(t),\msV[\theta])\geq \delta} \mu_\eps(\dif \fy) \notag\\
		& +
		\int_{B_\rho} \P\del[2]{\operatorname{dist}_{\mcC^{\alpha}}(\bar{u}^\fy_\eps(t),\msV[\theta])\geq \delta} \mu_\eps(\dif \fy).
		\label{pf:ldp_ub:split} \\
		=\ & 	\int_{B_\rho^{\texttt{c}}} \P\del[2]{\operatorname{dist}_{\mcC^{\alpha}}(\bar{u}^\fy_\eps(t),\msV[\theta])\geq \delta} \mu_\eps(\dif \fy) \notag\\
		&
		+ \int_{B_\rho} \P \del[2]{\operatorname{dist}_{\mcC^{\alpha}}(\bar{u}^\fy_\eps(t),\msV[\theta])\geq \delta, \ u^\fy_\eps \in \CE_{\rho,\delta,\theta}(\bar{n})}\mu_\eps(\dif \fy)\notag\\
		&+ \int_{B_\rho} \P\del[2]{\operatorname{dist}_{\mcC^{\alpha}}(\bar{u}^\fy_\eps(t),\msV[\theta])\geq \delta, \ u^\fy_\eps \in \CE_{\rho,\delta,\theta}(\bar{n})^{\texttt{c}}}\mu_\eps(\dif \fy) \notag\\
		\eqqcolon & \ (\RN{1}) + (\RN{2}) + (\RN{3}). \notag
	\end{align}
	In the above we recall the notation $\bar{u}^\fy_\eps$ for the solution to the infinitely renormalised dynamics, \eqref{eq:phi43_rigorous_SPDE} with $u^\fy_\eps(0) =\fy$.  Treating the first term enforces our choice of $\rho$. We apply Proposition~\ref{prop:measure_tail_bounds} with~$\rho = \rho(\theta) > 0$ and~$K > 0$ as obtained therein, to see that for any $\eps>0$,
	\begin{equation}\label{eq:ldp_upper_step1}
		(\RN{1}) = \int_{B_\rho^{\texttt{c}}} \P\del[2]{\operatorname{dist}_{\mcC^\alpha}(\bar{u}^\fy_\eps(t),\msV[\theta])\geq \delta} \mu_\eps(\dif \fy)
		\leq
		\mu_\eps(B_\rho^{\texttt{c}})
		\leq K 
		\exp\del[3]{-\frac{\theta}{\eps^2}}
	\end{equation}
	Concerning the second term, we apply Proposition~\ref{prop:bound_prob_H} to see that there exists an~$\bar{n}\coloneqq  \bar{n}(\rho,\,\delta,\, \theta ) \in \N$ and~$\eps_1(\gamma,\delta,\, \theta )\coloneqq  \eps_1 > 0$ such that for all $\eps\in (0,\eps_1)$ and $t>0$,
	\begin{equs}[eq:ldp_upper_step2]
		(\RN{2}) =&\int_{B_\rho}  \P\del[2]{\operatorname{dist}_{\mcC^\alpha}(\bar{u}^\fy_\eps(t),\msV[\theta])\geq \delta, \ \bar{u}^\fy_\eps \in \mcE_{\rho,\,\delta,\, \theta }(\bar{n})} \mu_\eps(\dif \fy)  \\
		\leq &\sup_{\fy \in \bar{B}_\rho} \P\del[2]{\bar{u}^\fy_\eps \in \CE_{\rho,\delta,\theta}(\bar{n})}\\
		\leq &
		\exp\del[3]{-\frac{\theta-\nicefrac{\gamma}{2}}{\eps^2}}.
	\end{equs}
	Finally, the third term is treated using Proposition~\ref{prop:ldp_upper_outside_H} to see that for $\bar{n}\coloneqq  \bar{n}(\rho,\,\delta,\, \theta ) \in \N$ as above, there exists a $t\coloneqq t(\bar{n})>0$ (see also Remark~\ref{rmk:choice_t_nbar}) and $\eps_2(t,\bar{n},\delta,\,\theta,\gamma)\coloneqq  \eps_2>0$ such that for all $\eps\in (0,\eps_2)$,
	\begin{equs}[eq:ldp_upper_step3]
		(\operatorname{III})
		= &\int_{B_\rho} \P\del[2]{\operatorname{dist}_{\mcC^{\alpha}}(\bar{u}^\fy_\eps(t),\msV[\theta])\geq \delta, \ \bar{u}^\fy_\eps \in \CE_{\rho,\delta,\theta}(\bar{n})^{\texttt{c}}}\mu_\eps(\dif \fy)\\
		\leq  & \ \sup_{\mfy \in \bar{B}_\rho}\P\del[2]{\operatorname{dist}_{\mcC^{\alpha}}(\bar{u}^\fy_\eps(t),\msV[\theta])\geq \delta, \ \bar{u}^\fy_\eps \in \CE_{\rho,\delta,\theta}(\bar{n})^{\texttt{c}}}\\
		\leq & \ \bar{n}	\exp\del[3]{-\frac{\theta-\nicefrac{\gamma}{2}}{\eps^2}}.
	\end{equs}
	Combining \eqref{eq:ldp_upper_step1}, \eqref{eq:ldp_upper_step2} and \eqref{eq:ldp_upper_step3} (along with the requisite choices of $\rho,\, t>0$) we see that for all $\eps \in (0,\eps_1\wedge \eps_2\wedge 1)$ it holds that
	\begin{equation*}
		\mu_\eps\del[1]{\fz \in \mcC^\alpha(\mbT^3): \ \operatorname{dist}_{\mcC^{\alpha}}(\fz,\msV[\theta])\geq \delta}
		\leq 
		\exp\del[3]{-\frac{\theta}{\eps^2}}
		+
		(1 + \bar{n})
		\exp\del[3]{-\frac{\theta - \nicefrac{\gamma}{2}}{\eps^2}}.
	\end{equation*} 
	This concludes the proof upon restricting $\eps \in (0,\eps_0)$ with $\eps_0<\eps_1\wedge \eps_2\wedge 1$ sufficiently small and making use of the identity $\msV = 2\msS$ proved by Theorem~\ref{th:V_equal_2S}.
\end{proof}

\appendix
\section{The Skeleton Equation}\label{sec:skeleton_equation}
We collect some results concerning well-posedness and controllability of the skeleton equation associated to \eqref{eq:phi43_formal_SPDE},
\begin{equation}\label{eq:nl_skeleton_equation}
	\begin{cases}
		\partial_t w^{h,\fz} - \Delta w^{h,\fz} = -(w^{h,\fz})^3-m^2 w^{h,\fz} + h,&\\
		w^{h,\fz}\tzero = \fz,
	\end{cases}
\end{equation}
where $h\in L^2_TL^2_x$ is an element of the Cameron--Martin space of the white noise. Many (but not all) results contained in the appendix can surely be found in disparate, alternative sources. However, we find it beneficial to collect them here together along with (at least) sketched proofs of these facts.
\subsection{Well-Posedness and A Priori Bounds}

We begin with the following local well-posedness result for \eqref{eq:nl_skeleton_equation}.
\begin{proposition}[Local Well-Posedness of the Skeleton Equation]\label{prop:nl_skeleton_LWP}
	Let $h\in L^2(\mbR_+\times \mbT^3)$, $\alpha \in (-\frac{2}{3},-\frac{1}{2})$ and $\fz \in \mcC^{\alpha}(\mbT^3)$. Then, there exists a $\bar{T}\coloneqq  \bar{T}(\alpha,\|\fz\|_{\mcC^\alpha},\|h\|_{L^2_T L^2_x})\in \mbR\cup\{+\infty\}$ such that a unique, mild solution, $w^{h,\fz}$, to \eqref{eq:nl_skeleton_equation} exists on $[0,\bar{T})$. In addition, 
	\begin{itemize}
		\item it either holds that
		\begin{equation}\label{eq:skeleton_eq_maximal}
			\bar{T}=\infty\quad \text{or}\quad \lim_{t\nearrow \bar{T}}\|w^{h,\fz}(t)\|_{\mcC^{\alpha}} =\infty,
		\end{equation}
		\item for $p\coloneqq  \frac{3}{1+\alpha}\in (6,9)$ there exists a weight, $\eta\coloneqq \eta (\alpha)\in (-\frac{5\alpha}{12},-\frac{\alpha}{3})$
		an exponent $\delta\coloneqq \delta(\eta,\alpha) >0$ and a finite constant $C_0\coloneqq  C_0(\eta,\alpha)>0$ such that 
		\begin{equation}\label{eq:nl_skeleton_Lp_growth}
			\sup_{t \in (0,\bar{T})} (t\wedge1)^{\eta} \|w^{h,\fz}(t)\|_{L^p_x} \leq C_0\,(T\wedge 1)^\delta (\|\fz\|_{\mcC^{\alpha}}+ \|h\|_{L^2_{\bar{T}}L^2_x}),
		\end{equation}
		\item for $\eta \in \left(-\frac{5\alpha}{12},-\frac{\alpha}{3}\right)$ as above, there exists a constant $C_1\coloneqq  C_1(\alpha,\eta)$ and $\delta\coloneqq \delta(\alpha,\eta)>0$ such that
		\begin{equation}\label{eq:nl_skeleton_H1_growth}
			\sup_{t \in (0,\bar{T})} (t\wedge 1)^{1+\eta}\|w^{h,\fz}(t)\|_{H^1_x} \leq C_1(T\wedge 1)^\delta(\|\fz\|_{\mcC^{\alpha}}+ \|h\|_{L^2_{\bar{T}}L^2_x}).
		\end{equation}
		\item if $ \fz \in H^1(\mbT^3)$ then in place of \eqref{eq:nl_skeleton_Lp_growth} and \eqref{eq:nl_skeleton_H1_growth} one has
		\begin{equation}\label{eq:nl_skeleton_H1_growth_reg_initial}
			\sup_{t\in [0,\bar{T})} \|w^{h,\fz}(t)\|_{H^1_x} \lesssim \left(\|\fz\|_{H^1_x} + \|h\|_{L^2_{\bar{T}}L^2_x}\right).
		\end{equation}
	\end{itemize}
\end{proposition}
\begin{proof}
	We only treat the case $\mfz \in \mcC^{\alpha}_x$ in detail, since the estimate \eqref{eq:nl_skeleton_H1_growth_reg_initial} is easily obtained by similar arguments. For $T>0$ and $p=\frac{3}{1+\alpha}$ 
	as above define the set
	\begin{equation*}
		\mfB_T \coloneqq  \left\{ w \in C_T\mcC^{\alpha}(\mbT^3)\,:\,  \|w\|_{C_{\eta;T}L^p_x} \leq 1 \right\},
	\end{equation*}
	along with the map $\Psi$ defined such that for all $w \in C_TL^p(\mbT^3)$ and $t\in [0,T]$,
	\begin{equation}\label{eq:nl_skel_sol_map_def}
		\Psi(w)(t) \coloneqq  e^{t (\Delta -m^2)} \fz - \int_0^{t} e^{(t -s)(\Delta -m^2)} w(s)^3 \dd s  + \int_0^{t} e^{(t -s)(\Delta -m^2)} h(s) \dd s. 
	\end{equation}
	We first show that there exists some $T_*\in (0,1)$ such that $\Psi$ maps $\mfB_{T_*} $ to itself. 
	Applying the Ornstein--Uhlenbeck semi-group estimate \eqref{eq:ou_reg} where appropriate along with the embedding \eqref{eq:besov_reg_embed}, for any $w\in \mfB_T$ and $t\in (0,T]$, we find,
	\begin{equs}
		\thinspace
		& \|\Psi(w)(t)\|_{L^p_x} \\
		\leq & \ \|e^{t(\Delta-m^2)}\fz \|_{L^p_x} + \int_0^t \|e^{(t-s)(\Delta-m^2)}w(s)^3 \|_{L^p_x}\dd s+\int_0^t \|e^{(t-s)(\Delta -m^2)} h(s)\|_{L^p_x}\dd s\\
		\lesssim_\alpha & \  t^{\frac{5\alpha}{12}} \|\fz \|_{\CB^{\nicefrac{5\alpha }{6}}_{p,1}} + \int_0^t (t-s)^{-\frac{3}{2}\left(\frac{3}{p}-\frac{1}{p}\right)}\|w(s)^3\|_{L^{\nicefrac{p}{3}}_x}\dd s +  \int_0^t (t-s)^{-\frac{3}{2}\left(\frac{1}{2}-\frac{1}{p}\right)}\|h(s)\|_{L^2}\dd s\\
		\leq & \ t^{\frac{5\alpha}{12}} \|\fz \|_{\CC^{\alpha}}+ \int_0^t (t-s)^{-\frac{3}{p}}\|w(s)\|^3_{L^{p}_x}\dd s+ \left(\int_0^t (t-s)^{-\frac{3p-6}{2p}}\dd s \right)^{\frac{1}{2}}\|h\|_{L^2_tL^2_x}\\
		\leq & \ t^{\frac{5\alpha}{12}}\|\fz \|_{\CC^{\alpha}} + \|w\|^3_{C_{\eta_0;t}L^p_x}\int_0^t (t-s)^{-\frac{3}{p}}s^{-3\eta_0}\dd s + \|h\|_{L^2_t L^2_x }t^{\frac{1}{2}- \frac{3p-6}{4p}}.
	\end{equs}
	Therefore, for any $T>0$, with $\eta_0$ as above, 
	\begin{equation}\label{eq:nl_Psi_apriori}
		\|\Psi(w)\|_{C_{\eta_0;T}L^p_x} \lesssim_\alpha T^{\eta_0 +\frac{5\alpha}{12}} \|\fz \|_{\CC^\alpha} + T^{1-\frac{3}{p}-2\eta_0}\|w\|_{C_{\eta_0;T}L^p_x} +T^{\frac{1}{2}- \frac{3p-6}{4p}+\eta_0}\|h\|_{L^2_T L^2_x} .
	\end{equation}
	Note that due to the restrictions placed on $\alpha,\, p$ and $\eta_0$ all the exponents in \eqref{eq:nl_Psi_apriori} are positive.
	Hence, noting that continuity of $t\mapsto \|\Psi(w)(t)\|_{L^p_x}$ can be easily shown, there exists a $T_* \in (0,1)$ such that $\Psi(w)\in \mfB_{T_*}$. In order to show that $\Psi$ is a contraction on $\mfB_{T_{**}}$, for $T_{**} \in (0,T_*]$, let $w\neq v \in \mfB_T$, so that for any $t\in (0,T]$, 
	\begin{align*}
		\|\Psi(w(t) - \Psi(v)(t)\|_{L^p} &\leq \int_0^t \|e^{(t-s)(\Delta+m^2)} (w(s)^3 - v(s)^3)\|_{L^p_x}\dd s \\
		&\lesssim_p  \int_0^t (t-s)^{-\frac{3}{2}\left(\frac{3}{p}-\frac{1}{p}\right)}\|w(s)^3-v(s)^3\|_{L^{p/3}_x}\dd s\\
		&\lesssim  \int_0^t (t-s)^{-\frac{3}{p}}\left(\|w(s)\|^{2}_{L^p_x} \vee \|v(s)\|^2_{L^p_x}\right) \|w(s)-v(s)\|_{L^{p}_x}\dd s\\
		&\leq  \left(\|w\|^2_{C_{\eta_0;T}L^p_x}\vee  \|v\|^2_{C_{\eta_0;T}L^p_x}\right)\|u-v\|_{C_{\eta_0;T}L^p_x}\int_0^t (t-s)^{-\frac{3}{p}}s^{-3\eta_0}\dd s\\
		&\leq \|w-v\|_{C_{\eta_0;T}L^p_x} T^{1-\frac{3}{p}-3\eta_0},
	\end{align*}
	where the last inequality follows from the assumptions on $p$ and $\eta_0$ as well as the fact that $w,\,v \in \mfB_T$. Therefore, choosing $T_{**} \in [0,T_*]$ sufficiently small, we see that $\Psi$ is a contraction mapping from $\mfB_{T_{**}}$ to itself; hence, by Banach's fixed point theorem, a unique fixed point, $w^{h,\fz}$, exists in $\mfB_{T_{**}}$. Uniqueness in all of $C_{\eta;T_{**}}L^p(\mbT^3)$ can be shown by similar steps, see for example the last steps in the proof of \cite[Thm.~3.4]{chevyrev_hambly_mayorcas_22_stoch}.
	
	To show that we may extend the solution beyond $T_{**}$ we use the semi-group property of the solution, i.e. for $t_0 \in [0,T_{**}]$ and any $t \in (0,\bar{T}-t_0)$, one has,
	\begin{equs}
		w^{h,\fz}(t_0+t) 
		& = e^{t(\Delta-m^2)} w^{h,\fz}(t_0) - \int_{t_0}^{t_0+t}e^{(t_0+t-s)(\Delta-m^2)} w^{h,\fz}(s)^3 \dd s \\
		& \quad + \int_{t_0}^{t_0+t} e^{(t_0+t-s)(\Delta-m^2)} h(s) \dd s,
	\end{equs}
	It is then easy to see that we may restart the equation, obtaining local existence and uniqueness of solutions in a similar manner as above, for as long as $\|w^{h,\fz}(t)\|_{\mcC{^\alpha}}<\infty$, proving \eqref{eq:skeleton_eq_maximal}
	
	In order to obtain the growth estimates \eqref{eq:nl_skeleton_Lp_growth} and \eqref{eq:nl_skeleton_H1_growth}, let $T \in [0,T_{**}]$ and we first return to \eqref{eq:nl_Psi_apriori}, now applied to the fixed point found above, on the interval $[0,T_{**})$. Using the fact that $T_{**}<1$, we find that there exists a $\delta\coloneqq \delta(\alpha,\eta_0)>0$ such that for any $T\in (0,T_{**}]$,
	\begin{equation*}
		(1-T^{1-\frac{3}{p}-2\eta_0})\|w^{h,\fz}\|_{C_{\eta_0;T}L^p_x} \lesssim T_{**}^\delta \left(\|\fz\|_{\CC^{\alpha}} +\|h\|_{L^2_{T_{**}} L^2_x}\right),
	\end{equation*}
	from which it follows that for any $T\in [0,T_{**}/2)$ there exists a $C\coloneqq  C(\eta_0,\alpha)$ such that,
	\begin{equation*}
		\|w^{h,\fz}\|_{C_{\eta_0;{T}}L^p_x} \lesssim  T^\delta_{**}\left(  \|\fz\|_{\CC^{\alpha}} +\|h\|_{L^2_{T} L^2_x}\right).
	\end{equation*}
	The extension to all $T\in (0,\bar{T})$ follows from the same re-starting procedure as above, repeating the same argument on intervals of length $1/2$ or less; this concludes the proof of \eqref{eq:nl_skeleton_Lp_growth}.
	
	To obtain \eqref{eq:nl_skeleton_H1_growth}, we first pass to a direct Fourier representation in the additive final term of  \eqref{eq:nl_skel_sol_map_def} and show that it enjoys $H^1(\mbT^3)$ regularity, continuously in time. For any $t\in (0,\bar{T})$,
	\begin{align}
		\left\| \int_0^{t} e^{(t -s)(\Delta -m^2)} h(s) \dd s  \right\|^2_{H^1_x} &= \sum_{k\in \mbZ^3}  \left|\,\left(1+|k|^2\right)^{\frac{1}{2}}\int_0^t e^{-(t-s)\left(|k|^2+m^2\right)}\langle h(s),e_k\rangle \,\dd s\,\right|^2 \notag\\
		&\leq  \sum_{k\in \mbZ^3}  \,\left(1+|k|^2\right)\int_0^t e^{-2(t-s)\left(|k|^2+m^2\right)}\dd s \int_0^t|\langle h(s),e_k\rangle |^2\,\dd s \notag \\
		&=\sum_{k\in \mbZ^3}  \,\frac{\left(1+|k|^2\right)}{2\left(m^2+|k|^2\right)} \left(1-e^{-2t\left(|k|^2+m^2\right)}\right) \int_0^t|\langle h(s),e_k\rangle |^2\,\dd s \notag\\
		&\lesssim_{\bar{T},m} \sum_{k\in \mbZ^3}\int_0^t|\langle h(s),e_k\rangle |^2\,\dd s  = \|h\|^2_{L^2_TL^2_x} \label{eq:SHE_H1_reg}.
	\end{align}
	Continuity in time of the left-hand side is easily shown.
	Then, again applying \eqref{eq:ou_reg}, along with the embedding $L^{p}(\mbT^3) \hookrightarrow L^6(\mbT^3)\hookrightarrow L^2(\mbT^3)$ and the estimate \eqref{eq:SHE_H1_reg}, for any $t_0 \in (0,\nicefrac{T_{**}}{2})$ and $t\in (0,t_0]$ it holds that
	\begin{align*}
		\|w^{h,\fz}(t_0+t)\|_{H_x^{1}} \leq & \, \|e^{t(\Delta-m^2)}u(t_0)\|_{H^1_x} + \int_{t_0}^{t_0+t} \|e^{(t+t_0-s)(\Delta-m^2)} w^{h,\fz}(s)^3 \|_{H^1_x}\dd s\\
		& +  \left\|\int_{t_0}^{t_0+t} e^{(t+t_0-s)(\Delta-m^2)} h(s)\dd s \right\|_{H^1_x} \\
		\lesssim & \,t^{-\frac{1}{2}}\|w^{h,\fz}(t_0)\|_{L^2_x} + \int_{t_0}^{t_0+t} (t+t_0 -s)^{-\frac{1}{2}}\|w^{h,\fz}(s)^3\|_{L^2_x} \dd s \\
		&+ \left\|\int_{t_0}^{t_0+t} e^{(t+t_0-s)(\Delta-m^2)} h(s)\dd s \right\|_{H^1_x}\\
		\lesssim_{\bar{T},m} & \, t^{-\frac{1}{2}}\|w^{h,\fz}(t_0)\|_{L^p_x} + t^{\frac{1}{2}}t_0^{-3\eta_0}\|w^{h,\fz}\|^3_{C_{\eta_0;T_{**}}L^p_x}+  \|h\|_{L^2_{\bar{T}}L^2_x}.
	\end{align*}
	Since $\|w^{h,\fz}\|_{C_{\eta_0;T_{**}}L^p_x}\leq 1$, multiplying both sides by $t_0^{\eta_0} t$ and taking suprema we find,
	\begin{align*}
		\sup_{t_0 \in (0,\nicefrac{T_{**}}{2})}	 \sup_{t \in (0,t_0]}  t_0^{\eta_0} t \|w^{h,\fz}(t_0+t)\|_{H^1_x} 
		\lesssim_{\bar{T},m} & \left(T_{**}^{\frac{1}{2}}+T_{**}^{\frac{3}{2}-2\eta_0} \right) \|w^{h,\fz}\|_{C_{\eta_0;T_{**}}}  +\|h\|_{L^2_{\bar{T}}L^2_x}  
	\end{align*}
	Therefore applying \eqref{eq:nl_skeleton_Lp_growth} and consolidating the relevant estimates, there exists a $\delta\coloneqq  \delta(\eta_0)>0$ such that
	\begin{align*}
		\sup_{t\in (0,\nicefrac{T_{**}}{2})}	  t^{1+\eta_0} \|w^{h,\fz}(t)\|_{H^1_x} \lesssim \, T_{**}^\delta \left(\|\fz\|_{\mcC^{\alpha}}+ \|h\|_{L^2_{\bar{T}}L^2_x}\right).
	\end{align*}
	Iterating the same arguments on successive time intervals as previously, we then obtain,
	\begin{equation}\label{eq:nl_skeleton_H2gamma_growth}
		\sup_{t\in (0,\bar{T})} (t\wedge 1)^{1+\eta_0}\|w^{h,\fz}(t)\|_{H^{1}_x} \lesssim (\bar{T}\wedge 1)^\delta\left(\|\fz\|_{\CC^\beta} +\|h\|_{L^2_{\bar{T}} L^2_x}\right)
	\end{equation}
	which concludes the proof of \eqref{eq:nl_skeleton_H1_growth}.
\end{proof}

The following proposition establishes an a priori estimate on the growth of certain $L^p(\mbT^3)$ norms of weak solutions to \eqref{eq:nl_skeleton_equation}; in combination with the maximal time of existence of \eqref{eq:skeleton_eq_maximal} and standard Besov embeddings, we will establish global well-posedness of \eqref{eq:nl_skeleton_equation}, this argument will be concluded in Theorem~\ref{th:nl_skeleton_gwp} below.
\begin{proposition} \label{prop:nl_skeleton_Lp_apriori}
	Let $T\in(0,\infty]$ and $w$ be a mild solution to \eqref{eq:nl_skeleton_equation} on $[0,T]$. Then, let $t_0\in [0,T)$ be such that $w(t_0) \in L^p(\mbT^3$ and $w \in C((t_0,T);H^1(\mbT^3))$ and for any $p> 1$ and $t\in (0,T-t_0]$ it holds that
	\begin{equation}\label{eq:nl_skeleton_Lp_identity}
		\begin{aligned}
			\frac{1}{p}\|w(t_0+t)\|^p_{L^p_x} = &\,\frac{1}{p}\|w(t_0)\|^{p}_{L^p_x} -\frac{2(p-1)}{p}\int_{t_0}^{t_0+t} \|\nabla (w(s)^{p/2})\|_{L^2_x}\dd s \\
			&-\int_{t_0}^{t_0+t} \langle w(s)^3+m^2 w(s),w(s)^{p-1}\rangle\dd s + \int_{t_0}^{t_0+t} \langle h(s),w(s)^{p-1}\rangle\dd s.
		\end{aligned}
	\end{equation}
	Then, for $p\geq 1$ there exists a constant $C\coloneqq  C(p,m) >0$ and an exponent $k(p)\in (1,2]$ such that for all $t\in (0,T-t_0]$, 
	\begin{equation}\label{eq:nl_skeleton_Lp_global}
		\|w(t_0+t)\|^p_{L^p_x} \lesssim C \,e^{-\frac{m^2}{2} t}\|w(t_0)\|^p_{L^p_x}  + \frac{1}{m^{2-k(p)}}\left(1-e^{-\frac{m^2}{2-k(p)} t}\right)^{\frac{2-k(p)}{2}}\|h\|_{L^2_{[t_0,t_0+t]}L^2_x}^{k(p)}.
	\end{equation}
\end{proposition}
\begin{proof}
	Choosing $p\geq 2$ to be even in \eqref{eq:nl_skeleton_Lp_identity} we first find the identity,
	\begin{equation}\label{eq:p_even_skel_energy}
		\begin{aligned}
			\frac{1}{p}\|w(t_0+t)\|^p_{L^p_x} = & \frac{1}{p}\|w(t_0)\|^{p}_{L^p_x} -\frac{2(p-1)}{p}\int_{t_0}^{t_0+t} \|\nabla (w(s)^{p/2})\|_{L^2_x}\dd s -\int_{t_0}^{t_0+t} \| w(s)^{p+2}\|_{L^1_x}\dd s\\
			&-m^2 \int_{t_0}^{t_0+t} \|w(s)^p\|_{L^1_x}\dd s + \int_{t_0}^{t_0+t} \langle h(s),w(s)^{p-1}\rangle\dd s.
		\end{aligned}
	\end{equation}
	The only term to be estimated is the last one which has no given sign. Applying H\"older's inequality directly gives,
	\begin{equation}\label{eq:holder_applied}
		|\langle h(s),w(s)^{p-1}\rangle|\leq \|h(s)\|_{L^2_x}\|w(s)^{p-1}\|_{L^2_x}.
	\end{equation}
	We recall the Sobolev inequality for non-mean free functions $f:\mbT^3\to \mbR$, see for example \cite[Cor.~2.2]{benyi_oh_13_sobolev_torus},
	\begin{equation}\label{eq:sobolev_ineq_torus}
		\frac{k}{3}\geq \frac{1}{r}-\frac{1}{q}\quad \text{then}\quad \|f\|_{L^q_x}\lesssim \|f\|_{W^{k,r}_x}.  
	\end{equation}
	Choosing $k=1$ and $r=2$, we may apply \eqref{eq:sobolev_ineq_torus} to $w(s)^{p/2}$ with $q=\frac{4p-4}{p}$ and using the equivalence of finite dimensional norms we obtain
	\begin{align}
		\|w(s)^{p-1}\|_{L^2_x} = \|w(s)^{p/2}\|^{\frac{p}{8p-8}}_{L^{(4p-4)/p}_x} \lesssim \|w(s)^{p/2}\|^{\frac{p}{8p-8}}_{H^1_x} \lesssim &\, \|w(s)^{p/2}\|^{\frac{p}{8p-8}}_{L^2_x} + \|\nabla (w(s)^{p/2})\|^{\frac{p}{8p-8}}_{L^2_x}\notag \\
		=& \|w(s)^p\|^{\frac{p}{16p-16}}_{L^1_x} + \|\nabla (w(s)^{p/2})\|^{\frac{p}{8p-8}}_{L^2_x}. \label{eq:sobolev_applied}
	\end{align}
	Combining \eqref{eq:holder_applied} with \eqref{eq:sobolev_applied} and the Cauchy--Schwarz inequality gives,
	\begin{equation*}
		|\langle h(s),w(s)^{p-1}\rangle| \lesssim \|h(s)\|_{L^2_x}\left(\|w(s)^p\|^{\frac{p}{16p-16}}_{L^1_x} + \|\nabla (w(s)^{p/2})\|^{\frac{p}{8p-8}}_{L^2_x}\right).
	\end{equation*}
	Applying Young's product inequality (in the form $ab \leq \frac{1}{\theta_1\eps}a^{\theta_1}+\frac{\eps}{\theta_2}b^{\theta_2}$) and noting that for $p\geq 2$ one has
	\begin{equation*}
		\frac{p}{16p-16} \vee \frac{p}{4p-4}\leq 2,
	\end{equation*}
	we see that there exist constants $k\coloneqq k(p) \in (1,2]$ and $C\coloneqq C(p,m)>0$ such that, 
	\begin{equation*}
		|\langle h(s),w(s)^{p-1}\rangle|  \lesssim \frac{C}{p} \|h(s)\|^{k(p)}_{L^2_x} + \frac{(p-1)m^2}{p}\|w(s)^p\|_{L^1_x} + \frac{p-1}{p}\|\nabla(w(s)^{p/2})\|_{L^2_x}.
	\end{equation*}
	Introducing this bound into \eqref{eq:p_even_skel_energy} shows that, 
	\begin{equation}\label{eq:skel_absorbed_energy}
		\begin{aligned}
			\|w(t_0+t)\|^p_{L^p_x} \leq &\,\|w(t_0)\|^p_{L^p_x} -(p-1)\int_{t_0}^{t_0+t} \|\nabla (w(s)^{p/2})\|_{L^2_x}\dd s -p\int_{t_0}^{t_0+t} \| w(s)^{p+2}\|_{L^{1}_x}\dd s\\
			&-m^2 \int_{t_0}^{t_0+t} \|w(s)^p\|_{L^1_x}\dd s +\int_{t_0}^{t_0+t} \|h(s)\|^{k(p)}_{L^2_x}\dd s.
		\end{aligned}
	\end{equation}
	Simply bounding the second and third negative terms on the right hand side by zero, we find a comparison with the ODE, written in suggestive notation,
	\begin{equation*}
		\frac{\dd}{\dd t} x(t_0+t) = -\frac{m^2}{2} x(t)  +  h(t)^{k}, \quad x(t_0+t)\big|_{t=0} =x(t_0),
	\end{equation*}
	which has the explicit solution, 
	\begin{align*}
		x(t_0+t) =&\, e^{-\frac{m^2}{2} t}x(t_0) + \int_{t_0}^{t_0+t}  e^{-\frac{m^2}{2}(t_0+t-s)}h(s)^{k} \dd s.
	\end{align*}
	Thus, by a suitable comparison, we obtain the chain of estimates,
	\begin{equs}
		\thinspace
		& \|w(t_0+t)\|^p_{L^p_x} \\
		\leq & \, e^{-\frac{m^2}{2} t}\|w(t_0)\|^p_{L^p_x}  +\int_{t_0}^{t_0+t} e^{-\frac{m^2}{2} (t_0+t-s)}\|h(s)\|^{k(p)}_{L^2_x}\dd s\\
		\leq & \, e^{-\frac{m^2}{2} t}\|w(t_0)\|^p_{L^p_x}  +\left(\int_{t_0}^{t_0+t} e^{-\frac{m^2}{2-k(p)} (t_0+t-s)}\dd s\right)^{\frac{2-k(p)}{2}}\left(\int_{t_0}^{t_0+t} \|h(s)\|^{2}_{L^2_x}\dd s\right)^{\frac{k(p)}{2}}\\
		=&\, e^{-\frac{m^2}{2} t}\|w(t_0)\|^p_{L^p_x}  + \left(\frac{2-k(p)}{m^2}\left(1-e^{-\frac{m^2}{2-k(p)} t}\right)\right)^{\frac{2-k(p)}{2}}\|h\|_{L^2_{[t_0,t_0+t]}L^2_x}^{k(p)}\\
		\lesssim &\, e^{-\frac{m^2}{2} t}\|w(t_0)\|^p_{L^p_x}  + \frac{1}{m^{2-k(p)}}\left(1-e^{-\frac{m^2}{2-k(p)} t}\right)^{\frac{2-k(p)}{2}}\|h\|_{L^2_{[t_0,t_0+t]}L^2_x}^{k(p)}
	\end{equs}
	we obtain the claimed estimate \eqref{eq:nl_skeleton_Lp_global} for all even integers $p\geq 2$. The estimate for any $p\geq 2$ follows from the ordering of the $L^p(\mbT^3)$ spaces. 
\end{proof}
\begin{remark}\label{rem:nl_skeleton_CDFI}
	Instead of ignoring the effect of the negative non-linear term in the above proof, one could instead apply Jensen's inequality to obtain,
	\begin{equation*}
		\|w\|^p_{L^p_x} \lesssim 	\|w^{p+2}\|^{\frac{p}{p+2}}_{L^1_x}
	\end{equation*}
	which combined with \eqref{eq:skel_absorbed_energy} and formally differentiating in time, gives,
	\begin{equation*}
		\frac{\dd}{\dd t}\|w(t)\|^p_{L^p_x} + c_1 \left(\|w(t)\|^p_{L^p_x}\right)^{\frac{p+2}{p}}\lesssim   \|h(t)\|^{k(p)}_{L^2_x},\quad \text{for all}\,\, t\in (t_0,T),
	\end{equation*}
	for the same $k(p)\in (1,2]$ as above. Then, applying \cite[Lem.~3.8]{tsatoulis_weber_18_spectral} (see also \cite[Lem.~7.3]{mourrat_weber_infinity}) one finds a stronger bound, which is independent of the initial data, of the following form
	\begin{equation*}
		\sup_{t\in (t_0,T)}\|w(t)\|_{L^p} \lesssim_p \max\left\{\frac{1}{\sqrt{t_0}} ,\|h\|^{\frac{k(p)}{p+2}}_{L^2_{[t_0,T]}L^2_x} \right\}.
	\end{equation*}
	It follows in particular, that for any $p \geq 2$ and $\alpha \in (-\frac{2}{3},-\frac{1}{2})$ as in Proposition~\ref{prop:nl_skeleton_LWP} one has
	\begin{equation}\label{eq:nl_skeleton_Lp_CDFI}
		\sup_{\fz \in \mcC^{\alpha}(\mbT^3)} \sup_{t\in (t_0,T)}	\|w(t)\|_{L^p} \lesssim_p \max\left\{\frac{1}{\sqrt{t_0}} ,\|h\|^{\frac{k(p)}{p+2}}_{L^2_{[t_0,T]}L^2_x} \right\}.
	\end{equation}
	Since \eqref{eq:nl_skeleton_Lp_CDFI} is independent of the initial data it is strictly stronger than the estimate derived in Proposition~\ref{prop:nl_skeleton_Lp_apriori}. However, the proof of \eqref{eq:nl_skeleton_Lp_CDFI} relies on the strongly damping cubic term in the $\Phi^4$ equation and since many other quantum field theories to which this method may conceivably be applied (e.g. sine-Gordon and Yang--Mills theories) do not contain an analogous term we only make use of the weaker estimate \eqref{eq:nl_skeleton_Lp_global}. This estimate is both sufficient for our purposes and likely to be replicable for other suitable theories. One advantage to note of the strategy to prove \eqref{eq:nl_skeleton_Lp_CDFI} is that, unlike the proof of Proposition~\ref{prop:nl_skeleton_Lp_apriori}, it does not require a positive mass $m^2>0$ in the equation.
\end{remark}
\begin{corollary}\label{cor:nl_skeleton_Lp_decay}
	Let $\mfz\in \mcC^\alpha(\mbT^3)$, $h\in L^2(\mbR_+\times \mbT^3)$ and $w^{h,\mfz}:[0,\infty) \times \mbT^3 \to \mbR$ be a weak solution to \eqref{eq:nl_skeleton_equation} such that \eqref{eq:nl_skeleton_Lp_identity} holds for some $p >1$. Then, for the same $p\geq 1$ one has
	\begin{equation*}
		\lim_{t \rightarrow \infty} \|w^{h,\mfz}(t)\|_{L^p_x}=0.
	\end{equation*}
\end{corollary}
\begin{proof}
	The claim follows directly from the a priori estimate \eqref{eq:nl_skeleton_Lp_growth}. Since $w^{h,\mfz}$ is a global weak solution, we may take the joint limit in \eqref{eq:nl_skeleton_Lp_identity} to give,
	\begin{equation*}
		\lim_{t \to \infty} \|w^{h,\mfz}(t)\|^p_{L^p_x} = \lim_{t_0 \to \infty} \lim_{t\to \infty} \|w^{h,\mfz}(t)\|^p_{L^p_x} = \frac{1}{m^{2-k(p)}} \,\lim_{t_0 \to \infty}  \|h\|^{k(p)}_{L^2_{[t_0,\infty]}L^2_x} =0.
	\end{equation*}
\end{proof}
The following a priori lemma demonstrates that decay to zero at $t=+\infty$ of a sufficiently high $L^p(\mbT^3)$ norm of the solution implies decay in the $H^1(\mbT^3)$ norm. Note that the proof also makes use of the assumed positive mass $m^2>0$.
\begin{lemma}\label{lem:nl_skeleton_upgrade_decay} \label{lem_decay_nonlin_eq_H1}
	Let~$\mfz \in \mcC^{\alpha}(\mbT^3)$, $h \in L^2(\mbR_+\times \mbT^3)$, $\alpha\in \left(-\frac{2}{3},-\frac{1}{2}\right)$,~$q\coloneqq  -\frac{1}{\alpha} \in \left(\frac{3}{2},2\right)$,  and $w^{h,\mfz}:\mbR_+\times \mbT^3 \to \mbR$ be a mild solution to \eqref{eq:nl_skeleton_equation} on $\mbR_+\times \mbT^3$ such that $\lim_{t\to \infty}\|w^{h,\mfz}(t)\|_{L^{3q}_x} = 0$. 
	Then it also holds that,
	\begin{equation}\label{eq:nl_skeleton_H1_decay}
		\lim_{t \to \infty} \|w^{h,\mfz}(t)\|_{H^1_x} =0.
	\end{equation}
\end{lemma}

\begin{proof}
	Recall that for any $t_0\geq 0$ and $t>0$, by the semi-group property of the solution, we have
	\begin{equs}
		w^{h,\mfz}(t_0+t)
		& = e^{t(\Delta-m^2)} w^{h,\mfz}(t_0) \\ 
		& \quad  -\int_{t_0}^{t_0+t} e^{(t_0+t-s)(\Delta-m^2)}w^{h,\mfz}(s)^3\dd s + \int_{t_0}^{t_0+t} e^{(t_0+t-s)(\Delta-m^2)} h(s) \dd s.
	\end{equs}
	Hence, applying \eqref{eq:ou_reg}
	along with \eqref{eq:SHE_H1_reg}, we obtain the chain of  estimates
	\begin{equs}
		\thinspace &
		\|w^{h,\mfz}(t_0+t)\|_{H^1_x} \\
		\leq & \, e^{- tm^2} \|e^{t\Delta}w^{h,\mfz}(t_0)\|_{H^1_x} 
		+ \int_{t_0}^{t_0+t} \|e^{(t_0+t-s)(\Delta-m^2)} w^{h,\mfz}(s)^3\|_{H^1_x} \dd s \\ 
		& + \, \left\|\int_{t_0}^{t_0+t} e^{(t_0+t-s)(\Delta-m^2)} h(s)\, \dd s\right\|_{H^1_x}\\
		\lesssim  \, & e^{-tm^2} (t\wedge 1)^{-\frac{3-q}{4q}}\|w^{h,\mfz}(t_0)\|_{L^{3q}_x}  \\
		& +e^{-(t_0+t)m^2} \int_{t_0}^{t_0+t} e^{s m^2} ((t_0+t-s)\wedge 1)^{-\frac{3-q}{4q}} \|w^{h,\mfz}(s)\|_{L^{3q}_x} \dd s+ \|h\|_{L^2_{[t_0,t_0+t]}L^2_x}\\
		\lesssim   \,& e^{-tm^2} (t\wedge 1)^{-\frac{3-q}{4q}}\|w^{h,\mfz}(t_0)\|_{L^{3q}_x}  + e^{-tm^2} t^{1-\frac{3-q}{4q}}\sup_{s \in [t_0,t_0+t]}\|w^{h,\mfz}(s)\|_{L^{3q}_x}  + \|h\|_{L^2_{[t_0,+\infty)}L^2_x}.
	\end{equs}
	Therefore, since for any fixed $t>0$ one has
	\begin{equation*}
		\lim_{s\to +\infty} \|w^{h,\mfz}(s)\|_{H^1_x} = \lim_{t_0 \to +\infty}\|w^{h,\mfz}(t_0+t)\|_{H^1_x},
	\end{equation*}
	using the assumption that $\lim_{t\to\infty} \|w^{h,\mfz}(t)\|_{L^{3q}_x}=0$, it holds that
	\begin{equation*}
		\lim_{s\to \infty} \|w^{h,\mfz}(s)\|_{H^1_x}  \lesssim \lim_{t_0 \to +\infty}  \|h\|_{L^2_{[t_0,+\infty)}L^2_x} =0.
	\end{equation*}
\end{proof}
\begin{theorem}[Global Well-Posedness of the Skeleton Equation]\label{th:nl_skeleton_gwp}
	Let $\alpha \in (-\frac{2}{3},-\frac{1}{2})$, $\fz \in \mcC^{\alpha}(\mbT^3)$ and $h\in L^2(\mbR_+\times \mbT^3)$. Then, there exists a unique solution $w^{h,\mfz}:(0,\infty]\to H^1(\mbT^3)$ to \eqref{eq:nl_skeleton_equation}. Furthermore,
	\begin{enumerate}[label=(\roman*)]
		\item \label{it:nl_skeleton_regularity} for any $T>0$ and $t_0\in (0,T]$ (where if $\mfz \in H^1(\mbT^3)$ we may allow $t_0=0$) 
		\begin{equation}\label{eq:nl_skeleton_global_bound}
			\|w^{h,\mfz}\|_{H^1_{[t_0,T]}H^{-1}_x} + \|w^{h,\mfz}\|_{C_{[t_0,T]}H^1_x}
			<\infty.
		\end{equation}
		
		\item  \label{it:nl_skeleton_global_decay} for any $p\geq 2$, it holds that 
		\begin{equation}\label{eq:nl_skeleton_global_decay}
			\lim_{t\to \infty} \|w^{h,\mfz}(t)\|_{\mcC^{\alpha}} = \lim_{t\to \infty}\|w^{h,\mfz}(t)\|_{L^p_x} = \lim_{t\to \infty}\|w^{h,\mfz}(t)\|_{H^1_x} =0,
		\end{equation}
		\item
		\label{it:nl_skeleton_quant_bounds}	for any $p\geq 2$, there exists an $\eta \in \left(-\frac{5\alpha}{12},-\frac{\alpha}{3}\right)$ such that for all $\mfz \in \mcC^{\alpha}(\mbT^3)$, $h\in L^2(\mbR_+\times \mbT^3)$ and $T\in (0,\infty)$,
		\begin{equation}\label{eq:nl_skeleton_global_increment_bound_2}
			\sup_{t \in (0,T)} (t\wedge1)^{\eta} \|w^{h,\fz}(t)\|_{L^p_x} \lesssim_{m,p,\alpha} \|\fz\|_{\mcC^{\alpha}}+ \|h\|_{L^2_{T}L^2_x}.
		\end{equation}
		while, if $\mfz = \in H^1(\mbT^3)$, we have
		\begin{equation}\label{eq:nl_skeleton_global_increment_bound_0_initial_2}
			\sup_{t\in [0,T)}\|w^{h,\fz}(t)\|_{L^p_x} \lesssim_{m,p} \|\mfz\|_{H^1_x} + \|h\|_{L^2_{T}  L^2_x}^2.
		\end{equation}
	\end{enumerate}
\end{theorem}
\begin{proof}
	Firstly, applying Lemma~\ref{prop:nl_skeleton_LWP} we see that there exists a unique, mild solution $w^{h,\mfz} \in C((0,\bar{T}];H^1(\mbT^3))$ for some $\bar{T} \in (0,1)$. Furthermore, this solution can be extended so long as $\|w^{h,\mfz}(t)\|_{\mcC^{\alpha}}<\infty$. Since one has $w^{h,\mfz}(t_0)\in H^1(\mbT^3)$ for all $t_0 \in (0,\bar{T}]$, one may argue using uniqueness of the mild solution in this regularity class and construction of weak solutions via the variational method (see for example \cite{prevot_rockner_07_concise,pardoux_21_SPDE}) to see that $t\mapsto w^{h,\mfz}(t)$ satisfies \eqref{eq:nl_skeleton_Lp_identity} on any interval $[t_0,\bar{T}]$ with $t_0 \in (0,\bar{T})$. As a result, we may apply Proposition~\ref{prop:nl_skeleton_Lp_apriori} for $p\geq 2$ sufficiently large we see that 
	\begin{equation*}
		\sup_{t\in [t_0,\bar{T}]} \|w^{h,\mfz}(t)\|_{\mcC^{\alpha}} <\infty.
	\end{equation*}
	So repeating the same arguments, we may extend the solution indefinitely. Uniqueness follows from similar arguments as in the proof of Proposition~\ref{prop:nl_skeleton_LWP} and the estimate \eqref{eq:nl_skeleton_global_bound}, for $\mfz \in \mcC^{\alpha}(\mbT^3)\setminus H^1(\mbT^3)$ follows directly. To see that we may take $t_0 =0$ in the case $\mfz \in H^1(\mbT^3)$ we make use of \eqref{eq:nl_skeleton_H1_growth_reg_initial}.
	
	To show the decay of solutions as in \eqref{eq:nl_skeleton_global_decay} it suffices to combine Corollary~\ref{cor:nl_skeleton_Lp_decay}, Lemma~\ref{lem:nl_skeleton_upgrade_decay} along with the embedding $L^{p}(\mbT^3)\hookrightarrow\mcC^{\alpha}(\mbT^3)$ for $p>1$ sufficiently large.

	The estimates \eqref{eq:nl_skeleton_global_increment_bound_2} and \eqref{eq:nl_skeleton_global_increment_bound_0_initial_2} follow from their counterparts given by Proposition~\ref{prop:nl_skeleton_LWP} but now extended to any $T\in (0,\infty)$ using the fact that $\sup_{t\geq 0} \|w^{h,\fz}(t)\|_{\mcC^{\alpha}}<\infty$, as shown in the proceeding paragraph.
\end{proof}
\subsection{Global Controllability}\label{sec:skeleton_controllability}
In this section, we study \emph{controllability} of the solution~$w^{h,0}$ to \eqref{eq:nl_skeleton_equation} with~$0$ initial data.
We first consider the linearized, skeleton equation, we consider the solution to
\begin{equation}\label{eq:lin_skeleton_equation}
	\begin{cases}
		\partial_t z^h - \Delta z^h = - m^2 z^h + h,&\\
		z^h\tzero = 0.
	\end{cases}
\end{equation}
for $h\in L^2_T L^2(\mbT^3)$.
\begin{lemma}\label{lem:lin_skeleton_reg}
	Let $T>0$, $m^2>0$, $h\in L^2_TL^2$ and $z^h$ solve \eqref{eq:lin_skeleton_equation}. Then there exists a constant $C\coloneqq  C(T,m)>0$ such that
	\begin{equation}\label{eq:lin_skeleton_growth}
		\|z^h\|_{C_TH^1_x} \, \vee \,\,	\|z^h\|_{L^2_TH^2_x} \leq C \|h\|_{L^2_TL^2_x}.
	\end{equation}
\end{lemma}
\begin{proof}
	The proof is straightforward and standard so we only present a sketch. Firstly, for any $t\in [0,T]$ and $\alpha\in [0,1]$, using the Fourier representation of the heat semi-group, one directly has
	\begin{align*}
		\|z^h(t)\|^2_{H^\alpha_x} &= \sum_{k\in \mbZ^3}  \left|\,\left(1+|k|^2\right)^{\frac{\alpha}{2}}\int_0^t e^{-(t-s)\left(|k|^2+m^2\right)}\langle h(s),e_k\rangle \,\dd s\,\right|^2\\
		&\leq  \sum_{k\in \mbZ^3}  \,\left(1+|k|^2\right)^{\alpha}\int_0^t e^{-2(t-s)\left(|k|^2+m^2\right)}\dd s \int_0^t|\langle h(s),e_k\rangle |^2\,\dd s\\
		&=\sum_{k\in \mbZ^3}  \,\frac{\left(1+|k|^2\right)^{\alpha}}{2\left(m^2+|k|^2\right)} \left(1-e^{-2t\left(|k|^2+m^2\right)}\right) \int_0^t|\langle h(s),e_k\rangle |^2\,\dd s\\
		&\lesssim{T,m} \|h\|^2_{L^2_TL^2_x}
	\end{align*}
	Hence, taking $\alpha =1$ gives $z^h\in L^\infty_T H^1_x$. However, it is readily checked that $t\mapsto \|z_t\|_{H^1_x}$ is continuous which proves the first estimate of \eqref{eq:lin_skeleton_growth}. To prove the second, take $\alpha=2$, apply Young's convolution inequality and arguing a posteriori that one may exchange summation and integration, one finds
	\begin{align*}
		\|z^h\|^2_{L^2_T H^2_x}=	\int_0^T \|z^h(t)\|^2_{H^2_x} \dd t &= \int_0^T  \sum_{k\in \mbN^3}  \bigg|\,\left(1+|k|^{2}\right)\int_0^t e^{-(t-s)(|k|^2+m^2)}\langle h(s),e_k\rangle \,\dd s\,\bigg|^2 \dd t\\
		&=  \sum_{k\in \mbN^3} \left\|\,\left(1+|k|^2\right) e^{-(\,\cdot\,)(|k|^2+m^2)}\ast \langle h(\,\cdot\,),e_k\rangle\right\|^2_{L^2_T}\\
		&\leq  \sum_{k\in \mbN^3} \left|\int_0^T \left(1+|k|^{2}\right) e^{-(t-s)(|k|^2+m^2)} \dd s\,\right|^2 \left\| \langle h(\,\cdot\,),e_k\rangle\right\|^2_{L^2_T}\\
		&= \sum_{k\in \mbN^3} \left(1-e^{-t(|k|^2+m^2)} \right)^2 \,\left\| \langle h(\,\cdot\,),e_k\rangle\right\|^2_{L^2_T}\\
		&\lesssim_{T,m} \|h\|^2_{L^2_TL^2_x}.
	\end{align*}
\end{proof}
To show controllability of the non-linear equation, we require some finer properties of the linear solution map.
\begin{lemma} \label{lem:ctrl_lineq}
	For $T > 0$ and $m^2>0$, defining the linear solution map~$\msL_T: L^2_TL^2(\mbT^3) \to H^1(\T^3)$, given by~$\msL_T h \coloneqq  z^h(T)$, defines a linear, bounded, compact and surjective operator for which there exists a $C\coloneqq C(T,m)>0$ such that
	\begin{equation}\label{eq:lin_sol_lwr_bnd}
		\inf_{\|\Psi\|_{H^{-1}_x\setminus\{0\}}} \frac{\|\msL_T^* \Psi\|_{L^2_TL^2_x}}{\|\Psi\|_{H^{-1}_x}} \geq C>0.
	\end{equation}
\end{lemma}

\begin{proof}
	Linearity and boundedness from $L^2_TL^2(\mbT^3) \to H^1(\T^3)$ follows directly from Lemma~\ref{lem:lin_skeleton_reg}. To demonstrate compactness, it suffices to argue that since $z^h \in L^2_T H^2(\mbT3)$, up to possibly redefining the image on a set of measure zero, we in fact have $\msL_T h \in H^2(\mbT^3) \cembed H^1(\mbT^3)$.
	Thus it only remains to check \eqref{eq:lin_sol_lwr_bnd}, which by the \emph{closed range theorem} (c.f.~\cite[Thm.~$4.13$]{rudin_func_ana}) implies surjectivity of the map.
	Since the operator~$(\Delta - m^2)$ is self-adjoint we see that
	\begin{equation*}
		(\msL_T^* \Psi)(s) = e^{(T-s)(\Delta - m^2)^*} \Psi = e^{(T-s)(\Delta - m^2)} \Psi.
	\end{equation*}
	Recalling that we write~$(e_k)_{k \in \Z^3}$ for the Fourier basis of~$L^2(\T^3)$ we set~$\Psi_k \coloneqq  \scal{\Psi,e_k}_{H^{-1}_x \x H^1_x}$ so that, for some $\tilde{C}\coloneqq \tilde{C}(T)>0$, we find
	\begin{equs}
		\norm[0]{\msL_T^* \Psi}_{L^2_TL^2_x}^2
		& =
		\int_0^T \norm[0]{e^{(T-s)(\Delta-m^2)} \Psi}_{L^2_x}^2 \dif s
		=
		\sum_{k \in \Z^3} \Psi_k^2 \int_0^T e^{-2(T-s) \del[0]{\abs[0]{k}^2 + m^2}} \dif s \\
		& =
		\frac{1}{2}  \sum_{k \in \Z^3} \sbr[1]{1 - e^{-2T\del[0]{\abs[0]{k}^2 + m^2}}} \frac{\Psi_k^2}{\del[0]{\abs[0]{k}^2 + m^2}} 
		\geq 
		\frac{1-e^{-2Tm^2}}{2} \sum_{k \in \Z^3} \frac{\Psi_k^2}{\del[0]{\abs[0]{k}^2 + m^2}} \\
		& \geq 
		\frac{1-e^{-2Tm^2}}{2m^2}  \norm[0]{\Psi}_{H_x^{-1}}^2.
	\end{equs}
	So that \eqref{eq:lin_sol_lwr_bnd} holds with
	\begin{equation*}
		C \coloneqq  \frac{\sqrt{1-e^{-2Tm^2}}}{\sqrt{2}m}.
	\end{equation*}
\end{proof}

Observe that while the operator~$\msL_T$ is surjective onto $H^1(\mbT^3)$ it is clearly not injective.  However, we may define its \emph{pseudo}-inverse~$\msL^+_T$, as described by the following lemma.
\begin{lemma} \label{lem:ctrl_lineq_pseudoinv}
	Let~$T > 0$ and $m^2>0$. Then, there exists a bounded, linear map
	\begin{equation*}
		\msL^+ : H^1(\mbT^3)\setminus\{0\} \to L^2_T L^2(\mbT^3),
	\end{equation*}
	such that
	\begin{itemize}
		\item for all $\phi \in H^1(\mbT^3)\setminus \{0\}$
		\begin{equation*}
			\msL_T \msL_T^+ \phi = \phi,
		\end{equation*}
		\item given any $\phi \in H^1(\mbT^3)\setminus \{0\}$, for all $g \in L^2_TL^2(\mbT^3)$ such that $\msL_T g = \phi$ it holds that
		\begin{equation*}
			\langle \msL^+_T \phi - g,\msL^+_T \phi \rangle_{L^2_TL^2_x} =0.
		\end{equation*}
	\end{itemize}
	Furthermore, recalling the constant $C\coloneqq C(T,m)>0$ from \eqref{eq:lin_sol_lwr_bnd}, one has
	\begin{equation}\label{eq:lin_pseud_inv_bnd}
		\sup_{\phi \in H^1_x \setminus\{0\}} \frac{\|\msL^+_T \phi\|_{L^2_TL^2_x}}{\|\phi\|_{H^1_x}} \leq \frac{1}{C}
	\end{equation}
\end{lemma}

\begin{proof}
	First note that since in general, one may also define the pseudo-inverse of an operator as selecting the pre-image of any element of smallest norm, the existence of a pseudo-inverse fails if and only if the operator itself does not have closed range. In our case, since we have shown $\msL_T :L^2_T L^2(\mbT^3) \to H^1(\mbT^3)$ to be surjective, its range is trivially closed. To obtain the estimate, \eqref{eq:lin_pseud_inv_bnd}, one uses the fact that $\msL_T$ is a compact operator to approximate it by finite rank operators and the lower bound \eqref{eq:lin_sol_lwr_bnd} to see that this approximations can be chosen to have a uniformly positive, smallest eigenvalue. Using this approximation, one also derives the identity,
	\begin{equation}\label{eq:MP_inverse_rep}
		\msL^+_T = \lim_{\iota \to 0} \left(\msL_T\msL_T^* + \iota \mbI_{H^1_x}\right)^{-1} \msL_T^*,
	\end{equation}
	and sees that \eqref{eq:lin_pseud_inv_bnd} holds uniformly along the approximating sequence. Since the sequence converges in the strong topology, we obtain \eqref{eq:lin_pseud_inv_bnd} in the limit.
\end{proof}
\paragraph*{The non-linear equation.} \label{sec_controllability_nonlin_eq}
We employ the results obtained in Lemmas~\ref{lem:ctrl_lineq} \& \ref{lem:ctrl_lineq_pseudoinv} to obtain local, exact, controllability of the non-linear, skeleton dynamics, \eqref{eq:nl_skeleton_equation}. We begin with a straightforward preliminary lemma.
\begin{lemma} \label{lem_estimate_nonlineq_hitpoint}
	Let~$T > 0$, ~$\phi \in H^1(\mbT^3)$ and $\msL^+_T\phi$ be the pseudo-inverse of $\msL_T$ as defined in Lemma~\ref{lem:ctrl_lineq_pseudoinv}.  Then, for $w^{\msL^+_T\phi,0}$ solving,
	\begin{equation}\label{eq:nl_skeleton_equation_from_0}
		\begin{cases}
			\partial_t w^{\msL^+_T \phi,0} - \Delta w^{\msL_T^+,0} = -(w^{\msL_T^+,0})^3-m^2 w^{\msL_T^+,0} + \msL_T^+ \phi,&\\
			w^{\msL_T^+,0}\tzero = 0,
		\end{cases}
	\end{equation}
	it holds that
	\begin{equation*}
		\|w^{\msL^+_T\phi,0}\|_{C_TL^6_x}\lesssim \|\phi\|_{H^1_x}.
	\end{equation*}
\end{lemma}
\begin{proof}
	This is a direct consequence of Theorem~\ref{th:nl_skeleton_gwp} and Lemma~\ref{lem:ctrl_lineq_pseudoinv}, in particular the inequalities~\eqref{eq:nl_skeleton_global_increment_bound_0_initial_2} and \eqref{eq:lin_pseud_inv_bnd}.
\end{proof}
We are now ready to prove global, exact controllability of the skeleton equation.
\begin{proposition}[Global Exact Controllability of~$w^{h,0}$] \label{prop:nl_skeleton_control}
	For any $T>0$ and $\phi \in  H^1(\mbT^3)$ there exists an $h(\phi)\in L^2_{T}L^2(\mbT^3)$ such that $w^{h,0}(T)= \phi$, where $w^{h,0}$ solves \eqref{eq:nl_skeleton_equation} with $\mfz=0$.
\end{proposition}

\begin{proof}
	For~$T > 0$, using item~\ref{it:nl_skeleton_regularity} of~Theorem~\ref{th:nl_skeleton_gwp} we define the family of maps
	\begin{equation}
		H^1(\mbT^3)	\ni	\phi \mapsto 
		f_T(\phi)\coloneqq w^{\msL_T^+ \phi,0}(T) \in H^1(\mbT^3),
		\label{prop_ctrl_nonlineq_pf_Upsilon}
	\end{equation}
	where~$w^{\msL_T^+ \phi,0}$ solves~\eqref{eq:nl_skeleton_equation_from_0} on $[0,T]$; note that here we use the fact that $\msL^+_T$ is a bounded, linear operator.
	We will apply the global, inverse function theorem to show that there exists a map
	\begin{equation*}
		H^1(\mbT^3) \ni \phi \mapsto f^{-1}_T(\phi) \in H^1(\mbT^3),
	\end{equation*}
	such that choosing $h(\phi)\coloneqq \msL^+_Tf^{-1}_T(\phi)$ we have
	\begin{equation*}
		w^{h(\phi)}(T) = \phi
	\end{equation*}
	as required. A local version of the inverse function theorem can be found as \cite[Thm.~4.F]{zeidler_applied_fa}. A global version is given as \cite[Prob.~4.12]{zeidler_applied_fa} (see also \cite[Prob.~4.13]{zeidler_applied_fa}). We are required to show first $f_T :H^1(\mbT^3)\to H^1(\mbT^3)$ is a $C^1$-Frech\'et map and then that
	\begin{enumerate}[label=(\roman*)]
		\item $D_\phi f_T(\phi) \coloneqq D_\phi w^{\msL_T^+ \phi}[\,\cdot\,](T) :H^1(\mbT^3)\to H^1(\mbT^3)$ is bijective for all $\phi \in H^1(\mbT^3)$, \label{it:f_derivative_bijective}
		\item $f_T(\phi) :H^1(\mbT^3)\to H^1(\mbT^3)$ is \emph{proper} (that is if $K\subset H^1(\mbT^3)$ is compact, then $f^{-1}(K)\subset H^{1}(\mbT^3)$ is compact).  \label{it:f_proper}
	\end{enumerate}
	To show that $f_T$ is $C^1$, first let~$\phi \in H^1(\mbT^3)$,~$t \in [0,T]$ and  then write
	\begin{equation}\label{eq:w_STphi_mild}
		w^{\msL_T^+ \phi,0}(t)
		= -\int_0^t e^{(t-s)(\Delta - m^2)} \left(w^{\msL_T^+ \phi,0}(s)\right)^3 \dif s 
		+ 
		z^{\msL_T^+ \phi}(t),
	\end{equation}
	where $z^{\msL_T^+ \phi}$ solves \eqref{eq:lin_skeleton_equation} with $h= \msL_T^+ \phi$. Since the map $\phi \mapsto z^{\msL_T^+ \phi}$ is linear and bounded, we have
	\begin{equation*}
		D_\phi z^{\msL_T^+ \phi} [\psi] = z^{S_T\psi},\quad \text{for all}\quad \psi \in H^1(\mbT^3).
	\end{equation*}
	Then, it is readily checked, for example applying the implicit function theorem, \cite[Thm.~2.3]{ambrosetti_prodi_93}, that the map $\phi \to w^{\msL_T^+ \phi}$ is F\'rechet differentiable (see for example the arguments proceeding \cite[Prop.~5.1]{tsatoulis_weber_18_spectral} or the proof of \cite[Lem.~5.11]{chevyrev_hambly_mayorcas_22_stoch} for details of a similar argument). Then, applying the standard rules of calculus, and equality of Fr\'echet and Gateaux derivatives in the case that the former exists, one find that for all $\psi \in H^1(\mbT^3)$, the derivative satisfies, for all $t\in [0,T]$,
	\begin{equation}\label{eq:w_STphi_deriative}
		D_\phi w^{\msL_T^+ \phi,0}[\psi](t) = 3 \int_0^t e^{(t-s)(\Delta-m^2)} \left(w^{\msL_T^+ \phi,0}(s)\right)^2 D_\phi w^{\msL_T^+ \phi,0}[\psi](s) \dd s + z^{S_T\psi}.
	\end{equation}
	Similar estimates as for the linear and non-linear skeleton equations, \eqref{eq:lin_skeleton_equation} \& \eqref{eq:nl_skeleton_equation} allow for control on $D_\phi w^{\msL_T^+ \phi}[\psi] \in C_T H^1(\mbT^3)$, hence showing that $\phi \mapsto f_T(\phi ) = w^{\msL_T^+,0}$ is a $C^1$-Frech\'et map.
	
	To show Point~\ref{it:f_derivative_bijective}, we first observe that by definition $z^{S_T\psi}(T)=\psi$, so that we have
	\begin{equation}
		D_\phi w^{\msL_T^+ \phi,0}[\psi](T) = 3 \int_0^T e^{(T-s)(\Delta-m^2)} \left(w^{\msL_T^+ \phi,0}(s)\right)^2 D_\phi w^{\msL_T^+ \phi,0}[\psi](s) \dd s +\psi,
	\end{equation}
	and hence, computing somewhat formally,
	\begin{equation*}
		D_\phi w^{\msL_T^+ \phi,0}[\psi](T) =\psi \exp\left(3 \int_0^T e^{(T-s)(\Delta-m^2)} \left(w^{\msL_T^+ \phi,0}(s)\right)^2  \dd s\right).
	\end{equation*}
	From this expression it is clear that $D_\phi f_T(\phi)[\,\cdot\,] = D_\phi w^{\msL_T^+ \phi,0}[\,\cdot\,](T) :H^1(\mbT^3)\to H^1(\mbT^3)$ is a global, $C^1$-diffeomorphism. 
	
	To check Point~\ref{it:f_proper} we appeal to \cite[Ex.~3]{zeidler_applied_fa}; to wit, $f$ is \emph{proper} if 
	\begin{equation}\label{eq:f_infinity_growth}
		\lim_{\|\phi\|_{H^1_x} \to \infty} \|f(\phi)\|_{H^1_x} = \lim_{\|\phi\|_{H^1_x} \to \infty} \|w^{\msL_T^+ \phi}(T)\|_{H^1_x} =+\infty.
	\end{equation}
	and we can write
	\begin{equation}\label{eq:f_decompose}
		f_T(\phi) = f_1(\phi) + f_2(\phi),
	\end{equation}
	where $f_1:H^1(\mbT^3)\to H^1(\mbT^3)$ is a \emph{compact} and $f_2 : H^1(\mbT^3)\to H^1(\mbT^3)$ is an \emph{homeomorphism}.
	
	To show the decomposition \eqref{eq:f_decompose} it suffices to use \eqref{eq:w_STphi_mild} and again use the fact that $z^{\msL_T^+ \phi}(T)=\phi$ to see that
	\begin{align}\label{eq:w_STphi_mild_at_T}
		f_T(\phi) = w^{\msL_T^+ \phi}(T) 
		=& -\int_0^Te^{(T-s)(\Delta-m^2)} \left(w^{\msL_T^+ \phi}(s)\right)^3 \dd s + \phi.
	\end{align}
	So that we may set $f_2(\phi)= \phi$, which as the identity is clearly a homeomorphism. To show that the integral operator
	\begin{equation*}
		H^1(\mbT^3)\ni \phi \mapsto f_1(\phi) \coloneqq -\int_0^Te^{(T-s)(\Delta-m^2)} \left(w^{\msL_T^+ \phi}(s)\right)^3 \dd s\in H^1(\mbT^3)
	\end{equation*}
	is compact, we note that if $\phi \in H^1(\mbT^3)$, then by Theorem~\ref{th:nl_skeleton_gwp}, one has $s\mapsto w^{\msL_T^+ \phi}(s) \in C_T L^{3q}(\mbT^3)$, for some $q\in [1,2)$. Hence,
	\begin{align*}
		\norm[3]{\int_0^T e^{(T-s)(\Delta-m^2)} \left(w^{\msL_T^+ \phi}(s)\right)^3 \dd s \,}_{H^{1+\eps}_x} 
		\lesssim \int_0^T (t-s)^{-\frac{1+\eps}{2} -\frac{3}{2}\left(\frac{1}{3q}-\frac{1}{2}\right)} \|w^{\msL_T^+ \phi}(s)\|^3_{L^{3q}_x} \dd s\\
		\lesssim T^{1-\frac{1+\eps}{2} -\frac{3(2-3q)}{12q}} \sup_{s\in [0,T]} \|w^{\msL_T^+ \phi}(s)\|_{H^1_x} <\infty.
	\end{align*}
	Since $H^{1+\eps}(\mbT^3)\cembed H^1(\mbT^3)$ we have shown that the integral operator is a compact map from $H^1(\mbT^3)$ to itself.
	
	To show the limit \eqref{eq:f_infinity_growth}, since $w^{\msL_T^+ \phi,0} \in C_TH^1(\mbT^3)$ (see Theorem~\ref{th:nl_skeleton_gwp}) by continuity in time it is enough to show that
	\begin{equation*}
		\lim_{\|\phi\|_{H^1_x}\to \infty}\sup_{t\in [0,T]}\|w^{\msL_T^+ \phi}(t)\|_{H^1_x} =\infty.
	\end{equation*}
	Therefore, assume for a contradiction, that there exists a sequence $\phi^n$ and a $C>0$ such that
	\begin{equation*}
		\|\phi^n\|_{H^1_x} \to \infty\quad \text{and}\quad \sup_{t\in [0,T]} \|w^{\msL_T^+ \phi^n,0}(t)\|_{H^1_x} \leq C<\infty.
	\end{equation*}
	Simply rearranging \eqref{eq:w_STphi_mild_at_T} and applying similar estimates as above with $q=1$ and $\eps=0$
	\begin{align*}
		\|\phi^n\|_{H^1_x} = &\norm[3]{ w^{\msL_T^+ \phi^n,0}(T) +\int_{0}^{T} e^{(t-s)(\Delta-m^2)}(w^{\msL_T^+ \phi^n,0}(s))^3  \dd s \,}_{H^1_x}\\
		\leq \, & \|w^{\msL_T^+ \phi^n,0}(T)\|_{H^1_x} + T^{\frac{3}{4}} \left(\sup_{s\in [0,T]}\|w^{\msL_T^+ \phi^n,0}(s)\|^3_{H^1_x}\right)\\ 
		\leq & \, C + T^{\frac{3}{4}} C^3 <\infty.
	\end{align*}
	Since the right hand side is uniform in $n$ this contradicts our assumption, showing that the limit \eqref{eq:f_infinity_growth} holds and the proof is complete.
\end{proof}

\section{The Dynamical $\boldsymbol{\Phi^4_3}$ Equation on the Torus} \label{sec:technical_proofs}

In this appendix, we collect various properties of the~$\Phi^4_3$ equation on the 3D torus. We begin with a recap of the local and global well-posedness theory for the $\Phi^4_3$ dynamics, as given by the theory of regularity structures, c.f. Section~\ref{app:sol_theory}. As discussed in the introduction, the appeal to regularity structures is essentially incidental. However, from our perspective the theory presents two advantages. Firstly, prior work has already obtained a large deviation principle for the dynamic $\Phi^4_3$ model using the theory of regularity structures, \cite{hairer_weber_ldp}. Secondly, the theory of regularity structures is largely systematised, giving a pathwise, local well-posedness for a large range of models at once, opening up the possibility of approaches similar to ours for many, interesting EQFT models.

In Section~\ref{sec:uniform_ldp}, we extend the main result of \cite{hairer_weber_ldp} to obtain a locally uniform large deviation principle, leveraging continuity of the solution map and the large deviation principle obtained on model space therein.  The locally uniform LDP is required in the proof of our main result, see the proofs of Theorem~\ref{thm:ldp_lower_bound}, Proposition~\ref{prop:bound_prob_H} and Proposition~\ref{prop:ldp_upper_outside_H}. Finally, in Section~\ref{sec:tail_bounds}, we recap results on ergodicity for the $\Phi^4_3$ dynamics, identification of the unique invariant measure with the $\Phi^4_3$ EQFT measure and verification of the appropriate Osterwalder--Schrader axioms. In that section we also present  suitable tail bounds on the invariant measure $\mu_{\eps}$ which are required in our main proofs.
\subsection[Local and Global Solution Theory for the~$\Phi^4_3$ Equation]{Local and Global Solution Theory for the~$\boldsymbol{\Phi^4_3}$ Equation} \label{app:sol_theory}

We first focus on the \emph{local} solution theory of the periodic $\Phi^4_3$ dynamics, presented in the language of regularity structures. We recall that the dynamic equation with smooth noise reads,
\begin{equation}
	\label{eq:phi43_SPDE_mollified_unrenormalised_app}
	\partial_t u_{\eps,\kappa}^\fz - \Delta u^\fz_{\eps,\kappa}
	= - \del[0]{u^\fz_{\eps,\kappa}}^3 - m^2 u^\fz_{\eps,\kappa} + \eps \xi_\kappa, \quad
	u^\fz_{\eps,\kappa}(0,\,\cdot\,) = \fz.
\end{equation}
Since the driving noise~$\xi_\kappa = \xi * \rho_\kappa$ denotes a space-time white noise~$\xi$ convolved with a mollifier~$\rho_\kappa$, this equation is classically well-posed at any scale~$\kappa > 0$. 
For~$\kappa  = 0$, however, the driving noise~$\xi$ is a distribution of (parabolic) H\"older--Besov regularity~$\alpha - 2 < -\frac{d+2}{2}$; hence, by Schauder theory, the solution to the corresponding linear equation has regularity~$\alpha < -\frac{d-2}{2}$. As soon as~$d \geq 2$, this is strictly negative and the cubic term in the non-linear equation becomes, a priori,  ill-defined.

In fact, already in dimension~$d = 2$ and for fixed~$\eps > 0$, it has been proved by Hairer, Ryser, and Weber~\cite{hairer-ryser-weber} that~$u_{\eps,\kappa}^\fz$ converges to~$0$ in a space of distributions in probability as~$\kappa \to 0$. If we denote by~$\Phi_{\tiny \text{clas}}(\fz,\eps \xi_\kappa) \coloneq u_{\eps,\kappa}^\fz$ the classical solution map to~\eqref{eq:phi43_SPDE_mollified_unrenormalised_app}, this immediately implies that~$\Phi_{\tiny \text{clas}}(\fz,\,\cdot\,)$ is discontinuous for~$d = 2$; simple power counting suggests that the situation is no better in higher dimensions.
The theory of regularity structures~\cite{hairer_rs}\footnote{For an introduction to the theory, see the two expository articles~\cite{hairer_intro, hairer_rs_dynamical_phi43}, the second of which is specialised to the~$\Phi^4_3$ equation. Since the case~$d = 2$ is analogous, but easier, we focus on~$d = 3$ from here on.} takes the following steps to overcome that problem:
\begin{itemize}
	\item Let~$\zeta$ be a smooth driving noise, for example~$\zeta = \xi_\kappa$ for~$\kappa > 0$. In a first step, one constructs a finite number of non-linear objects from~$\zeta$ which are indexed by \emph{trees}.
	The collection of these objects is called a (minimal) \emph{model}~$\Pi^\zeta$, see Definition~\ref{def:model} below for details.
	\item On the space of (minimal) models~$\MMm$, a \emph{non-linear} metric space with distance~$d_{\MMm}$, one defines a solution map~$\Phi(\fz, \cdot)$ in a way that is consistent with the classical solution theory, i.e. such that~$\Phi_{\tiny \text{clas}}(\fz,\zeta) = \Phi(\fz,\Pi^\zeta)$.
	This solution map~$\Phi$ is locally Lipschitz \emph{continuous} in both arguments, see Definition~\ref{def:solution_map} and Lemma~\ref{lem_uniform_lipschitz} below.
	\item As discussed above, one does not expect the naive models~$\Pi^{\eps\xi_\kappa}$ to converge as~$\kappa \to 0$.
	However, in the case of the periodic $\Phi^4_3$ dynamics, a relatively simple renormalisation procedure maps~$\Pi^{\eps\xi_\kappa} \rightsquigarrow \bar{\Pi}^{\eps\xi_\kappa}$ through two families of diverging constants~$C^{(1)}_{\eps,\kappa}$ and~$C^{(2)}_{\eps,\kappa}$ (see Definition~\ref{def:bphz_lift}). There exists a particular choice of such constants for which~$\bar{\Pi}^{\eps \xi_\kappa}$ converges to a limiting model~$\bar{\Pi}^\eps$ as~$\kappa \to 0$. This model is called the~\emph{BPHZ model}, see Proposition~\ref{prop:bphz_model_noise_intensity}.
	\item In order for the renormalisation procedure to be meaningful, one needs to associate~$\bar{u}_{\eps,\kappa}^\fz := \Phi(\fz,\bar{\Pi}^{\eps\xi_\kappa})$ to an actual equation.
	In fact, it can be shown that~$\bar{u}_{\eps,\kappa}^\fz$ solves~\eqref{eq:phi43_SPDE_mollified_renormalised}, which reads
	\begin{equation*} 
		\partial_t \bar{u}^\fz_{\eps,\kappa} - \Delta \bar{u}^\fz_{\eps,\kappa}
		= - \del[0]{\bar{u}^\fz_{\eps,\kappa}}^3 - \del[1]{m^2  -3 \eps^2 C_{\kappa}^{(1)} + 9 \eps^4 C_{\kappa}^{(2)}} \bar{u}^\fz_{\eps,\kappa} + \eps \xi_\kappa, \quad
		\bar{u}^\fz_{\eps,\kappa}(0,\,\cdot\,) = \fz.
	\end{equation*}
\end{itemize}
Finally, all these steps justify the following definition:

\begin{definition}[Rigorous meaning of~$\bar{u}_\eps^\fz$] \label{def:phi43_rigorous_SPDE}
	For~$\eps \in [0,1]$ and~$\fz \in \CC^\alpha(\mbT^3)$ with~$\alpha  \in (-\nicefrac{2}{3},-\nicefrac{1}{2})$,  
	we define
	\begin{equation} \label{eq:phi43_rigorous_SPDE}
		\bar{u}_{\eps}^\fz \coloneqq \Phi(\fz,\bar{\Pi}^\eps)
		= \lim_{\kappa \to 0} \Phi(\fz,\bar{\Pi}^{\eps\xi_\kappa}) \quad \text{in prob.} \
		in~ C_T\CC^\alpha(\mbT^3).
	\end{equation}
\end{definition}
We emphasise that~$\bar{u}^\fz_\eps$ does \emph{not} solve a stochastic PDE itself;
in line with the previous definition, the equation in~\eqref{eq:phi43_formal_SPDE} only represents the~\emph{formal} limit as~$\kappa \to 0$ of the stochastic PDEs for~$\bar{u}^\fz_{\eps,\kappa}$, see~\eqref{eq:phi43_SPDE_mollified_renormalised}.

\subsubsection{Trees, Models, and the Solution Map}	

In this section, we properly introduce the abstract objects alluded to in the outline above.
Particular examples of the regularity structure associated to the~$\Phi^4_3$ equation have been presented in numerous other works, originally in \cite[Sec.~9.2 and~9.4]{hairer_rs}, summarised in by~\cite[Sec.~2]{hairer_weber_ldp}, as well as \cite{hairer_rs_dynamical_phi43, hairer_matetski_18_discretisation}. As a result, we do not present all the required details here but refer the interested reader to the above works.
The main object to introduce is the set of \emph{trees} or \emph{symbols}, denoted by~$\CW_-$, of negative degree which do not contain polynomial factors. These are summarised in the following table:
\setcounter{table}{0}
\renewcommand{\thetable}{\arabic{table}}
{\small
	\begin{table}[h]
		\centering
		\renewcommand{\arraystretch}{1.6}
		\begin{tabular}{cccccccc} 	
			\toprule
			\centering
			$\tau$ & \centering $\<xi>$ & \centering $\<1>$ & \centering $\<2>$ 
			&  \centering $\<3>$ 
			&  \centering $\<22>$ & \centering $\<31>$ & \centering $\<32>$ \tabularnewline
			\midrule
			\centering $ \ \abs{\tau} \ $ & \centering $\ -\frac{5}{2} - \kappa \ $ & \centering $\ - \frac{1}{2} - \kappa \ $ & \centering $ \ -1 - 2\kappa \ $ 
			& \centering $ \ -\frac{3}{2} - 3\kappa\ $ 
			& \centering $ \ -4\kappa \ $ & \centering $ \ -4\kappa \ $ & \centering $ \ -\frac{1}{2} - 5\kappa \ $
			\tabularnewline
			\midrule
			\centering $ \ n_\tau \ $ & \centering $\ 1 \ $ & \centering $\ 1 \ $ 
			& \centering $ \ 2 \ $ 
			& \centering $ \ 3 \ $ & \centering $ \ 4 \ $ & \centering $ \ 4 \ $ & \centering $ \ 5 \ $ \tabularnewline
			\bottomrule
		\end{tabular}
		\caption{\label{table_trees} \hspace{-0.5em} List of trees~$\tau \in \CW_-$, their degree~$\abs[0]{\tau}$, and the number of leaves~$n_\tau$.}
	\end{table}
} 

\noindent	
We also set~$\CT_- := \spn \CW_-$ and fix a compact domain~$D \coloneqq I \times \mbT^3$ for some interval~$I \subseteq \R$.
One may think of~$I$ as the interval~$I_T \coloneqq [0,T]$ for some horizon~$T > 0$. 
For technical reasons, however, we are required, at various points, to enlarge the interval~$I$, see Remarks~\ref{rmk:recentering_maps} and~\ref{rmk:single_test_fct}.
The precise length of ~$I$ is essentially unimportant, so long as it contains~$[-2,T+2]$.
We write~$D_T$ instead of~$D$ when we wish to highlight its dependence on~$T$.
The following definition is taken from~\cite[p.~11]{hairer_weber_ldp}.

\begin{definition}[Minimal model space] \label{def:model}
	For any smooth map
	$\Pi: D \to \CL(\CT_-,C^\infty(D;\R))$
	we define 
	\begin{equs}[][def:model:hom_model_norm_trees]
		\threebars \Pi \tau\threebars := 
		\begin{cases}
			\sup_{\lambda \in (0,1]} \sup_{\varphi \in \CB_4} \sup_{z \in D} \thinspace \lambda^{-\thinspace\abs{\tau}} \abs[1]{\scal{\Pi_z \tau,\varphi_z^{\lambda}}}_{L^2} & \quad \text{if} \quad \tau \in \CW_-\setminus\{\<xi>, \<1>\}, \\[0.5em]
			\sup_{s \in \R}\sup_{\lambda \in (0,1]} \sup_{\varphi \in \CB_4} \sup_{z \in D} \thinspace \lambda^{-\thinspace\abs[0]{\<xi>}} \abs[1]{\scal{\boldsymbol{1}_{t \geq s}\Pi_z \<wn>,\varphi_z^{\lambda}}_{L^2}} & \quad \text{if} \quad \tau = \<xi>, \\[0.5em]
			\sup_{\lambda \in (0,1]} \sup_{\eta \in \CB_3} \sup_{z \in D} \thinspace \lambda^{-\thinspace\abs[0]{\<1>}} \abs[1]{\scal{\Pi_z\<1>(t,\cdot),\eta_x^{\lambda}}_{L^2}} & \quad \text{if} \quad \tau = \<1>,
		\end{cases}
	\end{equs}
	where
	\begin{itemize}
		\item the collections of test functions~$\CB_d$ for $d=3,\,4$ are given by 
		\begin{equation*}
			\CB_d := \left\{\psi \in C^3(\R^d;\R): \ \norm{\psi}_{C^3} \leq 1, \ \supp \, (\psi) \subseteq B(0,1)\right\}, \quad d \in \{3,4\},
		\end{equation*}
		\item the rescaled, shifted test functions~$\varphi_z^\lambda$ resp.~$\eta_x^\lambda$ are defined by
		\begin{equation*}
			\varphi_z^\lambda(\bar{z}) 
			\coloneq \lambda^{-5} \varphi(\lambda^{-2}(t - \bar{t}), \lambda^{-1}(x - \bar{x})), \quad
			\eta_x^\lambda(\bar{x}) 
			\coloneq \lambda^{-3} \eta(\lambda^{-1}(x - \bar{x}))
		\end{equation*}
		for~$z = (t,x)$~and~$\bar{z} = (\bar{t}, \bar{x})$ in~$D$.
	\end{itemize}
	The (inhomogeneous, minimal) \emph{model norm}~$\threebars \cdot \threebars$ and the distance~$d_{\MMm}$ are then given by 
	\begin{equation}
		\threebars \Pi\threebars := 
		\max_{\tau \in \CW_-} \threebars\Pi \tau\threebars, \quad
		d_{\MMm}(\Pi,\tilde{\Pi}) := \sup_{\tau \in \CW_-} \sup_{\lambda \in (0,1]} \sup_{\varphi \in \CB_4} \sup_{z \in D} \lambda^{\abs[0]{\tau}} \abs[0]{\scal{\Pi_z \tau - \bar{\Pi}_z \tau, \varphi^{\lambda}_z}}.
		\label{def:model_distance}
	\end{equation}
	and the space~$\MMm$ of (minimal, admissible) models is defined to be the completion of smooth maps~$\Pi$, as above, under the distance~$d_{\MMm}$. We also impose the algebraic constraint~\cite[Eq.~2.9]{hairer_weber_ldp}, the admissibility conditions~\cite[Eq.~2.10\thinspace (a-c)]{hairer_weber_ldp}
	and define the~\emph{homogeneous} model norm~$\barnorm{\Pi}$ by
	\begin{equation*}
		\barnorm{\Pi} := \max_{\tau \in \CW_-} \barnorm{\Pi \tau}, \quad
		\text{where}\quad \barnorm{\Pi \tau} := \threebars \Pi \tau \threebars^{\frac{1}{n_\tau}} 
		\quad
		\text{for}
		\quad
		\tau \in \CW_-.
	\end{equation*}
\end{definition}
As before, we write~$\MMm_T$ as well as $\threebars \cdot \threebars_T$ and~$\barnorm{\cdot}_T$ whenever we wish to highlight their dependence on~$T$.
The following lemma is immediate from the definitions.
\begin{lemma} \label{lem:conv_models}
	Convergence w.r.t.~$\threebars \cdot \threebars$ and~$\barnorm{\cdot}$ in~$\MMm$ are equivalent.
\end{lemma}

\begin{remark}[Recentering maps~$F$] \label{rmk:recentering_maps}
	The usual definition of a model, minimal or otherwise, is a pair of maps $(\Pi,F)$ with $\Pi$ as above and $F$ a \emph{recentring map}, see Remark~\ref{rem:model_extension} below. In the case of minimal models however, the implicit algebraic and admissibility conditions in De\-fi\-ni\-tion~\ref{def:model} guarantee that one can canonically define the action of the recentring maps~$F_z\tau$ for~$z \in D$ on the symbols~$\tau \in \CW_-$; further details are presented in~\cite[Sec.~2.3]{hairer_weber_ldp}.
	For the requisite analytical bounds to be satisfied, one will have to suitably enlarge the interval~$I$, see \cite[Remarks~2.3 and~2.4]{hairer_weber_ldp} and the discussion below their Remark~2.6.
\end{remark}

\begin{remark}[Choice of test function] \label{rmk:single_test_fct}
	As an immediate consequence of~\cite[Prop.~13.1 and~13.2]{caravenna_zambotti_2020}, one can equivalently state Definition~\ref{def:model} using only a \emph{single} test function~$\varphi$ (resp.~$\eta$), at the expense of suitably enlarging the domain~$D$. 
	Henceforth, in view of Proposition~\ref{prop:cdfi_model} below, we choose~$\varphi := \Psi$ where~$\Psi$ is introduced in~\cite[p.~$4$]{moinat_weber_20_phi43Loc}. 
	Since we do not explicitly refer to the properties of~$\Psi$, we do not give details of its construction here. 
\end{remark}

Definition~\ref{def:model} of the space of~\emph{minimal} models~$\MMm$ is very similar to that of the model space~$\MM$ given in~\cite[Def.~2.17]{hairer_rs}.
The following remark establishes a connection between both notions.

\begin{remark}[Minimal models and  models]\label{rem:model_extension}
	By~\cite[Thm.~2.10]{hairer_weber_ldp}, 
	\begin{itemize}
		\item any element~$\Pi \in \MMm$ can be \emph{uniquely} extended to a bona fide admissible model~$\bz = (\tilde{\Pi},\tilde{F}) \in \MM$ such that~$\Pi_z \tau = \tilde{\Pi}_z \tau$ for any~$z \in D$ and~$\tau \in \CW_-$;
		\item the map~$\EE: \MMm \to \MM, \ \Pi \mapsto~\bz = (\tilde{\Pi},\tilde{F})$ is locally Lipschitz continuous.
	\end{itemize}
	This is reminiscent of Lyons' extension theorem in rough paths theory.
	Conversely, if we choose \mbox{$\bz = (\tilde{\Pi},\tilde{F}) \in \MM$} and set~$\Pi_z := \tilde{\Pi}_z\sVert[0]_{\CW_-}$ for~$z \in D$, we obtain an element~$\Pi \in \MMm$.	
\end{remark}
We work with minimal admissible models since they do not act on polynomials and therefore transform \emph{homogeneously} when rescaling the noise~$\xi \rightsquigarrow \eps \xi$, see Lemma~\ref{lem:homog_model_norm} below. 
This is also the reason why we work with the \emph{homogeneous} model norm~$\barnorm{\cdot}_T$ rather than the inhomogeneous model norm~$\threebars \cdot \threebars_T$.
Both of these facts will be crucial in Sections~\ref{sec:uniform_ldp} and~\ref{sec:tail_bounds} below.

\begin{definition}[Renormalised Canonical lift] \label{def:bphz_lift}
	Let~$G$ denote the Green's function of the heat kernel~$(\partial_t - \Delta)$ on~$\R_+ \times \mbT^3$ with decomposition 
	\begin{equation}
		G = K + R
		\label{eq:kernel_decomposition}
	\end{equation} 
	where~$R$ is smooth and bounded and~$K$ satisfies~\cite[Assumption~5.1, 5.3, and 5.4]{hairer_rs}.  
	For~$\zeta \in C^\infty(D)$, arbitrary constants~$C_\zeta^{(1)}, C_\zeta^{(2)} \in \R$ and points $z,\,\bar{z} \in D$, we then set 
	\vspace{0.5em}
	\begin{itemize}
		\begin{minipage}{0.5\linewidth}
			\item $\bar{\Pi}^{\zeta}_z \<wn> \coloneqq \zeta$
			\item $\bar{\Pi}^\zeta_z \<1> \coloneqq K * \zeta$
			\item $\bar{\Pi}^\zeta_z \<2> \coloneqq \del[1]{\bar{\Pi}^\zeta_z \<1>}^2 - C_\zeta^{(1)}$
			\item $\bar{\Pi}^\zeta_z \<3> \coloneqq \del[1]{\bar{\Pi}^\zeta_z \<1>}^3 - 3 C_\zeta^{(1)} \, \bar{\Pi}^\zeta_z \<1>$
			\item $\bar{\Pi}^\zeta_z \<20> (\bar{z})\coloneqq (K * \bar{\Pi}^\zeta_z \<2>)(\bar{z}) - (K * \bar{\Pi}^\zeta_z \<2>)(z)  $
		\end{minipage}\hspace{0.2em}
		\begin{minipage}{0.5\linewidth}
			\item $\bar{\Pi}^\zeta_z \<30>(\bar{z}) \coloneqq \del[1]{K * \bar{\Pi}^\zeta_z \<3>}(\bar{z}) - \del[1]{K * \bar{\Pi}^\zeta_z \<3>}(z)$
			\item $\bar{\Pi}^\zeta_z \<31> \coloneqq \bar{\Pi}^\zeta_z \<30> \; \bar{\Pi}^\zeta_z \<1>$
			\item $\bar{\Pi}^\zeta_z \<22> \coloneqq \bar{\Pi}^\zeta_z \<20> \; \bar{\Pi}^\zeta_z \<2> - C_\zeta^{(2)}$
			\item $\bar{\Pi}^\zeta_z \<32> \coloneqq \bar{\Pi}^\zeta_z \<30> \; \bar{\Pi}^\zeta_z \<2> - 3 C_\zeta^{(2)} \bar{\Pi}^\zeta_z \<1>$
		\end{minipage}
	\end{itemize}
	For any choice of constants~$C_\zeta^{(i)}$, this defines an element~$\bar{\Pi}^\zeta \in \MMm$, see~\cite{hairer_weber_ldp}, more specifically their comments before Theorem~2.13 (incl. references).
	\begin{itemize}
		\item For~$C_\zeta^{(1)} = C_\zeta^{(2)} = 0$, we write~$\Pi^\zeta$ instead of~$\bar{\Pi}^\zeta$ and refer to it as the \emph{canonical lift} of~$\zeta$ to~$\MMm$.
		\item 	In the specific case~$\zeta = \eps\xi_\kappa$, we write~$C^{(i)}_{\eps,\kappa} := C^{(i)}_{\eps\xi_\kappa}$ and~$C^{(i)}_\kappa := C^{(i)}_{1,\kappa}$ for~$i = 1,2$ and call~$\bar{\Pi}^{\eps \xi_\kappa}$ the \emph{renormalised} canonical lift of~$\eps \xi_\kappa$.
	\end{itemize}
\end{definition}
The next lemma is a special case of~\cite[Prop.~2]{schoenbauer} and extends the definition of the canonical lift~$\Pi^\bullet$ to Cameron--Martin functions~$h \in L^2_T L^2(\mbT^3)$, for any~$T > 0$.
\begin{lemma}[Canonical lift of~Cameron--Martin functions] \label{lem:can_lift_cameron_martin_fcts}
	Let~$T > 0$. The map
	\begin{equation*}
		\CC^\infty(D_T) \ni \zeta \mapsto \Pi^\zeta \in \MMm_T
	\end{equation*}
	extends to a continuous map from~$L^2_T L^2(\mbT^3)$ into~$\MMm$. In addition, it is locally Lipschitz continuous in the sense that, for any~$M > 0$, 
	uniformly over all~$h_1, h_2 \in L^2_T L^2(\mbT^3)$ with~$\norm[0]{h_1}_{L^2 L^2_x} \vee \norm[0]{h_2}_{L^2 L^2_x} \leq M$, it holds that
	\begin{equation}
		\label{lem:can_lift_cameron_martin_fcts:lipschitz}
		d_{\MMm}(\Pi^{h_1},\Pi^{h_2}) \lesssim_{T,M} \norm[0]{h_1 - h_2}_{L^2_T L^2_x}.
	\end{equation}
	For~$w^{h,\fz}$ with~$h \in L^2_T L^2(\mbT^3)$ and~$\fz \in \CC^\alpha(\mbT^3)$ as in~\eqref{eq:nl_skeleton_equation}, we also have~$\Phi(\fz,\Pi^h) = w^{h,\fz}$. 
\end{lemma}

The following proposition provides us with a specific choice of constants~$C_{\eps,\kappa}^{(i)}$ which then remain fixed for the rest of the appendix. 

\begin{proposition}[BPHZ model] \label{prop:bphz_model_noise_intensity}
	Let~$\eps \in [0,1]$. 
	There exists a choice of constants~$C_{\kappa}^{(1)}$ and~$C_\kappa^{(2)}$ such that
	\begin{enumerate}[label=(\roman*)]
		\item \label{prop:bphz_model_noise_intensity:i} $C_{\eps,\kappa}^{(1)} = \eps^{2}\, C_{\kappa}^{(1)}$ and $C_{\eps,\kappa}^{(2)} = \eps^{4}\, C_{\kappa}^{(2)}$ for any~$\kappa > 0$, and
		\item \label{prop:bphz_model_noise_intensity:ii} the corresponding sequence~$(\bar{\Pi}^{\eps \xi_\kappa})_{\kappa > 0}$ converges in probability in~$(\MMm,d_{\MMm})$ to a limiting model~$\bar{\Pi}^{\eps}$.
	\end{enumerate}
	We call~$\bar{\Pi}^{\eps}$ the (minimal) \emph{BPHZ model} (with noise intensity~$\eps$) and also write~$\bar{\Pi} := \bar{\Pi}^1$.
\end{proposition}

\begin{proof} 
	Let us define the constants~$C_{\eps,\kappa}^{(1)}$ and~$C_{\eps,\kappa}^{(2)}$ as in~\cite[Eq.~(10.33)]{hairer_rs} and~\cite[Eq.~(10.41)]{hairer_rs}, respectively;
	the claim in~\ref{prop:bphz_model_noise_intensity:i} is immediate from that definition. It also follows from a more general formula for the BPHZ renormalisation constants, see~\cite[Rmk.~2.23]{rs_renorm} and the references therein.
	The assertion in~\ref{prop:bphz_model_noise_intensity:ii} is the content of~\cite[Thm.~10.7]{hairer_rs} (which, in turn, is a special case of the more general results in~\cite{chandra_hairer} or~\cite{hairer_steele_bphz}). 
\end{proof}

\begin{lemma}[Homogeneity of~$\bar{\Pi}^\eps$] \label{lem:homog_model_norm}
	Let~$T > 0$. For any~$z \in D_T$ and any~$\tau \in \CW_-$, we have~$\bar{\Pi}^{\eps}_z \tau = \eps^{n_\tau} \bar{\Pi}_z \tau$. 
	As a consequence, we also have~$\barnorm{\bar{\Pi}^\eps}_T = \eps \barnorm{\bar{\Pi}}_T$.
\end{lemma}

\begin{proof}
	By definition of~$\bar{\Pi}^{\eps}$, it suffices to check that the claims hold for~$\bar{\Pi}^{\eps \xi_\kappa}$ with~$\kappa > 0$ fixed and~$C_{\eps,\kappa}^{(i)}$ as in Proposition~\ref{prop:bphz_model_noise_intensity}\ref{prop:bphz_model_noise_intensity:i}.
	In that case, the claims follow immediately from Definitions~\ref{def:model} and~\ref{def:bphz_lift}.
\end{proof}

\begin{lemma}[Fernique] \label{lem:fernique_bphz}
	For any~$T > 0$, there exists some~$\Lambda_T > 0$ such that 
	\begin{equation*}
		K_T \coloneqq  \E\sbr[0]{\exp\del[1]{\Lambda_T \barnorm{\bar{\Pi}}_T^2}} < \infty.
		\label{lem:fernique_bphz:cst}
	\end{equation*}
\end{lemma}

\begin{proof}
	It is immediate to see that, for~$\kappa > 0$ and~$\tau \in \CW_-$, the norm~$\barnorm{\bar{\Pi}^{\xi_\kappa} \tau}_T^{n_\tau}$ is an element in the $n_\tau$-th inhomogeneous Wiener chaos.
	Since all Wiener chaoses are closed under convergence in probability~\cite{schreiber_1969}, the norm~$\barnorm{\bar{\Pi} \tau}_T^{n_\tau}$ lives in the $n_\tau$-th inhomogeneous Wiener chaos as well; for elements therein, it is well known that there exists some~$\Lambda_T^\tau > 0$ such that
	\begin{equation*}
		\E\sbr[0]{\exp\del[1]{\Lambda_T^\tau \barnorm{\bar{\Pi} \tau}_T^2}} < \infty.
	\end{equation*}
	For each~$\tau \in \CW_-$, we then choose~$p_\tau \in (1,\infty)$ such that~$\sum_{\tau \in \CW_-} \frac{1}{p_\tau} = 1$ and set
	\begin{equation*}
		\Lambda_T := \min_{\tau \in \CW_-} \frac{\Lambda_T^{\tau}}{p_\tau}.
	\end{equation*}	
	Since
	\begin{equation*}
		\barnorm{\bar{\Pi}}_T^2 = \max_{\tau \in \CW_-} \barnorm{\bar{\Pi} \tau}^2 
		\leq \sum_{\tau \in \CW_-} \barnorm{\bar{\Pi} \tau}^2
	\end{equation*}
	Hölder's inequality implies
	\begin{align*}
		K_T
		& \leq 
		\E\sbr[4]{\prod_{\tau \in \CW_-} \exp\del[1]{\Lambda_T \barnorm{\bar{\Pi} \tau}^2}}
		\leq 
		\prod_{\tau \in \CW_-} 
		\norm[1]{\exp\del[1]{\Lambda_T \barnorm{\bar{\Pi} \tau}^2}}_{L^{p_\tau}(\P)} \\
		& =
		\prod_{\tau \in \CW_-} 
		\E\sbr[1]{\exp\del[1]{\Lambda_T p_\tau \barnorm{\bar{\Pi} \tau}^2}}^{\frac{1}{p_\tau}}
		\leq
		\prod_{\tau \in \CW_-} 
		\E\sbr[1]{\exp\del[1]{\Lambda_T^\tau \barnorm{\bar{\Pi} \tau}^2}}^{\frac{1}{p_\tau}}
		< \infty.
	\end{align*}
\end{proof}

Finally, let us give some details regarding the solution map~$\Phi$.
For later reference, we already present the next two statements for \emph{arbitrary}~$T > 0$ since we know that~$\bar{u}_\eps^\fz$ exists forever, a fact that we recapitulate in the next subsection, see Proposition~\ref{prop:cdfi_model} and Corollary~\ref{coro:global_existence}.

\begin{definition}[Solution map] \label{def:solution_map}
	For any~$T > 0$, the solution map~$\Phi: \CC^{\alpha}(\mbT^3) \x \MMm_T \to C_T\CC^{\alpha}(\mbT^3)$ is defined as~$\Phi(\mfz,\Pi) := \CS_A(-m^2,\mfz,\Pi)$ where~$\CS_A$ is given in~\cite[Thm.~2.12]{hairer_weber_ldp}.
\end{definition}

The following lemma immediately follows from global existence and the joint local Lipschitz continuity of the solution map, see~\cite[Lem.~7.5 and Thm.~7.8]{hairer_rs} as well as~(the proof of)~\cite[Thm.~2.12]{hairer_weber_ldp} and the references therein.
\begin{lemma} \label{lem_uniform_lipschitz}
	For all~$R > 0$ and~$T > 0$, there exists some~$L_{R,T} > 0$ such that for all $\Pi, \tilde{\Pi} \in \MMm_T$ with~$\threebars \Pi \threebars_T \vee \threebars \tilde{\Pi} \threebars_T \leq R$ and $\fz, \tilde{\fz} \in \CC^\alpha(\mbT^3)$ with~$\norm[0]{\fz}_{\CC^\alpha_x} \vee \norm[0]{\tilde{\fz}}_{\CC^\alpha_x} \leq R$, it holds that
	\begin{equation} \label{lem_uniform_lipschitz_eq_IC}
		\norm[0]{\Phi(\fz, \Pi) - \Phi(\tilde{\fz}, \tilde{\Pi})}_{C_T \CC^\alpha_x} \leq L_{R,T}  \,
		\del[2]{\norm[0]{\fz - \tilde{\fz}}_{\CC^\alpha_x} + d_{\MMm}(\Pi,\tilde{\Pi})} \,.
	\end{equation}
	In particular, we have
	\begin{equation}
		\sup_{\fz \in \bar{B}_R(0)} \norm[0]{\Phi(\fz, \Pi) - \Phi(\fz, \tilde{\Pi})}_{C_T \CC^\alpha_x} \leq L_{R,T}  \; d_{\MMm}(\Pi,\tilde{\Pi})
		\label{lem_uniform_lipschitz_eq} \,.
	\end{equation}
\end{lemma}
\subsubsection{Coming Down from Infinity} \label{sec:cdfi}
In the previous subsection, we have sketched the \emph{local} solution theory for the~$\Phi^4_3$ equation using regularity structures; in this subsection, we will recap the strategy to show that this solution actually exists~\emph{globally} in time. 
This follows from the so-called \emph{coming down from infinity}~(CDFI) estimate which was first obtained by Mourrat and Weber~\cite{mourrat_weber_infinity} for the~$\Phi^4_3$ equation on the $3$D torus based on paraproduct estimates and paracontrolled calculus in the sense of~\cite{gip}.
In recent years, the same result has been obtained by Moinat and Weber~\cite{moinat_weber_20_phi43Loc} (see also~\cite{moinat_weber_20_RD}) in a setting that is closer to regularity structures, allowing for a streamlined proof which has then been generalised by Chandra, Moinat, and Weber~\cite{chandra_moinat_weber_23} to the case of the~$\Phi^4_{4-\iota}$ equation when~$\iota > 0$.

The following proposition presents a version of the CDFI estimate that is essentially equivalent to~\cite[Thm.~2.1]{moinat_weber_20_phi43Loc} but phrased in the language of the previous subsection;
this will be useful in Sections~\ref{sec:uniform_ldp} and~\ref{sec:tail_bounds} below.

\begin{proposition} \label{prop:cdfi_model}
	Let~$\fz \in \CC^{\alpha}(\T^3)$ and~$\eps \in [0,1]$. We denote by~$\<1b>^\eps$ the solution to the additive noise, stochastic heat equation, with noise intensity~$\eps$  and~$0$ initial condition, i.e.
	\begin{equation}
		(\partial_t - \Delta) \<1b>^{\eps} = \eps \xi, \quad \<1b>^\eps (0,\cdot) = 0.
		\label{eq:she_tree}
	\end{equation}
	Furthermore, set~$v_\eps^\fz := u_\eps^\fz - \<1b>^\eps$ with~$u_\eps^\fz$ as given by~\eqref{eq:phi43_rigorous_SPDE}.
	Then, for any~$t \in (0,1)$, the bound 
	\begin{equation}
		\sup_{(s,x) \in (t,1) \x \T^3} \abs[0]{v_\eps^\fz(s,x)} \lesssim \frac{1}{t^2} \vee \del[1]{\eps\barnorm{\bar{\Pi}}}^{\frac{2}{1 - 2\kappa}}
		\label{prop:cdfi_model_estimate}
	\end{equation}
	holds uniformly over all initial conditions~$\fz \in \CC^{\alpha}(\T^3)$, where~$\bar{\Pi}$ denotes the (minimal) BPHZ~model.
\end{proposition}
\begin{proof}
	It suffices to prove the claim when~$\eps = 1$; the case of general~$\eps \in [0,1]$ follows from Lemma~\ref{lem:homog_model_norm}.
	When $\eps =1$ the claimed statement of~\eqref{prop:cdfi_model_estimate} is equivalent to the content of~\cite[Thm.~2.1]{moinat_weber_20_phi43Loc} upon making the change
	\begin{equation*}
		\barnorm{\bar{\Pi}} \, \rightsquigarrow \, \max\left\{[\sigma]^{\frac{1}{n_\sigma}}_{\abs[0]{\sigma}}: \, \sigma \in L\right\}
		\label{prop:cdfi_model:pf_eq1}
	\end{equation*}
	for the following list of trees: 
	\begin{equation*}
		L := \{\<1b>, \, \<2b>, \, \<2b>_x, \, \<20b>, \, \<30b>, \, \<22b>, \, \<31b>, \, \<32b>\}.
	\end{equation*}
	These trees are closely related to those in~$\CW_-$, see table~\ref{table_trees}, but---despite superficial appearances---they are~\emph{not} identical.\footnote{The connection between the symbols in~$L$ and~$\CW_-$ is discussed in detail in~\cite[Sec.~3]{moinat_weber_20_phi43Loc}.}
	We therefore colour them in black rather than blue; however, we emphasise that the trees~$\sigma \in L$ will \emph{only} play a role within this proof. 
	The quantities~$[\sigma]_{\abs[0]{\sigma}}$ are appropriate norms, which we detail below, and~$n_\sigma$ denotes the number of leaves.
	
	In order to establish our claim, it therefore suffices to obtain the estimate
	\begin{equation}
		[\sigma]_{\abs[0]{\sigma}}^{\frac{1}{n_\sigma}} 
		\lesssim
		\barnorm{\bar{\Pi}} \quad \text{for all} \quad \sigma \in L.
		\label{prop:cdfi_model:pf_eq2}
	\end{equation}
	To this end, we recall the construction of the trees ~$\sigma \in L$ from~\cite[Sec.~2.2]{moinat_weber_20_phi43Loc}.
	First, consider a smooth noise~$\zeta$ which we may think of as~$\zeta = \xi_\kappa$.
	Since all estimates obtained therein  are stable in the limit~$\kappa \to 0$ and~since $\barnorm{\bar{\Pi}^{\xi_\kappa}} \to  \barnorm{\bar{\Pi}}$ (c.f. Proposition~\ref{prop:bphz_model_noise_intensity}), it suffices to prove the estimate~\eqref{prop:cdfi_model:pf_eq2} with both sides smoothed out at level~$\kappa>0$.
	Since we fixed $\eps =1$, we write~$\<1b>_\kappa$ for the solution to~\ref{eq:she_tree}, with~$\xi$ replaced by~$\xi_\kappa$ and then set
	\begin{equation}
		\<2b>_\kappa \coloneqq \<1b>_\kappa^2 - C_\kappa^{(1)}, \quad 
		\<3b>_\kappa \coloneqq \<1b>_\kappa^3 - 3 C_\kappa^{(1)}\<1b>_\kappa, \quad
		\<20b>_\kappa \coloneqq G * \<2b>_\kappa, \quad
		\<30b>_\kappa \coloneqq G * \<3b>_\kappa
		\label{prop:cdfi_model:trees_1}
	\end{equation}
	where~$G$ denotes the Green's function of the heat operator.
	Additionally, as in~\cite[Eq.~(2.16--2.19)]{moinat_weber_20_phi43Loc}, let~$\abs[0]{\<2>_x} := \abs[0]{\<2>} + 1$ and, with~$\varphi$ as in Remark~\ref{rmk:single_test_fct}, set
	\begin{equs}[][prop:cdfi_model:pf_eq3]
		\sbr[1]{\<2b>_{x,\kappa}}_{\abs[0]{\<2s>_x}} & \coloneqq 
		\sup_{z \in (0,1) \x \T^3} \sup_{\lambda < 1} \lambda^{-\abs[0]{\<2s>_x}} \abs[3]{\int (\bar{z} - z)\, \<2b>_\kappa(\bar{z}) \, \varphi_z^\lambda(\bar{z})\dif \bar{z}}, 
		\\
		\sbr[1]{\<22b>_\kappa}_{\abs[0]{\<22s>}} & \coloneqq 
		\sup_{z \in (0,1) \x \T^3} \sup_{\lambda < 1} \lambda^{-\abs[0]{\<22s>}} \abs[3]{\int \del[2]{\del[1]{\<20b>_\kappa(\bar{z}) - \<20b>_\kappa(z)}\<2b>_\kappa(\bar{z}) - C_\kappa^{(2)}}\varphi_z^\lambda(\bar{z})\dif \bar{z}}, 
		\\
		\sbr[1]{\<31b>_\kappa}_{\abs[0]{\<31s>}} & \coloneqq 
		\sup_{z \in (0,1) \x \T^3} \sup_{\lambda < 1} \lambda^{-\abs[0]{\<31s>}} \abs[3]{\int \del[2]{\del[1]{\<30b>_\kappa(\bar{z}) - \<30b>_\kappa(z)}\<1b>_\kappa(\bar{z})}\varphi_z^\lambda(\bar{z})\dif \bar{z}}  ,
		\\
		\sbr[1]{\<32b>_\kappa}_{\abs[0]{\<32s>}} 
		& \coloneqq \sup_{z \in (0,1) \x \T^3} \sup_{\lambda < 1} \lambda^{-\abs[0]{\<32s>}} \abs[3]{\int \del[2]{\del[1]{\<30b>_\kappa(\bar{z}) - \<30b>_\kappa(z)}\<2b>_\kappa(\bar{z}) - 3C_\kappa^{(2)} \<1b>_\kappa(\bar{z})}\varphi_z^\lambda(\bar{z})\dif \bar{z}}. 
	\end{equs}
	The norm~$[\sigma_\kappa]_{\abs[0]{\sigma}}$ for all the symbols~$\sigma \in L$ that do not appear in~\eqref{prop:cdfi_model:pf_eq3} is simply the corresponding (Hölder-) norm on~$\CC^{\abs[0]{\sigma}}((0,1) \x \T^3)$; the exception is~$\<1b>_\kappa$, in which case we use the slightly stronger norm
	\begin{equation*}
		[\<1b>_\kappa]_{\abs[0]{\<1s>}} := \sup_{t \in [0,1]}  \norm[0]{\<1b>_\kappa(t,\cdot)}_{\CC^\alpha(\T^3)}.
	\end{equation*}
	The remainder of the proof consists of associating each~$\sigma \in L$ with a corresponding tree\footnote{The action of any \emph{admissible} model, such as~$\bar{\Pi}$, on the symbols~$\<2> \textcolor{symbols}{X}, \<20>$ and~$\<30>$ is canonically defined from the action on~$\<2>$ resp.~$\<3>$; this follows from the admissibility conditions that are implicit in Definition~\ref{def:model}.
		Therefore, we have not added these symbols to the \emph{minimal} list~$\CW_-$.}~$\tau \in \CW_-$, and vice versa.
	
	Let us explain in detail how one obtains the estimate~$\sbr[0]{\sigma_\kappa}_{\abs[0]{\sigma}}  \lesssim \barnorm{\bar{\Pi}^{\xi_\kappa}}^{n_\sigma}$ for the most complicated tree~$\sigma = \<32b>_\kappa$.
	We graphically encode the decomposition~$G = K + R$ given in~\eqref{eq:kernel_decomposition} as 
	\begin{equation*}
		\<G> = \<K> + \<R>\,, \quad \<G>\, \coloneqq G\,, \quad  \<K> \coloneqq K\,, \quad \<R> \coloneqq R\,,
		\label{eq:kernel_decomposition_graphical}
	\end{equation*} 
	and accordingly have
	\begin{align*}
		\<1b>_\kappa(y) & = \<1k>_\kappa(y) + \<1r>_\kappa(y), \qquad
		\<2b>_\kappa(y) = \<2kk>_\kappa(y) + 2\<2kr>_\kappa(y) + \<2rr>_\kappa(y), \\[0.5em]
		\<30b>_\kappa(y) &
		= \<30k>_\kappa(y) + \<30r>_\kappa(y) \\[0.5em]
		& = 
		\del[1]{\<3kkk0k>_\kappa(y) + 3 \<3kkr0k>_\kappa(y) + 3 \<3krr0k>_\kappa(y) +  \<3rrr0k>_\kappa(y)} 
		+  
		\del[1]{\<3kkk0r>_\kappa(y) + 3 \<3kkr0r>_\kappa(y) + 3 \<3krr0r>_\kappa(y) +  \<3rrr0r>_\kappa(y)}. 
	\end{align*}
	For any of these terms we introduce the shorthand notation~$A_{z,\bar{z}} := A({\bar{z}}) - A(z)$
	and use the previous decomposition to express the trees within the integral defining~$[\<32b>_\kappa]_{\abs[0]{\<32s>}}$ as follows:\footnote{For notational purposes, we have also joined roots in the case of \emph{non-singular} products, e.g.~$\<2kr>_\kappa := \<1k>_\kappa \<1r>_\kappa$, as opposed to~\eqref{prop:cdfi_model:trees_1} where we need to renormalise the product. Similarly, $\<3kkr0k>_\kappa := G*\<3kkr>_\kappa = G * \del[0]{\<2kk>_\kappa \<1r>_\kappa} = G * \del[0]{\sbr[0]{(\<1k>_\kappa)^2 - C_\kappa^{(1)}}\<1r>_\kappa}$.} 
	\begin{equs}
		\thinspace & \,
		\del[1]{\<30b>_\kappa(\bar{z}) - \<30b>_\kappa(z)}\<2b>_\kappa(\bar{z}) - 3C_\kappa^{(2)} \<1b>_\kappa(\bar{z}) \\[0.5em]
		= \ & \,
		\del[2]{\del[1]{\<3kkk0k>_\kappa + 3 \<3kkr0k>_\kappa + 3 \<3krr0k>_\kappa +  \<3rrr0k>_\kappa} 
			+ 
			\del[1]{\<3kkk0r>_\kappa + 3 \<3kkr0r>_\kappa + 3 \<3krr0r>_\kappa +  \<3rrr0r>_\kappa}}_{z,\bar{z}}
		\del[1]{\<2kk>_\kappa + 2\<2kr>_\kappa + \<2rr>_\kappa}(\bar{z}) \\
		& - 3 C_\kappa^{(2)} \del[1]{\<1k>_\kappa + \<1r>_\kappa}(\bar{z}) \\[0.5em]
		= \ & \ 
		\del[2]{\del[1]{\<3kkk0k>_\kappa}_{z,\bar{z}} \<2kk>_\kappa(\bar{z}) - 3 C_\kappa^{(2)} \<1k>_\kappa(\bar{z})}
		+
		\del[2]{3 \del[1]{\<3kkr0k>_\kappa}_{z,\bar{z}} \<2kk>_\kappa(\bar{z}) - 3 C_\kappa^{(2)} \<1r>_\kappa(\bar{z})}
		+
		(\ldots)
	\end{equs}
	Note that we have suitably paired the two summands above with the corresponding counterterms in order to relate them to~$\bar{\Pi}^{\xi_\kappa}_z \<32>(\bar{z})$ and~$\bar{\Pi}^{\xi_\kappa}_z \<22>(\bar{z})$ from Definition~\ref{def:bphz_lift}, respectively.
	In total, making use of the simple estimates, 
	\begin{equation*}
		\norm[0]{R * \bar{\Pi}^{\xi_\kappa} \tau}_{L^\infty} \lesssim 
		\threebars \bar{\Pi}^{\xi_\kappa} \tau \threebars, \quad
		\threebars \bar{\Pi}^{\xi_\kappa} \<30> \threebars
		\lesssim
		\threebars \bar{\Pi}^{\xi_\kappa} \<3> \threebars
	\end{equation*}
	we find 
	\begin{equs}[][prop:cdfi_model:pf_eq4]
		\thinspace &
		\sbr[1]{\<32b>_\kappa}_{\abs[0]{\<32s>}} \\
		& \lesssim
		\threebars \bar{\Pi}^{\xi_\kappa} \<32> \threebars + \threebars \bar{\Pi}^{\xi_\kappa} \<22> \threebars \, \norm{\xi_\kappa}_{\CC^{\abs[0]{\<wns>}}}  
		+
		\threebars \bar{\Pi}^{\xi_\kappa} \<3> \threebars \del[1]{\, \norm{\xi_\kappa} \threebars \bar{\Pi}^{\xi_\kappa} \<1> \threebars + \norm{\xi_\kappa}_{\CC^{\abs[0]{\<wns>}}}^2} \\
		& \quad +  
		\del[2]{\threebars \bar{\Pi}^{\xi_\kappa} \<3> \threebars + \norm{\xi_\kappa} \threebars \bar{\Pi}^{\xi_\kappa} \<2> \threebars  + \norm{\xi_\kappa}^2 \threebars \bar{\Pi}^{\xi_\kappa} \<1> \threebars + \norm{\xi_\kappa}^3} 
		\del[2]{\, \threebars \bar{\Pi}^{\xi_\kappa} \<2> \threebars + \norm{\xi_\kappa} \threebars \bar{\Pi}^{\xi_\kappa} \<1> \threebars + \norm{\xi_\kappa}^2} \qquad
	\end{equs}	
	where we wrote~$\norm{\xi_\kappa} = \norm{\xi_\kappa}_{\CC^{\abs[0]{\<wns>}}} = \threebars \bar{\Pi}^{\xi_\kappa} \<wn> \threebars$ to lighten the notation.
	Since the total number of leaves in each product on the RHS of~\eqref{prop:cdfi_model:pf_eq4} equals~$n_{\<32s>} = 5$, we find that 
	\begin{equation} \label{prop:cdfi_model:pf_eq5}
		\sbr[1]{\<32b>_\kappa}_{\abs[0]{\<32s>}} 
		\lesssim
		\barnorm{\bar{\Pi}^{\xi_\kappa}}^5.
	\end{equation}
	This is exactly the claimed estimate~\eqref{prop:cdfi_model:pf_eq2} in the case~$\sigma_\kappa = \<32b>_\kappa$. 
	
	Using the same simple combinatorial argument, one can show that the following formula holds for the other trees in~$L$:
	
	\vspace{0.5em}
	\noindent
	\textit{%
		For each~$\sigma \in L\setminus\{\<32b>, \<2b>_x\}$, write\footnote{The operator~$\CI$ denotes the \enquote{abstract} integration operator associated to~$K$ within the theory of regularity structures, see~\cite[Def.~5.7]{hairer_rs}; however, we use the convention that~$\CI(\<wn>)^0 := \1$ and~$\CI \1 := \1$. Graphically,~$\CI$ is encoded by the blue straight line~$\<I>$.}~$\sigma = \CI\del[1]{\CI(\<wn>)^k} \CI(\<wn>)^m$ with~$k \in \{0,2\}$ and $m \in \{0,1,2,3\}$ such that~$n_\sigma = k + m$.
		We then have the estimate
		\begin{equs}[][prop:cdfi_model:pf_eq4]
			\thinspace
			\sbr[0]{\sigma_\kappa}_{\abs[0]{\sigma}}
			& 
			\lesssim
			\threebars \bar{\Pi}^{\xi_\kappa} \sigma \threebars \mathbf{1}_{k > 0}
			+
			\threebars \bar{\Pi}^{\xi_\kappa} \CI(\<wn>)^k \threebars \thinspace \del[3]{\sum_{\ell=1}^m \norm[0]{\xi_\kappa}^\ell \thinspace \threebars \bar{\Pi}^{\xi_\kappa} \CI(\<wn>)^{m-\ell}\threebars} \\
			& \quad +
			\del[3]{\sum_{r=0}^k \norm[0]{\xi_\kappa}^r \thinspace \threebars \bar{\Pi}^{\xi_\kappa} \CI(\<wn>)^{m-r} \threebars}
			\del[3]{\sum_{\ell=0}^m \norm[0]{\xi_\kappa}^\ell \thinspace \threebars\bar{\Pi}^{\xi_\kappa} \CI(\<wn>)^{m-\ell}\threebars} \\
		\end{equs} 
		where the norms~$\threebars\bar{\Pi}^{\xi_\kappa} \tau\threebars$ for~$\tau \in \CW_-$ were introduced in Definition~\ref{def:model}.
		For~$\sigma = \<2b>_x$, the estimate is of the same form as for~$\sigma = \<2b>$, only with a different implicit constant: The reason is that the distance~$\abs[0]{z - \bar{z}}$ is uniformly bounded for any~$z, \bar{z} \in (0,1) \x \T^3$. 
	}
	
	\vspace{0.5em}
	\noindent
	Using the bound~\eqref{prop:cdfi_model:pf_eq4}, the same argument as for~\eqref{prop:cdfi_model:pf_eq5} then implies the estimate
	\begin{equation*}
		\sbr[0]{\sigma_\kappa}_{\abs[0]{\sigma}}
		\lesssim
		\barnorm{\bar{\Pi}^{\xi_\kappa}}^{n_\sigma} \mathbf{1}_{k > 0} + \barnorm{\bar{\Pi}^{\xi_\kappa}}^k \barnorm{\bar{\Pi}^{\xi_\kappa}}^m
		+
		\barnorm{\bar{\Pi}^{\xi_\kappa}}^k \barnorm{\bar{\Pi}^{\xi_\kappa}}^m
		\simeq 
		\barnorm{\bar{\Pi}^{\xi_\kappa}}^{n_\sigma}
	\end{equation*}
	where we have used that~$k + m = n_\sigma$. 
	As argued before, sending~$\kappa \to 0$ concludes the proof.
\end{proof}
Since the estimate in \eqref{prop:cdfi_model_estimate} is independent of the initial condition, we immediately have the following corollary.

\begin{corollary}[Global existence] \label{coro:global_existence}
	For any~$\fz \in \CC^\alpha(\T^3)$ and~$\eps \geq 0$, the process~$\bar{u}^\fz_\eps$ defined in~\eqref{eq:phi43_rigorous_SPDE} almost surely is an element of~$C_T\CC^\alpha(\T^3)$ for any~$T > 0$.
\end{corollary}

\subsection[Uniform LDP for the~$\Phi^4_3$ Equation]{Uniform LDP for the~$\boldsymbol{\Phi^4_3}$ Equation} \label{sec:uniform_ldp}

As recalled in the introduction, an LDP for the~$\Phi^4_3$ dynamics~$(\bar{u}_\eps^\fz)_{\eps > 0}$ was obtained in~\cite[Thm.~$4.4$]{hairer_weber_ldp} for~\emph{fixed} initial condition~$\fz$.
In this subsection, we prove a slightly stronger version of their result which is~\emph{locally uniform} in~$\fz$.
However, some care is needed in stating the uniform LDP since one finds various, inequivalent concepts in the literature. We refer to the survey by Salins's ~\cite{salins_uniform_ldp} and Remark~\ref{rmk:locally_uniform_ldp} below.
Recall:
\begin{itemize}
	\item $\bar{B}_R(0)$ denotes a closed ball in~$\CC^\alpha(\T^3)$ with radius~$R > 0$, centred at the origin, 
	\item the dynamics~$\bar{u}_\eps^\fz$ introduced in Definition~\ref{def:phi43_rigorous_SPDE},
	\item with the notation of Lemma~\ref{lem:can_lift_cameron_martin_fcts}, it holds that~$w^{h,\fz} = \Phi(\fz, \Pi^h)$,
	\item $\threebars \cdot \threebars$ and $\barnorm{\cdot}$ induce the same topology on~$\MMm$, see Lemma~\ref{lem:conv_models}.
\end{itemize}

The following proposition is the main result of this subsection.
\begin{proposition}[Locally uniform FWLDP for the dynamics] \label{prop:uniform_ldp}
	For each~$R, T > 0$, the family
	\begin{equation*}
		\{\bar{u}_\eps^\fz: \ \eps > 0, \ \fz \in \bar{B}_R(0)\} \subseteq C_T \CC^{\alpha}(\T^3)
	\end{equation*}
	satisfies a \emph{uniform FWLDP} with good rate function\footnote{If we write~$\II_T(v)$ for~$v \in C_T \CC^\alpha(\T^3)$ without specifying the initial condition as a superscript, by convention, we always mean~$\II_T^{v(0)}(v)$.}
	\begin{equs}[][eq:uniform_ldp_rf]
		\II_{T}^\fz(v)
		& \coloneqq \inf\left\{ \frac{1}{2} \norm[0]{h}_{L^2_TL^2_x}^2: \ h \in L^2_TL^2_x, \ \Phi(\fz, \Pi^h) = v \right\}\\
		& = \frac{1}{2}\int_{0}^{T} \int_{\T^3} \del[1]{\partial_t v - \Delta v - m^2 v + v^3}^2 \dif x \dif t
	\end{equs}
	with the understanding in the second line that~$\II^\fz_{T}(v) \coloneqq  +\infty$ if
	\begin{equation*}
		\partial_t v - \Delta v - m^2 v + v^3 \notin L^2_TL^2(\mbT^3) \quad \text{or} \quad
		v(0) \neq \fz.
	\end{equation*}
	That is:
	\begin{enumerate}[label=(\roman*)]
		\item \label{prop:uniform_ldp_1} For any~$\fz \in \mcC^{\alpha}(\T^3)$ and~$\theta > 0$, the sublevel set~$\II_T^\fz[\theta] \equiv \{\II_T^\fz \leq \theta\}$ is compact.
		\item \label{prop:uniform_ldp_2} For any~$\gamma,\, \delta, \theta > 0$ there exists some~$\eps_0^{\downarrow} \coloneqq  \eps_0^{\downarrow}(T,\delta, \gamma,\theta) > 0$ such that for all~$\fz \in \bar{B}_R(0)$, all $v \in \II^\fz_T[\theta]$, and all~$\eps \leq \eps_0^{\downarrow}$ it holds that
		\begin{equation}
			\P\del[1]{\thinspace\norm{\bar{u}^\fz_\eps - v}_{C_T \CC^{\alpha}_x} < \delta} \geq \exp\del[2]{-\frac{\II^\fz_T(v) + \gamma}{\eps^2}}. 
			\label{prop:uniform_ldp_lb_eq}
		\end{equation}
		\item \label{prop:uniform_ldp_3} For any~$\gamma,\, \delta, \theta > 0$ there exists some~$\eps_0^{\uparrow} \coloneqq  \eps_0^{\uparrow}(T,\delta, \gamma,\theta) > 0$ such that for all~$\fz \in \bar{B}_R(0)$ and all~$\eps \leq \eps^{\uparrow}_0$ it holds that
		\begin{equation}
			\P\del[1]{\thinspace\operatorname{dist}_{C_T \CC^{\alpha}_x}\del[0]{\bar{u}^\fz_\eps, \II_T^\fz[\theta]} \geq \delta} \leq \exp\del[2]{-\frac{\theta - \gamma}{\eps^2}}. 
			\label{prop:uniform_ldp_ub_eq}
		\end{equation}
	\end{enumerate}
	For convenience, we set~$\eps_0 \coloneqq  \eps_0^{\downarrow} \wedge \eps_0^{\uparrow}$.
\end{proposition}

\begin{remark}[Variants of LDPs] \label{rmk:locally_uniform_ldp}
	We briefly comment on the most commonly employed variants of the uniform and non-uniform LDP as well as relations between them.
	\begin{enumerate}[label=(\roman*)]
		\item \label{rmk:locally_uniform_ldp:i} \textbf{Non-uniform LDPs}:
		There are two ways to phrase an LDP: One formulation due to Varadhan (which we use to present our main result, Theorem~\ref{th:main_result}) and another due to Freidlin and Wentzell.
		It is proved in~\cite[Chap.~$3$, Thm.~$3.3$]{freidlin-wentzell} that the statements concerning \emph{lower bounds} are equivalent in both formulations---but the equivalence of the respective \emph{upper bound} statements requires compact sublevel sets of the rate function.
		\item \label{rmk:locally_uniform_ldp:ii}  \textbf{Locally uniform LDPs}:
		Giving credit to Freidlin's and Wentzell's definition at the end of~\cite[Ch.~3, Sec.~3]{freidlin-wentzell}, Salins~\cite[Def.~2.1]{salins_uniform_ldp} calls the uniform LDP in the previous proposition a \enquote{FWULDP}---in contrast to a \enquote{DZULDP} à la Dembo and Zeitouni, see~\cite[Def.~2.2]{salins_uniform_ldp}
		---and there are counterexamples which show that a \enquote{FWULDP} does not generally imply a \enquote{DZULDP}, see~\cite[Prop.~3.4]{salins_uniform_ldp}.

		The \enquote{FWULDP} is sufficient for most but unfortunately not all proofs---the exception is given by Proposition~\ref{prop:bound_prob_H} in which we require the \enquote{DZULDP} upper bound.
		Luckily, however, we are able to show that, in our situation, the~\enquote{FWULDP} upper bound \emph{does imply} the required \enquote{DZULDP} upper bound, see Corollary~\ref{cor:DZULDP} below.
	\end{enumerate}
\end{remark}

\begin{remark}
	The LDP of~\cite{hairer_weber_ldp} actually allows for the possibility that the process~$\bar{u}^\fz_\eps$ blows up in finite time.
	However, as recapitulated in Appendix~\ref{app:sol_theory}, global existence of the renormalised~$\Phi^4_3$ dynamics has been established by now~(see Corollary~\ref{coro:global_existence} above), so the conditioning on times~$T > 0$ which are smaller than the explosion time of~$\bar{u}^\fz_\eps$ becomes superfluous;
	the previous proposition accounts for that.
\end{remark}

As in the proof of the non-uniform LDP given by \cite{hairer_weber_ldp}, our proof of Proposition~\ref{prop:uniform_ldp} will be based on the \emph{joint} local Lipschitz continuity of the solution map~$\Phi$, see Lemma~\ref{lem_uniform_lipschitz}.
The other main ingredient is the following LDP on minimal model space, see~\cite[Thm.~$4.3$]{hairer_weber_ldp}, which we present in the Freidlin--Wentzell formulation, cf.~Remark~\ref{rmk:locally_uniform_ldp}~\ref{rmk:locally_uniform_ldp:i}. 

\begin{proposition} \label{lem_ldp_models}
	For each~$T > 0$, the family~$\{\bar{\Pi}^\eps: \ \eps > 0\}$ satisfies an LDP on~$\MMm_T$ with good rate function 
	\begin{equation}
		\JJ_T(\Pi) \coloneqq  \inf\left\{ \frac{1}{2} \norm[0]{h}_{L^2_TL^2_x}^2: \ h \in L^2_TL^2_x, \ \Pi^h = \Pi \right\}.
		\label{lem_ldp_models_rf}
	\end{equation}
	That is:
	\begin{enumerate}[label=(\roman*)]
		\item For any~$\theta > 0$, the sublevel set~$\JJ_T[\theta] \equiv \{\Pi \in \MMm: \, \JJ_T(\Pi) \leq \theta\}$ is compact.
		\item For any~$\gamma,\, \delta, \theta > 0$ and~$\Pi \in \MMm_T$, there exists some~$\eps_0^{\CM,\downarrow} = \eps_0^{\CM,\downarrow}(T,\delta, \gamma,\theta) > 0$ such that for all~$\eps \leq \eps_0^{\CM,\downarrow}$ we have
		\begin{equation}
			\P\del[1]{\thinspace d_{\MMm}\del[1]{\bar{\Pi}^\eps, \Pi} < \delta} \geq \exp\del[2]{-\frac{\JJ_T(\Pi) + \gamma}{\eps^2}}. 
			\label{lem_ldp_models_lb_eq}
		\end{equation}
		\item For any~$\gamma,\, \delta, \theta > 0$ there exists some~$\eps_0^{\CM,\uparrow} = \eps_0^{\CM,\uparrow}(T,\delta, \gamma,\theta) > 0$ such that for all~$\eps \leq \eps_0^{\CM,\uparrow}$ we have
		\begin{equation}
			\P\del[1]{\thinspace\operatorname{dist}_{\MMm}\del[0]{\bar{\Pi}^\eps, \JJ_T[\theta]} \geq \delta} \leq \exp\del[2]{-\frac{\theta - \gamma}{\eps^2}}. 
			\label{lem_ldp_models_ub_eq}
		\end{equation}
	\end{enumerate}
	For convenience, we set~$\eps_0^{\CM} \coloneqq  \eps_0^{\CM,\downarrow} \wedge \eps_0^{\CM,\uparrow}$.
\end{proposition}
We will also need the following elementary observation, a consequence of sub-additivity of~probability measures.
\begin{lemma} \label{lem_bound_prob_sum}
	For any~$\delta \geq 0$ and~$\delta_1, \delta_2 \geq 0$ such that~$\delta_1 + \delta_2 = \delta$, we have
	\begin{equation*}
		\P(X_1 + X_2 \geq \delta)
		\leq
		\P(X_1 \geq \delta_1) + \P(X_2 \geq \delta_2).
		\label{lem_bound_prob_sum_eq}
	\end{equation*} 
\end{lemma}

Finally, we have gathered all the ingredients to give a proof of~Proposition~\ref{prop:uniform_ldp}.

\begin{proof}[Proof of Proposition~\ref{prop:uniform_ldp}] \label{prop:uniform_ldp_pf}
	We proceed in three steps that mirror the three claims.
	\vspace{-0.8em}
	\paragraph*{$\boldsymbol{\triangleright}$ \textbf{Proof of~\ref{prop:uniform_ldp_1}}.}
	The statement in~\ref{prop:uniform_ldp_1} deals with the case of \emph{fixed}~$\fz \in \mcC^{\alpha}(\T^3)$ and was proved by Hairer and Weber~\cite[Thm.~4.4]{hairer_weber_ldp}.
	\vspace{-0.8em}
	\paragraph*{$\boldsymbol{\triangleright}$ \textbf{Proof of~\ref{prop:uniform_ldp_2}}.}
	Observe that~$\II_T^\fz(v) < \infty$ implies that we can find a \emph{unique} (and \emph{deterministic}) $h \coloneqq h^v \in L^2_T L^2(\T^3)$ given by
	\begin{equation*}
		h = \partial_t v - \Delta v + v^3 + m^2 v
	\end{equation*}
	such that 
	\begin{equation}
		v 
		= w^{h,\fz} 
		\equiv \Phim(\fz,\Pi^h), 
		\quad \II_T^\fz(v) = \frac{1}{2} \norm[0]{h}_{L^2_T L^2_x}^2.
		\label{prop:uniform_ldp_pf_eq1}
	\end{equation}
	This follows immeditely from the definition of the rate function~$\II_T^\fz$ given in~\eqref{eq:uniform_ldp_rf} above. 
	We then choose~$\d_1,\d_2 > 0$ such that~$\delta = \delta_1 + \delta_2$ and apply Lemma~\ref{lem_bound_prob_sum} with 
	\begin{equation}
		X_1 \coloneqq  \norm[0]{\bar{u}^\fz_\eps - v}_{C_T\CC^\alpha_x} \mathbf{1}_{\eps \sbarnorm{\bbzm} < M_1}, \quad
		X_2 \coloneqq  \norm[0]{\bar{u}^\fz_\eps - v}_{C_T\CC^\alpha_x} \mathbf{1}_{\eps \sbarnorm{\bbzm} \geq M_1},
		\label{prop:uniform_ldp_pf_aux1}
	\end{equation} 
	where~$M_1 > 0$ is arbitrary for the moment; it will be chosen appropriately in~\eqref{prop:uniform_ldp_pf_lb_M} below.
	This leads to the following estimate:
	\begin{equs}[][prop:uniform_ldp_pf_lb_elem]
		\thinspace 
		&
		\P\del[1]{\norm[0]{\bar{u}_\eps^\fz - v}_{C_T\CC^\alpha_x} < \delta} \\
		= & \,
		1 - \P\del[1]{\norm[0]{\bar{u}_\eps^\fz - v}_{C_T\CC^\alpha_x} \geq \delta}  \\
		\geq  & \, 
		1 
		- 
		\P\del[1]{\norm[0]{\bar{u}_\eps^\fz - v}_{C_T\CC^\alpha_x} \1_{\eps \sbarnorm{\bbzm} \leq M_1} \geq \delta_1}
		-
		\P\del[1]{\norm[0]{\bar{u}_\eps^\fz - v}_{C_T\CC^\alpha_x} \1_{\eps \sbarnorm{\bbzm} > M_1} \geq \delta_2} \\
		=: & \,
		1 - P_1 - P_2.
	\end{equs}
	We start by estimating~$P_2$. Since~$\d_2 > 0$, we have the elementary inclusion
	\begin{equation*}
		\left\{\norm[0]{\bar{u}_\eps^\fz - v}_{C_T\CC^\alpha_x} \1_{\eps \sbarnorm{\bbzm} > M_1} \geq \delta_2\right\}
		\subseteq
		\{\eps \barnorm{\bbzm} > M_1\}
	\end{equation*}
	which, by sub-additivity of~$\P$ and Markov's inequality, leads to the estimate
	\begin{equation}
		P_2
		\equiv
		\P\del[1]{\norm[0]{\bar{u}_\eps^\fz - v}_{C_T\CC^\alpha_x} \1_{\eps \sbarnorm{\bbzm} > M_1} \geq \delta_2}
		\leq 
		\P\del[2]{\barnorm{\bbzm} > \frac{M_1}{\eps}}
		\leq
		K_T \exp\del[2]{-\frac{\Lambda_T M_1^2}{\eps^2}}
		\label{prop:uniform_ldp_pf_lb_P2}
	\end{equation}
	for any~$\eps > 0$ with~$K_T \in (0,\infty)$ and~$\Lambda_T > 0$ as in Lemma~\ref{lem:fernique_bphz}.
	For~$P_1$, we wish to apply Lemma~\ref{lem_uniform_lipschitz}, the joint local Lipschitz continuity of the solution map~$\Phim$.
	To this end, note that by Lemma~\ref{lem:homog_model_norm}, one has
	\begin{equation}
		\bar{u}^\fz_\eps = \Phim(\fz,\bbzmeps), \quad \barnorm{\bbzmeps} \, \1_{\eps \sbarnorm{\bbzm} \leq M_1} = \eps \barnorm{\bbzm} \, \1_{\eps \sbarnorm{\bbzm} \leq M_1} \leq M_1, 
		\label{prop:uniform_ldp_pf_lb_ueps}
	\end{equation}
	and we can choose~$M_1 > \barnorm{\Pi^{h^v}}$ since~$h^v$ is deterministic.
	By Lemma~\ref{lem_uniform_lipschitz} and Lemma~\ref{lem:conv_models}, there exists some~$\rho_1 \coloneqq \rho_1(\delta_1,M_1,T) > 0$ such that 
	\begin{align*} 
		P_1 
		& \equiv
		\P\del[1]{\norm[0]{\Phim(\fz,\bbzmeps) - \Phim(\fz,\Pi^{h^{v}})}_{C_T\CC^\alpha_x}  \1_{\sbarnorm{\bbzm^\eps} \leq M_1} \geq \delta_1} 
		\leq 
		\P\del[1]{d_{\MMm}(\bbzmeps,\Pi^{h^{v}}) \1_{\sbarnorm{\bbzmeps} \leq M_1} \geq \rho_1}. 
	\end{align*} 
	and the LDP lower bound~\eqref{lem_ldp_models_lb_eq} gives the estimate
	\begin{equation*}
		1 - P_1 
		\geq \P\del[1]{d_{\MMm}(\bbzmeps,\Pi^{h^{v}}) < \rho_1}
		\geq 
		\exp\del[2]{-\frac{\JJ_T(\Pi^{h^{v}}) + \nicefrac{\gamma}{2}}{\eps^2}}
	\end{equation*}
	for all~$\eps < \eps^{\CM,\downarrow}_0(T, \rho_1, \nicefrac{\gamma}{2},\theta)$. 
	Since
	\begin{equation*}
		\JJ_T(\Pi^{h^{v}})
		=
		\frac{1}{2} \norm[0]{h^v}_{L^2_T L^2_x}^2
		=
		\II_T^{\fz}(v),
	\end{equation*}
	the previous bound is equivalent to
	\begin{equation}
		1 - P_1 
		\geq 
		\exp\del[2]{-\frac{\II_T^\fz(v) + \nicefrac{\gamma}{2}}{\eps^2}}.
		\label{prop:uniform_ldp_pf_lb_P1}
	\end{equation}
	In summary, our estimates in~\eqref{prop:uniform_ldp_pf_lb_elem},~\eqref{prop:uniform_ldp_pf_lb_P2}, and~\eqref{prop:uniform_ldp_pf_lb_P1} lead to the inequality
	\begin{align*}
		\P\del[1]{\norm[0]{\bar{u}_\eps^\fz - v}_{C_T\CC^\alpha_x} < \delta}
		& \geq
		\exp\del[2]{-\frac{\II_T^\fz(v) + \nicefrac{\gamma}{2}}{\eps^2}}
		-
		K_T \exp\del[2]{-\frac{\Lambda_T M_1^2}{\eps^2}} \\
		& = 
		\exp\del[2]{-\frac{\II_T^\fz(v) + \nicefrac{\gamma}{2}}{\eps^2}}
		\del[3]{
			1
			-
			K_T \exp\del[2]{-\frac{\Lambda_T M_1^2 - \II_T^\fz(v) - \nicefrac{\gamma}{2}}{\eps^2}}
		}
	\end{align*}
	for all~$\eps < \eps^{\CM,\downarrow}_0(T, \rho_1, \nicefrac{\gamma}{2},\theta)$. 
	Finally, we observe that~$v \in \II_T^\fz[\theta]$ implies
	\begin{equation*}
		1
		-
		K_T \exp\del[2]{-\frac{\Lambda_T M_1^2 - \II_T^\fz(v) - \nicefrac{\gamma}{2}}{\eps^2}}
		\geq
		1
		-
		K_T \exp\del[2]{-\frac{\Lambda_T M_1^2 - \theta - \nicefrac{\gamma}{2}}{\eps^2}}
		=:
		E(\gamma, \theta, M_1, \eps)
	\end{equation*}
	and by choosing~$M_1 = M_1(h^v,\gamma,\theta)$ such that
	\begin{equation}
		M_1 > \barnorm{\Pi^{h^v}} \vee \sqrt{\frac{\theta + \nicefrac{\gamma}{2}}{\lambda}} 
		\label{prop:uniform_ldp_pf_lb_M}
	\end{equation}
	we can guarantee that~$E(\gamma, \theta, M_1, \eps)$ is monotonously increasing in~$\eps$.
	As a consequence, we can find~$\eps_0^\downarrow \leq \eps^{\CM,\downarrow}_0(T, \rho_1, \nicefrac{\gamma}{2},\theta)$ such that, for all~$\eps \leq \eps_0^\downarrow$, we have
	\begin{equation*}
		E(\gamma, \theta, M_1, \eps) \geq \exp\del[2]{-\frac{\nicefrac{\gamma}{2}}{\eps^2}}.
	\end{equation*}
	Altogether, this implies that the estimate
	\begin{equation*}
		\P\del[1]{\norm[0]{\bar{u}_\eps^\fz - v}_{C_T\CC^\alpha_x} < \delta}
		\geq 	
		\exp\del[2]{-\frac{\II_T^\fz(v) + \gamma}{\eps^2}}
	\end{equation*}
	holds for all~$\eps \leq \eps_0^\downarrow$ and thereby completes the proof of the lower bound.
	\vspace{-0.8em}
	\paragraph*{$\boldsymbol{\triangleright}$ \textbf{Proof of~\ref{prop:uniform_ldp_3}}.}
	Since the singleton~$\{\bbzmeps\}$ is closed and~$\JJ_T[\theta]$ is compact, we can find a~$\bzm \in \JJ_T[\theta]$ such that 
	\begin{equation*}
		\dist_{\MMm}(\bbzmeps,\JJ_T[\theta]) = d_{\MMm}(\bbzmeps,\bzm).
		\label{prop:uniform_ldp_pf_ub_eq1}
	\end{equation*}
	Note that~$\bzm$ is at minimal distance from the \emph{random} element~$\bbzmeps$ relative to the other elements in~$\JJ_T[\theta]$, i.e.~$\bzm$ itself is \emph{random}, too.
	From the proof of~\cite[Thm.~$4.4$]{hairer_weber_ldp}, we know that
	\begin{equation*}
		\{\Pi^h: \ h \in L^2_T L^2(\T^3)\} = \{\JJ_T < \infty\}.
	\end{equation*}
	Our choice of~$\bzm$, i.e. the fact that~$\JJ_T(\bzm) \leq \theta$, thus implies the existence of~$h^{\bzm} \in L^2_TL^2(\T^3)$ such that
	\begin{equation} \label{prop:uniform_ldp_pf_consistency_Pi}
		\bzm = \Pi^{h^{\bzm}}, \quad \frac{1}{2}\norm[0]{h^{\bzm}}_\CH^2 \equiv \JJ_T(\bzm) \leq \theta
	\end{equation}
	Again, note that this control~$h^{\bzm}$ is \emph{random} as well due to its dependence on the random model~$\bzm$.
	However, for each fixed \emph{realisation} of~$\bzm$, it is \emph{unique}:
	If we suppose that there were~$h_1,h_2 \in \CH$ such that~$\Pi^{h_1} = \bzm = \Pi^{h_2}$, then the definition of the canonical lift~$\Pi^{\bullet}$ implies that
	\begin{equation*}
		h_1 = \Pi^{h_1} \<xi> = \Pi^{h_2} \<xi> = h_2.
	\end{equation*}
	We then set~$v^{\bzm} \coloneqq  \Phim(\fz,\Pi) \equiv w^{h^{\Pi}, \fz}$ which, by the form of the rate function~$\II_T^\fz$ in~\eqref{eq:uniform_ldp_rf} above, implies that
	\begin{equation} \label{prop:uniform_ldp_pf_lb_energy_bound}
		\II_T^\fz\del[0]{v^{\bzm}} = \frac{1}{2} \norm[0]{h^{\bzm}}_{L^2_T L^2_x}^2 \leq \theta, 
	\end{equation}
	i.e.~$v^{\bzm} \in \II_T^\fz[\theta]$. 
	At the same time, we know that the map
	\begin{equation*}
		\Pi^{\bullet}: L^2_T L^2_x \to \MMm_T, \quad h \mapsto \Pi^h
	\end{equation*}
	is locally Lipschitz continuous w.r.t.~$\threebars \cdot \threebars_T$, see Lemma~\ref{lem:can_lift_cameron_martin_fcts} above; we denote the corresponding Lipschitz constant, which is implicit in~\eqref{lem:can_lift_cameron_martin_fcts:lipschitz}, by~$c_T$.
	Then, Lemma~\ref{lem:can_lift_cameron_martin_fcts} combined with~\eqref{prop:uniform_ldp_pf_consistency_Pi} and~\eqref{prop:uniform_ldp_pf_lb_energy_bound} implies that the bound~$\threebars \Pi \threebars_T \leq c_T \, \theta$  holds~uniformly over all realisations of~$\bzm$; by definition of~$\barnorm{\cdot}_T$, this implies the bound
	\begin{equation}
		\barnorm{\Pi}_T \leq A_T
		\label{prop:uniform_ldp_pf_ub_eq2}
	\end{equation}
	where~$A_T$ is some function of~$c_T\theta$.
	As in the proof of the lower bound, we choose~$\d_1,\d_2 > 0$ with~$\delta = \delta_1 + \delta_2$ and apply Lemma~\ref{lem_bound_prob_sum} with~$X_1$ and~$X_2$ as in~\eqref{prop:uniform_ldp_pf_aux1}, only with~$M_1$ replaced by~$M_2$; the latter is arbitrary for now and  will be chosen later.
	We obtain the estimate
	\begin{equs}
		\thinspace & 
		\P\del[1]{\dist_{C_T\CC^\alpha_x}(\bar{u}_\eps^\fz, \II_T^{\fz}[\theta]) \geq \delta} \\
		\leq & \
		\P\del[1]{\norm[0]{\bar{u}_\eps^\fz - v^{\bzm}}_{C_T\CC^\alpha_x} \geq \delta} \\
		\leq & \
		\P\del[1]{\norm[0]{\bar{u}_\eps^\fz - v^{\bzm}}_{C_T\CC^\alpha_x} \1_{\eps \sbarnorm{\bbzm} \leq M_2} \geq \delta_1}
		+
		\P\del[1]{\norm[0]{\bar{u}_\eps^\fz - v^{\bzm}}_{C_T\CC^\alpha_x} \1_{\eps \sbarnorm{\bbzm} > M_2} \geq \delta_2} \\
		=: & \  
		Q_1 + Q_2.
	\end{equs}
	The probability~$Q_2$ can be estimated exactly as~$P_2$ in~\eqref{prop:uniform_ldp_pf_lb_P2} above, which gives the bound
	\begin{equation*}
		Q_2
		\leq
		K \exp\del[2]{-\frac{\lambda M_2^2}{\eps^2}}
		\label{prop:uniform_ldp_pf_ub_Q2}.
	\end{equation*}
	Thanks to~\ref{prop:uniform_ldp_pf_ub_eq2} and, as before,~\eqref{prop:uniform_ldp_pf_lb_ueps}, we can apply Lemma~\ref{lem_uniform_lipschitz} to estimate~$Q_1$ similarly to~$P_1$ if we choose~$M_2 > A_T$.
	Then, by Lemma~\ref{lem_uniform_lipschitz}, there exists~$\beta_1 \coloneqq  \beta_1(\delta_1, M_2,T) > 0$ such that
	\begin{align*} 
		Q_1 
		& \leq 
		\P\del[1]{d_{\MMm}(\bbzmeps,\Pi) \1_{\sbarnorm{\bbzmeps} \leq M_2} \geq \beta_1} \\
		& \leq 
		\P\del[1]{d_{\MMm}(\bbzmeps,\Pi) \geq \beta_1}
		= 
		\P\del[1]{\dist_{\MMm}(\bbzmeps, \JJ_T[\theta]) \geq \beta_1}
		\leq
		\exp\del[2]{-\frac{\theta - \nicefrac{\gamma}{2}}{\eps^2}}
	\end{align*} 
	for all~$\eps \leq \eps_0^{\CM,\uparrow}(T, \beta_1, \nicefrac{\gamma}{2}, \theta)$.
	Note that we have additionally used the LDP upper bound~\eqref{prop:uniform_ldp_ub_eq}.
	
	One may now proceed similarly as in the proof of the lower bound, arrive at a similar constraint for~$M_2$ as for~$M_1$ in~\eqref{prop:uniform_ldp_pf_lb_M}, and finally choose an appropriate~$\eps_0^{\uparrow} \leq \eps_0^{\CM,\uparrow}(T,\beta_1,\nicefrac{\gamma}{2},\theta)$ such that the claim in~\eqref{prop:uniform_ldp_ub_eq} follows.
	We refrain from giving all the details.
\end{proof}
As discussed in Remark~\ref{rmk:locally_uniform_ldp} above, we require both forms of uniform LDPs, at least w.r.t. the upper bound. 
The following lemma establishes a key property on the sub-level sets of the dynamic rate function which will enable us to show that the FWULDP upper bound in Proposition~\ref{prop:uniform_ldp} implies a DZULDP upper bound, see Corollary~\ref{cor:DZULDP} below.

Before stating the lemma, let us introduce some notation. Given a set $A \subseteq C_T\mcC^\alpha(\mbT^3)$ and an element $w\in C_T\mcC^{\alpha}(\mbT^3)$, we set
\begin{equation*}
	d_{T;\alpha}(w,A) \coloneqq \inf_{v\in A} \|w-v\|_{C_T\mcC_x^{\alpha}}.
\end{equation*}
\begin{lemma}\label{lem:hausdorff_convergence}
	Let $\theta \geq 0$, $T\in (0,+\infty)$ and $\mfz \in \mcC^{\alpha}(\mbT^3)$. Then for any sequence $(\mfz_n)_{n\geq 0}\subset \mcC^{\alpha}(\mbT^3)$ such that $\lim_{n\to \infty}\|\mfz-\mfz_n\|_{\mcC^\alpha} = 0$ it holds that
	\begin{equation*}
		\lim_{n\to \infty} \max\left\{ \sup_{w \in \msI^{\mfz_n}_T[\theta]} d_{T;\alpha}(w,\msI^\mfz_T[\theta]), \,\sup_{v\in \msI^\mfz_T[\theta]}d_{T;\alpha}(v,\msI^{\mfz_n}_T[\theta]) \right\} =0.
	\end{equation*}
\end{lemma}
\begin{proof}
	We show that each term inside the maximum goes to zero individually. 
	
	To show the first, take any sequence $(w^n)_{n\geq 0}\subset C_T\mcC^{\alpha}(\mbT^3)$ where each $w^n \in \msI^{\mfz_n}_T[\theta]$. We will show that there exists $v\in \msI^\mfz_T[\theta]$ such that $\|w^n-v\|_{C_T\mcC^{\alpha}_x} \to 0$. Since, by assumption, $w^n \in \msI^{\mfz_n}_T[\theta]$ for each $n\geq 0$ there exists an $h^n \in L^2_TL^2_x$ such that $\frac{1}{2}\|h^n\|^2_{L^2_TL^2_x} \leq \theta$ and $w^n = w^{h_n,\fz_n}$. Since $L^2_TL^2(\mbT^3)$ is reflexive, there exists an $\bar{h} \in L^2_TL^2(\mbT^3)$ such that the sequence $(h^n)_{n\geq 0}$ converges with respect to the weak $L^2_TL^2(\mbT^3)$ topology to $\bar{h}$. In addition, one readily checks that\footnote{The notation $\mcC^{\alpha-2}_\text{par}([0,T]\times \mbT^3)$ indicates a space of $(\alpha-2)$-H\"older regular distributions over the cylinder $[0,T]\times \mbT^3$ equipped with the parabolic scaling $(2,1,1,1)$, i.e. derivatives in time cost twice as much as derivatives in space. This is the natural scale of spaces in which to measure the regularity of realisations of the space-time white noise. Since this notation is confined to the current proof we do not give a full definition, instead referring the reader to \cite[Def.~A.15]{martini_mayorcas_24_small_noise} and \cite[Def.~2.3 \& Thm.~2.7]{chandra_weber}.} $L^2_TL^2(\mbT^3) \cembed  \mcC^{\alpha-2}_{\text{par}}([0,T]\times \mbT^3)$ so that we also have $\lim_{n\to \infty}\|h^n-\bar{h}\|_{\mcC^{\alpha-2}_{\text{par}}}=0$. Then, it is a straightforward but tedious exercise to check that one can in fact build all the diagrams of Table \ref{table_trees} from $h^n$, each being uniformly bounded in a Banach space which compactly embeds into the native topology of each diagram in the model, in particular those induced by \eqref{def:model:hom_model_norm_trees}. To wit, consider the equivalent diagram to $\<2b>$, from Lemma~\ref{lem:lin_skeleton_reg} (and recalling the definition therein) one has
	\begin{equation*}
		z^n \coloneqq z^{h^n} \in C_TH^1 (\mbT^3) \embed C_T L^6(\mbT^3),
	\end{equation*}
	uniformly in $n$ and similarly, $z^{\bar{h}} \in C_TH^1 (\mbT^3) \embed C_T L^6(\mbT^3)$. Thus, by simple application of the product rule and the fact that the map $(f,g)\to f\times g$ is a continuous map from $L^{p_1}(\mbT^3)\times L^{p_2}(\mbT^3)\to L^{p_3}(\mbT^3)$ for $\frac{1}{p_3}= \frac{1}{p_2}+ \frac{1}{p_1}$, one finds that uniformly in $n\geq 0$
	\begin{equation*}
		(z^{n})^2 \in C_T W^{1,\nicefrac{3}{2}}(\mbT^3)\embed \mcB^{1}_{\nicefrac{3}{2},\infty}(\mbT^3) \embed  C_T \mcC^{-1}(\mbT^3) \cembed C_T\mcC^{2\alpha}(\mbT^3).
	\end{equation*}
	Similarly, $(z^{\bar{h}})^2 \in C_T \mcC^{-1}(\mbT^3) \cembed C_T\mcC^{2\alpha}(\mbT^3)$ and so we conclude that one also has
	\begin{equation*}
		\lim_{n\to \infty}\| (z^n)^2 - (z^{\bar{h}})^2\|_{C_T\mcC^{2\alpha}_x}=0.
	\end{equation*}
	We leave it as an exercise to the reader to check that a similar procedure holds for the remaining diagrams of Table \ref{table_trees}, only pointing out that for the higher regularity trees, it becomes necessary to use the recentring encoded in Definition~\ref{def:bphz_lift} (ignoring the renormalisation constants which are not needed here) along with the Morrey--Sobolev embedding.
	
	It follows from this discussion that can one build canonical models (in the sense of Definition~\ref{def:model})  $\Pi^n$ from $h^n$ and $\Pi$ from $\bar{h}$ which convergence in the appropriate topology so that we can in fact apply Lemma~\ref{lem_uniform_lipschitz} to see that there exists a $v\in C_T\mcC^{\alpha}(\mbT^3)$ such that $\lim_{n\to \infty} \|w^n-v\|_{C_T\mcC^{\alpha}_x} =0$ and in the sense of distributions,
	\begin{equation*}
		\partial_t v- \Delta v = - v^3 - m^2 v + \bar{h},\quad v\tzero = \mfz.
	\end{equation*}
	Finally, using the fact that the norm is weakly lower-semi-continuous in Hilbert spaces, one has
	\begin{equation*}
		\msI^\mfz_T(v) \leq \liminf_{n\to \infty} \msI^{\mfz_n}_T(w^n) = \frac{1}{2}\liminf_{n\to \infty} \|h^n\|^2_{L^2_TL^2_x} \leq \theta.
	\end{equation*}
	
	Showing that the second supremum converges is simpler. Given $v\in \msI^\mfz_T[\theta]$, we have the existence of an $h\in L^2_TL^2_x$ with $\frac{1}{2}\|h\|_{L^2_TL^2_x}\leq \theta$ such that $\msI^\mfz_T(v)= \frac{1}{2}\|h\|_{L^2_TL^2_x}$ and $v = w^{h,\fz}$. Therefore, we define a sequence $(w^n)_{n\geq 0}\subset C_T\mcC^{\alpha}(\mbT^3)$, each solving in the sense of distributions
	\begin{equation*}
		\partial_t w^n - \Delta w^n +m^2 w^n = - (w^n)^3 + h, \quad w^n\tzero = \mfz_n, 
	\end{equation*}
	i.e.~$w^n = w^{h_n,\fz_n}$.
	The fact that this sequence converges to $v$ follows again from continuity of the solution map (see Lemma~\ref{lem_uniform_lipschitz}). To see that each $w^n$ lives in the right sub-level set, we recall that by definition we have $\msI^{\mfz_n}_T(w^n) =\frac{1}{2}\|h\|^2_{L^2_TL^2_x} \leq \theta$. 
\end{proof}
Appealing to \cite[Thm.~2.7]{salins_uniform_ldp}  the above lemma shows that we also obtain a uniform LDP in the sense of Dembo--Zeitouni. Since the use of this form of the uniform LDP is confined to the proof of Proposition~\ref{prop:bound_prob_H} we do not give its full definition, instead referring to \cite[Def.~2.2]{salins_uniform_ldp}.
\begin{corollary}[Locally uniform DZLDP for the dynamics]\label{cor:DZULDP}
	Let~$T > 0$ and $\msA$ be the collection of all compact subsets of $\mcC^{\alpha}(\mbT^3)$. Then, for any $A\in \msA$, the family
	\begin{equation*}
		\{\bar{u}_\eps^\fz: \ \eps > 0, \ \fz \in A\} \subseteq C_T \CC^{\alpha}(\T^3)
	\end{equation*}
	satisfies a \emph{DZULDP} (see \cite[Def.~2.2]{salins_uniform_ldp}) with good rate function $\msI^\mfz_T$. In particular, for any $A\in \msA$ and $F \subseteq C_T\mcC^{\alpha}(\mbT^3)$ closed we have the upper bound
	\begin{equation*}
		\limsup_{\eps \to 0} \eps ^2 \log \sup_{\fy \in A} \P\del[1]{\bar{u}^\fy_\eps \in F} 
		\leq 
		-\inf_{\fy \in A} \inf_{v \in F} \II_T^\fz(v) 
	\end{equation*}
\end{corollary}
\begin{proof}
	Lemma~\ref{lem:hausdorff_convergence} shows that \cite[Ass.~2.6 d)]{salins_uniform_ldp} is satisfied in our case. Point \cite[Ass.~2.6 a) \& c)]{salins_uniform_ldp} are satisfied with $\mcE_0 =\mcC^{\alpha}(\mbT^3)$. Finally, inspecting the proof we see that our FWULDP given by Proposition~\ref{prop:uniform_ldp} also holds if we restrict the family of measures therein as here, to initial data $\mfz\in A$ where $A\in \msA$ is any compact subset of $\mcC^{\alpha}(\mbT^3)$. Thus \cite[Thm.~2.7]{salins_uniform_ldp} applies and so we see that the FWULDP obtained by Proposition~\ref{prop:uniform_ldp} implies an equivalent DZULDP, see \cite[Def.~2.2]{salins_uniform_ldp}.
\end{proof}
\subsection[Tail Bounds for the~$\Phi^4_3$ Measure]{Tail Bounds for the~$\boldsymbol{\Phi^4_3}$ Measure} \label{sec:tail_bounds}

In this section, we recapitulate the strategy to \emph{rigorously} construct the measure~$\mu_\eps$, informally described in~\eqref{eq:phi43_formal_measure}, as the unique invariant measure of the Markov process~$\bar{u}^\fz_\eps$ given by Definition~\ref{def:phi43_rigorous_SPDE}.
Furthermore, we establish exponential tail bounds on~$\mu_\eps$ which are quantitative in $\eps >0$, see Proposition~\ref{prop:measure_tail_bounds} below, which play a key role in the proof of our main result, Theorem~\ref{th:main_result}.

To begin with, the \emph{existence} of an invariant measure for~$\bar{u}_\eps^\fz$ follows from the CDFI results, recovered in Subsection~\ref{sec:cdfi} by means of a Krylov--Bogoliubov argument, see~\cite{mourrat_weber_infinity} for details.
\emph{Uniform exponential convergence} in total variation distance to a \emph{unique} invariant measure~$\mu_\eps$ then follows from a combination of CDFI with the following results, see (the proof of)~\cite[Coro.~1.13]{hairer_schoenbauer_support}:

\begin{itemize}
	\item The \emph{strong Feller} property, obtained by Hairer and Mattingly in~\cite{hairer-mattingly}, which implies continuity of the transition probabilities associated with~$\bar{u}_\eps^\fz$ in total variation distance.
	\item The \emph{support theorem} of Hairer and Sch\"onbauer~\cite{hairer_schoenbauer_support}, which is generality applicable but specialised to~$\Phi^4_3$ in~\cite[Thm.~1.12]{hairer_schoenbauer_support}. In our case, this asserts that the law of~$\bar{u}_\eps^\fz$ has full support in $\mcC^{\alpha}(\mbT^3)$.
\end{itemize}
The following works connect the measure~$\mu_\eps$ as obtained via the above procedure to the programme of constructive QFT (CQFT):
\begin{itemize}
	\item In~\cite{hairer_matetski_18_discretisation}, Hairer and Matetski showed that, for small coupling strengths, the measure~$\mu_\eps$ is \emph{identical} to the one constructed via lattice approximations, see for example~\cite{feldman_74_finite_vol}.
	Note, however, that in the CQFT programme of the~1970's, the measure~$\mu_\eps$ had only been constructed for small coupling constants, hence the constraint in the previous statement;
	the SPDE techniques summarised earlier apply at~\emph{any coupling strength} and thus generalise the earlier CQFT constructions.
	\item The Osterwalder--Schrader axioms~\cite{osterwalder_schrader_I, osterwalder_schrader_II}, which connect a Euclidean QFT to its relativistic counterpart, were verified for~$\Phi^4_3$ by Gubinelli and Hofmanova in~\cite{gubinelli_hofmanova} using PDE techniques.
	\item Finally, Hairer and Steele~\cite{hairer_steele_22} showed that, in fact, the measure~$\mu_\eps$ (for fixed~$\eps > 0$) satisfies an exponential tail bound that is much stronger than the one required by the OS axioms.
\end{itemize}
The following proposition is the main result of this section: It provides an exponential tail bound for the measure~$\mu_\eps$ in the temperature parameter~$\eps$.

\begin{proposition}[Tail Bound] \label{prop:measure_tail_bounds} 
	For any~$\theta > 0$ there exists some $\rho \coloneqq \rho(\theta) > 0$ such that for all~$\eps~\in~(0,1]$ 
	\begin{equation}
		\mu_\eps\del[1]{B_\rho^{\texttt{c}}} \leq K \exp\del[3]{-\frac{\theta}{\eps^2}}.
		\label{prop:measure_tail_bounds:eq}
	\end{equation}
	where~$K \coloneqq K_{\nicefrac{1}{2}}$ is the constant from Lemma~\ref{lem:fernique_bphz}, the Fernique estimate.
\end{proposition}

\begin{remark} \label{rmk:quantitative_tail_bound}
	The previous tail bound could likely be strengthened if we were to track the dependence on the temperature~$\eps$ in the arguments of~\cite{hairer_steele_22}, resulting in a tighter dependence of~$\rho$ on~$\theta$.
	We do not need such precise information here and therefore do not investigate this assertion further.
\end{remark}

\begin{proof}[Proof of Proposition~\ref{prop:measure_tail_bounds}]
	Since the measure $\mu_{\eps}$ is invariant for the renormalised dynamics $t\mapsto u^\fz_{\eps}(t)$, with $\fz \sim \mu_{\eps}$, for any~$t \in [0,\infty)$ it holds that
	\begin{equation*}
		\mu_\eps(\bar{B}_\rho^{\mathtt{c}}) = \int_{\CC^\alpha(\T^3)} \P\del[1]{\bar{u}_\eps^\fz(t) \in \bar{B}_\rho^{\mathtt{c}}} \mu_\eps(\dif \fz).
	\end{equation*}
	Therefore, the claimed bound in~\eqref{prop:measure_tail_bounds:eq} is a consequence of the following lemma.
\end{proof}

\begin{lemma} \label{lem:tail_bound_dynamics}
	Let~$t \in (0,1)$. For each~$\theta > 0$, there exists a exists some~$\rho = \rho(t,\theta) > 0$ such that for all~$\eps \in (0,1]$,  the bound
	\begin{equation} \label{eq:outside_ball_dynamics}
		\P\del[1]{\norm[0]{\bar{u}_{\eps}^\fz(t)}_{\mcC^{\alpha}} \geq \rho}
		\leq K_t \exp\del[3]{-\frac{\theta}{\eps^2}}. 
	\end{equation}
	holds uniformly over all initial conditions~$\fz \in \CC^\alpha(\T^3)$ and~$K_t \in (0,\infty)$ denotes the constant in Lemma~\ref{lem:fernique_bphz}, the Fernique estimate.
\end{lemma}
Let us highlight that it is \emph{only} within the proof of that lemma that we use the CDFI result of Moinat and Weber, Proposition~\ref{prop:cdfi_model}, which is specific for the~$\Phi^4_3$ equation and does not hold for other interesting models from QFT.
We direct the reader to the \enquote{proof strategy} paragraph of the introduction, specifically item~\ref{model_specific:1} on page~\pageref{model_specific:1}, where this issue has been extensively discussed.
\begin{proof}	
	By Proposition~\ref{prop:cdfi_model}, we have
	\begin{equation}  \label{lem:tail_bound_dynamics:pf_eq1}
		\norm[0]{\bar{u}^\fz_\eps(t)}^q_{\mcC^{\alpha}}
		\leq
		\norm[0]{v^\fz_{\eps}(t)}^q_{\mcC^{\alpha}} 
		+ 
		\norm[0]{\<1b>^{\eps}(t)}^q_{\mcC^{\alpha}}
		\lesssim
		\frac{1}{t^{2q}} + \eps^{\frac{2q}{1 - 2\kappa}} \del[1]{\barnorm{\bbzm}_t}^{\frac{2q}{1 - 2\kappa}} 
	\end{equation}
	where we have additionally used homogeneity of~~$\barnorm{\cdot}$ with respect to rescaling the noise~$\xi \rightsquigarrow \eps \xi$, recall Lemma~\ref{lem:homog_model_norm}.
	Choosing~$q = 1-2\kappa$ and denoting the implicit constant in~\eqref{lem:tail_bound_dynamics:pf_eq1} by~$C$, Markov's inequality implies the estimate
	\begin{align*}
		\P\del[1]{\norm[0]{\bar{u}_{\eps}^\fz(t)}_{\mcC^{\alpha}} \geq \rho}
		& = 
		\P\del[1]{\norm[0]{\bar{u}_{\eps}^\fz(t)}_{\mcC^{\alpha}}^{1-2\kappa} \geq \rho^{1-2\kappa}}
		\leq 
		\P\del[3]{\frac{1}{\eps^2 t^{2(1-2\kappa)}} +  \barnorm{\bar{\Pi}}_t^2 \geq \frac{\rho^{1 - 2\kappa}}{C\eps^{2}}} \\[0.5em]
		& \leq
		\frac{\exp\del[1]{\frac{\Lambda_t }{\eps^2 t^{2(1-2\kappa)}}} \E\sbr[1]{\exp\del[0]{\Lambda_t  \barnorm{\bar{\Pi}}_t^2}}}{\exp\del[1]{\Lambda_t  \rho^{1-2\kappa} C^{-1} \eps^{-2}}} \\
		& =
		K_t \exp\del[2]{-\Lambda_t  \sbr[2]{\frac{\rho^{1-2\kappa}}{C} - \frac{1}{t^{2(1-2\kappa)}}}\frac{1}{\eps^{2}}}
	\end{align*}
	where $\Lambda_t > 0$ and~$K_t \in (0,\infty)$ denote the constants in Lemma~\ref{lem:fernique_bphz}.
	The claim now follows by choosing~$\rho > 0$ such that 
	\begin{equation*}
		\Lambda_t  \sbr[2]{\frac{\rho^{1-2\kappa}}{C} - \frac{1}{t^{2(1-2\kappa)}}} \geq \theta.
	\end{equation*}
\end{proof}

\bibstyle{alpha}
\bibliography{bibfile.bib}

\begin{thebibliography}{ABDVG22}

\bibitem[ABDVG22]{albeverio_borasi_deVecchi_gubinelli_22_grassmannian}
Sergio Albeverio, Luigi Borasi, Francesco~C. De~Vecchi, and Massimiliano
  Gubinelli.
\newblock Grassmannian stochastic analysis and the stochastic quantization of
  {E}uclidean fermions.
\newblock {\em Probab. Theory Related Fields}, 183(3-4):909--995, 2022.
\newblock DOI:
  \href{https://doi.org/10.1007/s00440-022-01136-x}{\texttt{10.1007/s00440-022-01136-x}}.

\bibitem[AD85]{azencott_doss_1985}
R.~Azencott and H.~Doss.
\newblock {L'Equation de Schr{\"o}dinger quand $\hbar$ tend vers zero; une
  approche probabiliste}.
\newblock In Sergio Albeverio, Philippe Combe, and Madeleine Sirugue-Collin,
  editors, {\em {Stochastic Aspects of Classical and Quantum Systems}}, pages
  1--17, Berlin, Heidelberg, 1985. Springer Berlin Heidelberg.
\newblock DOI:
  \href{https://doi.org/10.1080/17442508808833536}{\texttt{10.1080/17442508808833536}}.

\bibitem[AP93]{ambrosetti_prodi_93}
Antonio Ambrosetti and Giovanni Prodi.
\newblock {\em {A Primer of Nonlinear Analysis}}.
\newblock Cambridge University Press, 1993.
\newblock DOI:
  \href{https://doi.org/10.1002/zamm.19930731218}{\texttt{10.1002/zamm.19930731218}}.

\bibitem[AVG21]{albeverio_deVecchi_gubinelli_21_elliptic}
Sergio Albeverio, Francesco C.~De Vecchi, and Massimiliano Gubinelli.
\newblock {The elliptic stochastic quantization of some two dimensional
  Euclidean QFTs}.
\newblock {\em Annales de l'Institut Henri Poincar{\'e}, Probabilit{\'e}s et
  Statistiques}, 57(4):2372 -- 2414, 2021.
\newblock DOI:
  \href{https://doi.org/10.1214/20-AIHP1145}{\texttt{10.1214/20-AIHP1145}}.

\bibitem[BA88]{benArous_88_stationary}
G\'{e}rard Ben~Arous.
\newblock Methods de {L}aplace et de la phase stationnaire sur l'espace de
  {W}iener.
\newblock {\em Stochastics}, 25(3):125--153, 1988.
\newblock DOI:
  \href{https://doi.org/10.1080/17442508808833536}{\texttt{10.1080/17442508808833536}}.

\bibitem[Bai23]{bailleul_23_phi43_uniqueness}
Ismael Bailleul.
\newblock {Uniqueness of the $\Phi^4_3$ measures on closed Riemannian
  $3$-manifolds}, 2023.
\newblock preprint, DOI:
  \href{https://doi.org/10.48550/arXiv.2306.07616}{\texttt{10.48550/arXiv.2306.07616}}.

\bibitem[Bar22a]{barashkov_22_stochcontrol_sine_gordon}
Nikolay Barashkov.
\newblock {A stochastic control approach to Sine Gordon EQFT}, 2022.
\newblock preprint, DOI: \href{
  https://doi.org/10.48550/arXiv.2203.06626}{10.48550/arXiv.2203.06626}.

\bibitem[{Bar}22b]{barashkov_phd_thesis}
{Barashkov, Nikolay}.
\newblock {\em {A variational approach to Gibbs measures on function spaces}}.
\newblock PhD thesis, Rheinische Friedrich-Wilhelms-Universität Bonn, 2022.
\newblock DOI:
  \href{https://hdl.handle.net/20.500.11811/9684}{\texttt{20.500.11811/9684}}.

\bibitem[BC24a]{bringmann_cao_24_sineGordon}
Bjoern Bringmann and Sky Cao.
\newblock Global well-posedness of the dynamical sine-gordon model up to
  $6\pi$, 2024.

\bibitem[BC24b]{bringmann_cao_24_abelianHiggs}
Bjoern Bringmann and Sky Cao.
\newblock Global well-posedness of the stochastic abelian-higgs equations in
  two dimensions, 2024.

\bibitem[BCCH20]{rs_renorm}
Yvain Bruned, Ajay Chandra, Ilya Chevyrev, and Martin Hairer.
\newblock Renormalising {SPDE}s in regularity structures.
\newblock {\em Journal of the European Mathematical Society}, 23(3):869--947,
  2020.
\newblock DOI: \href{http://dx.doi.org/10.4171/JEMS/1025}{10.4171/JEMS/1025}.

\bibitem[BCD11]{bahouri_chemin_danchin_11_fourier}
Hajer Bahouri, Jean-Yves Chemin, and Rapha{\"e}l Danchin.
\newblock {\em Fourier Analysis and Nonlinear Partial Differential Equations},
  volume 343 of {\em Grundlehren der mathematischen Wissenschaften}.
\newblock Springer, 2011.
\newblock DOI:
  \href{https://doi.org/10.1007/978-3-642-16830-7}{\texttt{10.1007/978-3-642-16830-7}}.

\bibitem[BDFT23a]{bailleul_dang_ferdinand_to_23_phi43_harmonic}
Ismael Bailleul, Nguyen~Viet Dang, {Léonard} Ferdinand, and Tat~Dat {Tô}.
\newblock {Global harmonic analysis for $\Phi^4_3$ on closed Riemannian
  manifolds}, 2023.
\newblock preprint, DOI:
  \href{https://doi.org/10.48550/arXiv.2306.07757}{\texttt{10.48550/arXiv.2306.07757}}.

\bibitem[BDFT23b]{bailleul_dang_ferdinand_to_23_manifold_measures}
Ismael Bailleul, Nguyen~Viet Dang, {Léonard} Ferdinand, and Tat~Dat {Tô}.
\newblock {$\Phi^4_3$ measures on compact Riemannian $3$-manifolds}, 2023.
\newblock preprint, DOI:
  \href{https://doi.org/10.48550/arXiv.2304.10185}{\texttt{10.48550/arXiv.2304.10185}}.

\bibitem[BDM08]{budhiraja_dupuis_maroulas_2008_ldp_weak_conv}
Amarjit Budhiraja, Paul Dupuis, and Vasileios Maroulas.
\newblock {Large deviations for infinite dimensional stochastic dynamical
  systems}.
\newblock {\em The Annals of Probability}, 36(4):1390 -- 1420, 2008.
\newblock DOI:
  \href{https://doi.org/10.1214/07-AOP362}{\texttt{10.1214/07-AOP362}}.

\bibitem[BFZ24]{bai_feng_zhao_24_burgers}
Rui Bai, Chunrong Feng, and Huaizhong Zhao.
\newblock {Large deviations principle for invariant measures of stochastic
  Burgers equations}, 2024.
\newblock DOI:
  \href{https://doi.org/10.48550/arXiv.2409.14234}{\texttt{10.48550/arXiv.2409.14234}}.

\bibitem[BG20]{barashkov_gubinelli_20_variational}
Nikolay Barashkov and Massimiliano Gubinelli.
\newblock {A variational method for {$\Phi^4_3$}}.
\newblock {\em Duke Math. J.}, 169(17):3339--3415, 2020.
\newblock DOI:
  \href{https://doi.org/10.1215/00127094-2020-0029}{\texttt{10.1215/00127094-2020-0029}}.

\bibitem[BG21]{barashkov_gubinelli_21_girsanov}
Nikolay Barashkov and Massimiliano Gubinelli.
\newblock {The {$\Phi^4_3$} measure via {G}irsanov's theorem}.
\newblock {\em Electron. J. Probab.}, 26:Paper No. 81, 29, 2021.
\newblock DOI:
  \href{https://doi.org/10.1214/21-ejp635}{\texttt{10.1214/21-ejp635}}.

\bibitem[BG23]{barashkov_gubinelli_22_variational_volume}
Nikolay Barashkov and Massimiliano Gubinelli.
\newblock {On the variational method for Euclidean quantum fields in infinite
  volume}.
\newblock {\em Probability and Mathematical Physics}, 4(4):761--801, 2023.
\newblock DOI:
  \href{https://doi.org/10.2140/pmp.2023.4.761}{10.2140/pmp.2023.4.761}.

\bibitem[BHZ19]{bhz}
Yvain Bruned, Martin Hairer, and Lorenzo Zambotti.
\newblock {Algebraic renormalisation of regularity structures}.
\newblock {\em Inventiones mathematicae}, 215(3):1039--1156, 2019.
\newblock DOI:
  \href{https://doi.org/10.1007/s00222-018-0841-x}{\texttt{10.1007/s00222-018-0841-x}}.

\bibitem[BO13]{benyi_oh_13_sobolev_torus}
\'{A}rp\'{a}d B\'{e}nyi and Tadahiro Oh.
\newblock The {S}obolev inequality on the torus revisited.
\newblock {\em Publ. Math. Debrecen}, 83(3):359--374, 2013.
\newblock DOI:
  \href{https://doi.org/10.5486/PMD.2013.5529}{\texttt{10.5486/PMD.2013.5529}}.

\bibitem[CC18]{catellier_chouk_paracontrolled_phi43}
Rémi Catellier and Khalil Chouk.
\newblock {Paracontrolled distributions and the 3-dimensional stochastic
  quantization equation}.
\newblock {\em The Annals of Probability}, 46(5):2621 -- 2679, 2018.
\newblock DOI:
  \href{https://doi.org/10.1214/17-AOP1235}{\texttt{10.1214/17-AOP1235}}.

\bibitem[CC23]{cao_chatterjee_1}
Sky Cao and Sourav Chatterjee.
\newblock {The Yang--Mills heat flow with random distributional initial data}.
\newblock {\em Communications in Partial Differential Equations},
  48(2):209--251, 2023.
\newblock DOI:
  \href{https://doi.org/10.1080/03605302.2023.2169937}{\texttt{10.1080/03605302.2023.2169937}}.

\bibitem[CC24]{cao_chatterjee_2}
Sky Cao and Sourav Chatterjee.
\newblock {A state space for {3D} {E}uclidean {Y}ang-{M}ills theories}.
\newblock {\em Communications in Mathematical Physics}, 405(1), 2024.
\newblock DOI:
  \href{https://doi.org/10.1007/s00220-023-04870-y}{\texttt{10.1007/s00220-023-04870-y}}.

\bibitem[CCHS22a]{chandra_chevyrev_hairer_shen_22_2dYM}
Ajay Chandra, Ilya Chevyrev, Martin Hairer, and Hao Shen.
\newblock {Langevin dynamic for the {2D} {Y}ang--{M}ills measure}.
\newblock {\em Publications mathématiques de l'IHÉS}, 2022.
\newblock DOI:
  \href{https://doi.org/10.1007/s10240-022-00132-0}{\texttt{10.1007/s10240-022-00132-0}}.

\bibitem[CCHS22b]{chandra_chevyrev_hairer_shen_22_3dYMH}
Ajay Chandra, Ilya Chevyrev, Martin Hairer, and Hao Shen.
\newblock Stochastic quantisation of {Y}ang--{M}ills--{H}iggs in {3D}, 2022.
\newblock preprint, DOI:
  \href{https://doi.org/10.48550/arXiv.2201.03487}{\texttt{10.48550/arXiv.2201.03487}}.

\bibitem[CD19]{cerrai_debussche_19_ldp_phi2nd}
Sandra Cerrai and Arnaud Debussche.
\newblock Large deviations for the dynamic {$\Phi^{2n}_d$} model.
\newblock {\em Appl. Math. Optim.}, 80(1):81--102, 2019.
\newblock DOI:
  \href{https://doi.org/10.1007/s00245-017-9459-4}{\texttt{10.1007/s00245-017-9459-4}}.

\bibitem[CdLFW24]{chandra_feltes_weber_24_apriori}
Ajay Chandra, Guilherme de~Lima~Feltes, and Hendrik Weber.
\newblock A priori bounds for 2-d generalised parabolic anderson model, 2024.

\bibitem[Cer03]{cerrai_2003}
Sandra Cerrai.
\newblock {Stochastic reaction-diffusion systems with multiplicative noise and
  non-Lipschitz reaction term}.
\newblock {\em Probability Theory and Related Fields}, 125(2):271--304, 2003.
\newblock DOI:
  \href{https://doi.org/10.1007/s00440-002-0230-6}{\texttt{10.1007/s00440-002-0230-6}}.

\bibitem[CF11]{cerrai_freidlin_11_rde_quasi}
Sandra Cerrai and Mark Freidlin.
\newblock {Approximation of Quasi-Potentials and Exit Problems for
  Multidimensional {RDE's} with Noise}.
\newblock {\em Transactions of the American Mathematical Society},
  363(7):3853--3892, 2011.
\newblock DOI:
  \href{https://doi.org/10.1090/S0002-9947-2011-05352-3}{\texttt{10.1090/S0002-9947-2011-05352-3}}.

\bibitem[CF23]{chandra_ferdinand_23_tensor}
Ajay Chandra and Léonard Ferdinand.
\newblock {A Stochastic Analysis Approach to Tensor Field Theories}, 2023.
\newblock preprint, DOI:
  \href{https://doi.org/10.48550/arXiv.2306.05305}{\texttt{10.48550/arXiv.2306.05305}}.

\bibitem[CH16]{chandra_hairer}
Ajay {Chandra} and Martin {Hairer}.
\newblock {An analytic BPHZ theorem for regularity structures}.
\newblock December 2016.
\newblock preprint, DOI:
  \href{https://doi.org/10.48550/arXiv.1612.08138}{\texttt{10.48550/arXiv.1612.08138}}.

\bibitem[Che19]{chevyrev_19_YM}
Ilya Chevyrev.
\newblock Yang-{M}ills measure on the two-dimensional torus as a random
  distribution.
\newblock {\em Comm. Math. Phys.}, 372(3):1027--1058, 2019.

\bibitem[Che22]{chevyrev_review_ym23d}
Ilya Chevyrev.
\newblock {Stochastic quantization of {Y}ang--{M}ills}.
\newblock {\em Journal of Mathematical Physics}, 63(9):091101, 2022.
\newblock DOI:
  \href{https://doi.org/10.1063/5.0089431}{\texttt{10.1063/5.0089431}}.

\bibitem[CHM22]{chevyrev_hambly_mayorcas_22_stoch}
Ilya Chevyrev, Ben Hambly, and Avi Mayorcas.
\newblock {A stochastic model of chemorepulsion with additive noise and
  nonlinear sensitivity}.
\newblock {\em Stochastics and Partial Differential Equations: Analysis and
  Computations}, 2022.
\newblock DOI:
  \href{https://doi.org/10.1007/s40072-022-00244-y}{\texttt{10.1007/s40072-022-00244-y}}.

\bibitem[CHP23]{chandra_hairer_peev_23_yukawa}
Ajay Chandra, Martin Hairer, and Martin Peev.
\newblock {A Dynamical {{Y}ukawa$_{2}$} Model}.
\newblock arXiv:2305.07388, 2023.
\newblock preprint, DOI:
  \href{https://doi.org/10.48550/arXiv.2305.07388}{\texttt{10.48550/arXiv.2305.07388}}.

\bibitem[CHS18]{chandra_hairer_shen_18_sineGordon_subCrit}
Ajay Chandra, Martin Hairer, and Hao Shen.
\newblock {The dynamical sine-Gordon model in the full subcritical regime},
  2018.
\newblock preprint, DOI:
  \href{https://doi.org/10.48550/arXiv.1808.02594}{\texttt{10.48550/arXiv.1808.02594}}.

\bibitem[CMW23]{chandra_moinat_weber_23}
Ajay Chandra, Augustin Moinat, and Hendrik Weber.
\newblock A priori bounds for the $\phi ^4$ equation in the full sub-critical
  regime.
\newblock {\em Archive for Rational Mechanics and Analysis}, 247(48), 2023.
\newblock DOI:
  \href{https://doi.org/10.1007/s00205-023-01876-7}{\texttt{10.1007/s00205-023-01876-7}}.

\bibitem[CP22]{cerrai_paskal_22_NS_LDP}
S.~Cerrai and N.~Paskal.
\newblock Large deviations principle for the invariant measures of the 2{D}
  stochastic {N}avier-{S}tokes equations with vanishing noise correlation.
\newblock {\em Stoch. Partial Differ. Equ. Anal. Comput.}, 10(4):1651--1681,
  2022.

\bibitem[CR04]{cerrai_roeckner_ldp_dynamics}
Sandra Cerrai and Michael Röckner.
\newblock {Large deviations for stochastic reaction-diffusion systems with
  multiplicative noise and non-Lipshitz reaction term}.
\newblock {\em The Annals of Probability}, 32(1B):1100 -- 1139, 2004.
\newblock DOI:
  \href{https://doi.org/10.1214/aop/1079021473}{\texttt{10.1214/aop/1079021473}}.

\bibitem[CR05]{cerrai_roeckner_ldp_measure}
Sandra Cerrai and Michael Röckner.
\newblock {Large deviations for invariant measures of stochastic
  reaction–diffusion systems with multiplicative noise and non-Lipschitz
  reaction term}.
\newblock {\em Annales de l'Institut Henri Poincare (B) Probability and
  Statistics}, 41(1):69--105, 2005.
\newblock DOI:
  \href{https://doi.org/10.1016/j.anihpb.2004.03.001}{\texttt{10.1016/j.anihpb.2004.03.001}}.

\bibitem[CS23]{chevyrev_shen_23}
Ilya Chevyrev and Hao Shen.
\newblock {Invariant measure and universality of the 2D Yang--Mills Langevin
  dynamic}, 2023.
\newblock preprint, DOI:
  \href{https://doi.org/10.48550/arXiv.2302.12160}{\texttt{10.48550/arXiv.2302.12160}}.

\bibitem[CW17]{chandra_weber}
Ajay Chandra and Hendrik Weber.
\newblock Stochastic {PDE}s, regularity structures, and interacting particle
  systems.
\newblock {\em Annales de la Facult{\'e} des Sciences de Toulouse
  Math{\'e}matiques}, 26(4):847--909, 1 2017.

\bibitem[CZ20]{caravenna_zambotti_2020}
Francesco Caravenna and Lorenzo Zambotti.
\newblock Hairer’s reconstruction theorem without regularity structures.
\newblock {\em EMS~Surv.~Math.~Sci.~7}, 2:207–251, 2020.
\newblock DOI:
  \href{https://doi.org/10.4171/EMSS/39}{\texttt{10.4171/EMSS/39}}.

\bibitem[DE97]{dupuis_ellis_97_weak}
Paul Dupuis and Richard~S. Ellis.
\newblock {\em A weak convergence approach to the theory of large deviations}.
\newblock Wiley Series in Probability and Statistics: Probability and
  Statistics. John Wiley \& Sons, Inc., New York, 1997.
\newblock DOI:
  \href{https://doi.org/10.1002/9781118165904}{\texttt{10.1002/9781118165904}}.

\bibitem[DGR24]{duch_gubinelli_rinaldi_24_parabolic}
Paweł Duch, Massimiliano Gubinelli, and Paolo Rinaldi.
\newblock Parabolic stochastic quantisation of the fractional $\phi^4_3$ model
  in the full subcritical regime, 2024.

\bibitem[DH87]{damgaard_hueffel_1987_stoch_quant}
Poul~H. Damgaard and Helmuth Hüffel.
\newblock Stochastic quantization.
\newblock {\em Physics Reports}, 152(5):227--398, 1987.
\newblock DOI:
  \href{https://doi.org/10.1016/0370-1573(87)90144-X}{\texttt{10.1016/0370-1573(87)90144-X}}.

\bibitem[Dos85]{doss_1985}
Halim Doss.
\newblock Démonstrations probabiliste de certains développements
  asymptotiques quasi classiques.
\newblock {\em Bulletin des Sciences Mathématiques, 2eme Série},
  109:179--208, 1985.

\bibitem[Duc22]{duch_22_flow}
Pawe\l{} Duch.
\newblock Flow equation approach to singular stochastic pdes.
\newblock arXiv:2109.11380, 2022.
\newblock preprint, DOI:
  \href{https://doi.org/10.48550/arXiv:2109.11380}{\texttt{10.48550/arXiv:2109.11380}}.

\bibitem[Duc23]{duch_23_renormalisation}
Pawe\l{} Duch.
\newblock {Renormalization of singular elliptic stochastic PDEs using flow
  equation}, 2023.
\newblock preprint, DOI:
  \href{https://doi.org/10.48550/arXiv.2201.05031}{\texttt{10.48550/arXiv.2201.05031}}.

\bibitem[EO71]{eckmann_osterwalder_71_uniqueness}
Jean-Pierre Eckmann and Konrad Osterwalder.
\newblock On the uniqueness of the {H}amiltionian and of the representation of
  the {${\rm CCR}$} for the quartic boson interaction in three dimensions.
\newblock {\em Helv. Phys. Acta}, 44:884--909, 1971.

\bibitem[Eva10]{evans_10_partial}
Lawrence~Craig Evans.
\newblock {\em Partial {D}ifferential {E}quations}.
\newblock American Mathematical Society, 2010.
\newblock DOI:
  \href{https://doi.org/10.1090/gsm/019}{\texttt{10.1090/gsm/019}}.

\bibitem[EW24]{esquivel_weber_24_apriori}
Salvador Esquivel and Hendrik Weber.
\newblock A priori bounds for the dynamic fractional $\phi^4$ model on
  $\mathbb{T}^3$ in the full subcritical regime, 2024.

\bibitem[Fel74]{feldman_74_finite_vol}
Joel Feldman.
\newblock The $\lambda \varphi^4_3$ field theory in a finite volume.
\newblock {\em Communications in Mathematical Physics}, 37(2):93--120, 1974.
\newblock DOI:
  \href{https://doi.org/10.1007/BF01646205}{\texttt{10.1007/BF01646205}}.

\bibitem[FJL82]{faris-jona-lasinio}
William~G. Faris and Giovanni Jona-Lasinio.
\newblock Large fluctuations for a nonlinear heat equation with noise.
\newblock {\em J. Phys. A: Math. Gen.}, 15:3025--3055, 1982.
\newblock DOI:
  \href{https://doi.org/10.1088/0305-4470/15/10/011}{\texttt{10.1088/0305-4470/15/10/011}}.

\bibitem[FK22]{friz_klose_22_precise}
Peter~K. Friz and Tom Klose.
\newblock Precise {L}aplace asymptotics for singular stochastic {PDE}s: the
  case of 2{D} g{PAM}.
\newblock {\em J. Funct. Anal.}, 283(1):Paper No. 109446, 86, 2022.
\newblock DOI:
  \href{https://doi.org/10.1016/j.jfa.2022.109446}{\texttt{10.1016/j.jfa.2022.109446}}.

\bibitem[FO76]{feldman_osterwalder_76_wightman}
Joel~S. Feldman and Konrad Osterwalder.
\newblock The {W}ightman axioms and the mass gap for weakly coupled {$(\Phi
  \sp{4})\sb{3}$} quantum field theories.
\newblock {\em Ann. Physics}, 97(1):80--135, 1976.
\newblock DOI:
  \href{https://doi.org/10.1016/0003-4916(76)90223-2}{\texttt{10.1016/0003-4916(76)90223-2}}.

\bibitem[Fre88]{freidlin88}
Mark~I. Freidlin.
\newblock {Random Perturbations of Reaction-Diffusion Equations: The
  Quasi-Deterministic Approximation}.
\newblock {\em Transactions of the American Mathematical Society},
  305(2):665--697, 1988.
\newblock DOI:
  \href{https://doi.org/10.2307/2000884}{\texttt{10.2307/2000884}}.

\bibitem[FW12]{freidlin-wentzell}
Mark~I. {Freidlin} and Alexander~D. Wentzell.
\newblock {\em Random Perturbations of Dynamical Systems}, volume 260 of {\em
  Grundlehren der mathematischen Wissenschaften}.
\newblock Springer, third edition, 2012.
\newblock DOI:
  \href{https://doi.org/10.1007/978-3-642-25847-3}{\texttt{10.1007/978-3-642-25847-3}}.

\bibitem[GH21]{gubinelli_hofmanova}
Massimiliano Gubinelli and Martina Hofmanová.
\newblock {A PDE Construction of the Euclidean $\Phi ^4_3$ Quantum Field
  Theory}.
\newblock {\em Communications in Mathematical Physics}, 384(1):1 -- 75, 2021.
\newblock DOI:
  \href{https://doi.org/10.1007/s00220-021-04022-0}{\texttt{10.1007/s00220-021-04022-0}}.

\bibitem[GHR23]{gubinelli_hofmanova_rana_23_exp_decay}
Massimiliano Gubinelli, Martina Hofmanová, and Nimit Rana.
\newblock {Decay of correlations in stochastic quantization: the exponential
  Euclidean field in two dimensions}.
\newblock arXiv:2305.12017, 2023.
\newblock preprint, DOI:
  \href{https://doi.org/10.48550/arXiv.2305.12017}{\texttt{10.48550/arXiv.2305.12017}}.

\bibitem[GIP15]{gip}
Massimiliano Gubinelli, Peter Imkeller, and Nicolas Perkowski.
\newblock Paracontrolled {D}istributions and singular {PDE}s.
\newblock {\em Forum of Mathematics, Pi}, 3(6):1--75, 2015.
\newblock DOI:
  \href{https://doi.org/10.1017/fmp.2015.2}{\texttt{10.1017/fmp.2015.2}}.

\bibitem[GJ73]{glimm_jaffe_73_positivity}
James Glimm and Arthur Jaffe.
\newblock Positivity of the {$\phi \sp{4}\sb{3}$} {H}amiltonian.
\newblock {\em Fortschr. Physik}, 21:327--376, 1973.
\newblock DOI:
  \href{https://doi.org/10.1002/prop.19730210702}{\texttt{10.1002/prop.19730210702}}.

\bibitem[GJ81]{glimm_jaffe}
James Glimm and Arthur Jaffe.
\newblock {\em {Quantum Physics. A Functional Integral Point of View}}.
\newblock Springer, 1981.
\newblock DOI:
  \href{https://doi.org/10.1007/978-1-4612-4728-9}{\texttt{10.1007/978-1-4612-4728-9}}.

\bibitem[Gli68]{glimm_68_boson}
James Glimm.
\newblock Boson fields with the {$\mathpunct{:}\Phi \sp{4}\mathpunct{:}$}
  interaction in three dimensions.
\newblock {\em Comm. Math. Phys.}, 10:1--47, 1968.
\newblock DOI:
  \href{https://doi.org/10.1007/BF01654131}{\texttt{10.1007/BF01654131}}.

\bibitem[GM24]{gubinelli_meyer_24_sineGordon}
Massimiliano Gubinelli and Sarah-Jean Meyer.
\newblock {The FBSDE approach to sine-Gordon up to $6\pi$}, 2024.
\newblock preprint, DOI:
  \href{https://doi.org/10.48550/arXiv.2401.13648}{10.48550/arXiv.2401.13648}.

\bibitem[GST24]{gess_seong_tsatsoulis_24}
Benjamin Gess, Kihoon Seong, and Pavlos Tsatsoulis.
\newblock {Low temperature expansion for the Euclidean $\Phi^4_2$-measure},
  2024.
\newblock DOI:
  \href{https://doi.org/10.48550/arXiv.2404.14539}{\texttt{10.48550/arXiv.2404.14539}}.

\bibitem[Hai14]{hairer_rs}
Martin Hairer.
\newblock {A theory of regularity structures}.
\newblock {\em Inventiones mathematicae}, 198:269--504, 2014.
\newblock DOI:
  \href{https://doi.org/10.1007/s00222-014-0505-4}{10.1007/s00222-014-0505-4}.

\bibitem[Hai15a]{hairer_intro}
Martin Hairer.
\newblock Introduction to regularity structures.
\newblock {\em Braz. J. Probab. Stat.}, 29(2):175--210, 2015.
\newblock DOI: \href{https://doi.org/10.1214/14-BJPS241}{10.1214/14-BJPS241}.

\bibitem[Hai15b]{hairer_rs_dynamical_phi43}
Martin Hairer.
\newblock Regularity structures and the dynamical $\phi^4_3$ model.
\newblock {\em Current Developments in Mathematics}, 2014:1--49, 2015.
\newblock DOI:
  \href{https://dx.doi.org/10.4310/CDM.2014.v2014.n1.a1}{\texttt{10.4310/CDM.2014.v2014.n1.a1}}.

\bibitem[HM18a]{hairer_matetski_18_discretisation}
Martin Hairer and Konstantin Matetski.
\newblock {Discretisations of rough stochastic PDEs}.
\newblock {\em The Annals of Probability}, 46(3):1651 -- 1709, 2018.
\newblock DOI:
  \href{https://doi.org/10.1214/17-AOP1212}{\texttt{10.1214/17-AOP1212}}.

\bibitem[HM18b]{hairer-mattingly}
Martin Hairer and Jonathan Mattingly.
\newblock {The strong Feller property for singular stochastic PDEs}.
\newblock {\em Annales de l'Institut Henri Poincaré, Probabilités et
  Statistiques}, 54(3):1314 -- 1340, 2018.
\newblock DOI:
  \href{https://doi.org/10.1214/17-AIHP840}{\texttt{10.1214/17-AIHP840}}.

\bibitem[HRW12]{hairer-ryser-weber}
Martin Hairer, Marc Ryser, and Hendrik Weber.
\newblock {T}riviality of the {2D} stochastic {A}llen-{C}ahn equation.
\newblock {\em Electron. J. Probab.}, 17, 2012.
\newblock
  \href{http://dx.doi.org/10.1214/EJP.v17-1731}{\texttt{10.1214/EJP.v17-1731}}.

\bibitem[HS16]{hairer_shen_16_dyn_sineGordon}
Martin Hairer and Hao Shen.
\newblock {The Dynamical Sine-Gordon Model}.
\newblock {\em Communications in Mathematical Physics}, 341(3):933--989, 2016.
\newblock DOI:
  \href{https://doi.org/10.1007/s00220-015-2525-3}{\texttt{10.1007/s00220-015-2525-3}}.

\bibitem[HS22a]{hairer_schoenbauer_support}
Martin Hairer and Philipp Schönbauer.
\newblock The support of singular stochastic partial differential equations.
\newblock {\em Forum of Mathematics, Pi}, 10, 2022.
\newblock DOI:
  \href{https://doi.org/10.1017/fmp.2021.18}{\texttt{10.1017/fmp.2021.18}}.

\bibitem[HS22b]{hairer_steele_22}
Martin Hairer and Rhys Steele.
\newblock {The $\Phi _3^4$ Measure Has Sub-Gaussian Tails}.
\newblock {\em Journal of Statistical Physics}, 186(38):3025--3055, 2022.
\newblock DOI:
  \href{https://doi.org/10.1007/s10955-021-02866-3}{\texttt{10.1007/s10955-021-02866-3}}.

\bibitem[HS23]{hairer_singh_23_RegStruct_manifold}
Martin Hairer and Harprit Singh.
\newblock Regularity structures on manifolds and vector bundles, 2023.
\newblock preprint, DOI:
  \href{https://doi.org/10.48550/arXiv.2308.05049}{\texttt{10.48550/arXiv.2308.05049}}.

\bibitem[HS24]{hairer_steele_bphz}
Martin {Hairer} and Rhys {Steele}.
\newblock {The BPHZ Theorem for Regularity Structures via the Spectral Gap
  Inequality}.
\newblock {\em {Archive for Rational Mechanics and Analysis}}, 248:9, 2024.
\newblock DOI:
  \href{https://doi.org/10.1007/s00205-023-01946-w}{\texttt{10.1007/s00205-023-01946-w}}.

\bibitem[HW15]{hairer_weber_ldp}
Martin Hairer and Hendrik Weber.
\newblock Large deviations for white-noise driven, nonlinear stochastic {PDEs}
  in two and three dimensions.
\newblock {\em Annales de la Facult{\'e} des Sciences de Toulouse},
  24(1):55--92, 2015.
\newblock DOI:
  \href{https://doi.org/10.5802/afst.1442}{\texttt{10.5802/afst.1442}}.

\bibitem[JLM90]{jona-lasinio_mitter_ldp_phi42}
Giovanni Jona-Lasinio and Pronob~K. Mitter.
\newblock {Large deviation estimates in the stochastic quantization of
  $\phi^4_2$}.
\newblock {\em Communications in Mathematical Physics}, 130(1):111 -- 121,
  1990.
\newblock DOI:
  \href{https://doi.org/10.1007/BF02099877}{\texttt{10.1007/BF02099877}}.

\bibitem[Klo25]{klose_2025}
Tom Klose.
\newblock The constant coefficient in precise {L}aplace asymptotics for g{PAM}.
\newblock {\em Ann. Inst. H. Poincaré Probab. Statist.}, 2025.
\newblock (To appear).

\bibitem[Kup16]{kupiainen_renorm_group_spde}
Antti Kupiainen.
\newblock {Renormalization Group and Stochastic PDEs}.
\newblock {\em Annales Henri Poincaré}, 17(3):497 -- 535, 2016.
\newblock DOI:
  \href{https://doi.org/10.1007/s00023-015-0408-y}{\texttt{10.1007/s00023-015-0408-y}}.

\bibitem[LN06]{levy_norris_06}
Thierry L{\'e}vy and James~R. Norris.
\newblock {Large Deviations for the Yang-Mills Measure on a Compact Surface}.
\newblock {\em Communications in Mathematical Physics}, 261(2):405--450, 2006.
\newblock DOI:
  \href{https://doi.org/10.1007/s00220-005-1450-2}{\texttt{10.1007/s00220-005-1450-2}}.

\bibitem[LRV17]{lacoin_rhodes_vargas_17_semiclassical_liouville}
Hubert Lacoin, R\'{e}mi Rhodes, and Vincent Vargas.
\newblock Semiclassical limit of {L}iouville field theory.
\newblock {\em J. Funct. Anal.}, 273(3):875--916, 2017.
\newblock DOI:
  \href{https://doi.org/10.1016/j.jfa.2017.04.012}{\texttt{10.1016/j.jfa.2017.04.012}}.

\bibitem[LRV22]{lacoin_rhodes_vargas_22_semiclassical_conformal}
Hubert Lacoin, R\'{e}mi Rhodes, and Vincent Vargas.
\newblock The semiclassical limit of {L}iouville conformal field theory.
\newblock {\em Ann. Fac. Sci. Toulouse Math. (6)}, 31(4):1031--1083, 2022.
\newblock DOI:
  \href{https://doi.org/10.5802/afst.1713}{\texttt{10.5802/afst.1713}}.

\bibitem[MM24]{martini_mayorcas_24_small_noise}
Adrian Martini and Avi Mayorcas.
\newblock An additive-noise approximation to keller-segel-dean-kawasaki
  dynamics: Small-noise results, 2024.

\bibitem[MW17a]{mourrat_weber_17_GWP}
Jean-Christophe Mourrat and Hendrik Weber.
\newblock {Global well-posedness of the dynamic $\Phi^{4}$ model in the plane}.
\newblock {\em The Annals of Probability}, 45(4):2398--2476, 2017.
\newblock DOI:
  \href{https://doi.org/10.1214/16-AOP1116}{\texttt{10.1214/16-AOP1116}}.

\bibitem[MW17b]{mourrat_weber_infinity}
Jean-Christophe Mourrat and Hendrik Weber.
\newblock {The {D}ynamic {{$\Phi^4_3$}} {M}odel {C}omes {D}own from
  {I}nfinity}.
\newblock {\em Communications in Mathematical Physics}, 356(3):673--753, Dec
  2017.
\newblock DOI:
  \href{https://doi.org/10.1007/s00220-017-2997-4}{\texttt{10.1007/s00220-017-2997-4}}.

\bibitem[MW20a]{moinat_weber_20_RD}
Augustin Moinat and Hendrik Weber.
\newblock Local bounds for stochastic reaction diffusion equations.
\newblock {\em Electron. J. Probab.}, 25:26 pp., 2020.
\newblock DOI:
  \href{https://doi.org/10.1214/19-EJP397}{\texttt{10.1214/19-EJP397}}.

\bibitem[MW20b]{moinat_weber_20_phi43Loc}
Augustin Moinat and Hendrik Weber.
\newblock {{S}pace-{T}ime {L}ocalisation for the {D}ynamic~$\Phi^4_3$ {M}odel}.
\newblock {\em Communications on Pure and Applied Mathematics},
  73(12):2519--2555, 2020.
\newblock DOI:
  \href{https://doi.org/10.1002/cpa.21925}{\texttt{10.1002/cpa.21925}}.

\bibitem[Nel66]{nelson_66_derivation}
Edward Nelson.
\newblock {Derivation of the Schr\"odinger Equation from Newtonian Mechanics}.
\newblock {\em Phys. Rev.}, 150:1079--1085, Oct 1966.
\newblock DOI:
  \href{https://doi.org/10.1103/PhysRev.150.1079}{\texttt{10.1103/PhysRev.150.1079}}.

\bibitem[OS73]{osterwalder_schrader_I}
Konrad Osterwalder and Robert Schrader.
\newblock {Axioms for Euclidean Green's functions}.
\newblock {\em Communications in Mathematical Physics}, 31(2):83--112, 1973.
\newblock DOI:
  \href{https://doi.org/10.1007/BF01645738}{\texttt{10.1007/BF01645738}}.

\bibitem[OS75]{osterwalder_schrader_II}
Konrad Osterwalder and Robert Schrader.
\newblock {Axioms for Euclidean Green's functions. II}.
\newblock {\em Communications in Mathematical Physics}, 42(3):281--305, 1975.
\newblock DOI:
  \href{https://doi.org/10.1007/BF01608978}{\texttt{10.1007/BF01608978}}.

\bibitem[Par21]{pardoux_21_SPDE}
{É}tienne Pardoux.
\newblock {\em Stochastic Partial Differential Equations}.
\newblock Springer Cham, 1 edition, 2021.
\newblock DOI:
  \href{https://doi.org/10.1007/978-3-030-89003-2}{\texttt{10.1007/978-3-030-89003-2}}.

\bibitem[PR07]{prevot_rockner_07_concise}
Claudia {Prévôt} and Michael {Röckner}.
\newblock {\em A Concise Course on Stochastic Partial Differential Equations}.
\newblock Springer Berlin, Heidelberg, 1 edition, 2007.
\newblock DOI:
  \href{https://doi.org/10.1007/978-3-540-70781-3}{\texttt{10.1007/978-3-540-70781-3}}.

\bibitem[PW81]{parisi-wu}
Giorgio Parisi and Yong-shi Wu.
\newblock Perturbation theory without gauge fixing.
\newblock {\em Sci. Sinica}, 24(4):483--496, 1981.
\newblock Available at
  \href{https://www.openaccessrepository.it/record/18105/files/LNF_81_017.pdf}{INFN
  Open Access Repository}.

\bibitem[Rud91]{rudin_func_ana}
Walter Rudin.
\newblock {\em Functional Analysis}.
\newblock McGraw-Hill, second edition, 1991.

\bibitem[Sal19]{salins_uniform_ldp}
Michael Salins.
\newblock {Equivalences and counterexamples between several definitions of the
  uniform large deviations principle}.
\newblock {\em Probability Surveys}, 16(none):99 -- 142, 2019.
\newblock DOI:
  \href{https://doi.org/10.1214/18-PS309}{\texttt{10.1214/18-PS309}}.

\bibitem[Sch69]{schreiber_1969}
Michel Schreiber.
\newblock {Fermeture en probabilité de certains sous-espaces d'un
  espace~$L^2$}.
\newblock {\em Zeitschrift für Wahrscheinlichkeitstheorie und Verwandte
  Gebiete}, 14(1):36--48, 1969.
\newblock DOI:
  \href{https://doi.org/10.1007/BF00534116}{\texttt{10.1007/BF00534116}}.

\bibitem[{Sch}23]{schoenbauer}
Philipp {Sch{\"o}nbauer}.
\newblock {Malliavin calculus and densities for singular stochastic partial
  differential equations}.
\newblock {\em Probability Theory and Related Fields}, 186(3):643--713, 2023.
\newblock DOI:
  \href{https://doi.org/10.1007/s00440-023-01207-7}{\texttt{10.1007/s00440-023-01207-7}}.

\bibitem[She21]{shen_abelian}
Hao Shen.
\newblock {Stochastic Quantization of an Abelian Gauge Theory}.
\newblock {\em Communications in Mathematical Physics}, 384(3):1445--1512,
  2021.
\newblock DOI:
  \href{https://doi.org/10.1007/s00220-021-04114-x}{\texttt{10.1007/s00220-021-04114-x}}.

\bibitem[Sow92a]{sowers_ldp_measure}
Richard Sowers.
\newblock {Large deviations for the invariant measure of a reaction-diffusion
  equation with non-Gaussian perturbations}.
\newblock {\em Probability Theory and Related Fields}, 92(3):1432--2064, 1992.
\newblock DOI:
  \href{https://doi.org/10.1007/BF01300562}{\texttt{10.1007/BF01300562}}.

\bibitem[Sow92b]{sowers_ldp_dynamics}
Richard~B. Sowers.
\newblock {Large Deviations for a Reaction-Diffusion Equation with Non-Gaussian
  Perturbations}.
\newblock {\em The Annals of Probability}, 20(1):504 -- 537, 1992.
\newblock DOI:
  \href{https://doi.org/10.1214/aop/1176989939}{\texttt{10.1214/aop/1176989939}}.

\bibitem[Tri92]{triebel_92_theory}
Hans Triebel.
\newblock {\em Theory of function spaces. {II}}, volume~84 of {\em Monographs
  in Mathematics}.
\newblock Birkh\"{a}user Verlag, Basel, 1992.
\newblock DOI:
  \href{https://doi.org/10.1007/978-3-0346-0419-2}{\texttt{10.1007/978-3-0346-0419-2}}.

\bibitem[TW18]{tsatoulis_weber_18_spectral}
Pavlos Tsatsoulis and Hendrik Weber.
\newblock {Spectral gap for the stochastic quantization equation on the
  2-dimensional torus}.
\newblock {\em Annales de l'Institut Henri Poincaré, Probabilités et
  Statistiques}, 54(3):1204 -- 1249, 2018.
\newblock DOI:
  \href{https://doi.org/10.1214/17-AIHP837}{\texttt{10.1214/17-AIHP837}}.

\bibitem[TW20]{tsatsoulis_weber_20_exponential}
Pavlos Tsatsoulis and Hendrik Weber.
\newblock {Exponential loss of memory for the 2-dimensional Allen–Cahn
  equation with small noise}.
\newblock {\em Probability Theory and Related Fields}, 177(1):257--322, 2020.
\newblock DOI:
  \href{https://doi.org/10.1007/s00440-019-00945-x}{\texttt{10.1007/s00440-019-00945-x}}.

\bibitem[VFG22]{devecchi_fresta_gubinelli_22_fermionic}
Francesco C.~De Vecchi, Luca Fresta, and Massimiliano Gubinelli.
\newblock {A stochastic analysis of subcritical Euclidean fermionic field
  theories}, 2022.
\newblock preprint, DOI: \href{https://doi.org/10.48550/arXiv.2210.15047
  }{\texttt{10.48550/arXiv.2210.15047}}.

\bibitem[Zei95]{zeidler_applied_fa}
Eberhard Zeidler.
\newblock {\em Applied Functional Analysis. Main Principles and Their
  Applications}, volume 109 of {\em Applied Mathematical Sciences}.
\newblock Springer, 1995.
\newblock DOI:
  \href{https://doi.org/10.1007/978-1-4612-0821-1}{\texttt{10.1007/978-1-4612-0821-1}}.

\end{thebibliography}
		
\subsection*{Acknowledgements}
Both authors thank Ilya Chevyrev for his hospitality and financial support during a research visit at the University of Edinburgh in April~2023 and for suggesting a strategy that lead to a significant simplification in the proof of Theorem~\ref{th:V_equal_2S}.
Both authors are grateful to Rui Bai who pointed out a gap in an earlier version of the proof of Proposition~\ref{prop:bound_prob_H} which has now been closed.
TK thanks Nils Berglund for many interesting discussions on the topic of this work and related problems.
AM additionally thanks the University of Warwick for financial and housing support during his visit in August~2023 and support from the DFG research unit FOR2402.
TK gratefully acknowledges financial support via Giuseppe Cannizzaro’s UKRI FL Fellowship \enquote{Large-scale universal behaviour of Random Interfaces and Stochastic Operators} MR/W008246/1.
		
		\vspace{3em}
		\noindent
		\begin{minipage}{.5\textwidth}
			\small
			\textbf{Tom Klose} \\
			University of Oxford \\ 
			Mathematical Institute \\
			Woodstock Road \\
			Oxford, OX2 6GG, United Kingdom \\
			{\it E-mail address:}
			{\tt tom.klose@maths.ox.ac.uk}
		\end{minipage}%
		\begin{minipage}{.5\textwidth}
			\small
			\textbf{Avi Mayorcas} \\
			University of Bath\\
			Department of Mathematical Sciences\\
			Bath, BA2 7AY, United Kingdom\\
			{\it E-mail address:}
			{\tt am2735@bath.ac.uk}
		\end{minipage}
															
\end{document}